\documentclass{memo-l}

\usepackage{amsmath, euscript}
\usepackage{cases}
\usepackage{mathrsfs}
\usepackage{bbm}
\usepackage{amssymb}
\usepackage{txfonts}
\usepackage{amscd}
\usepackage{amsfonts,latexsym,amsmath,amsthm,amsxtra,mathdots,amssymb,latexsym,mathabx}
\usepackage[all,cmtip]{xy}
\usepackage{color}
\usepackage{multicol}
\usepackage{hyperref}
\usepackage{tikz}
\usepackage[toc,page]{appendix}
\usepackage{chngcntr}

\allowdisplaybreaks

\newsavebox{\foobox}
\newcommand{\slantbox}[2][.3]
{%
	\mbox
	{%
		\sbox{\foobox}{#2}%
		\hskip\wd\foobox
		\pdfsave
		\pdfsetmatrix{1 0 #1 1}%
		\llap{\usebox{\foobox}}%
		\pdfrestore
	}%
}

\def\mod{\mathrm{mod}\ }

 \newcommand{\BB}{{\mathbb {B}}}
\newcommand{\BC}{{\mathbb {C}}} \newcommand{\BD}{{\mathbb {D}}}
 \newcommand{\BF}{{\mathbb {F}}}
 \newcommand{\BH}{{\mathbb {H}}}
 \newcommand{\BJ}{{\slantbox{$\mathbb{J}$}}}
 \newcommand{\BL}{{\mathbb {L}}}
 \newcommand{\BN}{{\mathbb {N}}}
 
\newcommand{\BQ}{{\mathbb {Q}}} \newcommand{\BR}{{\mathbb {R}}}
\newcommand{\BS}{{\mathbb {S}}} 
\newcommand{\BU}{{\mathbb {U}}} 
 \newcommand{\BX}{{\mathbb {X}}}
 \newcommand{\BZ}{{\mathbb {Z}}}

\newcommand{\fr}{{\mathfrak{r}}} \newcommand{\fC}{{\mathfrak{C}}}
\newcommand{\GL}{{\mathrm{GL}}}
\renewcommand{\Im}{{\mathfrak{Im}\,}}

\newcommand{\fc}{{\mathfrak{c}}} 
\newcommand{\PGL}{{\mathrm{PGL}}} 
\renewcommand{\Re}{{\mathfrak{Re}\,}}

\newcommand{\PSL}{{\mathrm{PSL}}}

\def\frc{\mathfrak{c}}

\def\sl{\mathfrak{sl}}
\def\-{^{-1}}

\def\Ssis{\mathscr S_{\mathrm {sis}}}
\def\Ssiss{\mathscr T_{\mathrm {sis}}}

\def\Msis{\mathscr M_{\mathrm {sis}}}

\def\Nsis{\mathscr N_{\mathrm {sis}}}

\def\SS{\mathscr S }

\def\Hrd {\mathscr H_{\mathrm {rd}}}

\def\hh{{ \text{\usefont{T1}{pzc}{m}{it}{h}}} }

\def\hld{{ \text{\usefont{T1}{pzc}{m}{it}{h}}} _{(\umu, \udelta)}}
\def\hmld{{ \text{\usefont{T1}{pzc}{m}{it}{h}}} _{(-\umu, \udelta)}}
\def\hmum{{ \text{\usefont{T1}{pzc}{m}{it}{h}}} _{(\umu, \um)}}
\def\hmmum{{ \text{\usefont{T1}{pzc}{m}{it}{h}}} _{(-\umu, \um)}}

\def\Hsl{\EuScript H_{(\usigma, \ulambda)}}
\def\Hld{\EuScript H_{(\umu, \udelta)}}
\def\Hmum{\EuScript H_{(\umu, \um)}}
\def\Hmld{\EuScript H_{(-\umu, \udelta)}}
\def\Hmmum{\EuScript H_{(-\umu, -\um)}}

\def\BRx{\mathbb R^\times}
\def\BCx{\mathbb C^\times}
\def\BFx{\mathbb F^\times}

\def\ux{\boldsymbol x}
\def\uy{\boldsymbol y}
\def\udelta{\boldsymbol {\delta}}
\def\ulambda{\boldsymbol \lambda}
\def\uLambda{\boldsymbol \varLambda}
\def\usigma{\boldsymbol \varsigma}
\def\unu{\boldsymbol \nu}
\def\umu{\boldsymbol \mu}

\def\urho{\boldsymbol \varrho}
\def\ue{\boldsymbol e}
\def\um{\boldsymbol m}
\def\utheta{\boldsymbol \theta}

\def\-{^{-1}}
\def\ut{\boldsymbol t}
\def\ualpha{\boldsymbol \alpha}
\def\ukappa{\boldsymbol \kappa}
\def\lp {\left (}
\def\rp {\right )}
\def\lpp {\left \{ }
\def\rpp {\right \} }
\def\EC{\EuScript C}
\def\EJ{\EuScript J}
\def\EM{\EuScript M}
\def\EF{\EuScript F}
\def\ET{\EuScript T}
\def\ES{\EuScript S}
\def\EH{\EuScript H}
\def\Voronoi{Vorono\" \i \hskip 3 pt}

\def\BZT{\BZ/2 \BZ}

\def\uk{\boldsymbol k}

    \newcommand{\ra}{\rightarrow}

    \newcommand{\ds}{\displaystyle}
    \newcommand{\sstyle}{\scriptstyle}

    \newcommand{\SL}{{\mathrm{SL}}}

    \newcommand{\sgn}{{\mathrm{sgn}}} \newcommand{\Tr}{{\mathrm{Tr}}}\newcommand{\delete}[1]{}
    
    \newcommand{\ccirc}{\raisebox{15 \depth}{${\sstyle \mathrm o}$}}

\makeatletter
\newsavebox\myboxA
\newsavebox\myboxB
\newlength\mylenA

\newcommand*\xoverline[2][0.75]{%
	\sbox{\myboxA}{$\m@th#2$}%
	\setbox\myboxB\null
	\ht\myboxB=\ht\myboxA%
	\dp\myboxB=\dp\myboxA%
	\wd\myboxB=#1\wd\myboxA
	\sbox\myboxB{$\m@th\overline{\copy\myboxB}$}
	\setlength\mylenA{\the\wd\myboxA}
	\addtolength\mylenA{-\the\wd\myboxB}%
	\ifdim\wd\myboxB<\wd\myboxA%
	\rlap{\hskip 0.5\mylenA\usebox\myboxB}{\usebox\myboxA}%
	\else
	\hskip -0.5\mylenA\rlap{\usebox\myboxA}{\hskip 0.5\mylenA\usebox\myboxB}%
	\fi}
\makeatother

\newtheorem{theorem}{Theorem}[section]
\newtheorem{lem}[theorem]{Lemma}
\newtheorem{cor}[theorem]{Corollary}
\newtheorem{thm}[theorem]{Theorem}
\newtheorem{prop}[theorem]{Proposition}

\theoremstyle{definition}
\newtheorem{defn}[theorem]{Definition}
\newtheorem{example}[theorem]{Example}
\newtheorem{term}[theorem]{Terminology}

\newtheorem*{acknowledgement}{Acknowledgements}

\newtheorem{observation}[theorem]{Observation}
\newtheorem{notation}[theorem]{Notation}

\theoremstyle{remark}
\newtheorem{rem}[theorem]{Remark}

\numberwithin{section}{chapter}
\numberwithin{equation}{section}

\renewcommand*\thesection{\arabic{section}}

\counterwithout{section}{chapter}

\renewcommand{\qedsymbol}{Q.E.D.}

\makeindex

\begin{document}

\frontmatter

\title{Theory of Fundamental Bessel Functions of High Rank}


\author{Zhi Qi}
\address{Department of Mathematics\\ The Ohio State University\\100 Math Tower\\231 West 18th Avenue\\Columbus, OH 43210\\USA}
\email{qi.91@buckeyemail.osu.edu}
\thanks{}


\date{}



\dedicatory{To Hui}

\maketitle

\tableofcontents

\begin{abstract}
	In this article, we shall study fundamental Bessel functions for $\GL_n(\BF)$ arising
	from the \Voronoi summation formula for any rank $n$ and field $\BF = \BR$ or $\BC$, with focus on developing their analytic and asymptotic theory. 
	The main implements and subjects of our study of  fundamental Bessel functions are their formal integral representations and Bessel differential equations. We shall prove the asymptotic formulae for  fundamental Bessel functions and explicit connection formulae for the Bessel differential equations.
\end{abstract}

%

\chapter*{Introduction}

\addtocontents{toc}{\protect\setcounter{tocdepth}{0}}

\section*{Number Theoretic Motivations}

The oscillatory exponential function $e (x) = e^{2 \pi i x}$ arises as the integral kernel of the Fourier transform in Poisson's summation formula, which has played a very important role in analysis and number theory. The Poisson summation asserts the identity
\begin{equation*}
\sum_{n = - \infty}^{\infty} \upsilon (n) = \sum_{n = - \infty}^{\infty} \widehat \upsilon (n),
\end{equation*}
for at least all Schwartz functions $\upsilon$, with $\widehat \upsilon$ the Fourier transform of $\upsilon$,
\begin{equation*}
\widehat \upsilon  (x) = \int_{- \infty}^{\infty} \upsilon (y) e (- x y) d y.  
\end{equation*}
Riemann's proof of the functional equation of his zeta function relies on the Poisson summation formula. Also, Hecke used the higher {dimensional} generalization of the formula for the zeta function associated with an arbitrary number field. Furthermore, Tate's thesis reinterprets these using the Poisson summation formula for the adele ring.

\vskip 5 pt

(Classical) Bessel functions occur in Vorono\"i's summation formula as well as Petesson's and Kuznetsov's trace formula for $\GL_2 (\BR)$. These formulae have become   fundamental analytic tools for a number of deep results in analytic number theory, most notably for the subconvexity problem for automorphic $L$-functions. A version of the \Voronoi summation formula, which is not in the most general form, reads as follows (see for example \cite[Theorem A.4]{KMV} and \cite[Proposition 1]{Harcos-Michel}),
\begin{align*}
\sum_{n=1}^{\infty} \sqrt n\, \rho^+_F (n) e \lp \frac {a n} c \rp \upsilon (n) = \frac {1} {c} \sum_{\pm} \sum_{ n=1 }^{\infty} \sqrt n \, \rho^{\pm}_F (n)  e \lp \mp \frac {\overline a n} c \rp \Upsilon \lp \pm \frac n {c^2} \rp.
\end{align*}
In this formula, $a$, $\overline a$ and $c$ are integers such that $(a, c) = 1$ and $a \overline a \equiv 1 (\mod c)$, $\rho^{\pm}_F (n)$ are certain normalized Fourier coefficients of a holomprhic or Maa\ss  \hskip 3 pt cusp form for   $\SL_2 (\BZ)$, $\upsilon$ is a smooth weight function compactly supported on $ (0, \infty)$ and $\Upsilon$ is the Hankel transform of $\upsilon$,
\begin{equation*} 
\Upsilon (x) = \int_{0}^{\infty} \upsilon (y) J_F (xy) dy, \hskip 10 pt   x \neq 0,
\end{equation*}
where, if $F$ is a Maa\ss  \hskip 3 pt form of eigenvalue $\frac 1 4 + t^2$ and weight $k$,
\begin{align}
\nonumber J _F (x) & = - \frac \pi {\cosh (\pi t)} \left( Y_{ 2i t} (4 \pi \sqrt x) + Y_{- 2i t} (4 \pi \sqrt x) \right) \\
\nonumber & = \frac { \pi i} {\sinh (\pi t)} \left( J_{ 2i t} (4 \pi \sqrt x) - J_{- 2i t} (4 \pi \sqrt x) \right) \\
\label{1eq: n=2, Maass form, Bessel functions, k even}  & = \pi i \lp e^{-\pi t } H^{(1)}_{ 2i t} (4\pi \sqrt x) - e^{ \pi t} H^{(2)}_{ 2i t} (4 \pi \sqrt x) \rp, \\ 
\nonumber J_F (- x) & = 4 \cosh (\pi t) K_{2 i t} (4 \pi \sqrt x) \\
\nonumber & = \frac { \pi i} {\sinh (\pi t)} \left( I_{ 2i t} (4 \pi \sqrt x) - I_{- 2i t} (4 \pi \sqrt x) \right), \hskip 10 pt x > 0,
\end{align}
for $k$ even,
\begin{align}
\nonumber J _F (x) & = - \frac \pi {\sinh (\pi t)} \left( Y_{ 2i t} (4 \pi \sqrt x) - Y_{- 2i t} (4 \pi \sqrt x) \right) \\
\nonumber & =  \frac { \pi i} {\cosh (\pi t)} \left( J_{ 2i t} (4 \pi \sqrt x) + J_{- 2i t} (4 \pi \sqrt x) \right) \\
\label{1eq: n=2, Maass form, Bessel functions, k odd} & = 
\pi i \lp e^{-\pi t } H^{(1)}_{ 2i t} (4\pi \sqrt x) + e^{ \pi t} H^{(2)}_{ 2i t} (4 \pi \sqrt x) \rp \\
\nonumber  J_F (- x) &  = 4 \sinh (\pi t) K_{2 i t} (4 \pi \sqrt x) \\
\nonumber & = \frac { \pi i} {\cosh (\pi t)} \left(I_{ 2i t} (4 \pi \sqrt x) - I_{- 2i t} (4 \pi \sqrt x) \right), \hskip 10 pt x > 0,
\end{align}
for $k$ odd, 
and if $F$ is a holomorphic cusp form of weight $k$,
\begin{equation}\label{1eq: n=2, holomorphic form, Bessel functions}
J_F (x) = 2 \pi i^k J_{k-1} (4 \pi \sqrt x), \hskip 10 pt J_F (-x) = 0, \hskip 10 pt x > 0.
\end{equation}
Thus the integral kernel $J _F$ has an expression in Bessel functions, where, in standard notation, $J_{\nu}$, $Y_\nu$, $H^{(1)}_\nu$, $H^{(2)}_\nu$, $I_\nu$ and $K_\nu$ are the various Bessel functions (see for instance \cite{Watson}). 
Here, 
the following connection formulae {\rm (\cite[3.61 (3, 4, 5, 6), 3.7 (6)]{Watson})} have been applied in  \eqref{1eq: n=2, Maass form, Bessel functions, k even} and \eqref{1eq: n=2, Maass form, Bessel functions, k odd},
\begin{align}
& Y_\nu (x) = \frac {J_\nu (x) \cos (\pi \nu) - J_{- \nu} (x)}{\sin ( \pi \nu)}, \hskip 10 pt Y_{-\nu} (x) = \frac {J_\nu (x) - J_{- \nu} (x) \cos (\pi \nu)}{\sin ( \pi \nu)},\\
& \label{2eq: connection formulae}
H^{(1)}_\nu (x) = \frac {J_{-\nu} (x) - e^{- \pi i \nu} J_\nu (x) }{i \sin ( \pi \nu)}, \hskip 14 pt H^{(2)}_\nu (x) = \frac { e^{\pi i \nu} J_\nu (x) - J_{-\nu} (x)}{i \sin ( \pi \nu)},\\
& \label{1eq: connection formula K} K_{\nu} (x) =  \frac {\pi \lp I_{-\nu} (x) - I_\nu (x) \rp} {2 \sin (\pi \nu)}.
\end{align}
The theory of Bessel functions has been extensively studied since the early 19th century, and we refer the reader to Watson's beautiful book \cite{Watson} for an encyclopedic treatment. 

\vskip 5 pt

The \Voronoi summation formula for $\GL_n( \BZ) $ with $n \geqslant 3$ is formulated in the work of Miller and Schmid \cite{Miller-Schmid-2006, Miller-Schmid-2009} (see also \cite{Goldfeld-Li-1, Goldfeld-Li-2}),  in which Hankel transforms are the Archimedean ingredient that relates the weight functions on two sides of the identity.  The notion of {\it automorphic distributions} 
is used for their proof of this formula, and is also used to derive the analytic continuation and the functional equation of the $L$-function of a cuspidal $\GL_n(\BZ)$-automorphic representation of $\GL_n (\BR)$.
As the foundation of automorphic distributions, the harmonic analysis over $\BR$ is studied in \cite{Miller-Schmid-2006} from the viewpoint of gamma factors and signed Mellin transforms.
As explained in \cite{Miller-Schmid-2004-1}, the cases $n = 1, 2$ can also be incorporated into their framework.

More recently, using the global theory of $\GL_n \times \GL_1$-Rankin-Selberg $L$-functions, Inchino and Templier \cite{Ichino-Templier} extended Miller and Schmid's work and proved the \Voronoi summation formula for any irreducible cuspidal automorphic representation of $\GL_n$ over an arbitrary number field for $n \geq 2$. 
According to \cite{Ichino-Templier}, the associated Hankel transform over an Archimedean local field is obtained from the corresponding local functional equations for $\GL_n \times \GL_1$-Rankin-Selberg zeta integrals over the field. 

\vskip 5 pt

In order to motivate our study, let us give some detailed descriptions of Hankel transforms of any rank over the real numbers in \cite{Miller-Schmid-2006, Miller-Schmid-2009}.

Suppose that $(\ulambda, \udelta) = (\lambda_1, ..., \lambda_n, \delta_1, ..., \delta_n) \in \BC^n \times (\BZ/2 \BZ)^n$ is a certain parameter of a cuspidal $\GL_n(\BZ)$-automorphic representation of $\GL_n (\BR)$. Miller and Schmid  give  two expressions for the associated Hankel transform.

\vskip 5 pt 

The first expression of the Hankel transform  associated with $(\ulambda, \udelta)$ is based on  gamma factors and signed Mellin transforms as follows.

Let $\mathscr S (\BR )$ denote the space of Schwartz functions on $\BR $. For $\lambda \in \BC$, $j \in \BN = \{ 0, 1, 2, ... \} $ and $\eta \in \BZ/2\BZ= \{ 0, 1 \}$, let $\upsilon$ be a smooth function on $\BR^\times = \BR \smallsetminus \{ 0 \}$ such that $\sgn (x)^{\eta }  |x|^{ \lambda } \lp \log |x| \rp^{-j} \upsilon (x) \in \mathscr S (\BR)$.
For $\delta \in \BZ/ 2\BZ$, the \textit{signed Mellin transform $\EuScript M_{\delta} \upsilon$ with order $\delta$ of  $\upsilon$} is defined by
\begin{equation*} 
\EuScript M_{\delta} \upsilon (s) = \int_{\BR^\times} \upsilon (x) \sgn (x)^\delta |x|^{s } d^\times x.
\end{equation*}
Here $d^\times x = |x|\- d x$ is the standard multiplicative Haar measure on $\BR^\times$. The Mellin inversion formula is
\begin{equation*}
\upsilon (x) = \sum_{\delta \in \BZ/ 2\BZ } \frac { \sgn (x)^\delta} {4 \pi i} \int_{(\sigma)} \EuScript M_\delta \upsilon (s) |x|^{-s} d s,  \hskip 10 pt \sigma >  \Re \lambda,
\end{equation*}
where the contour of integration $ (\sigma)$ is the vertical line   from $\sigma - i \infty$ to $\sigma + i \infty$.

Let $\mathscr S (\BR ^\times)$ denote the space of smooth functions on $\BR^\times $ whose derivatives are rapidly decreasing at both zero and infinity.
We associate with $\upsilon \in \mathscr S (\BR ^\times)$ 
a function $\Upsilon$ on $\BR ^\times$ satisfying the following two identities
\begin{equation*}
\EuScript M_\delta \Upsilon (s ) =  \left( \prod_{l      = 1}^n G_{\delta_{l     } + \delta} (s - \lambda_l     ) \right) \EuScript M_\delta \upsilon ( 1 - s),  \hskip 10 pt \delta \in \BZ/2 \BZ,
\end{equation*}
where $G_{\delta} (s)$ denotes the gamma factor
\begin{equation*} 
G_\delta (s) = i^\delta \pi^{ \frac 1 2 - s} \frac {\Gamma \lp \frac 1 2 ( {s + \delta} ) \rp} {\Gamma \lp \frac 1 2 ( {1 - s + \delta} ) \rp} = 
\left\{ \begin{split}
& 2(2 \pi)^{-s} \Gamma (s) \cos \left(\frac {\pi s} 2 \right), \hskip 10 pt \text { if } \delta = 0,\\
& 2 i (2 \pi)^{-s} \Gamma (s) \sin  \left(\frac {\pi s} 2 \right), \hskip 9 pt \text { if } \delta = 1.
\end{split} \right.
\end{equation*}
$\Upsilon$ is called the \textit{Hankel transform of index $(\ulambda, \udelta)$ of $\upsilon$}.\footnote{Note that if $\upsilon$ is the $f$ in \cite{Miller-Schmid-2009} then $|x| \Upsilon((-)^n x)$ is their $F(x)$.}
According to \cite[\S 6]{Miller-Schmid-2006}, $\Upsilon$ is smooth on $\BR^\times$ and decays rapidly at infinity, along with all its derivatives. At the origin,  $\Upsilon$ has singularities of some very particular type. Indeed,  $\Upsilon (x) \in \sum_{l      = 1}^n \sgn (x)^{\delta_l     } |x|^{\lambda_l     }  \mathscr S (\BR)$ when no two components of $\ulambda$ differ by an integer, and in the nongeneric case powers of $\log |x|$ will be included. 

By the Mellin inversion,
\begin{equation}\label{1eq: Psi defined by Gamma functions}
\Upsilon (x) 
= \sum_{\delta \in \BZ/ 2\BZ } \frac { \sgn (x)^\delta} {4 \pi i} \int_{(\sigma)} \left( \prod_{l      = 1}^n G_{\delta_{l     } + \delta} (s - \lambda_l      ) \right) \EuScript M_\delta \upsilon (1 - s) |x|^{ - s} d s,
\end{equation}
for $\sigma > \max  \left\{ \Re  \lambda_l      \right \} $.\\

An alternative description of $\Upsilon$ is given by the \textit{Fourier type transform}, in symbolic notion, as follows
\begin{equation}\label{1eq: generalized Fourier transform}
\begin{split}
\Upsilon (x) =  \frac 1 { |x|} \int_{\BR^{\times \, n}} \upsilon \left( \frac {x_1 ... x_n} { x} \right) \left( \prod_{l      = 1}^{n} \left( \sgn (x_l     )^{\delta_l     } |x_l     |^{- \lambda_l     } e ( x_l     ) \right) \right) d x_n d x_{n-1} ... d x_1.
\end{split}
\end{equation}
The integral  converges when performed as {\it iterated integral} in the order $d x_n d x_{n-1} ... d x_1$, starting from $x_n$, then $x_{n-1}$, ..., and finally $x_1$, provided $\Re  \lambda_1 >  ... > \Re  \lambda_{n-1} > \Re  \lambda_n$, and it has meaning for arbitrary values of $\ulambda \in \BC^n$ by analytic continuation. 

\vskip 5 pt

When applying the \Voronoi summation formula for $\GL_3(\BZ)$ (currently, there is no application for rank $n \geqslant 4$), one is normally reduced to the asymptotic behaviour of Hankel transforms. According to \cite{Miller-Schmid-2009}, though \textit{less suggestive than} (\ref{1eq: generalized Fourier transform}), the expression (\ref{1eq: Psi defined by Gamma functions}) of Hankel transforms is \textit{more useful in applications}. 
Indeed, {all} the applications of the \Voronoi summation formula in analytic number theory so far are based on (\ref{1eq: Psi defined by Gamma functions}) with exclusive use of Stirling's asymptotic formula of the Gamma function (see \cite{Ivic,Miller-Wilton,XLi,Blomer} and Appendix \ref{appendix: asymptotic}). Applications include the estimates for additively twisted sums for $\GL_3$ (\cite{Miller-Wilton}) and for  shifted convolution sums for $\GL_3 \times \GL_2$ (\cite{Munshi-Convolution}), the subconvexity and non-vanishing of central values of automorphic $L$-functions for $\GL_3$ and $\GL_3 \times \GL_2$ (\cite{XLi,XLi2011,Blomer}).  On the other hand, there is no occurrence of the Fourier type integral transform \eqref{1eq: generalized Fourier transform} in the literature other than Miller and Schmid's foundational work. It will however be shown in this article that the expression \eqref{1eq: generalized Fourier transform} should {\it not}  be  of only aesthetic interest.

\vskip 5 pt

\section*{Outline of Article}

The main purpose of this article is to develop the analytic theory of \textit{fundamental Bessel functions} and \textit{Bessel kernels}\footnote{The Bessel functions and Bessel kernels studied here are called \textit{fundamental} in order to be distinguished from the Bessel functions for $\GL_n (\BF)$ arising in the Kuznetsov trace formula, with $\BF = \BR$ or $\BC$. 
	The latter should be regarded as the foundation of harmonic analysis on $\GL_n (\BF)$. Some evidences show that fundamental Bessel functions are actually the building blocks of the Bessel functions for $\GL_n (\BR)$. See \cite[\S 3.2]{Qi-Thesis} for the case of $\GL_3 (\BF)$.
	
	Throughout this article, we shall drop the adjective {\it fundamental} for brevity. Moreover, the usual Bessel functions will be referred to as classical Bessel functions.}, with particular focus on their asymptotic formulae. Other subjects of this paper include analytic investigations of Hankel transforms and representation theoretic interpretations of Bessel kernels.

\vskip 5 pt

Enlightened by the work of Inchino and Templier \cite{Ichino-Templier},  we shall first study Hankel transforms over both Archimedean local fields $\BR$ and $\BC$. With some refinements, our treatment of Hankel transforms over $\BR$ is essentially the same as that in \cite{Miller-Schmid-2004}. The author however has resisted the temptation of establishing here the distribution theory on Hankel transforms over $\BC$ from the perspective of  \cite{Miller-Schmid-2004}, mainly because it would take us too far afield from the analytic theory of Bessel functions. It is very likely that this will lead to the theory of automorphic distributions on $\GL_n (\BC)$ with respect to congruence subgroups, as well as the \Voronoi summation formula for cuspidal automorphic representations of $\GL_n (\BC)$. Such a formula  in this generality is already covered by \cite{Ichino-Templier}, but this approach would still be of its own interest.

In Chapter \ref{chap: Hankel transforms}, it will be shown that there are two expressions for Hankel transforms over either $\BR$ or $\BC$ and they admit integral kernels in the sense that
\begin{equation*} 
\Upsilon (x) = \int_{\BF^\times} \upsilon (y) J (xy ) d y,
\end{equation*} 
with $\BF =\BR$ or $\BC$. Thus the asymptotics of Hankel transforms may be deduced from those of their Bessel kernels. Furthermore, we shall be interested in a type of Hankel transforms over $\BR_+$ which is more fundamental. We now give a brief introduction of such Hankel transforms and their integral kernels.

\vskip 5 pt

Given $(\ulambda, \usigma) \in \BC^n \times \{+, -\}^n$, with every Schwartz function $\upsilon$ on $\BR_+$, which decays rapidly at both zero and infinity, we   associate a function $\Upsilon$ on $\BR_+$ satisfying the identity
	\begin{equation*}
	\EM \Upsilon (s ) = \lp \prod_{ l = 1}^n G  (s - \lambda_l, \varsigma_l) \rp \EM \upsilon ( 1 - s),
	\end{equation*}
where $\EM$ is the usual Mellin transform over $\BR_+$ and $G  (s, \pm)$ is the gamma factor defined by
\begin{equation*}
G  (s, \pm) = \Gamma (s) e\lp \pm \frac s 4\rp.
\end{equation*}
$\Upsilon$ is called the \textit{Hankel transform of index $(\ulambda, \usigma)$ of $\upsilon$}. The Mellin inversion yields the first expression of this Hankel transform,
\begin{equation}\label{1eq: Psi defined by Gamma functions, R+}
\Upsilon (x) 
=  \frac {1} {2 \pi i} \int_{(\sigma)} \left( \prod_{l      = 1}^n G  (s - \lambda_l      , \varsigma_l ) \right) \EuScript M \upsilon (1 - s) x^{ - s} d s, \hskip 10 pt \sigma > \max  \left\{ \Re  \lambda_l      \right \}.
\end{equation}
On the other hand, under the condition $\Re  \lambda_1 >  ... > \Re  \lambda_{n-1} > \Re  \lambda_n$, the second expression is given by
	\begin{equation} \label{0eq: Fourier type integral, R+}
	 \Upsilon (x) = \frac 1 {x} \int_{\BR_+^{  n}} \upsilon \lp \frac {x_1 ... x_n} x \rp \lp \prod_{l      = 1}^{n }    x_{l     }^{ - \lambda_{l     }} e^{\varsigma_l i x_l}  \rp dx_n ... d x_1.
	\end{equation}
We stress that these two expressions, \eqref{1eq: Psi defined by Gamma functions, R+} and \eqref{0eq: Fourier type integral, R+}, of Hankel transforms over $\BR_+$ are intimately close to those of Hankel transforms over $\BR$, namely (\ref{1eq: Psi defined by Gamma functions}) and (\ref{1eq: generalized Fourier transform}), respectively.
We also have the kernel formula
\begin{equation*}
\Upsilon (x) = \int_{\BR_+} \upsilon (y) J (xy; \usigma, \ulambda  ) d y,
\end{equation*}
in which the kernel $J (x; \usigma, \ulambda ) $ is called the \textit{Bessel function of index $(\usigma, \ulambda)$}.

\vskip 5 pt


The first expression (\ref{1eq: Psi defined by Gamma functions, R+}) of the Hankel transform of index $(\ulambda, \usigma)$ yields a formula of the Bessel function $J (x; \usigma, \ulambda ) $ as a certain Mellin-Barnes type integral  involving the Gamma function (see \eqref{3eq: definition of J (x; sigma)}).  Moreover, the analytic continuation of $J ( x; \usigma, \ulambda  )$ from $ \BR_+ $ onto the Riemann surface $\BU$, the universal cover of $\BC \smallsetminus \{ 0 \}$, can be realized as a Barnes type integral via modifying the integral contour of a  Mellin-Barnes type integral (see \eqref{3eq: Barnes contour integral 1}). 

As alluded to above, the asymptotics of $J (x; \usigma, \ulambda ) $ may be obtained from applying Stirling's asymptotic formula  to the Mellin-Barnes type integral  (see Appendix \ref{appendix: asymptotic}). This is the only known method   in the literature. There are however two limitations of this method. Firstly, it is {\it not} applicable to a Barnes type integral and therefore the domain of the asymptotic expansion can not be extended from $\BR_+$.  Secondly, it is {\it only}  applicable when $\ulambda$ is regarded  as fixed constant and hence the dependence on $\ulambda$ of the error term can not be clarified.

The second expression \eqref{0eq: Fourier type integral, R+}  of the Hankel transform of index $(\ulambda, \usigma)$ leads to a representation of the Bessel function $J ( x; \usigma, \ulambda  )$ as  formal integral
\begin{equation}\label{1eq: formal integral J nu (x; sigma)}
J_{\unu} (x; \usigma) = \int_{\BR_+^{n-1}} \left(\prod_{l      = 1}^{n-1} t_l     ^{ \nu_l      - 1} \right) e^{i x \left(\varsigma_n t_1 ... t_{n-1} + \sum_{l      = 1}^{n-1} \varsigma_l      t_l     \- \right)} d t_{n-1} ... dt_1,
\end{equation}
with $\nu_l = \lambda_l - \lambda_{n}$, provided that the index $\ulambda$ satisfies the condition $\sum_{l=1}^n \lambda_l = 0$. Accordingly, we define the complex hyperplane  $\BL^{n-1} = \left \{ \ulambda \in \BC^n :  \sum_{l      = 1}^n \lambda_{l     } = 0 \right\} $.

The novelty of this article is an approach to Bessel functions and Bessel kernels starting from the second expression (\ref{0eq: Fourier type integral, R+}) of  Hankel transforms over $\BR_+$. This approach is more accessible from the perspective of harmonic analysis,  at least in symbolic notions. Once we can make sense of the symbolic notions in (\ref{1eq: formal integral J nu (x; sigma)}), some well-developed methods from analysis and differential equations may be exploited so that we are able to understand Bessel functions and Bessel kernels to a much greater extent. 


\vskip 5 pt

Chapter \ref{chap: analytic theory} is devoted to the investigations of the formal integrals $J_{\unu} (x; \usigma)$ and the Bessel differential equations satisfied by $J (x; \usigma, \ulambda ) $.

\vskip 5 pt

First of all, we must justify   the formal integral  $J_{\unu} (x; \usigma)$ as a representation of  $J (x; \usigma, \ulambda ) $. For this, we partition the formal integral $J_{\unu} (x; \usigma)$ according to some partition of unity on $\BR^{n-1}_+$, and then 
repeatedly apply \textit{two} kinds of partial integration operators on each resulting integral. In this way, $J_{\unu} (x; \usigma)$ can be transformed into a finite sum of \textit{absolutely convergent} multiple integrals. This sum of integrals is regarded as the rigorous definition of  $J_{\unu} (x; \usigma)$. 
Furthermore, it is shown that
\begin{equation*}
J (x; \usigma, \ulambda) = J_{\unu} (x; \usigma),
\end{equation*}
where $ J_{\unu} (x; \usigma)$ on the right is now rigorously understood.

Either adapting  techniques or   applying  results from {the method of stationary phase} due to H\"ormander, we then study the asymptotic behaviour of each oscillatory multiple integral in the rigorous definition of $J_{\unu} (x; \usigma)$, and hence  $J_{\unu} (x; \usigma)$ itself, for large argument. Even in the classical case $n=2$, our method is entirely new, as the coefficients in the asymptotic expansions are formulated in a way that is quite different from what is known in the literature (see \S \ref{sec: remark on the coefficients in the asymptotics}).

When all the components of $\usigma$ are identically $\pm$, we denote $J (x; \usigma, \ulambda)$, respectively  $J_{\unu} (x; \usigma)$, by $H^\pm (x; \ulambda)$, respectively $H_{\unu}^\pm (x )$, and call it an \textit{$H$-Bessel function}\footnote{If a statement or a formula includes $\pm$ or $\mp$, then it should be read with $\pm$ and $\mp$ simultaneously replaced by  either  $+$ and $-$ or $-$ and $+$.}. 
This pair of $H$-Bessel functions will be of paramount significance in our treatment.
It is shown that $H^{\pm} (x; \ulambda) = H_{\unu}^\pm (x )$ admits an analytic continuation from $\BR_+$ onto the half-plane $\BH^{\pm} = \left\{ z \in \BC \smallsetminus \{0\} : 0 \leq \pm \arg z \leq \pi \right\}$. Furthermore, we have, roughly speaking, the following asymptotic on $\BH^{\pm}$,
\begin{equation}\label{1eq: asymptotic expansion 1}
\begin{split}
H^{\pm}  (z; \ulambda) \sim n^{- \frac 1 2}   (\pm 2 \pi i)^{ \frac { n-1} 2} e^{ \pm i n z} z^{ - \frac { n-1} 2} ,  \hskip 10 pt \text{ as } |z| \ra \infty.
\end{split}
\end{equation}

All the other Bessel functions are called \textit{$K$-Bessel functions} and are shown to be Schwartz functions at infinity. 

\vskip 5 pt

In the second part of Chapter  \ref{chap: analytic theory}, we discover and study the ordinary differential equation, namely \textit{Bessel equation}, satisfied by the Bessel function $J (x; \usigma, \ulambda)$.

Given $\ulambda \in \BL^{n-1}$, there are exactly two Bessel equations
\begin{equation}\label{1eq: Bessel equations}
\sum_{j = 1}^{n} V_{n, j} (\ulambda) x^{j} w^{(j)} + \left( V_{n, 0} (\ulambda) - \varsigma (in)^n  x^n \right) w  = 0, \hskip 10 pt \varsigma \in \{+, - \},
\end{equation}
where $V_{n, j} (\ulambda)$ is some explicitly given symmetric polynomial in $\ulambda$ of degree $n-j$. We call $\varsigma$ the \textit{sign} of the Bessel equation (\ref{1eq: Bessel equations}). Then $J (x; \usigma, \ulambda)$ satisfies the Bessel equation of sign $S_n(\usigma) = \prod_{l      = 1}^n \varsigma_{l     }$.

Replacing $x$ by $z$ to stand for complex variable in the Bessel equation (\ref{1eq: Bessel equations}),   the domain is extended from $\BR_+$ to $\BU$. According to the theory of linear ordinary differential equations with analytic coefficients, $J (x; \usigma, \ulambda)$ admits an analytic continuation onto $\BU$.

Firstly, since zero is a regular singularity, the Frobenius method may be exploited to find a solution $J_{l     } (z; \varsigma, \ulambda)$ of (\ref{1eq: Bessel equations}), for each $l      = 1, ..., n$, defined by the following series,
\begin{equation*}
J_{l     } (z; \varsigma, \ulambda) = \sum_{m=0}^\infty \frac { (\varsigma i^n)^m  z^{ n (- \lambda_{l      } + m)} } { \prod_{k = 1}^n \Gamma \lp  { \lambda_{ k } - \lambda_{l     }}  + m + 1 \rp}.
\end{equation*}
$J_{l     } (z; \varsigma, \ulambda)$ will be called  {\it Bessel functions of the first kind}, since
they generalize the Bessel functions $J_{\nu} (z)$ and the modified Bessel functions $I_{\nu} (z)$ of the first kind.

It turns out that each $J (z; \usigma, \ulambda)$ may be expressed in terms of $J_{l     } \left(z; S_n(\usigma), \ulambda\right)$. This leads to the following connection formula
\begin{equation}\label{1eq: connection formula 1}
J(z ; \usigma, \ulambda) = e \lp \pm \frac {\sum_{l      \in L_ \mp (\usigma)} \lambda_{l      }} 2 \rp H^\pm \Big(e^{\pm \pi i \frac { n_{\mp} (\usigma)} n} z; \ulambda \Big),
\end{equation}
where $L_\pm (\usigma) = \{ l      : \varsigma_l      = \pm \}$ and $n_\pm (\usigma) = \left| L_\pm (\usigma) \right|$. Thus the Bessel function $J(z ; \usigma, \ulambda)$ is determined up to a constant by the pair of integers $(n_+ (\usigma), n_- (\usigma) )$, called the \textit{signature} of 
$J(z ; \usigma, \ulambda)$. 

Secondly, infinity is an irregular singularity of rank one. 
Let $\xi$ be an $n$-th root of $\varsigma 1$.   Then there exists a {\it unique} solution $J (z; \ulambda; \xi)$ of the Bessel equation of sign $\varsigma$ satisfying the asymptotic 
\begin{equation}\label{1eq: asymptotic expansion 2}
J (z; \ulambda; \xi) \sim e^{i n \xi z} z^{- \frac { n-1} 2},  \hskip 10 pt \text{ as } |z| \ra \infty,
\end{equation}
 on the sector
\begin{equation*}
\BS_{\xi } = \left\{ z \in \BU : \left| \arg z - \arg ( i \xoverline \xi) \right|  < \frac {\pi} n \right \},
\end{equation*}
or any of its open subsector. Furthermore, the asymptotic in \eqref{1eq: asymptotic expansion 2} may be extended onto a wider sector,
\begin{equation*}
\BS'_{\xi } = \left\{ z \in \BU : \left| \arg z - \arg ( i \xoverline \xi) \right|  < \pi + \frac {\pi} n  \right \}.
\end{equation*}
We remark  that it has been successful in attaining an optimal error estimate with respect to the index $\ulambda$ in the asymptotic formula. 

For a $2n$-th root of unity $\xi$, $J( z; \ulambda; \xi)$ is called a \textit{Bessel function of the second kind}. We have the following formula that relates all the the Bessel functions of the second kind to either $J (z; \ulambda;   1)$ or $J (z; \ulambda; - 1)$  upon rotating the argument by a $2n$-th root of unity,
\begin{equation}\label{1eq: connection formula 2}
J( z; \ulambda; \xi) = (\pm \xi)^{ \frac {n-1} 2} J(\pm \xi z; \ulambda; \pm 1).
\end{equation}



\vskip 5 pt

The third part of Chapter \ref{chap: analytic theory} is on the connection formulae among Bessel functions $J (z; \usigma, \ulambda)$, $J_{l     } (z; \varsigma, \ulambda) $ and $ J (z; \ulambda; \xi) $.
The most fundamental identity is
\begin{equation}\label{1eq: identity H}
H^{\pm} (z; \ulambda) = n^{- \frac 1 2} (\pm 2 \pi i)^{ \frac { n-1} 2} J (z; \ulambda; \pm 1).
\end{equation}
The theory of formal integrals and the asymptotic theory of Bessel equations are both indispensable in establishing this idenity. Its proof however is simply comparing the asymptotic formulae of $H^{\pm} (z; \ulambda)$ and $J (z; \ulambda; \pm 1)$ in (\ref{1eq: asymptotic expansion 1}) and \eqref{1eq: asymptotic expansion 2} and using the inclusion  $\BS_{\pm 1}  \subset \BH^{\pm} $. We stress that the identity \eqref{1eq: identity H} can {not} be proven by the existing  method for the asymptotic of $ H^{\pm} (x; \ulambda)  $ in the literature. The reason is that Stirling's asymptotic formula can only yield an asymptotic formula of $H^{\pm} (z; \ulambda)$ on $\BR_+$, but $\BR_+ $ does not intersect $\BS_{\pm 1}$! See Remark \ref{rem: Stirling not working}, \ref{rem: appendix only for R+} for more details.  With \eqref{1eq: identity H},
it follows from (\ref{1eq: connection formula 1}) and (\ref{1eq: connection formula 2}) that 
\begin{equation*}
\begin{split}
J (z; \usigma, \ulambda) = \frac { (\mp 2\pi i)^{\frac { n-1} 2} } {\sqrt n} e \lp \pm \frac {(n-1) n_\pm(\usigma)} {4 n} \mp \frac {\sum_{l      \in L_\pm (\usigma)} \lambda_{l      }} 2 \rp J \Big(z; \ulambda; \mp e^{\mp \pi i \frac { n_\pm(\usigma) } n} \Big).
\end{split}
\end{equation*}
This actually implies the exponential decay of $K$-Bessel functions on $\BR_+$. 
Furthermore, the identity \eqref{1eq: identity H} also yields connection formulae between the two kinds of Bessel functions,  in terms of a certain {\it Vandermonde matrix} and its inverse, solving the connection formula problem for Bessel equations (see Corollary \ref{cor: J(z; lambda; xi) and the J-Bessel functions} and \ref{8cor: inverse connection}). Note that \eqref{2eq: connection formulae} and \eqref{1eq: connection formula K} are   rank-two examples of such connection formulae.

\vskip 5 pt

In Chapter \ref{chap: Bessel Kernels}, we return to the study of Bessel kernels over   real and complex numbers using results in Chapter \ref{chap: analytic theory} for Bessel functions. 

Since a real Bessel kernel $J_{(\umu, \udelta)} (\pm x)$, with $\umu \in \BC^n$ and $\udelta \in (\BZT)^n$, is a signed {\it finite} sum of $J \big(2 \pi x^{\frac 1 n}; \usigma, \umu \big)$ (see \eqref{3eq: Bessel kernel, R, connection}), it is now very well understood  from the investigations of  $J  (x; \usigma, \ulambda )$ in Chapter \ref{chap: analytic theory}.

For a complex Bessel kernel $J_{(\umu, \um)} (z)$, with $\umu \in \BC^n$ and $\um \in \BZ^n$, some extra work is required. We first prove two connection formulae between  $J_{(\umu, \um)} (z)$ and the two kinds of Bessel functions $J_l       (  z ; +, \umu \pm \tfrac 1 2 \um  )$ and $J  ( z ; \umu \pm \tfrac 1 2 \um; \xi )$, with $\xi^{\, n} = 1$, arising from the Bessel equation of {positive} sign and order $n$. The connection formulae between the two  kinds of Bessel functions play a crucial role in the proof. Note that in the simplest rank-one case both connection formulae reduce to the identity
$$ e(z + \overline z) = e(z) e(\overline z). $$
Using the asymptotic formula \eqref{1eq: asymptotic expansion 2} for
$J (z; \ulambda; \xi)$ on $\BS'_{\xi } $, we obtain the asymptotic
\begin{equation}\label{0eq: complex asymptotic}
\begin{split}
J _{(\umu, \um)} \left(z^n \right)  \sim \sum_{\xi^{\, n} = 1 } \frac { e \big( n \big( \xi z +   \overline {\xi z } \big) \big) }   { n |z|^{n-1} [\xi z]^{|\um|} } , \hskip 10 pt \text{ as } |z| \ra \infty,
\end{split}
\end{equation}
where $|\um| = \sum_{l=1}^n m_l$ and $[z] = z/|z|$.

\section*{Applications in Analytic Number Theory}

As mentioned earlier, along with the asymptotic formula for Bessel kernels for  $\GL_2 (\BR)$ or $\GL_3 (\BR)$, obtained from the  Stirling asymptotic formula, the \Voronoi summation formula for  $\GL_2  $ or $\GL_3  $ over $\BQ$ has already had many applications in analytic number theory.  

The very first application of the $\GL_3$ \Voronoi summation formula is Miller's bound for the additive twists for $\GL_3 $-automorphic   forms over $\BQ$ \cite{Miller-Wilton}. In a recent paper \cite{Qi-Wilton}, the author has generalized Miller's bound for $\GL_3$ as well as the classical bound of Wilton for $\GL_2$ to an arbitrary number field.



The author certainly hopes that the present work will have more applications in analytic number theory  in the future.

\section*{Formulae for Bessel Functions and Representation Theory}

In the literature of representation theory, the special functions arising in the Kuznetsov trace formula or the  relative trace formula of Jacquet are also called Bessel functions.

In \S \ref{sec: Bessel, GL2(F)}  we shall prove a kernel formula for an infinite dimensional irreducible admissible representation of  $\GL_2 (\BF)$, with $\BF = \BR$ or $\BC$, showing that the action of the long Weyl element
on the Kirillov model is essentially a Hankel transform over $\BF$. It follows the consensus that for $\GL_2 (\BF)$ the Bessel functions occurring in the Kuznetsov trace formula should coincide with those in the \Voronoi summation formula.
This let us prove and generalize the Kuznetsov trace formula for $\PSL_2(\BZ[i])\backslash \PSL_2( \BC)$ in \cite{B-Mo}, in the same way that \cite{CPS} does for the Kuznetsov trace formula for $\PSL_2(\BZ )\backslash \PSL_2( \BR)$ in \cite{Kuznetsov}. See \cite{Qi-Kuz}.

In the work of Baruch and Mao \cite{BaruchMao-Real}, two formulae on the Fourier transform of classical Bessel functions due to Hardy and Weber are interpreted as  the Shimura-Waldspurger correspondence. Moreover, in the framework of the relative trace formula of Jacquet, the theory in \cite{BaruchMao-Real}, along with the corresponding non-Archimedean theory in \cite{BaruchMao-NA}, is used    to produce a Waldspurger-type formula over a {\it totally real} number field \cite{BaruchMao-Global}. With the motivation of generalizing their work  to an arbitrary number field, the author recently proved an analogous formula for   Bessel functions for $\PGL_2 (\BC)$ in \cite{Qi-Sph, Qi-II-G}.

Another important formula for Bessel functions attached to irreducible unitary {\it tempered} representations of $\SL_2 (\BF)$ is the Bessel-Plancherel formula, which is an analogue of Harish-Chandra's Plancherel formula. The formula for $\SL_2 (\BR)$ is a combination of the inversion formulae of Kontorovich, Lebedev and Kuznetsov, and that for $\SL_2 (\BC)$ is discovered in \cite{B-Mo, B-Mo2}. The formula for $\SL_2 (\BR)$ is also interpreted as the Whittaker-Plancherel formula in \cite{BaruchMao-Whittaker}.

In \cite[\S 3.2]{Qi-Thesis}, the author found a formula of Bessel functions for $\GL_3(\BF)$, which are two-variable, in terms of fundamental Bessel kernels. It is formally derived from Rankin-Selberg $\GL_3 \times \GL_2$ local functional equations and the $\GL_2$ Bessel-Plancherel formula. 
To be precise, the fundamental Bessel kernels occurring in the formula are for $\GL_3 \times \GL_2$ and $\GL_2 $, in which the involved representations of $\GL_2$ are all tempered. This gives the first instance of Bessel functions for groups of higher rank. 

\section*{Some Open Problems}

\subsection*{Connections to Other Generalizations of Bessel Functions}
\ 
\vskip 5 pt

Recently, the author came across a paper of Buttcane \cite{Buttcane-Kuz}. He discovered the two-variable Bessel functions in the Kuznetsov trace formula for $\GL_3(\BR)$ as    power series expansions which generalize the  classical Bessel functions of the first kind. It would be interesting  to verify the connection of his Bessel functions for $\GL_3 (\BR)$ to our fundamental Bessel kernels for    $\GL_3 (\BR) \times \GL_2 (\BR)$ and $\GL_2  (\BR)$ via the formula in \cite[\S 3.2]{Qi-Thesis}.

There are two more types of generalized Bessel functions, which are communicated  to  the author by Roman Holowinsky and the referee. 
The first is a class of matrix-valued analogues of classical Bessel and Whittaker functions. See for example \cite[\S 1.2.2]{Terras-2} and the refereneces therein. 
The second generalization was introduced by Everitt and Markett \cite{EW-Bessel}.  
Their Bessel functions satisfy certain higher-order differential equations of even order. It would be interesting to find the connection between these and the work in this paper. 

\subsection*{Generalizations of Other Integral Representations of Bessel Functions}
\ 
\vskip 5 pt

There are various integral representations of classical Bessel functions due to Bessel, Poisson, Hankel, Mehler, Sonine, Schl\"afli, Basset, Barnes and many others (see for example \cite{Watson}). These integral representations are elegant and useful. For example, the classical derivation  after Hankel of the asymptotic formula for Hankel functions makes use of his integral representation (\cite[7.2]{Watson}). 

The question is, can these integral representations   be generalized for our Bessel functions? Indeed, our Mellin-Barnes or Barnes type integrals and formal integrals are generalizations of the Barnes integrals and   variants of the Mehler-Sonine integrals (\cite[6.21 (10, 11), 6.22 (13), 6.5, 6.51]{Watson}), respectively. How about   other integral representations? In particular, is it admissible to  generalize   Hankel's integral representation and apply its generalization to the asymptotic problem along the classical line? \\

\subsection*{Further Study on the Bessel Equations}

The simplest examples of our Bessel equations of order three are
\begin{align}\label{0eq: Bessel equation, n=3}
x^3 \frac {d^3 w} {d x^3} + 3 x^2 \frac {d^2 w} {d x^2} + \lp 1 - 9 \nu^2 \rp x \frac {d  w} {d x } \mp 27 i x^3 w = 0,
\end{align}
which have index $ (  \nu, 0, - \nu) $. We would be interested in the cases $\nu =  t$ and $\nu = i t$,  $t >0 $. These cases are of interest in number theory because  $\lp k-1, 0, - k +1 \rp$ and $( 2it, 0, - 2it)$ are the parameters of the symmetric square lifts of a holomorphic cusp form of weight $k$ and a Maa\ss \hskip 3pt   form of eigenvalue $\frac 1 4 + t^2$ for $\SL_2 (\BZ)$  respectively. 

While there are many problems in the theory of special functions and differential equations, we propose here the asymptotics problem of   the Bessel functions for \eqref{0eq: Bessel equation, n=3} when $ |\nu| =t$  is large. For classical Bessel functions, some basic results on this problem may be found in \cite[Chapter 8]{Watson}. Further, by methods in differential equations, the work of Olver \cite{Olver-Bessel,Olver-1} also contributes some deep results, including certain {\it uniform}  asymptotic expansions which involve exponential  or Airy function for the Bessel or modified Bessel functions; see also \cite{Olver}. It would certainly be interesting to see what special functions will replace   the exponential and Airy function in our asymptotic problem for \eqref{0eq: Bessel equation, n=3}. \\

\begin{acknowledgement}
	The material in this article forms the first two chapters of the author's Ph.D. thesis \cite{Qi-Thesis}.
	The author is grateful to his advisor, Roman Holowinsky, who brought him to the area of analytic number theory, gave him much enlightenment, guidance and encouragement. The author would like to thank James W. Cogdell, Ovidiu Costin, Stephen D. Miller, Zhilin Ye, Pengyu Yang and Cheng Zheng for valuable comments and helpful discussions. The author would like to acknowledge the inspiration from the series of papers of Stephen D. Miller and Wilfried Schmid on the \Voronoi summation formula, without which this paper would never exist. The author is also indebted to the  referee for several suggestions which helped improving the paper.
\end{acknowledgement}

\aufm{Zhi Qi}

\addtocontents{toc}{\protect\setcounter{tocdepth}{1}}


%

\chapter*{Notation}

\addtocontents{toc}{\protect\setcounter{tocdepth}{0}}

\begin{itemize}
	\item[-] Denote $\BN=\{0,1,2...\}$ and $\BN_+ = \{1, 2, 3, ...\}$. 
	\item[-] The group $\BZ/2\BZ$ is usually identified with the two-element set $\{0, 1 \} $. For $m \in \BZ$ define $\delta (m) = m (\mod 2)$.
	\item[-] Denote $\BR_+ = (0, \infty)$, $\overline \BR_+ =[0, \infty)$, $\BR^\times = \BR \smallsetminus \{ 0 \}$ and $\BC^\times = \BC \smallsetminus \{ 0 \}$.
	\item[-] Denote by $\BU \cong \BR _+ \times \BR$ the universal cover of $\BC \smallsetminus \{ 0 \}$. Each element $z \in \BU$ is denoted by $z = x e^{i \omega}  = e^{\log x + i \omega}$, with $(x, \omega) \in \BR _+ \times \BR$.
	\item[-] Define $z^\lambda = e^{\lambda \log z }$ and  $\overline z = e^{- \log z}$ for $z \in \BU, \lambda \in \BC$. Let $1 = e^{ 0}$, $-1 = e^{\pi i}$ and $\pm i = e^{\pm \frac 1 2 \pi i}$.
	\item[-] For $m \in \BZ$ define $\delta (m) \in \BZT$ by $\delta(m) = m (\mod 2)$.
	\item[-] For $z \in \BC$ let $e(z) = e^{2 \pi i z}$.
	\item[-] For $s \in \BC$ and $\alpha \in \BN$, let
	$[s]_{\alpha} =  \textstyle \prod_{\kappa = 0}^{\alpha-1} (s - \alpha)$ and $ (s)_{\alpha} =  \textstyle \prod_{\kappa = 0}^{\alpha-1} (s + \alpha)$  if  $\alpha\geq 1$, and let $[s]_{0} = (s)_0 = 1$.

	\item[-] For $\ulambda = (\lambda_1, ..., \lambda_n) \in \BC^n$  denote $|\ulambda| = \sum_{l      = 1}^n \lambda_{l     }$ (this notation is also used for $(\BZT)^n = \{0, 1\}^n$ and $\BZ^n$ viewed as subsets of $\BC^n$).

	\item[-] Define the hyperplane $\BL^{n-1} = \left \{ \ulambda \in \BC^n : |\ulambda|  = \sum_{l      = 1}^n \lambda_{l     } = 0 \right\} $.
	\item[-] Denote by $\ue^n$ the $n$-tuple $(1, ..., 1)$. 
	
	\item[-] For $l = 1, 2, ..., n$, denote $\ue_l      = (\underbrace {0,..., 0, 1}_{l     }, 0 ..., 0)$ and $\ue^{l     } = (\underbrace {1,..., 1}_{l      }, 0 ..., 0)$.
	\item[-]  For $\um = (m_1, ..., m_n)\in \BZ^n$ define $ \|\um\| = (|m_1|, ..., |m_n|)$. 
	
	\item[-] For $\usigma = (\varsigma_1, ..., \varsigma_n) \in \{+, -\}^n$  denote $S_n (\usigma) = |\usigma| = \prod_{l      = 1}^n \varsigma_{l     }$.
	\item[-] For $\usigma   \in \{+, -\}^n$ define $L_\pm (\usigma) = \{ l      : \varsigma_l      = \pm \}$ and $n_\pm (\usigma) = \left| L_\pm (\usigma) \right|$.
	\item[-] For $\usigma  \in \{+, -\}^n$ and $\udelta \in (\BZT)^n$ denote $\usigma^{\udelta} = \prod_{l      = 1}^n \varsigma_{l     }^{\delta_l}$.
\end{itemize}

\addtocontents{toc}{\protect\setcounter{tocdepth}{1}}


\mainmatter
%
%
%

\chapter{Hankel Transforms and Bessel Kernels}\label{chap: Hankel transforms}

This chapter is devoted to the study of Hankel transforms over $\BR_+$, $\BR$ and $\BC$ and the associated Bessel functions and Bessel kernels.  

In \S \ref{sec: notation} some basic notions are introduced, such as gamma factors, Schwartz spaces, the Fourier transform and Mellin transforms. The three  kinds of Mellin transforms  $\EM$, $\EM_{\BR}$ and $\EM_{\BC}$ are first defined over the Schwartz spaces over $\BR_+ $, $\BRx $ and $\BCx $ respectively.

In \S \ref{sec: all Ssis}  the definitions of the Mellin transforms  $\EM$, $\EM_{\BR}$ and $\EM_{\BC}$ are extended onto certain function spaces  $\Ssis (\BR_+)$, $\Ssis (\BR^\times)$ and $\Ssis (\BCx)$ respectively. We shall precisely characterize their image spaces $\Msis$, $\Msis^\BR$ and  $\Msis^\BC$ under their corresponding Mellin transforms. In spite of their similar constructions, the  analysis of the Mellin transform $\EM_{\BC}$ is much more elaborate than that of   $\EM_{\BR}$ or  $\EM$.

In \S \ref{sec: Hankel transforms}, based on gamma factors and Mellin transforms, we shall construct Hankel transforms upon suitable subspaces of the $\Ssis$ function spaces just introduced in \S \ref{sec: all Ssis} and show that they admit integral  kernels, namely Bessel kernels, in the form of Mellin-Barnes type integrals.

In \S \ref{sec: classical cases}  we shall compute Bessel functions and Bessel kernels in the classical cases.

In \S \ref{sec: Fourier type transforms}  we  shall first introduce the Schmid-Miller transforms in companion with the Fourier transform and then use them to establish a Fourier type integral transform expression of a Hankel transform.

In \S \ref{sec: integral representations}  we shall introduce   certain integrals, derived from the Fourier type integral transforms given in \S \ref{sec: Fourier type transforms}, that represent Bessel functions and Bessel kernels. For Bessel functions and real Bessel kernels, these integrals are only formal and never absolutely converge. In the complex case, however, some range of index can be found where such integrals are absolutely convergent.

In Appendix \ref{sec: special example}, for an arbitrary rank, we shall give a prototypical example of Bessel functions which represents their asymptotic nature.

\section{Preliminaries}\label{sec: notation}


\subsection{Gamma Factors}
\ 
\vskip 5 pt 


\subsubsection{} We define  the gamma factor
\begin{equation}\label{1def: G pm (s)}
	G (s, \pm) = \Gamma (s) e\lp \pm \frac s 4\rp.
\end{equation}
For $( \usigma, \ulambda) = (\varsigma_1, ..., \varsigma_n, \lambda_1, ..., \lambda_n) \in  \{ +, - \}^n \times \BC^{n}$ let
\begin{equation}\label{1def: G(s; sigma; lambda)}
	G(s; \usigma, \ulambda) = \prod_{l      = 1}^n G  (s - \lambda_l     , \varsigma_l  ).
\end{equation}

\subsubsection{} For $\delta \in \BZ/ 2\BZ = \{0, 1\}$, we define the gamma factor
\begin{equation} \label{1def: G delta}
	G_\delta (s) = i^\delta \pi^{ \frac 1 2 - s} \frac {\Gamma \lp \frac 1 2 ({s + \delta} ) \rp} {\Gamma \lp \frac 1 2 ({1 - s + \delta} ) \rp} = 
	\left\{ \begin{split}
		& 2(2 \pi)^{-s} \Gamma (s) \cos \left(\frac {\pi s} 2 \right), \hskip 10pt \text { if } \delta = 0,\\
		& 2 i (2 \pi)^{-s} \Gamma (s) \sin  \left(\frac {\pi s} 2 \right), \hskip 9 pt \text { if } \delta = 1.
	\end{split} \right.
\end{equation}
Here, we have used the duplication formula and Euler's reflection formula for the Gamma function,
\begin{equation*}
	\Gamma (1-s) \Gamma (s) = \frac \pi {\sin (\pi s)}, \hskip 10 pt \Gamma (s) \Gamma \lp s + \frac 1 2 \rp = 2^{1-2s} \sqrt \pi \Gamma (2 s).
\end{equation*}
Let $(\umu, \udelta) = (\mu_1, ..., \mu_n, \delta_1, ..., \delta_n) \in \BC^{n} \times (\BZ/2 \BZ)^n$ and define 
\begin{equation}\label{1def: G (lambda, delta)}
	G_{(\umu, \udelta) } (s) = \prod_{l      = 1}^n G_{\delta_{l     } } (s - \mu_l     ).
\end{equation}
One   observes the following simple functional relation
\begin{equation}\label{1eq: G mu delta (1-s) = G - mu delta (s)}
	G_{(\umu, \udelta) } (1-s)  G_{(- \umu, \udelta) } (s) = (-1)^{|\udelta|}.
\end{equation}

\subsubsection{}

For $m \in \BZ$, we define the gamma factor
\begin{equation}\label{1def: G m (s)}
	G_m (s) = i^{|m| } (2\pi)^{1-2 s } \frac { \Gamma \lp s + \frac 1 2{|m|}   \rp} { \Gamma \lp 1 - s + \frac 1 2{|m|}   \rp }.
\end{equation}
Let $(\umu, \um) = (\mu_1, ..., \mu_n, m_1, ..., m_n) \in \BC^{n} \times \BZ^n$ and define 
\begin{equation}\label{1def: G (mu, m)}
	G_{(\umu, \um)} (s) = \prod_{l      = 1}^n G_{m_{l     } } (s - \mu_l     ).
\end{equation}
We have the functional relation
\begin{equation}\label{1eq: G mu m (1-s) = G - mu m (s)}
	G_{(\umu, \um) } (1-s)  G_{(- \umu, \um) } (s) = (-1)^{|\um|}.
\end{equation}

\subsubsection{Relations between the Three Types of Gamma Factors} \label{sec: Gamma factor, R and C}

We first observe that
\begin{equation*}
	G_{\delta} (s) = (2\pi)^{-s} \lp G  (s, +) + (-)^\delta G (s, -) \rp.
\end{equation*}
Hence
\begin{equation}\label{1eq: G (lambda, delta) = G (s; sigma, lambda)}
	G_{(\umu, \udelta)} (s) = \sum_{\usigma \in \{+, -\}^n } \usigma ^{\udelta} (2\pi)^{|\umu| - ns}   G (s;  \usigma, \umu), \hskip 10 pt  \usigma ^{\udelta} = \prod_{l     =1}^n \varsigma _{l     }^{\delta_l     },\,  |\umu| = \sum_{l     =1}^n \mu_l     .
\end{equation}

Euler's reflection formula and certain trigonometric identities yield
\begin{equation}\label{1f: G m (s) = G 1 G delta(m)}
	\begin{split}
		i G_m (s) & = i^{|m|+1} 2 (2 \pi)^{-2s } \Gamma \lp s + \frac {|m|} 2 \rp \Gamma \lp s - \frac {|m|} 2 \rp \sin \lp \pi \lp  s - \frac {|m|} 2 \rp \rp\\
		& = G_{\delta( m) + 1} \lp \hskip - 1 pt s - \frac {| m|} 2 \hskip - 1 pt \rp  G_{0} \lp \hskip - 1 pt s + \frac {|m |} 2 \hskip - 1 pt \rp = G_{\delta( m)} \lp \hskip - 1 pt s - \frac {| m|} 2 \hskip - 1 pt \rp  G_{1} \lp \hskip - 1 pt s + \frac {|m |} 2 \hskip - 1 pt \rp,
	\end{split}
\end{equation}
with  $\delta (m) = m (\mod 2)$. Consequently,  $G_{(\umu, \um)} (s)$  may be viewed as a certain $  G_{(\boldsymbol \eta, \udelta) } (s)$ of doubled rank. 

\begin{lem}\label{1lem: complex and real gamma factors}
	
	Suppose that $(\umu, \um) \in \BC^{n} \times \BZ^n$ and $(\boldsymbol \eta, \udelta) \in  \BC^{2n} \times (\BZ/2 \BZ)^{2n}$ are subjected to one of the following two sets of relations
	\begin{align} \label{1eq: relation between (mu, m) and (lambda, delta), 1}
		& \eta_{2 l      - 1} = \mu_l      + \frac {|m_l      |} 2, \   \eta_{2 l      } = \mu_l      - \frac {|m_l      |} 2,\ \delta_{2l      - 1} = \delta( m) + 1,\ \delta_{2l      } = 0; \\
		\label{1eq: relation between (mu, m) and (lambda, delta), 2}
		&  \eta_{2 l      - 1} = \mu_l      + \frac {|m_l      |} 2, \   \eta_{2 l      } = \mu_l      - \frac {|m_l      |} 2, \ \delta_{2l      - 1} =  \delta( m), \ \hskip 17 pt \delta_{2l      } = 1.
	\end{align}
	Then  $i^n G_{(\umu, \um)} (s) = G_{(\boldsymbol \eta, \udelta) } (s)$.
\end{lem}

\subsubsection{Stirling's Asymptotic Formula} Fix $s_0 \in \BC$, and let $|\arg s| < \pi - \epsilon$, $0 < \epsilon < \pi$. We have the following asymptotic as  $|s| \ra \infty$
\begin{equation*} 
	\log \Gamma (s_0 + s) \sim \left( s_0 + s - \frac 1 2 \right) \log s - s + \frac 1 2 \log (2 \pi).
\end{equation*}
If one writes  $s_0 = \rho_0 + i t_0$ and $s = \rho + i t$, $\rho \geq 0$, then the right hand side is equal to
\begin{equation*}
	\begin{split}
		& \left(\rho_0 + \rho - \frac 1 2 \right) \log \sqrt {t^2 + \rho^2} - (t_0 + t) \arctan \left( \frac t \rho \right) - \rho + \frac 1 2 \log (2 \pi) \\
		& + i (t_0 +  t) \log {\sqrt {t^2 + \rho^2}} - i t + i  \left(\rho_0 + \rho - \frac 1 2 \right) \arctan \left( \frac t \rho \right),
	\end{split}
\end{equation*}
and therefore
\begin{equation}\label{1eq: Stirling's formula}
	|\Gamma (s_0 + s)| \sim \sqrt {2 \pi} \lp t^2 + \rho^2 \rp^{\frac 1 2 \lp \rho_0 + \rho - \frac 1 2  \rp} e^{- (t_0 + t) \arctan \left(   t / \rho \right) - \rho}.
\end{equation}

\begin{defn}
	{\rm (1).} For a finite  closed interval $[a, b] \subset \BR$ define the closed vertical strip $\BS [a, b] = \{ s \in \BC : \Re s \in [a, b]  \}.$ The open vertical strip $\BS (a, b)$ for a finite open interval $(a, b)$ is similarly defined.
	
	{\rm (2).} For  $ \lambda \in \BC$ and $r > 0$, define $ \BB_{r} (\lambda) = \left\{ s \in \BC : |s - \lambda| < r \right \} $ to be the disc of  radius $r$ centered at $s = \lambda$.
\end{defn}

\begin{lem} \label{1lem: vertical bound}
	We have
	\begin{equation}\label{1eq: vertical bound, G (s; sigma, lambda)}
		\begin{split}
			G(s; \usigma, \ulambda) \lll_{\,\ulambda,\, a,\, b,\, r} & (|\Im s| + 1)^{n \lp \Re s - \frac 1 2 \rp - \Re |\ulambda|},
		\end{split}
	\end{equation}
	for all $s \in \BS [a, b] \smallsetminus \bigcup_{l      = 1}^n \bigcup_{\kappa \in \BN} \BB_r (\lambda_{l     } - \kappa)$, with small $r > 0$, 
	\begin{equation}\label{1eq: vertical bound, G (lambda, delta) (s)}
		\begin{split}
			G_{(\umu, \udelta)} (s ) \lll_{\,\umu,\, a,\, b,\, r} & (|\Im s| + 1)^{n \lp \Re s - \frac 1 2 \rp - \Re |\umu|},
		\end{split}
	\end{equation}
	for all $s \in \BS [a, b] \smallsetminus \bigcup_{l      = 1}^n \bigcup_{\kappa \in \BN} \BB_r (\mu_{l     } - \delta_{l     } - 2 \kappa)$, and
	\begin{equation}\label{1eq: vertical bound, G (mu, m) (s)}
		G_{(\umu, \um)} \lp   s   \rp \lll_{\,\umu,\, a,\, b,\, r} \prod_{l      = 1}^n (|\Im s| + |m_l     | + 1)^{2 \Re s - 2\Re \mu_{l     } - 1},
	\end{equation}
	for all $2 s \in \BS [a, b] \smallsetminus \bigcup_{l      = 1}^n \bigcup_{\kappa \in \BN} \BB_r (2 \mu_{l     } - |m_{l     }| - 2 \kappa)$.
	
	In other words, if $\ulambda$ and $ \umu $ are given, then $G(s; \usigma, \ulambda)$, $G_{(\umu, \udelta)} (s )$ and $G_{(\umu, \um)} (s)$ are all of moderate growth with respect to $\Im s$,  uniformly on vertical strips {\rm(}with bounded width{\rm)}, and moreover  $G_{(\umu, \um)} (s )$ is also of uniform moderate growth with respect to $\um$.
\end{lem}

\subsection{Basic Notions for $\BR_+$, $\BR^\times$ and $\BC^\times$}\label{sec: R+, Rx and Cx}

Define $\BR_+ = (0, \infty)$, $\BR^\times = \BR \smallsetminus \{ 0 \}$ and $\BC^\times = \BC \smallsetminus \{ 0 \}$. We observe the isomorphisms
$\BR^\times \cong \BR_+ \times  \{+, -\}\ $($\cong \BR_+ \times \BZT$) and $\BC^\times \cong \BR_+ \times \BR/ 2\pi \BZ$, the latter being realized via the {polar coordinates} $z = x e^{ i \phi}$.

\vskip 5 pt 
\subsubsection{}
Let $|\ \,|$ denote the ordinary absolute value on either $\BR$ or $\BC$, and set $\|\, \|_\BR = |\  \,|$ for $\BR$ and $\|\, \| = \|\, \|_\BC = |\  |^2$ for $\BC$. Let $d x$ be the Lebesgue measure on $\BR$, and let  $d^\times x = |x|\- d x$ be the standard choice of the multiplicative Haar measure on $\BR^\times $.
Similarly, let $d z$ be \textit{twice} the ordinary Lebesgue measure on $\BC$  and choose the standard multiplicative Haar measure $d^\times z = \| z\|^{-1} d z$ on $\BC^\times $. Moreover, in the polar coordinates, one has $d^\times z = 2 d^\times x d \phi$.
For $x \in \BR^\times$ the sign function $\sgn(x)$ is equal to $ x /{|x|}$, whereas for $z \in \BC^\times $ we introduce the notation $[z] =  z /{|z|}$. 


Henceforth, we shall let $\BF$ be either $\BR$ or $\BC$, and occasionally let $x, y$ denote elements in $\BF$ even if $\BF=\BC$.
 
\vskip 5 pt 
\subsubsection{}
For  $\delta\in \BZT$,  we define the space $C_\delta^\infty (\BRx)$ of all smooth functions $\varphi \in C ^\infty (\BRx)$ satisfying the parity condition 
\begin{equation} \label{1eq: delta condition, R}
	\varphi (-x) = (-)^\delta \varphi (x). 
\end{equation} 
Observe that a function $\varphi \in C_\delta^\infty (\BRx)$ is determined by its restriction on $\BR _+$, namely, $\varphi (x) = \sgn (x)^\delta \varphi (|x|)$. Therefore, 
\begin{equation}\label{1eq: C delta = sgn delta C}
	C_\delta^\infty (\BRx) = \sgn (x)^\delta C^\infty (\BR _+) = \left \{ \sgn (x)^\delta \varphi  (|x|) : \varphi \in C^\infty ( \BR_+) \right \}.
\end{equation}
For a smooth function $\varphi  \in C ^\infty (\BRx)$, we define $\varphi_{\delta} \in C^\infty (\BR _+)$ by
\begin{equation}\label{1eq: upsilon delta}
	\varphi_{\delta} (x) = \frac 1 2 \lp \varphi (x) + (-)^\delta \varphi (-x) \rp, \hskip 10 pt x \in \BR _+.
\end{equation}
Clearly,
\begin{equation}\label{1eq: C = C0 + C1}
	\varphi (x) = \varphi_0 (|x|) + \sgn(x) \varphi_1 (|x|).
\end{equation}

For $m \in \BZ$, we define the space $C_m^\infty (\BCx)$ of all smooth functions $\varphi \in C ^\infty (\BCx)$ satisfying
\begin{equation} \label{1eq: m condition, C}
	\varphi \big( x e^{i\phi} \cdot e^{i \phi'} \big) = e^{i m \phi'} \varphi \lp x e^{i \phi}\rp.
\end{equation} 
A function $\varphi \in C_m^\infty (\BCx)$ is determined by its restriction on $\BR _+$, namely, $\varphi (z) = [z]^m  \varphi (|z|)$, or, in the polar coordinates, $\varphi (x e^{i\phi} ) = e^{i m \phi } \varphi (x)$. Therefore,
\begin{equation}\label{1eq: C m = [z] m C}
	C_m^\infty (\BRx) = [z]^m C^\infty (\BR _+) = \left \{ [z]^m \varphi  (|z|) = e^{i m \phi } \varphi (x) : \varphi \in C^\infty ( \BR_+) \right \}.
\end{equation}
For a smooth function $\varphi  \in C ^\infty (\BCx)$, we let $\varphi_{ m} \in C^\infty (\BR _+)$ denote the $m$-th Fourier coefficient of $\varphi$ given by
\begin{equation}\label{1eq: Fourier coefficients of upsilon}
	\varphi_m (x) = \frac 1 { {2 \pi}} \int_0^{2\pi} \varphi \lp x e^{i \phi} \rp  e^{- i m \phi} d \phi.
\end{equation}
One has the Fourier expansion of $\varphi$,
\begin{equation}\label{1eq: Fourier series expansion}
	\varphi \lp x e^{i \phi}\rp = \sum_{m \in \BZ} \varphi_m (x) e^{i m \phi}.
\end{equation}

\subsubsection{}
Subsequently, we shall encounter various subspaces of $C^\infty (\BFx)$, with $\BF = \BR, \BC$, for instance, $\SS (\BF)$, $\SS (\BFx)$, $\Ssis (\BFx)$, $\Ssis^{(\umu, \udelta)} (\BRx)$ and $\Ssis^{(\umu, \um)} (\BCx)$. Here, we list three central questions that will be the guidelines of our investigations of these function spaces.

For now, we let $D $ be a subspace of $C^\infty (\BFx)$. For $\BF = \BR$ (respectively $\BF = \BC$), we shall add a superscript or subscript $\delta$ (respectively $m$) to the notation of $D$, say $D_\delta$ (respectively $D_m$), to denote the space of $\varphi \in D$ satisfying \eqref{1eq: delta condition, R} (respectively \eqref{1eq: m condition, C}). In view of \eqref{1eq: C delta = sgn delta C} (respectively \eqref{1eq: C m = [z] m C}), there is a subspace of $C^\infty (\BR _+)$, say $E_\delta$ (respectively $E_m$), such that 
$D_\delta = \sgn(x)^\delta E_\delta$ (respectively $D_m = [z]^m E_m$).

Firstly, we are interested in the question,
\begin{itemize}
	\item[] `` How to characterize the space $E_\delta$ (respectively $E_m$)?''.
\end{itemize}

Moreover, the subspaces $D \subset C^\infty (\BFx)$ that we shall consider always satisfy the following two hypotheses, 
\begin{itemize}
	\item[-] $\varphi \in D $ implies $\varphi_\delta \in E_\delta$ for $\BF = \BR$ (respectively, $\varphi \in D $ implies $\varphi_m \in E_m$ for $\BF = \BC$), and
	\item[-] $D$ is closed under addition.
\end{itemize}
For $\BF = \BR$, under these two hypotheses, it follows from \eqref{1eq: C = C0 + C1} that
$$D = D_0 \oplus D_1 \cong  E_0 \times E_1.$$
For $\BF = \BC$, in view of \eqref{1eq: Fourier series expansion}, the map that sends $\varphi$ to the sequence $\left\{\varphi_m \right \}  $ of its Fourier coefficients is injective. The second question arises,
\begin{itemize}
	\item[] `` What is the image of $D$ in $\prod_{m \in \BZ} E_m $ under this map?'', or equivalently,
	\item[] `` What conditions should a sequence $\left\{\varphi_m \right \}  \in \prod_{m \in \BZ} E_m$ satisfy in order for the Fourier series defined by \eqref{1eq: Fourier series expansion} giving a function $\varphi \in D$?''.
\end{itemize}

Finally, after introducing the Mellin transform $\EM_\BF$, 
we shall focus on the question,
\begin{itemize}
	\item[] `` What is the image of $D$ under the Mellin transform  $\EM_\BF$?''. 
\end{itemize}

\subsection{Schwartz Spaces} \label{sec: Schwartz spaces}

We say that a function $\varphi \in C^\infty (\BR_+)$ is {smooth at zero} if all of its derivatives admit asymptotics as below,
\begin{equation}\label{1eq: asymptotics R+}
	\varphi^{(\alpha)} (x) = \alpha!  a_{ \alpha}  + O_{ \alpha } \lp x \rp \text{ as } x \ra 0, \text{ for any } \alpha \in \BN, \text{ with } a_\alpha \in\BC.
\end{equation}

\begin{rem}
	Consequently, one has the asymptotic expansion $\varphi (x) \sim \sum_{\kappa =0}^\infty a_\kappa  x^\kappa$, which means that 
	$\varphi (x) = \sum_{\kappa  = 0}^{A } a_{\kappa } x^{\kappa } + O_{ A } \lp x^{A + 1 }\rp$ as $x \ra 0$ for any $A \in \BN$.
	It is however not required that the series $\sum_{\kappa =0}^\infty a_\kappa  x^\kappa $ be convergent for any $x \in \BR^\times$.
	
	Actually, \eqref{1eq: asymptotics R+} is equivalent to the following
	\begin{equation}\label{1eq: asymptotics R+, 2}
		\varphi^{(\alpha)} (x) = \sum_{\kappa  = \alpha}^{\alpha + A} a_{\kappa } [\kappa ]_{\alpha} x^{\kappa - \alpha} + O_{ \alpha, \, A } \lp x^{A + 1 }\rp \text{ as } x \ra 0, \text{ for any } \alpha, A \in \BN.
	\end{equation}
	
	Another observation is that, for a given constant $1> \rho > 0$, \eqref{1eq: asymptotics R+} is equivalent to the following seemingly weaker statement,
	\begin{equation}\label{1eq: asymptotics R+, 3}
		\begin{split}
			\varphi^{(\alpha)} (x) = \alpha!  a_{ \alpha}  + O_{ \alpha,\, \rho } \lp x^{\,\rho} \rp \text{ as } x \ra 0, \text{ for any } \alpha \in \BN, \text{ with } a_\alpha \in\BC.
		\end{split}
	\end{equation}
\end{rem}

Let $C^\infty (\overline \BR_+)$ denote the subspace of $C^\infty (\BR_+)$ consisting of smooth functions on $\BR_+$ that are also smooth at zero.

Let $\mathscr S (\overline \BR_+)$ denote the space of functions in $C^\infty (\overline \BR_+)$ that rapidly decay at infinity along with all of their derivatives. 
Let  $\mathscr S (\BF)$ denote the Schwartz space on $\BF$, with $\BF = \BR, \BC$.

Let $\mathscr S (\BR_+)$ denote the space of Schwartz functions on $ \BR_+ $, that is, smooth functions  on $ \BR_+ $ whose derivatives rapidly decay at {\it both} zero and infinity.
Similarly, we denote by $\mathscr S (\BF^\times )$ the space of Schwartz functions on $ \BF^\times$.


The following lemma provides  criteria for characterizing functions in these Schwartz spaces, especially functions in $\mathscr S (\BC)$ or $\mathscr S (\BCx )$ in the polar coordinates. Its proof is left as an easy excise in analysis for the reader.
\begin{lem}\label{lem: Schwartz}
	Let notations be as above.
	
	{\rm (1.1).} Let $\varphi \in C^\infty (\overline \BR _+)$ satisfy the asymptotics \eqref{1eq: asymptotics R+}. Then $\varphi \in \mathscr S (\overline \BR_+)$ if and only if $\varphi$ also satisfies
	\begin{equation}\label{1eq: Schwartz R+}
		x^{\alpha + \beta} \varphi ^{(\alpha)} (x) \lll_{\,\alpha,\, \beta} 1\, \text{ for all } \alpha, \beta \in \BN.
	\end{equation}

	{\rm (1.2).} A smooth function $\varphi$ on $\BR_+$ belongs to $ \mathscr S (\BR_+)$ if and only if $\varphi $ satisfies \eqref{1eq: Schwartz R+} with $\beta \in \BN$ replaced by $\beta \in \BZ$.
	
	Let $\varphi \in \mathscr S (\overline \BR_+)$ and $a_\alpha$ be as in \eqref{1eq: asymptotics R+}. Then $\varphi \in \SS ( \BR _+)$ if and only if $a_\alpha = 0$ for all $ \alpha \in \BN$.

	{\rm (2.1).}  A smooth function $\varphi $ on $\BRx$ extends to a function in $\mathscr S (\BR )$ if and only if
	\begin{itemize}
		\item[-] $\varphi $ satisfies \eqref{1eq: Schwartz R+} with  $x^{\alpha + \beta}$ replaced by $|x|^{\alpha + \beta}$, and
		\item[-] all the derivatives of $\varphi $ admit asymptotics
		\begin{equation} \label{1eq: asymptotics R}
			\varphi^{(\alpha)} (x) = \alpha! a_{ \alpha} + O_{ \alpha } \lp |x| \rp \text{ as } x \ra 0, \text{ for any }  \alpha \in \BN, \text{ with } a_\alpha \in \BC.
		\end{equation}
		 
	\end{itemize}

	{\rm (2.2).} Let  $\varphi$ be a smooth function on $\BRx$. Then $ \varphi \in \mathscr S (\BRx )$ if and only if $\varphi $ satisfies \eqref{1eq: Schwartz R+} with $x^{\alpha + \beta}$ replaced by $|x|^{\alpha + \beta}$ and $\beta \in \BN$ by $\beta \in \BZ$.
	
	Suppose $\varphi \in \mathscr S (\BR )$, then $\varphi \in \mathscr S (\BRx )$ if and only if $ \varphi ^{(\alpha)} (0) = 0$ for all $\alpha \in \BN$, or equivalently, $a_\alpha = 0$ for all $\alpha \in \BN$, with $a_\alpha $ given in \eqref{1eq: asymptotics R}.

	{\rm (3.1).}  Write $\partial_x = \partial/ \partial x$ and $\partial_\phi = \partial/ \partial \phi$.  In the polar coordinates, a smooth function $\varphi \lp x e^{i \phi}\rp \in C^\infty (\BC ^\times)$ extends to a function in $\mathscr S (\BC )$ if and only if
	\begin{itemize}
		\item[-] $\varphi \lp x e^{i \phi}\rp $ satisfies
		\begin{equation}\label{1eq: Schwartz Cx}
			x^{ \alpha + \beta} \partial_x ^{\alpha }  \partial_\phi ^{\gamma} \varphi \lp x e^{i \phi}\rp \lll_{\,\alpha,\, \beta,\, \gamma} 1\, \text{ for all } \alpha, \beta, \gamma \in \BN,
		\end{equation}
		\item[-] all the partial derivatives of $\varphi $ admit asymptotics
		
		\vskip 7 pt
		\item[\noindent \refstepcounter{equation}(\theequation) \label{1eq: Schwartz Cx asymptotic} \hskip 8 pt]
		$\ds x^\alpha \partial_x ^{\alpha }  \partial_\phi ^{\beta} \varphi \lp x e^{i \phi}\rp = \sum_{|m| \leq \alpha + \beta} \ \sum_{\sstyle |m| \leq \kappa \leq \alpha + \beta \atop {\sstyle \kappa \equiv m (\mod 2)}} a_{m, \kappa } [\kappa ]_{\alpha}   (im)^{\beta} x^{ \kappa }  e^{i m \phi} + O_{\alpha,\, \beta} \lp x^{\alpha + \beta+1} \rp$
		\vskip 5 pt
		\noindent as $  x \ra 0$, for any $ \alpha, \beta \in \BN$,
		with $a_{m, \kappa } \in\BC$ for $\kappa \geq |m|$ and $\kappa \equiv m (\mod 2)$.

	\end{itemize}
	
	Let $\varphi \in \mathscr S (\BC )$ and $\varphi_m$ be the $m$-th Fourier coefficient of $ \varphi $ given by \eqref{1eq: Fourier coefficients of upsilon},
	then it follows from {\rm(\ref{1eq: Schwartz Cx}, \ref{1eq: Schwartz Cx asymptotic})} that
	\begin{itemize}
		\item[-]  $ \varphi _m$ satisfies
		\begin{equation}\label{1eq: bounds for Fourier coefficients}
			\begin{split}
				x^{ \alpha + \beta} \varphi^{(\alpha)}_m (x) & \lll_{\,\alpha,\, \beta,\, A} (|m| + 1)^{- A} \, \text{ for all } \alpha, \beta, A \in \BN,
			\end{split}
		\end{equation}
		\item[-] all the derivatives of $\varphi_m$ admit asymptotics
		
		\vskip 5 pt
		\item[\noindent \refstepcounter{equation}(\theequation) \label{1eq: varphi m asymptotic, 0}  \hskip 8 pt]
		{ \hfill  $
			\ds \varphi^{(\alpha)} _m (x) = \sum_{\kappa  = \alpha}^{\alpha + A} a_{m, \kappa } [\kappa ]_{\alpha} x^{\kappa - \alpha}  +  O_{\alpha,\, A } \lp (|m| + 1)^{-A} x^{ A + 1} \rp 
			$ \hfill}
		
		\vskip 5 pt
		\item[ ]  as   $x \ra 0$, for any given $ \alpha, A \in \BN$, with $a_{m, \kappa } \in \BC$ satisfying $a_{m, \kappa }= 0 $ if either $\kappa < |m| $ or $\kappa \notequiv m (\mod 2)$.
		
	\end{itemize}
	Observe that \eqref{1eq: varphi m asymptotic, 0} is equivalent to the following two conditions,
	\begin{itemize}
		\vskip 5 pt
		\item[\noindent \refstepcounter{equation}(\theequation) \label{1eq: phi m, 1}  \hskip 8 pt]
		$\ds \varphi^{(\alpha)} _m (x) = \alpha ! a_{m, \alpha } + O_{\alpha } \lp x \rp $ as $x \ra 0$,  
		for any $ \alpha \geq |m|$,   with $ a_{m, \alpha } \in \BC $ satisfying $ a_{m, \alpha }= 0 $ if $ \alpha \notequiv m (\mod 2)$,
		
		\vskip 7 pt
		\item[\noindent \refstepcounter{equation}(\theequation) \label{1eq: phi m, 2}  \hskip 8 pt]
		for any given $ \alpha, A \in \BN$, $\varphi^{(\alpha)} _m (x) =  O_{\alpha,\, A } \lp (|m| + 1)^{-A} x^{ A + 1} \rp $ as $  x \ra 0$,   if $|m| > \alpha + A$.
		
	\end{itemize}
	In particular, $\varphi_m \in \mathscr S (\overline \BR _+)$.
	
	Conversely, if a sequence $\lpp \varphi_m \rpp $ of functions in $C^\infty (\BR _+)$ satisfies \eqref{1eq: bounds for Fourier coefficients}, \eqref{1eq: phi m, 1} and \eqref{1eq: phi m, 2}, then the Fourier series defined by $\lpp \varphi_m \rpp $, that is, the right hand side of \eqref{1eq: Fourier series expansion}, is a Schwartz function on $\BC$.

	{\rm (3.2).} In the polar coordinates, a smooth function $\varphi \lp x e^{i \phi}\rp \in C^\infty (\BC ^\times)$ is a Schwartz function on $\BC^\times$ if and only if $\varphi$ satisfies  \eqref{1eq: Schwartz Cx} with $\beta \in \BN$ replaced by $\beta \in \BZ$.

	Let $\varphi \in \mathscr S (\BC^\times)$ and $\varphi_m$ be the $m$-th Fourier coefficient of $ \varphi $, then it is necessary that $\varphi _m$ satisfies  \eqref{1eq: bounds for Fourier coefficients}  with $\beta \in \BN$ replaced by $\beta \in \BZ$.
	In particular, $\varphi_m \in \mathscr S (\BR _+)$.
	
	Conversely, if a sequence $\left\{\varphi_m \right \} $ of functions in $C^\infty (\BR _+)$ satisfies the condition \eqref{1eq: bounds for Fourier coefficients}  with $\beta \in \BN$ replaced by $\beta \in \BZ$, then the Fourier series defined by $\left\{\varphi_m \right \} $ gives rise to a Schwartz function on $\BC ^\times$.
	
	Let $\varphi \in \mathscr S (\BC)$ and $a_{m, \kappa}$ be given in \eqref{1eq: Schwartz Cx asymptotic}, \eqref{1eq: varphi m asymptotic, 0} or \eqref{1eq: phi m, 1}. $\varphi \in \SS ( \BCx)$ if and only  if  $a_{m, \kappa} = 0$ for all $m \in \BZ, \kappa \in \BN$.
	
\end{lem}

\subsubsection{Some Subspaces of $\SS (\overline \BR_+)$} \label{sec: Schwartz subspaces}

In the following, we introduce several subspaces of  $\SS (\overline \BR_+)$ which are closely related to $\SS (\BR)$ and $\SS (\BC)$.

We first define for $\delta \in \BZT $  the subspace $C^\infty_{\delta} (\overline \BR_+) \subset C^\infty (\overline \BR_+)$ of functions with an asymptotic expansion of the form $ \sum_{\kappa =0}^\infty a_{ \kappa} x^{\delta + 2 \kappa}$ at zero.

\begin{rem}\label{rem: C=C0+C1}
	A question arises, ``whether $C^\infty (\overline \BR_+) = C^\infty_0 (\overline \BR_+) + C^\infty_1 (\overline \BR_+)${\rm ?}''. 
	
	The answer is affirmative. 
	
	To see this, we define the space $C_\delta^\infty (\BR)$ of smooth functions $\varphi$ on $\BR$ satisfying \eqref{1eq: delta condition, R}.
	One has $\sgn (x)^\delta \varphi (|x|) \in C_\delta^\infty (\BR)$ if $\varphi  \in C^\infty_{\delta} (\overline \BR_+)$, and conversely, $\varphi \restriction_{\BR_+} \in  C^\infty_{\delta} (\overline \BR_+) $ if $\varphi \in C_\delta^\infty (\BR)$. Thus, with the simple observation $C^\infty ( \BR ) = C^\infty_0 ( \BR ) \oplus C^\infty_1 ( \BR )$, one sees that $C^\infty_0 (\overline \BR_+) + C^\infty_1 (\overline \BR_+)$ is the subspace  of $C^\infty (\overline \BR_+)$ consisting of functions on $\BR _+$ that admit a smooth extension onto $\BR$.
	
	On the other hand, the Borel theorem {\rm(}\cite[1.5.4]{Nara}{\rm)}, which is a special case of the Whitney extension theorem  {\rm (}\cite[1.5.5, 1.5.6]{Nara}{\rm)}, states that for any sequence $\lpp a_\alpha\rpp $ of constants there exists a smooth function $\varphi \in C^\infty (\BR)$ such that $ \varphi^{(\alpha)} (0) = \alpha! a_\alpha $. Clearly, this theorem of Borel implies our assertion above.
	
	In {\rm \S \ref{sec: refinements Msis Sis, R+}}  we shall give an alternative proof of this  using the Mellin transform. See Remark {\rm \ref{rem: S=S0+S1}}.
\end{rem}

We define $ \SS_\delta (\overline \BR_+) = \SS (\overline \BR_+) \cap C_\delta^\infty (\overline \BR_+)$. The following identity is obvious 
\begin{equation*}
	\SS_\delta (\overline \BR_+) = x^\delta \SS_0 (\overline \BR_+).
\end{equation*}
In view of Lemma \ref{lem: Schwartz} (1.2), we have $\SS_0  (\overline \BR_+) \cap \SS_1 (\overline \BR_+) = \SS (\BR _+)$.

If we let $\SS_\delta (\BR) $ be the space of functions $\varphi \in \SS (\BR)$ satisfying \eqref{1eq: delta condition, R}, then 
$$ \SS_\delta (\BR) = \sgn (x)^\delta \SS_\delta (\overline \BR_+) = \left \{ \sgn (x)^\delta \varphi  (|x|) : \varphi \in \SS_\delta (\overline \BR_+) \right \}.$$ 
Clearly, $\SS (\BR) = \SS_0 (\BR ) \oplus \SS_1 (\BR )$.

We define the subspace $\SS_m (\overline \BR _+) \subset \SS_{\delta(m)}  (\overline \BR _+)$, with  $\delta (m) = m (\mod 2)$,  of functions with an asymptotic expansion of the form $ \sum_{\kappa = 0}^\infty a_{ \kappa} x^{|m| + 2 \kappa}$ at zero.  
We have
\begin{equation*}
	\SS_m (\overline \BR_+) = x^{|m|} \SS_0 (\overline \BR_+).
\end{equation*}
If we define $\SS_m (\BC)$ to be the space of $\varphi \in \SS (\BC)$ satisfying \eqref{1eq: m condition, C}, then 
$$ \SS_m (\BC) = [z]^m \SS_m (\overline \BR_+) = \left \{ [z]^m \varphi  (|z|) = e^{im\phi} \varphi (x) : \varphi \in \SS_m (\overline \BR_+) \right \}.$$ 
The last two paragraphs in  Lemma  \ref{lem: Schwartz} (3.1) can be recapitulated as below
$$
\SS (\BC) \xrightarrow{\cong}  \left\{ \lpp \varphi_m\rpp \in \prod_{m \in \BZ} \SS_m (\overline \BR _+) : \varphi_m \text{ satisfies (\ref{1eq: bounds for Fourier coefficients}, \ref{1eq: phi m, 1}, \ref{1eq: phi m, 2})} \right \} \twoheadrightarrow  \SS_m (\overline \BR _+),
$$
where the first map sends $\varphi \in \SS (\BC)$ to the sequence $ \lpp \varphi _m \rpp $ of its Fourier coefficients, and the second is the $m$-th projection.
According to Lemma  \ref{lem: Schwartz} (3.1), the first map is an isomorphism, and the second projection is surjective.
 
\vskip 5 pt 
\subsubsection{$\SS_\delta (\BRx)$ and $\SS_m (\BCx)$} \label{sec: S delta and Sm}
Let $\delta \in \BZ/2 \BZ$ and $m \in \BZ$. We define $\SS_\delta (\BRx) = \SS (\BRx) \cap \SS_\delta (\BR)$ and $\SS_m (\BCx) = \SS (\BCx) \cap \SS_m (\BC)$.
Clearly, $\SS_\delta (\BRx) = \sgn(x)^\delta \SS (\BR _+)$ and $\SS_m (\BCx) = [z]^m \SS (\BR _+)$.


\subsection{The Fourier Transform}
According to the local theory in Tate's thesis 
for an Archimedean local field $\BF$, the Fourier transform $\widehat \varphi = \EF \varphi$ of a Schwartz function $\varphi\in \SS(\BF)$ is defined by
\begin{equation}\label{1eq: Fourier, R}
	\widehat \varphi(y) = \int_{\BF} \varphi(x) e(- \Lambda (xy)) d x,
\end{equation}
with
\begin{equation} \label{1eq: Lambda (x)}
	\Lambda (x) = 
	\left\{ \begin{split}
		& x , & &  \text{ if } \mathbb{F}=\mathbb{R};\\
		& \Tr (x) = x + \overline x, & & \text{ if } \mathbb{F}=\mathbb{C}.
	\end{split} \right.
\end{equation}
The Schwartz space $\SS(\BF)$ is invariant under the Fourier transform. Moreover, with our choice of measure in \S \ref{sec: R+, Rx and Cx}, the following inversion formula holds
\begin{equation}\label{1eq: Fourier, C}
	\widehat{\widehat {\varphi}} (x) = \varphi (-x), \hskip 10 pt x \in \BF.
\end{equation}

\subsection{The Mellin Transforms $\EM $, $\EM_{\delta}$ and $\EM_m$}\label{sec: Mellin preliminaries}

Corresponding to $\BR_+$, $\BRx$ and $\BCx$, there are three kinds of Mellin transforms  $\EM $, $\EM_{\delta}$ and $\EM_m$.

\begin{defn}[Mellin transforms]
	\ 
	
	{\rm (1).} The {\rm Mellin transform} $\EM  \varphi$ of a Schwartz function $\varphi \in \mathscr S (\BR_+)$
	is given by
	\begin{equation} \label{1def: Mellin transform}
		\EM  \varphi (s) = \int_{\BR_+} \varphi (x) x ^{s } d^\times x.
	\end{equation}
	
	{\rm (2).} For $\delta \in \BZ/ 2\BZ$, the {\rm (signed) Mellin transform $\EM _{\delta} \varphi$ with order $\delta$} of a Schwartz function $\varphi \in \mathscr S (\BR ^\times)$
	is defined by
	\begin{equation} \label{1def: signed Mellin transform}
		\EM _{\delta} \varphi (s) = \int_{\BR^\times} \varphi (x) \sgn (x)^\delta |x|^{s } d^\times x.
	\end{equation} 
	Moreover, define $\EM_{\BR} = (\EM_0, \EM_1)$.
	
	{\rm (3).} For $m \in \BZ$, the {\rm Mellin transform $\EM _m \varphi $ with order $m$} of a Schwartz function $\varphi \in \mathscr S (\BC^\times)$ is defined by
	\begin{equation}\label{1def: Mellin transform over complex}
		\EM _m \varphi (s) = \int_{\BC^\times} \varphi (z) [z]^{ m} \|z\| ^{\frac 12 s  }\ d^\times z = 2 \int_0^\infty \int_0^{2 \pi} \varphi \lp x e^{i \phi} \rp e^{ i m \phi} d\phi \cdot x^{ s } d^\times x.
	\end{equation}
	Moreover, define $\EM_{\BC} = \prod_{m \in \BZ} \EM_{- m} $.
	
\end{defn}

\begin{observation}\label{observ: Mellin} For $\varphi \in \mathscr S (\BRx)$, we have
	\begin{equation}\label{1eq: M delta = M}
		\EM _\delta \varphi (s) =  2 \EM  \varphi_{\delta}  (s), \hskip 10 pt \delta \in \BZT.
	\end{equation}
	Similarly, for $\varphi \in \mathscr S (\BCx)$, we have
	\begin{equation}\label{1eq: Mm = M}
		\EM _{-m} \varphi (s) = 4 \pi \EM  \varphi_{ m} ( s), \hskip 10 pt m \in \BZ.
	\end{equation}
	The relations  \eqref{1eq: M delta = M} and \eqref{1eq: Mm = M} reflect the identities $\BR^\times \cong \BR_+ \times  \{+, -\}  $ and $\BC^\times \cong \BR_+ \times \BR/ 2\pi \BZ$ respectively.
\end{observation}

\begin{lem}[Mellin inversions] \label{lem: Mellin inversions} Let $\sigma  $ be real. Denote by  $  (\sigma)  $ the vertical line from $\sigma - i \infty$ to $\sigma + i \infty$. 
	
	{\rm (1).} For $\varphi \in \mathscr S (\BR_+)$, we have
	\begin{equation}\label{1eq: Mellin inversion}
		\varphi (x) =  \frac 1{2 \pi i} \int_{(\sigma)} \EuScript M \varphi (s) x^{-s} d s.
	\end{equation}
	
	{\rm (2).} For $\varphi \in \mathscr S (\BRx)$, we have
	\begin{equation}\label{1eq: Mellin inversion, R}
		\varphi (x) =  \frac { 1} {4 \pi i} \sum_{\delta \in \BZ/ 2\BZ } \sgn (x)^\delta \int_{(\sigma)} \EuScript M_\delta \varphi (s) |x|^{-s} d s.
	\end{equation}
	
	{\rm (3).} For $\varphi \in \mathscr S (\BCx)$, we have
	\begin{equation}\label{1eq: Mellin inversion, C}
		\varphi (z) =  \frac { 1} {8 \pi^2 i} \sum_{m \in \BZ}  [z]^{-m} \int_{(\sigma)} \EuScript M_{m} \varphi (s) \|z\|^{- \frac 1 2 s  } d s,
	\end{equation}
	or, in the polar coordinates,
	\begin{equation}\label{1eq: Mellin inversion, C, polar}
		\varphi \lp x e^{i \phi}\rp = \frac { 1} {8 \pi^2 i}  \sum_{m \in \BZ} e^{- i m \phi} \int_{(\sigma)} \EuScript M_{m} \varphi (s) x^{- s} d s.
	\end{equation}
\end{lem}

\begin{defn}\
	
	{\rm (1).} Let $\mathscr H_{\mathrm {rd}}$ denote the space of all entire functions $H(s)$ on the complex plane that  rapidly decay along vertical lines, uniformly on vertical strips.  
	
	{\rm (2).} Define $\mathscr H_{\mathrm {rd}}^{\BR} = \mathscr H_{\mathrm {rd}} \times \mathscr H_{\mathrm {rd}}$.
	
	{\rm (3).} Let $\mathscr H_{\mathrm {rd}}^{\BC}$ be the subset of $\, \prod_{\BZ} \Hrd$ consisting of sequences $\lpp H_m  (s) \rpp $ of entire functions in $\Hrd$ satisfying the following condition,
	\begin{itemize}
		\item[\noindent \refstepcounter{equation}(\theequation) \label{1eq: rapid decay, C} \hskip 8 pt]
		for any given $ \alpha, A \in \BN$ and vertical strip $ \BS[a, b]$,  
		
		\vskip 5 pt
		{\centering $\ds  H_m  (s) \lll_{\, \alpha,\, A,\, a,\, b} (|m|+1)^{-A}  (|\Im s| + 1)^{-\alpha} \text{ for all } s \in \BS[a, b].$}
	\end{itemize}
\end{defn}


\begin{cor}\label{cor: Mellin, Schwartz}\
	
	{\rm (1).} The Mellin transform $\EM$ and its inversion establish an isomorphism between $\SS (\BR _+)$ and $\Hrd$.
	
	{\rm (2).} For each $\delta \in \BZT$, $\EM_\delta$ establishes an isomorphism between $\SS_\delta (\BRx)$ and $\Hrd $. Hence, $\EM_{\BR} $ establishes an isomorphism between $\SS (\BRx)$ and $\Hrd^{\BR}$.
	
	{\rm (3).} For each $m\in \BZ$, $\EM_{-m}$ establishes an isomorphism between $\SS_m (\BCx)$ and $\Hrd $. Moreover, $\EM_{\BC}$ establishes an isomorphism between $\SS (\BCx)$ and $\Hrd^{\BC}$.
\end{cor}
\begin{proof}
	(1) is a well-known consequence of Lemma \ref{lem: Mellin inversions} (1), whereas (2) directly follows from (1) and Lemma  \ref{lem: Mellin inversions} (2). As for (3), in addition to (1) and Lemma \ref{lem: Mellin inversions} (3),  Lemma \ref{lem: Schwartz} (3.2) is also required for the rapid decay in $m$.
\end{proof}

\section{The Function Spaces $\Ssis (\BR_+)$, $\Ssis (\BR^\times)$ and $\Ssis (\BC^\times)$}\label{sec: all Ssis}



The goal of this section is to extend the definitions of the Mellin transforms $\EM$, $\EM_{\BF}$ and generalize the   settings in   \S \ref{sec: Mellin preliminaries} to the function spaces $\Ssis (\BR _+)$, $\Msis$, $\Ssis (\BFx)$ and $\Msis^{\BF}$. These spaces are much more sophisticated than $\SS (\BR_+)$,  $\Hrd$, $\SS (\BFx)$ and $\Hrd^\BF$ but most suitable for investigating Hankel transforms over $\BR_+$ and $\BFx$. 

We shall first construct  the function spaces $\Ssis (\BR_+)$, $\Msis$ and   establish an isomorphism between them using the Mellin transform $\EM$.
Based on these, we shall then turn to the spaces $\Ssis (\BFx)$, $\Msis^{\BF}$ and the Mellin transform $\EM_{\BF}$. 
The case $\BF = \BR$ has been worked out in \cite[\S 6]{Miller-Schmid-2006}. Since $\BRx \cong \BR_+ \times \{+, -\}$ is simply two copies of $\BR _+$, the properties of $\Ssis (\BRx)$ and $\Msis^\BR$ are in substance the same as those of $\Ssis (\BR_+)$ and $\Msis$.
In the case $\BF = \BC$, $\Ssis (\BCx)$ and $\Msis^\BC$ can be constructed in a parallel way. 
The study on $\BCx$ is however  much more elaborate, since $\BCx \cong \BR_+ \times \BR/ 2\pi \BZ$ and the analysis on the circle $ \BR/ 2\pi \BZ$ is also  taken into account. 


\subsection{The Spaces $\Ssis (\BR_+)$  and $\Msis$}
\label{sec: Ssis and Msis, R+} \footnote{According to \cite[Definition 6.4]{Miller-Schmid-2006},  a function in $\Ssis (\BR_+)$ is said to have a \textit{simple singularity} at zero. Thus the subscript ``{sis}'' stands for ``{simple singularity}''.}
 
\vskip 5 pt 



\subsubsection{The Spaces $ x^{-\lambda}  (\log x)^j \SS (\overline \BR _+)$ and $\Msis^{\lambda, j}$}
Let $\lambda \in \BC$ and $j \in \BN$. 

We define $$ x^{-\lambda}  (\log x)^j \SS (\overline \BR _+) = \left\{ x^{-\lambda}  (\log x)^j \varphi (x) : \varphi \in \SS (\overline \BR _+) \right\} .$$

We say that a meromorphic function $H(s)$ has a pole of {pure order} $j + 1$  at $s = \lambda $ if the principal part of $H(s)$ at $s = \lambda $ is $ a {(s - \lambda )^{-j-1}}$ for some constant $a \in \BC$.  Of course, $H (s) $ does not have a genuine pole at $s = \lambda $ if $a = 0$.
We define the space $\Msis^{\lambda, j}$ of all meromorphic functions $H(s)$ on the complex plane such that
\begin{itemize}
	\item[-] the only possible singularities of $H  (s)$ are poles of pure order  $j+1$ at the points in $ \lambda - \BN = \left\{ \lambda - \kappa : \kappa \in \BN \right \}$, and
	\item[-] $H (s)$ decays rapidly along vertical lines, uniformly  on vertical strips, 
	that is,
	\item[\noindent \refstepcounter{equation}(\theequation)\hskip 15 pt \label{2eq: Mellin rapid decay, 0}]
	\hskip 5pt for any given $ \alpha \in \BN$,  vertical strip $ \BS [a, b] $ and  $ r > 0$,
	$$H (s) \lll_{\,\lambda,\, j,\, \alpha, \, a,\, b,\, r} (|\Im s| + 1)^{-\alpha} \text{ for all } \textstyle s \in \BS [a, b] \smallsetminus \bigcup_{\kappa \in \BN} \BB_{r} (\lambda - \kappa).$$
\end{itemize}

The constructions of the Mellin transform $\EM$ and its inversion (\ref{1def: Mellin transform}, \ref{1eq: Mellin inversion}) identically extend from $\SS (\BR _+)$ onto $\Ssis^{\lambda, j} (\BR _+)$, except that the conditions $\Re s > \Re \lambda$ and $\sigma > \Re \lambda$ are required to guarantee convergence.

\begin{lem}\label{2lem: Mellin isomorphism, Msis lambda j}
	Let $\lambda \in \BC$ and $j \in \BN$. The Mellin transform $\EM$ and its inversion establish an isomorphism of between $ x^{-\lambda}  (\log x)^j \SS (\overline \BR _+) $ and $\Msis^{\lambda, j}$.
\end{lem}

This lemma  is essentially  \cite[Lemma 6.13, Corollary 6.17]{Miller-Schmid-2006}.
Nevertheless, we shall include its proof as the reference for the constructions of $\Nsis^{\BC, \lambda, j}$ and $\Msis^{\BC}$ in \S \ref{sec: Msis C} as well as the proof of Lemma {\rm \ref{2lem: Ssis to Msis, C}}. 

\begin{proof}
	Let $\upsilon (x) = x^{-\lambda} (\log x)^j \varphi (x)$ for some $\varphi \in \SS (\overline \BR _+)$. Suppose that the derivatives of $\varphi$ satisfy  \eqref{1eq: asymptotics R+, 2} and \eqref{1eq: Schwartz R+}, that is, asymptotic expansions at zero and the Schwartz condition at infinity.
	
	\vskip 3 pt
	{\textsc {Claim}} 1. Let  
	\begin{equation*} 
		H (s) = \EM \upsilon (s) = \int_0^\infty \upsilon (s) x^{s-1} d x, \hskip 10 pt \Re s > \Re \lambda.
	\end{equation*}
	Then $H$ admits a meromorphic continuation onto the whole complex plane.
	The only singularities of $H $ are poles of pure order $j+1$ at the points in $ \lambda - \BN$. More precisely, $H (s) $ has a pole at $s = \lambda - \kappa$ of principal part $ {(-)^j j! a_\kappa } {(s - \lambda + \kappa)^{-j-1}}$.
	Moreover,  $H $ decays rapidly along vertical lines, uniformly on vertical strips. 
	To be concrete, we have
	\begin{itemize}
		\item[\noindent \refstepcounter{equation}(\theequation)\hskip 15 pt \label{2eq: Mellin rapid decay}]
		for any given $ \alpha, A \in \BN, b \geq a > \Re \lambda - \alpha - A - 1$   and   $ r > 0$,
		\vskip 5 pt
		{\centering $H (s) \lll_{\,\lambda, j,\, \alpha,\, A,\, a,\, b,\, r} (|\Im s| + 1)^{-\alpha} \text{ for all } \textstyle s \in \BS [a, b] \smallsetminus \bigcup_{\kappa = 0}^{\alpha + A} \BB_{r} (\lambda - \kappa).$}
	\end{itemize}
	We remark that \eqref{2eq: Mellin rapid decay, 0} and \eqref{2eq: Mellin rapid decay} are equivalent.
	
	\vskip 5 pt
	
	{\textsc {Proof of Claim}} 1. In view of $\EM \lp x^{-\lambda} (\log x)^j \varphi (x) \rp (s) = \EM \lp (\log x)^j \varphi (x) \rp (s - \lambda)$, one may assume $\lambda = 0$. As such, $\upsilon (x) = (\log x)^j \varphi (x)$.  
	
	Let $A \in \BN$. We have for $\Re s > 0$
	\begin{equation}\label{2eq: poles of Mellin}
		\begin{split}
			\EM \upsilon (s) = & \int_0^1 (\log x)^j \lp \varphi (x) - \sum_{\kappa = 0}^A a_{\kappa} x^{\kappa} \rp x^{s -1} d x + \sum_{\kappa = 0}^A \frac { (-)^j j!  a_\kappa } { \lp s + \kappa\rp^{j+1} }\\
			& + \int_1^\infty (\log x)^j \varphi (x) x^{s -1} d x.
		\end{split}
	\end{equation}
	Here, we have used $$\int_0^1 (\log x)^j x^{s -1} d x = \frac { (-)^j j! } { s^{j+1} }, \hskip 10 pt \Re s > 0.$$
	In view of $\varphi (x) - \sum_{\kappa = 0}^A a_{\kappa} x^{\kappa} = O_A (x^{A+1})$, the first integral in \eqref{2eq: poles of Mellin} converges in the half-plane $\left \{ s : \Re s > - A - 1 \right\}$. The last integral converges for all $s$ on the whole complex plane due to the rapid decay  of $\varphi$. Thus $H (s) = \EM \upsilon (s)$ admits a  meromorphic extension onto $\left \{ s : \Re s > - A - 1 \right\}$ and, since $A$ was arbitrary, onto the whole complex plane, with poles of pure order  $j+1$ at the points in $- \BN$. 
	
	For any given $\alpha \in \BN$, repeating partial integration  $\alpha $ times to the defining integral of $\EM \upsilon (s)$ yields
	\begin{equation*}
		(-)^\alpha (s )_{\alpha} \EM \upsilon (s) = \EM \upsilon^{(\alpha)} (s + \alpha).
	\end{equation*}
	In view of this, we first expand  $\EM \upsilon^{(\alpha)} (s + \alpha)$ according to the expansion of $\upsilon^{(\alpha)} (x) = (d/dx)^\alpha \big( \lp\log x\rp^j \varphi (x) \big) $.  We then write  each   term in the expansion of $\EM \upsilon^{(\alpha)} (s + \alpha)$ in the same fashion as \eqref{2eq: poles of Mellin} and apply  \eqref{1eq: asymptotics R+, 2} and \eqref{1eq: Schwartz R+} to estimate the first and the last integral respectively. We conclude that
	\begin{equation*}
		\EM \upsilon (s) \lll_{\,j ,\, \alpha,\, A,\, a,\, b} \frac 1 { \left| (s)_{\alpha} \right|} \lp 1 + \sum_{\kappa = 0}^{\alpha + A} \lp \frac 1 {|s + \kappa|} + ... +  \frac 1 {|s + \kappa|^{j+1}} \rp \rp,
	\end{equation*}
	for all $s \in \BS [a, b]$, with $b \geq a > -\alpha - A - 1$. In particular, \eqref{2eq: Mellin rapid decay} is proven.
	
	\vskip 7 pt
	
	Let $H \in\Msis^{\lambda, j}$. Suppose that the  principal part of $H (s)$ at $s = \lambda - \kappa$ is equal to $ {(-)^j j! a_\kappa } {(s + \lambda + \kappa)^{-j-1}}$ and that $H (s)$ satisfies the condition \eqref{2eq: Mellin rapid decay}.
	
	\vskip 3 pt
	{\textsc {Claim}} 2. If we denote by $\upsilon (x)$ the following integral
	\begin{equation*}
		\upsilon (x) = \frac 1 {2\pi i}\int_{(\sigma)} H(s) x^{-s} d s, \hskip 10 pt \sigma > \Re \lambda,
	\end{equation*}
	then all the derivatives of $\varphi (x) = x^{\lambda} (\log x)^{- j}\upsilon (x)$ satisfy the asymptotics in \eqref{1eq: asymptotics R+, 3} at zero and rapidly decay at infinity.

	\vskip 5 pt
	
	{\textsc {Proof of Claim}} 2. Again, let us assume $\lambda = 0$. 
	
	Let $1 > \rho > 0$. We left shift the  contour of integration from $  (\sigma) $ 
	to $  (- \rho)  $. When moving across $s = 0$, we obtain
	$a_{0} (\log x)^j$ in view of Cauchy's differentiation formula\footnote{Recall   Cauchy's differentiation formula,  $$f^{(j)} (\zeta) = \frac {j !} {2\pi i} \sideset{}{_{ \partial \BB_r (\zeta) } }\oint \frac {f(s)} {(s-\zeta)^{j+1}} d s,$$ where $f$ is a holomorphic function on a neighborhood of the closed disc $\overline \BB_r (\zeta)$ centered at $\zeta$, and the integral is  taken counter-clockwise on the circle $\partial \BB_r (\zeta) $. In the present situation, this formula is applied for $f(s) = x^{- s}$.}. It follows that
	\begin{equation*}
		\upsilon (x) = a_{0} (\log x)^j + \frac 1 {2\pi i} \int_{(- \rho )}  H(s) x^{-s} d s.
	\end{equation*}
	Using \eqref{2eq: Mellin rapid decay} with $r$ small, say $r < \rho $, to estimate the above integral, we arrive at
	\begin{equation*}
		\upsilon (x) = a_{0} (\log x)^j + O \big(x^{\, \rho } \big) = (\log x)^j \lp a_{0} + O (x^{\, \rho}) \rp, \ \text{ as } x \ra 0.
	\end{equation*}
	Thus  $\varphi (x) = (\log x)^{- j} \upsilon (x)$ satisfies the asymptotic \eqref{1eq: asymptotics R+, 3} with $\alpha = 0$. For the general case $\alpha \in \BN$, we have
	\begin{equation}\label{2eq: upsilon (alpha)}
		\upsilon^{(\alpha)} (x) = (-)^\alpha \frac {1} {2\pi i} \int_{(\sigma)} (s)_{\alpha} H(s) x^{-s - \alpha} d s.
	\end{equation}
	Shifting the contour from  $  (\sigma) $ to  $ (- \alpha - \rho) $ and following the same lines of arguments as above, combined with some straightforward algebraic manipulations, one may show \eqref{1eq: asymptotics R+, 3} by an induction. 
	
	We are left to show the Schwartz condition for $\varphi (x) = (\log x)^{- j} \upsilon (x)$, or equivalently, that for $\upsilon (x)$. Indeed, the bound \eqref{1eq: Schwartz R+} for $\upsilon^{(\alpha)} (x)$ follows from right shifting the contour of the integral in \eqref{2eq: upsilon (alpha)} to the vertical line $ (\beta) $ 
	and applying the 
	estimates in \eqref{2eq: Mellin rapid decay}. 
\end{proof}

\subsubsection{The Spaces $\Ssis (\BR_+)$ and $\Msis $}

Let $\lambda , \lambda' \in \BC$. We write $\lambda \preccurlyeq_1 \lambda'$ if $\lambda'  - \lambda  \in \BN$ and  
$\lambda \sim_1 \lambda'$ if $\lambda'  - \lambda  \in \BZ$. Observe that ``$\preccurlyeq_1$'' and ``$\sim_1$'' define an  order relation and an equivalence relation on $\BC$ respectively.

Define
\begin{equation*}
	\Ssis (\BR _+) = \sum_{\lambda \in \BC} \sum_{j \in \BN} x^{-\lambda}  (\log x)^j \SS (\overline \BR _+),
\end{equation*}
where the sum $\sum_{\lambda \in \BC}  \sum_{j \in \BN}$ is in the {\it algebraic} sense.
It is clear that
$\lambda  \preccurlyeq_1 \lambda' $ if and only if $ x^{-\lambda}  (\log x)^j \SS (\overline \BR _+) \subseteq  x^{-\lambda'}  (\log x)^{j} \SS (\overline \BR _+)$. 
One also observes that $ x^{-\lambda}  (\log x)^j \SS (\overline \BR _+) \cap  x^{-\lambda'}  (\log x)^{j'} \SS (\overline \BR _+) = \SS (\BR_+)$ if either $j \neq j'$ or $\lambda \nsim_1 \lambda'$. Therefore,
\begin{equation}\label{2eq: decompose Ssis (R+)}
	\begin{split}
		\Ssis (\BR_+)/ \SS (\BR _+) = \bigoplus_{ \omega\, \in \BC/ {\scriptscriptstyle \sim_1} } \bigoplus_{j \in \BN} \ \varinjlim_{\lambda \in \omega}  \big( x^{-\lambda}  (\log x)^j \SS (\overline \BR _+) \big) \big/\SS (\BR _+).
	\end{split}
\end{equation}
Here the direct limit $\varinjlim_{\lambda \in \omega} $ is taken on the totally ordered set $(\omega, \preccurlyeq_1)$ and may be simply viewed as the union $ \bigcup_{\lambda \in \omega}$.
More precisely, each function $\upsilon \in \Ssis (\BR_+)$ can be expressed as a sum
\begin{equation*}
	\upsilon (x) = \upsilon^{0} (x) + \sum_{\lambda \in \Lambda} \sum_{j=0}^{N } x^{- \lambda } (\log x)^j \upsilon_{\lambda, j } (x),
\end{equation*}
with $\Lambda \subset \BC$ a finite set such that $\lambda  \nsim_1 \lambda'$ for any two distinct points $\lambda, \lambda' \in \Lambda$,  $ N \in \BN$, $\upsilon^{0} \in \SS (\BR _+)$ and $\upsilon_{\lambda, j } \in \SS (\overline \BR _+)$. 
This expression is unique up to addition of Schwartz functions in $\SS (\BR _+)$.

On the other hand, we define the space $\Msis$ of all meromorphic functions $H $   satisfying the following conditions,
\begin{itemize}
	\item[-] the poles of $H$ lie in a finite number of sets $ \lambda - \BN $, 
	\item[-] the orders of the poles of $H$ are uniformly bounded, and
	\item[-] $H$ decays rapidly along vertical lines, uniformly on vertical strips. 
\end{itemize}
Appealing to certain Gamma identities for  the Gamma function in \cite[Lemma 6.24]{Miller-Schmid-2006}, one may show, in the same way as \cite[Lemma 6.35]{Miller-Schmid-2006}, that
\begin{equation*}
	\Msis = \sum_{\lambda \in \BC} \sum_{j \in \BN} \Msis^{\lambda, j}.
\end{equation*}
We have $\Msis^{\lambda, j} \subseteq \Msis^{\lambda', j}$ if and only if $\lambda \preccurlyeq_1 \lambda'$, and $\Msis^{\lambda, j} \cap \Msis^{\lambda', j'} = \Hrd $ if either $j \neq j'$ or $\lambda \nsim \lambda'$. Therefore
\begin{equation}\label{2eq: decompose Msis}
	\Msis / \Hrd
	= \bigoplus_{\omega\, \in \BC/ {\scriptscriptstyle \sim_1} } \bigoplus_{j \in \BN} \ \varinjlim_{\lambda \in \omega}  \Msis^{\lambda, j} \big/ \Hrd.
\end{equation}

The following lemma is a direct consequence of Lemma \ref{2lem: Mellin isomorphism, Msis lambda j}.
\begin{lem}\label{2lem: Sis and Msis, R+}
	The Mellin transform $\EM$ is an isomorphism between $\Ssis (\BR _+)$ and $\Msis$ which respects their decompositions \eqref{2eq: decompose Ssis (R+)} and \eqref{2eq: decompose Msis}.
\end{lem}

\subsubsection{More Refined Decompositions of $\Ssis (\BR _+)$ and $\Msis$}\label{sec: refinements Msis Sis, R+}

Alternatively, we define  an order relation on $\BC$,  $\lambda  \preccurlyeq_2 \lambda' $ if $\lambda'  - \lambda  \in 2 \BN$, as well as  an equivalence relation, $\lambda \sim_2 \lambda'$ if $\lambda'  - \lambda  \in 2 \BZ$. 


Define  $\Nsis^{\lambda, j}$ in the same way as $\Msis^{\lambda, j}$ with $ \lambda - \BN$ replaced by $ \lambda - 2 \BN$. Under the isomorphism via $\EM$ in Lemma \ref{2lem: Mellin isomorphism, Msis lambda j}, $\Nsis^{\lambda, j}$ is then isomorphic to $x^{-\lambda} (\log x)^j \SS_0 (\overline \BR _+)$.

According to \cite[Lemma 6.35]{Miller-Schmid-2006}, we have the following decomposition,
\begin{equation} \label{2eq: Msis = Nsis + Nsis}
	\Msis^{\lambda, j} / \Hrd = \Nsis^{\lambda, j}  / \Hrd \oplus \Nsis^{\lambda-1, j}  / \Hrd.
\end{equation}
Inserting this into \eqref{2eq: decompose Msis}, we obtain   the following refined decomposition of $\Msis / \Hrd$
\begin{equation*}
	\begin{split}
		\bigoplus_{\omega\, \in \BC/ {\scriptscriptstyle \sim_1} } \bigoplus_{j \in \BN} \ \varinjlim_{\lambda \in \omega} \lp \Nsis^{\lambda, j} \big/ \Hrd \oplus 
		\Nsis^{\lambda-1, j} \big/ \Hrd \rp 
		= \bigoplus_{\omega\, \in \BC/ {\scriptscriptstyle \sim_2} } \bigoplus_{j \in \BN} \ \varinjlim_{\lambda \in \omega}  \Nsis^{\lambda, j} \big/ \Hrd.
	\end{split}
\end{equation*}
Under the isomorphism via $\EM$ in Lemma \ref{2lem: Sis and Msis, R+}, the reflection of this refinement  on the decomposition of $\Ssis (\BR _+)/ \SS (\BR _+)$ is
\begin{equation*} 
	\begin{split}
		\bigoplus_{\omega\, \in \BC/ {\scriptscriptstyle \sim_2} } \bigoplus_{j \in \BN} \ \varinjlim_{\lambda \in \omega} \big( x^{-\lambda} (\log x)^j \SS_0 (\overline \BR _+) \big) \big/\SS (\BR _+).
	\end{split}
\end{equation*}

\begin{lem}\label{2lem: refined decomp, R+}
	We have the following refinements of the decompositions {\rm (\ref{2eq: decompose Ssis (R+)}, \ref{2eq: decompose Msis})},
	\begin{equation} \label{2eq: refined, Ssis (R+)}
		\begin{split}
			\Ssis (\BR_+)/ \SS (\BR _+) = \bigoplus_{\omega\, \in \BC/ {\scriptscriptstyle \sim_2} } \bigoplus_{j \in \BN} \ \varinjlim_{\lambda \in \omega}  \big( x^{-\lambda} (\log x)^j \SS_0 (\overline \BR _+) \big) \big/\SS (\BR _+).
		\end{split}
	\end{equation}
	\begin{equation}\label{2eq: refined, Msis}
		\Msis / \Hrd
		= \bigoplus_{\omega\, \in \BC/ {\scriptscriptstyle \sim_2} }  \bigoplus_{j \in \BN} \ \varinjlim_{\lambda \in \omega} \Nsis^{\lambda, j} \big/ \Hrd.
	\end{equation}
	The Mellin transform $\EM$ respects these two decompositions.
\end{lem}

\begin{cor}\label{2cor: refined decomp, R+} Let $\delta \in \BZT $ and $m \in \BZ$, and recall the definitions of  $\SS_\delta (\overline \BR _+)$ and $\SS_m (\overline \BR _+)$ in {\rm \S \ref{sec: Schwartz subspaces}}.
	
	{\rm (1).}  The Mellin transform $\EM$ respects the following decompositions,
	\begin{equation}\label{2eq: refined, Ssis (R+), delta}
		\begin{split}
			\Ssis (\BR_+)/ \SS (\BR _+) = \bigoplus_{\omega\, \in \BC/ {\scriptscriptstyle \sim_2} } \bigoplus_{j \in \BN} \ \varinjlim_{\lambda \in \omega}  \big( x^{-\lambda} (\log x)^j \SS_\delta (\overline \BR _+) \big) \big/\SS (\BR _+),
		\end{split}
	\end{equation}
	\begin{equation}\label{2eq: refined, Msis, delta}
		\Msis / \Hrd
		= \bigoplus_{\omega\, \in \BC/ {\scriptscriptstyle \sim_2} } \bigoplus_{j \in \BN} \ \varinjlim_{\lambda \in \omega} \Nsis^{\lambda - \delta, j} \big/ \Hrd.
	\end{equation}
	
	{\rm (2).} The Mellin transform $\EM$ respects the following decompositions,
	\begin{equation}\label{2eq: refined, Ssis (R+), m}
		\begin{split}
			\Ssis (\BR_+)/ \SS (\BR _+) = \bigoplus_{\omega\, \in \BC/ {\scriptscriptstyle \sim_2} } \bigoplus_{j \in \BN} \ \varinjlim_{\lambda \in \omega}  \big( x^{-\lambda} (\log x)^j \SS_m (\overline \BR _+) \big) \big/\SS (\BR _+),
		\end{split}
	\end{equation}
	\begin{equation}\label{2eq: refined, Msis, m}
		\Msis / \Hrd
		= \bigoplus_{\omega\, \in \BC/ {\scriptscriptstyle \sim_2} } \bigoplus_{j \in \BN} \ \varinjlim_{\lambda \in \omega} \Nsis^{\lambda - |m|, j} \big/ \Hrd.
	\end{equation}
\end{cor}

\begin{proof}
	These follow from Lemma \ref{2lem: refined decomp, R+} in conjunction with  $x^{\delta} \SS_0 (\overline \BR_+) = \SS_\delta (\overline \BR_+) $ and $x^{|m|} \SS_0 (\overline \BR_+) = \SS_m (\overline \BR_+)$.
\end{proof}

\begin{rem}\label{rem: S=S0+S1}
	Set $\lambda = 0$ and $j = 0$ in \eqref{2eq: Msis = Nsis + Nsis}. It follows from the isomorphism $\EM$ the decomposition as below,
	\begin{equation*}
		\SS  (\overline \BR _+) /\SS (\BR _+) = \SS_0 (\overline \BR _+) /\SS (\BR _+) \oplus x \SS_0 (\overline \BR _+) /\SS (\BR _+).
	\end{equation*}
	Since $x \SS_0 (\overline \BR_+) = \SS_1 (\overline \BR_+) $, one obtains $\SS (\overline \BR _+) = \SS_0 (\overline \BR _+) + \SS_1 (\overline \BR _+)$ and therefore $C^\infty (\overline \BR _+) = C^\infty_0 (\overline \BR _+) + C^\infty_1 (\overline \BR _+)$. See Remark {\rm \ref{rem: C=C0+C1}}.
\end{rem}

\subsection{The Spaces $\Ssis (\BRx)$ and $\Msis^\BR$}\label{sec: Ssis and Msis, R}


Following \cite[(6.10)]{Miller-Schmid-2006}, we  write $(\lambda , \delta ) \preccurlyeq ( \lambda' , \delta')$ if $\lambda'  - \lambda  \in \BN$ and $\lambda'  - \lambda  \equiv \delta' + \delta  (\mod 2)$ and $(\lambda , \delta ) \sim ( \lambda' , \delta')$ if $\lambda'  - \lambda - (\delta' + \delta) \in 2 \BZ$. Again, these define an order relation and an equivalence relation on  $\BC \times \BZT$.

\vskip 5 pt 
\subsubsection{The Space $\Ssis (\BRx)$}\label{sec: defn Ssis (Rx)}



According to  \cite[Definition 6.4]{Miller-Schmid-2006} and \cite[Lemma 6.35]{Miller-Schmid-2006}, define
\begin{equation*}
	\Ssis (\BRx) =  \sum_{\delta \in \BZT} \ \sum_{\lambda \in \BC } \ \sum_{j \in \BN}  \sgn(x)^{\delta} |x|^{-\lambda} (\log |x|)^j \SS (\BR).
\end{equation*}
We have the following decomposition,
\begin{align*}
	\Ssis (\BRx)/ \SS (\BRx)  
	=   \bigoplus_{\omega\, \in \BC \times \BZT / {\scriptscriptstyle \sim } } \bigoplus_{j \in \BN}  \ \varinjlim_{(\lambda, \delta)\, \in \omega} \big( \sgn(x)^{\delta} |x|^{-\lambda} (\log |x|)^j \SS (\BR) \big) \big/\SS (\BRx).
\end{align*}
It follows from $\sgn(x) |x| \SS (\BR) = x \SS (\BR)  \subset \SS (\BR)$ that
\begin{equation*}
	\Ssis (\BRx)/ \SS (\BRx) = \bigoplus_{\omega\, \in \BC/ {\scriptscriptstyle \sim_2} } \bigoplus_{j \in \BN} \ \varinjlim_{\lambda \in \omega}  \big(  |x|^{-\lambda} (\log |x|)^j \SS (\BR) \big) \big/\SS (\BRx).
\end{equation*}
We let $\Ssis^\delta (\BRx)$ denote the space of functions $\upsilon \in \Ssis (\BRx)$ satisfying the parity condition \eqref{1eq: delta condition, R}. Clearly, $\Ssis (\BRx) = \Ssis^0 (\BRx) \oplus \Ssis^1 (\BRx)$. Then, 
\begin{equation*}
	\Ssis^\delta (\BRx)/ \SS_\delta (\BRx) = \bigoplus_{\omega\, \in \BC/ {\scriptscriptstyle \sim_2} } \bigoplus_{j \in \BN} \ \varinjlim_{\lambda \in \omega}    \big(  |x|^{-\lambda} (\log |x|)^j \SS_\delta (\BR) \big) \big/\SS_\delta (\BRx),
\end{equation*}
where $\SS_\delta (\BR)$ and $\SS_\delta (\BRx)$ are defined in \S \ref{sec: Schwartz subspaces} and \S \ref{sec: S delta and Sm} respectively.
Since 
$ \SS_\delta (\BR) = \sgn(x)^\delta \SS_\delta (\overline \BR _+)  $, 
\begin{equation}\label{2eq: decompose Ssis (Rx), 2}
	\Ssis^\delta (\BRx)/ \SS_\delta (\BRx) = \bigoplus_{\omega\, \in \BC/ {\scriptscriptstyle \sim_2} } \bigoplus_{j \in \BN} \ \varinjlim_{\lambda \in \omega}  \big( \sgn(x)^{\delta } |x|^{-\lambda} (\log |x|)^j \SS_\delta (\overline \BR_+) \big) \big/\SS_\delta (\BRx).
\end{equation}
Then, $  |x|^{- \delta} \SS_{\delta} (\overline \BR_+) = \SS_{0} (\overline \BR_+)$ together with $ \SS_0 (\overline \BR_+)/\SS (\BR _+) \oplus |x| \SS_0 (\overline \BR_+)/\SS (\BR _+) = \SS (\overline \BR_+)/\SS (\BR _+)$ (see Remark \ref{rem: S=S0+S1}) yields
\begin{equation}\label{2eq: decompose Ssis (Rx)}
	\Ssis^\delta (\BRx)/ \SS_\delta (\BRx) = \bigoplus_{\omega\, \in \BC/ {\scriptscriptstyle \sim_1} } \bigoplus_{j \in \BN} \ \varinjlim_{\lambda \in \omega}  \big( \sgn(x)^{\delta } |x|^{-\lambda} (\log |x|)^j \SS (\overline \BR_+) \big) \big/\SS_\delta (\BRx).
\end{equation}
In particular,
\begin{equation*} 
	\Ssis^\delta (\BRx) = \sgn(x)^{\delta } \Ssis (\BR_+) = \left \{ \sgn (x)^\delta \upsilon  (|x|) : \upsilon \in \Ssis ( \BR_+) \right \}.
\end{equation*}

\subsubsection{The Space $\Msis^\BR $}
We simply define $\Msis^{\BR} = \Msis \times \Msis$. 
 
\vskip 5 pt 

\subsubsection{Isomorphism between  $\Ssis (\BRx)$ and $\Msis^\BR$ via the Mellin Transform $\EM_{\BR}$}

Let  $\upsilon \in \Ssis (\BRx)$. Since $\upsilon_\delta \in \Ssis (\overline \BR_ +)$, the identity  $\EM_\delta \upsilon (s) = 2 \EM \upsilon_\delta (s)$ in \eqref{1eq: M delta = M}
extends the definition of the Mellin transform $\EM_\delta$ onto the space  $\Ssis (\BRx)$. Therefore, as a consequence of Lemma \ref{2lem: Mellin isomorphism, Msis lambda j},   \ref{2lem: Sis and Msis, R+} and Corollary \ref{2cor: refined decomp, R+} (1), the following lemma is readily established.

\begin{lem}\label{2lem: Ssis to Msis, R}
	For $\delta \in \BZT$, the Mellin transform $\EM_\delta$ establishes an isomorphism between the spaces $\Ssis^\delta (\BRx)$ and $\Msis$ which respects their decompositions \eqref{2eq: decompose Ssis (Rx)} and \eqref{2eq: decompose Msis}  as well as \eqref{2eq: decompose Ssis (Rx), 2} and \eqref{2eq: refined, Msis, delta}. Therefore, $\EM^\BR = (\EM_0, \EM_1)$ establishes an isomorphism between $ \Ssis (\BRx) = \Ssis^0 (\BRx) \oplus \Ssis^1 (\BRx)$ and $\Msis^\BR = \Msis \times \Msis$.
\end{lem}


\subsubsection{An Alternative Decomposition of $\Ssis^\delta (\BRx)$}\label{sec: alternative decomp, R} 

The following lemma follows from Corollary \ref{2cor: refined decomp, R+} (1) (compare \cite[Corollary 6.17]{Miller-Schmid-2006}).

\begin{lem}\label{2lem: Ssis delta, M delta}
	Let $\delta \in \BZ/2 \BZ$.  
	The Mellin transform $\EM_\delta$ respects the following  decompositions,
	\begin{equation}\label{2eq: refinement Ssis, R}
		\begin{split}
			 \Ssis^\delta (\BRx)/ & \SS_\delta (\BRx) \\
			& =  \bigoplus_{\omega\, \in \BC \times \BZT/ {\scriptscriptstyle \sim } } \bigoplus_{j \in \BN} \ \varinjlim_{(\lambda, \epsilon)\, \in \omega} \big( \sgn(x)^{ \epsilon} |x|^{-\lambda} (\log |x|)^j \SS_{ \epsilon + \delta} ( \BR ) \big) \big/\SS_\delta (\BRx),
		\end{split}
	\end{equation}
	\begin{equation}\label{2eq: refinement Msis, R}
		\Msis / \Hrd
		= \bigoplus_{\omega\, \in \BC \times \BZT/ {\scriptscriptstyle \sim } } \bigoplus_{j \in \BN} \ \varinjlim_{(\lambda, \epsilon)\, \in \omega} \Nsis^{\lambda - (\epsilon + \delta), j} \big/ \Hrd.
	\end{equation}
\end{lem}

\subsection{The Spaces $\Ssis (\BCx)$ and $\Msis^\BC$}\label{sec: Ssis (Cx) and Msis C}

We write $(\lambda , m ) \preccurlyeq ( \lambda' , m')$ if $\lambda'  - \lambda  \in |m'-m| + 2\BN$ and  $(\lambda , m) \sim ( \lambda' , m')$ if $\lambda'  - \lambda - |m' - m| \in 2 \BZ$.  These define an order relation and  an equivalence relation on $\BC \times \BZ$. 

\vskip 5 pt 

\subsubsection{The Space $\Ssis (\BCx)$}
In parallel to \S \ref{sec: defn Ssis (Rx)}, we first define
\begin{equation*}
	\Ssis (\BCx) =  \sum_{m \in \BZ } \ \sum_{\lambda \in \BC } \ \sum_{j \in \BN}  [z]^{-m} |z|^{-\lambda} (\log |z|)^j \SS (\BC).
\end{equation*}
We have the following decomposition,
\begin{equation*}
	\begin{split}
		\Ssis (\BCx)/ \SS (\BCx) =  \bigoplus_{\omega\, \in \BC \times \BZ / {\scriptscriptstyle \sim } } \bigoplus_{j \in \BN}  \ \varinjlim_{(\lambda, m)\, \in \omega}   \big( [z]^{-m} |z|^{-\lambda} (\log |z|)^j \SS (\BC) \big) \big/\SS (\BCx).
	\end{split}
\end{equation*}
It follows from $[z] |z| \SS (\BC) = z \SS (\BC)  \subset \SS (\BC)$ that
\begin{equation}\label{2eq: decompose Ssis (Cx)}
	\Ssis (\BCx)/ \SS (\BCx) =  \bigoplus_{\omega\, \in \BC/ {\scriptscriptstyle \sim_2} } \bigoplus_{j \in \BN} \ \varinjlim_{\lambda \in \omega}  \big(  |z|^{-\lambda} (\log |z|)^j \SS (\BC) \big) \big/\SS (\BCx).
\end{equation}
We let $\Ssis^m (\BCx)$ denote the space of functions $\upsilon \in \Ssis (\BCx)$ satisfying \eqref{1eq: m condition, C}.
Then,
\begin{equation*}
	\Ssis^m (\BCx)/ \SS_m (\BCx) =  \bigoplus_{\omega\, \in \BC/ {\scriptscriptstyle \sim_2} } \bigoplus_{j \in \BN} \ \varinjlim_{\lambda \in \omega}   \big(  |z|^{-\lambda} (\log |z|)^j \SS_m (\BC) \big) \big/\SS_m (\BCx),
\end{equation*}
where $\SS_m (\BC)$ and $\SS_m (\BCx)$ are defined in \S \ref{sec: Schwartz subspaces} and \S \ref{sec: S delta and Sm} respectively.
Since 
$ \SS_m (\BC) = [z]^m \SS_m (\overline \BR _+) $, 
\begin{equation}\label{2eq: decompose Ssis m (Cx), 2}
	\Ssis^m (\BCx)/ \SS_m (\BCx) =  \bigoplus_{\omega\, \in \BC/ {\scriptscriptstyle \sim_2} } \bigoplus_{j \in \BN} \ \varinjlim_{\lambda \in \omega} \big( [z]^{m } |z|^{-\lambda} (\log |z|)^j \SS_m (\overline \BR_+) \big) \big/\SS_m (\BCx).
\end{equation}
Then, $  |z|^{- |m|} \SS_{m} (\overline \BR_+) = \SS_{0} (\overline \BR_+)$ together with  $ \SS_0 (\overline \BR_+)/\SS (\BR _+) \oplus |z| \SS_0 (\overline \BR_+)/\SS (\BR _+) = \SS (\overline \BR_+)/\SS (\BR _+)$  yields
\begin{equation}\label{2eq: decompose Ssis m (Cx)}
	\Ssis^m (\BCx)/ \SS_m (\BCx) =  \bigoplus_{\omega\, \in \BC/ {\scriptscriptstyle \sim_1} } \bigoplus_{j \in \BN} \ \varinjlim_{\lambda \in \omega} \big( [z]^{m } |z|^{-\lambda} (\log |z|)^j \SS (\overline \BR_+) \big) \big/\SS_m (\BCx).
\end{equation}
In particular,
\begin{equation*} 
	\Ssis^m (\BCx) = [z]^{m } \Ssis (\BR_+) = \left \{ [z]^m \upsilon  (|z|) : \upsilon \in \Ssis ( \BR_+) \right \}.
\end{equation*}

\subsubsection{The Space $\Msis^\BC$}\label{sec: Msis C}

For $\lambda \in \BC $ and $j \in \BN$, we define the space  $\Nsis^{\BC, \lambda, j}$ of all sequences $\lpp H_m (s) \rpp $ of meromorphic functions  such that
\begin{itemize}
	\item[-] the only singularities of $H_m (s)$ are poles of pure order  $j+1$ at the points in  $ \lambda - |m| - 2 \BN $,
	\item[-] Each $H_m (s)$ decays rapidly along vertical lines, uniformly on vertical strips 
	(see \eqref{2eq: Mellin rapid decay, 0}), and
	\item[-] $H_m (s)$ also decays rapidly with respect to $m$, uniformly on vertical strips, 
	in the sense that
	\item[\noindent \refstepcounter{equation}(\theequation)  \label{2eq: rapid decay Nsis, C, 0} \hskip 8 pt] for any given $ \alpha, A \in \BN $ and vertical strip $ \BS [a, b]$,
	\vskip 5 pt 
	\noindent $H_m (s) \lll_{\, \lambda, j,\, \alpha,\, A,\, a,\, b } (|m| + 1)^{-A} (|\Im s| + 1)^{-\alpha}  \text{ for all }  s \in \BS[a, b],$
	if $|m| > \Re \lambda - a$.
\end{itemize}

Observe that the first two conditions amount to $H_m \in \Nsis^{ \lambda - |m|, j}$. Therefore, $\Nsis^{\BC, \lambda, j} \subset \prod_{m\in \BZ} \Nsis^{ \lambda - |m|, j}$.


Define   the space $\Msis^{\BC}$ of all sequences $\lpp H_m \rpp $ of meromorphic functions  such that
\begin{itemize}
	\item[-] the poles of each $H_m$ lie in  $ \lambda - |m| - 2 \BN $, for a finite number of $\lambda$,
	\item[-] the orders of the poles of $H_m$ are uniformly bounded,
	\item[-] Each $H_m $ decays rapidly along vertical lines, uniformly on vertical strips, 
	and
	\item[-] $H_m $ decays rapidly with respect to $m$, uniformly on vertical strips. 
\end{itemize}
Using the refined Stirling's asymptotic formula \eqref{1eq: Stirling's formula} in  place of \cite[(6.22)]{Miller-Schmid-2006} and the following bound  in  place of \cite[(6.23)]{Miller-Schmid-2006}
\begin{equation*}
	\begin{split}
		\left|\frac { \Gamma^{(j)} \lp \frac 1 2 {(s - \lambda + |m| )}    \rp } { \Gamma \lp \frac 1 2 {(s + |m|) }    \rp } \right| \lll_{\,\lambda, j,\, a,\, b,\, r}   
		\lp |\Im s| + |m| + 1 \rp^{- \frac 1 2 {\Re \lambda}  }
	\end{split}
\end{equation*}
for $\lambda \in \BC$, $j \in \BN$,  $  s \in \BS [a, b]  \smallsetminus \bigcup_{\sstyle \kappa \geq |m| \atop \sstyle \kappa \equiv m (\mod 2)} \BB_{r} (\lambda - \kappa)$, with $ r > 0$, we may follow the same lines of the proofs of \cite[Lemma 6.24]{Miller-Schmid-2006} and \cite[Lemma 6.35]{Miller-Schmid-2006} to show that 
\begin{equation*}
	\Msis^\BC = \sum_{\lambda \in \BC} \sum_{j \in \BN} \Nsis^{\BC, \lambda, j},
\end{equation*}
and consequently
\begin{equation}\label{2eq: decompose Msis, C}
	\Msis^\BC / \Hrd^\BC
	=  \bigoplus_{\omega\, \in \BC/ {\scriptscriptstyle \sim_2} } \bigoplus_{j \in \BN} \ \varinjlim_{\lambda \in \omega} \Nsis^{\BC, \lambda, j} \big/ \Hrd^\BC.
\end{equation}

\subsubsection{Isomorphism between $\Ssis (\BCx)$ and $\Msis^\BC$ via the Mellin Transform $\EM_{\BC}$}

For $\upsilon \in \Ssis (\BCx)$, its $m$-th Fourier coefficient $\upsilon_m$ is a function in $ \Ssis (\overline \BR_ +)$. Hence the identity $\EM_{-m} \upsilon (s) = 4 \pi \EM \upsilon_m (s)$ in \eqref{1eq: Mm = M} extends the definition of the Mellin transform $\EM_{-m}$ onto the space $\Ssis (\BCx)$.

\begin{lem}\label{2lem: Ssis to Msis, C}
	For $m \in \BZ$, the Mellin transform $\EM_{-m}$ establishes an isomorphism between the spaces $\Ssis^{ m} (\BCx)$ and $\Msis$ which respects their decompositions \eqref{2eq: decompose Ssis m (Cx)} and \eqref{2eq: decompose Msis} as well as \eqref{2eq: decompose Ssis m (Cx), 2} and \eqref{2eq: refined, Msis, m}.
	Furthermore, $\EM_{\BC} = \prod_{m\in \BZ} \EM_{-m}$ establishes an isomorphism between $ |z|^{-\lambda} (\log |z|)^j \SS (\BC ) $ and $\Nsis^{\BC, \lambda, j}$ for any $\lambda \in \BC$ and $j \in \BN$, and hence an isomorphism between $ \Ssis(\BCx) $ and $\Msis^\BC$ which respects their decompositions  \eqref{2eq: decompose Ssis (Cx)} and \eqref{2eq: decompose Msis, C}.
\end{lem}

\begin{proof}
	For $\upsilon \in \Ssis^m (\BCx)$, one has $\upsilon \lp x e^{i \phi}\rp = e^{im \phi} \upsilon_m (x)$ and  $\upsilon_m \in \Ssis (\overline \BR_ +)$. Thus the first assertion follows immediately from Lemma \ref{2lem: Sis and Msis, R+} and Corollary \ref{2cor: refined decomp, R+} (2).

	Now let $\varphi \in \SS (\BC)$ and $ \upsilon (z) = |z|^{-\lambda} (\log |z|)^j \varphi (z)$. Clearly, their $m$-th Fourier  coefficients are related by $\upsilon_m (x) = x^{-\lambda} (\log x)^j \varphi_{ m} (x)$.
	Since $\varphi_{ m} \in  \SS_m (\overline \BR _+)$, it follows from Corollary \ref{2cor: refined decomp, R+} (2) that $H_{ m} = \EM_{-m} \upsilon = 4\pi \EM \upsilon_m$ lies in $ \Nsis^{\lambda - |m|, j}$, and therefore we are left to show \eqref{2eq: rapid decay Nsis, C, 0}. Recall that in the proof of Lemma \ref{2lem: Mellin isomorphism, Msis lambda j} we turned to verify \eqref{2eq: Mellin rapid decay} instead of \eqref{2eq: Mellin rapid decay, 0}. Likewise, it is more convenient to verify the following equivalent statement of \eqref{2eq: rapid decay Nsis, C, 0},
	\begin{itemize}
		\item[\noindent \refstepcounter{equation}(\theequation) \label{2eq: rapid decay Nsis, C, 1} \hskip 8 pt] for any given $\alpha,  A \in \BN, b \geq a > \Re \lambda - \alpha - A - 1$, 
		\vskip 5 pt 
		{\noindent $H_m (s) \lll_{\, \lambda, j,\, \alpha,\, A,\, a,\, b } (|m| + 1)^{-A} (|\Im s| + 1)^{-\alpha}  \text{ for all }  s \in \BS[a, b] $,  if $ |m| > \alpha + A$.}
	\end{itemize}
	
	According to Lemma \ref{lem: Schwartz} (3.1), $\varphi_{ m}$ satisfies the conditions (\ref{1eq: bounds for Fourier coefficients}, \ref{1eq: phi m, 2}). 
	Suppose $|m| > \alpha + A$. One directly applies \eqref{1eq: bounds for Fourier coefficients} and \eqref{1eq: phi m, 2} to bound the following integral by a constant multiple of $(|m| + 1)^{- A}$,
	\begin{equation*}
		(-)^\alpha (s - \lambda )_{\alpha} \EM \upsilon_m (s) = \int_0^\infty \frac {d^\alpha} {d x^\alpha} \lp  (\log x)^j \varphi_{ m} (x) \rp  x^{s - \lambda + \alpha - 1} d x.
	\end{equation*}
	This proves \eqref{2eq: rapid decay Nsis, C, 1} for $H_{ m} = 4\pi \EM \upsilon_m$.
	Therefore, the sequence $\lpp \EM_{-m} \upsilon \rpp $ belongs to $\Nsis^{\BC, \lambda, j}$. 
	
	Conversely, let $\lpp H_m \rpp  \in \Nsis^{\BC, \lambda, j}$, and let $ 4 \pi \upsilon_m$ be the Mellin inversion of $H_{ m}$, \begin{equation*} 
		\upsilon_m (x) = \frac 1 {8\pi^2 i}\int_{(\sigma)} H_{ m} (s) x^{-s} d s, \hskip 10 pt \sigma > \Re \lambda - |m|.
	\end{equation*} 
	Since $H_{ m} \in \Nsis^{\lambda - |m|, j}$, Corollary \ref{2cor: refined decomp, R+} (2) implies that $\upsilon_m (x) \in  x^{- \lambda} (\log x)^j  \SS_m (\overline \BR _+)$ and hence $\varphi_m (x) = x^{ \lambda} (\log x)^{-j} \upsilon_m (x)$ lies in $\SS_m (\overline \BR _+)$.
	This proves \eqref{1eq: phi m, 1}.
	Similar to the proof of Lemma \ref{2lem: Mellin isomorphism, Msis lambda j}, right shifting of the contour of integration combined with \eqref{2eq: rapid decay Nsis, C, 1} yields \eqref{1eq: bounds for Fourier coefficients}, whereas left shifting combined with \eqref{2eq: rapid decay Nsis, C, 1} yields \eqref{1eq: phi m, 2}\footnote{ Actually, $O_{\alpha,\, A} \lp (|m| + 1)^{- A} x^{A+1} \rp$ in  \eqref{1eq: phi m, 2} should be replaced by  $O_{\alpha,\, A, \, \rho} \lp (|m| + 1)^{- A} x^{A+ \rho} \rp$, $1 >  \rho > 0$. Moreover, one observes that the left contour shift here does not cross any pole.}. 
	
	The proof of the second assertion is completed.
\end{proof}

\subsubsection{An Alternative Decomposition of $\Ssis^m (\BCx)$}\label{sec: alternative decomp, C}

The following lemma follows from Corollary \ref{2cor: refined decomp, R+} (2).

\begin{lem}\label{2lem: Ssis m, M -m}
	Let $m \in \BZ $. 
	The Mellin transform $\EM_{-m} $ respects the following  decompositions,\begin{equation}\label{2eq: refinement Ssis, C}
		\begin{split}
			 \Ssis^m (\BCx)/  \SS_m & (\BCx) \\
			& =    \bigoplus_{\omega\, \in \BC \times \BZ / {\scriptscriptstyle \sim } } \bigoplus_{j \in \BN} \ \varinjlim_{(\lambda, k)\, \in \omega} \big( [z]^{-k} |z|^{-\lambda} (\log |z|)^j \SS_{m + k} ( \BC) \big) \big/\SS_m (\BCx),
		\end{split}
	\end{equation}
	\begin{equation}\label{2eq: refinement Msis, C}
		\Msis / \Hrd
		= \bigoplus_{\omega\, \in \BC \times \BZ / {\scriptscriptstyle \sim } } \bigoplus_{j \in \BN} \ \varinjlim_{(\lambda, k)\, \in \omega}  \Nsis^{\lambda - |m + k|, j} \big/ \Hrd.
	\end{equation}
\end{lem}

\section {Hankel Transforms and Bessel Kernels} \label{sec: Hankel transforms}

This section is arranged as follows. We start with the type of Hankel transforms over $\BR_+$ whose kernels are the Bessel functions that will be extensively studied in Chapter \ref{chap: analytic theory}. After this, we introduce two auxiliary Hankel transforms and Bessel kernels over $\BR_+$. Finally, we proceed to construct and study Hankel transforms and their Bessel kernels over $\BFx$, with $\BF = \BR, \BC$. 

\begin{defn}\label{3defn: ordered set}
	Let $ (\BX, \preccurlyeq) $ be an ordered set satisfying the condition that
	\begin{equation}\label{3eq: condition on order}
	\text{``}\lambda \preccurlyeq \lambda' \text{ or } \lambda' \preccurlyeq \lambda \text{'' is an equivalence relation.}
	\end{equation} 
	We denote the above equivalence relation  by $ \lambda \sim \lambda'$.  Given $ \ulambda = (\lambda_1, ..., \lambda_n) \in \BX^{n }$, the set $\{1, ... n\}$ is partitioned into several pair-wise disjoint subsets $L_\alpha$, $\alpha = 1, ..., A$, such that 
	$$ \lambda_{l     } \sim  \lambda_{l     '}  \text{ if and only if } l     , l     ' \text{ are in the same } L_\alpha.$$
	Each $\Lambda ^{\alpha} = \left \{ \lambda_{l     } \right \}_{l      \in L_\alpha}$\footnote{Here, $\left \{ \lambda_{l     } \right \}_{l      \in L_\alpha}$ is considered as a set, namely, $\lambda_{l     }$ are counted without multiplicity.} 
	is a totally ordered set. Let $B_{\alpha} = \left| \Lambda ^{\alpha} \right|$ and  label the elements of $\Lambda ^{\alpha}$  in the descending order, $ \lambda_{\alpha, 1} \succ ... \succ \lambda_{\alpha, B_\alpha}.$
	For $ \lambda_{\alpha, \beta} \in \Lambda ^{\alpha}$, let $M_{\alpha, \beta}$ denote the multiplicity of $\lambda_{\alpha, \beta}$ in $\ulambda$, that is, $M_{\alpha, \beta} = \left|\left\{ l      : \lambda_l      = \lambda_{\alpha, \beta} \right \} \right|$, and define $N_{\alpha, \beta} = \sum_{\gamma =1}^{\beta} M_{\alpha, \gamma} = \left|\left\{ l      : \lambda_{\alpha, \beta} \preccurlyeq \lambda_l      \right \} \right|$.
	
	$\ulambda$ is called generic if $\lambda_{l     } \nsim  \lambda_{l     '}$ for any $l      \neq l     '$.
\end{defn}

We recall that the ordered sets $(\BC, \preccurlyeq_1)$,  $(\BC, \preccurlyeq_2)$,  $(\BC \times \BZT, \preccurlyeq )$ and $(\BC \times \BZ, \preccurlyeq )$
defined in \S \ref{sec: all Ssis} all satisfy \eqref{3eq: condition on order}.

\begin{figure}
	\begin{center}
		\begin{tikzpicture}

		\draw [->] (-3.5,0) -- (1.5,0);
		\draw [->] (0, -3) -- (0, 3);
		\draw [-, thick] (1, -0.8) -- (1, 0.8);
		\draw [-, thick] (-0.5 , -3) -- (-0.5, -1.75);
		\draw [-, thick] (-0.5 , 1.75) -- (-0.5, 3);
		\draw [->] (-0.5 , 2) -- (-0.5, 2.5);
		\draw [->] (-0.5 , -2.6) -- (-0.5, -2.5);
		\draw [dotted] (-0.5 , -2) -- (-0.5, 2);
		\draw [-, thick] (1, -0.8) to [out=-90, in=90] (-0.5, -1.75);
		\draw [-, thick] (1, 0.8) to [out=90, in=-90] (-0.5, 1.75);
		
		\draw[fill=white] (-0.5, 0) circle [radius=0.04];
		\node [ below ] at (-0.65, 0.05) {\footnotesize $\sigma$};
		
		\draw[fill=white] (0.5, 0.3 ) circle [radius=0.04];
		\node [above ] at (0.5, 0.3 )  {\footnotesize $\lambda_{l     } - \kappa_{l     } $};
		\draw[fill=white] (-0.7, 0.3 ) circle [radius=0.04];
		\draw[fill=white] (-1.9, 0.3 ) circle [radius=0.04];
		\draw[fill=white] (-3.1, 0.3 ) circle [radius=0.04];
		
		\node [left ] at (-0.5, 2.2 )  {\footnotesize $\EC^d_{(\ulambda, \ukappa)}$};

		\draw [->] (7+-3.5,0) -- (7+1.5,0);
		\draw [->] (7+0, -1.7) -- (7+0, 1.7);
		\draw [-, thick] (7+0.8, -0.81) -- (7+0.8, 0.81);
		\draw [-, thick] (7+-3.5, 0.8) -- (7+0.8, 0.8);
		\draw [-, thick] (7+-3.5, -0.8) -- (7+0.8, -0.8);
		\draw [->] (7+-1, 0.8) -- (7+-2, 0.8);
		\draw [->] (7+-3.5, -0.8) -- (7+-2, -0.8);
		
		\draw[fill=white] (7+0.5, 0.3 ) circle [radius=0.04];
		\node [above ] at (7+0.5, 0.3 )  {\footnotesize $\lambda_{l     }$};
		\draw[fill=white] (7+-0.7, 0.3 ) circle [radius=0.04];
		\draw[fill=white] (7+-1.9, 0.3 ) circle [radius=0.04];
		\draw[fill=white] (7+-3.1, 0.3 ) circle [radius=0.04];
		
		\node [right ] at (7+0.8, 0.8)  {\footnotesize $\EC'_{\ulambda}$};
		
		\end{tikzpicture}
	\end{center}
	\caption{$\EC^d_{(\ulambda, \ukappa)}$ and $\EC'_{\ulambda}$}\label{fig: C d lambda}
\end{figure}

\begin{defn}\label{3defn: C d lambda}
	Let $d = 1$ or $2$, $\ulambda \in \BC^{n}$ and $\ukappa \in \BN^n$. Put $\sigma < \frac d 2 + \frac 1 {n} (\Re |\ulambda| - 1 )$ and choose
	a contour $ \EC^d _{(\ulambda, \ukappa)}$ {\rm(}see Figure {\rm \ref{fig: C d lambda}}{\rm)} such that
	\begin{itemize}
		\item[-] $\EC^d _{(\ulambda, \ukappa)}$ is  upward directed  from $\sigma - i \infty$ to $\sigma + i \infty$,
		\item[-] all the sets $\lambda_{l     } - \kappa_{l     } - \BN$ lie on the left  side of $\EC _{(\ulambda, \ukappa)}$, and
		\item[-] if $s \in \EC^d _{(\ulambda, \ukappa)}$ and $|\Im s| $ is sufficiently large, say $|\Im s| - \max  \left\{ |\Im \lambda_{l     }| \right \} \ggg 1$, then $\Re s = \sigma$.
	\end{itemize}
	For $\ulambda \in \BC$, we denote $\EC_{\ulambda} = \EC^1 _{(\ulambda, \boldsymbol 0)}$. For $(\umu, \udelta) \in \BC^n \times (\BZT)^n$, we denote $\EC_{(\umu, \udelta)} = \EC^1 _{(\umu, \udelta)}$. For $(\umu, \um) \in \BC^n \times \BZ^n$, we denote $\EC_{(\umu, \um)} = \frac 1 2\cdot \EC^2 _{(2 \umu, \|\um\|)}$.
\end{defn}

\begin{defn}\label{3defn: C ' lambda}
	For  $\ulambda \in \BC^{n}$, choose a contour $\EC'_{\ulambda}$ illustrated in Figure {\rm \ref{fig: C d lambda}} such that 
	\begin{itemize}
		\item[-] $\EC'_{\ulambda}$ starts from and returns to $- \infty$ counter-clockwise,
		\item[-] $\EC'_{\ulambda}$ consists two horizontal infinite half lines, 
		\item[-] $\EC'_{\ulambda}$ encircles all the sets $\lambda_{l     } - \BN$, and
		\item[-] $ \Im s \lll \max  \{ |\Im \lambda_l     | \} + 1$ for all $s \in \EC'_{\ulambda}$.
	\end{itemize}
\end{defn}

\subsection {\texorpdfstring{The Hankel Transform $ \Hsl$ and the Bessel Function $J(x; \usigma, \ulambda)$}{The Hankel Transform $ H_{(\varsigma, \lambda)}$ and the Bessel Function $J(x; \varsigma, \lambda)$}}

\subsubsection{\texorpdfstring{The Definition of  $ \Hsl$}{The Definition of  $ H_{(\varsigma, \lambda)}$}}

Consider the ordered set $(\BC, \preccurlyeq_1)$. For  $\ulambda \in \BC^{n}$, let notations $\lambda_{\alpha, \beta}$, $B_\alpha$, $M_{\alpha, \beta}$ and $N_{\alpha, \beta}$ be as in Definition \ref{3defn: ordered set}.
We define the following subspace of $ \Ssis (\BR _+)$,
\begin{equation}\label{3eq: Ssis(R+)}
\Ssis^{ \ulambda } (\BR _+) = \sum_{ \alpha = 1}^A \sum_{\beta = 1}^{B_\alpha} \sum_{j=0}^{N_{\alpha, \beta} - 1} x^{- \lambda_{\alpha, \beta}} (\log x)^{j } \SS (\overline \BR _+).
\end{equation}
\begin{prop}\label{3prop:H sigma lambda} 
	Let $( \usigma, \ulambda) \in  \{ +, - \}^n \times \BC^{n}$. Suppose $\upsilon \in \SS (\BR _+)$. Then there exists a unique function $\Upsilon \in \Ssis^{ \ulambda } (\BR _+) $ satisfying the following identity,
	\begin{equation}\label{3eq: Hankel transform identity 0, R+}
	\EM \Upsilon (s ) = G (s; \usigma, \ulambda) \EM \upsilon ( 1 - s).
	\end{equation}
	We call $\Upsilon$ the Hankel transform of  $\upsilon$ over $\BR _+$ of index $( \usigma, \ulambda)$ and write  $\Hsl \upsilon  = \Upsilon$.
\end{prop}
\begin{proof}
	Recall the definition  of $G (s; \usigma, \ulambda)$ given by (\ref{1def: G pm (s)}, \ref{1def: G(s; sigma; lambda)}),
	\begin{equation*}
	G (s; \usigma, \ulambda) = e \lp  \frac {\sum_{l      = 1}^n \varsigma_l      (s - \lambda_l     )} 4 \rp
	\prod_{l      = 1}^n \Gamma \lp s - \lambda_{l     } \rp.
	\end{equation*}
	The product in the above expression   may be rewritten as below
	\begin{equation*}
	\prod_{ \alpha = 1}^A \prod_{\beta = 1}^{B_\alpha}  \Gamma \lp s - \lambda _{\alpha, \beta} \rp^{M_{\alpha, \beta} }.
	\end{equation*}
	Thus the singularities of $G (s; \usigma, \ulambda)$ are poles at the points in $ \lambda _{\alpha, 1} - \BN$, $\alpha =1,..., A$. More precisely, $G (s; \usigma, \ulambda)$ has a pole of pure order $N_{\alpha, \beta}$ at $\lambda \in \lambda_{\alpha, 1} - \BN$ if one let  $\beta = \max \left\{ \beta' : \lambda \preccurlyeq_1 \lambda_{\alpha, \beta' } \right \} $. Moreover, in view of \eqref{1eq: vertical bound, G (s; sigma, lambda)} in Lemma \ref{1lem: vertical bound}, $G (s; \usigma, \ulambda)$ is of uniform moderate growth on vertical strips. 
	
	On the other hand, according to Corollary \ref{cor: Mellin, Schwartz} (1), $\EM \upsilon (1-s)$ uniformly rapidly decays on vertical strips. 
	
	Therefore, the product $G (s; \usigma, \ulambda) \EM \upsilon ( 1 - s)$ on the right hand side of \eqref{3eq: Hankel transform identity 0, R+} is a meromorphic function in the space $\sum_{ \alpha = 1}^A \sum_{\beta = 1}^{B_\alpha} \sum_{j=0}^{N_{\alpha, \beta} - 1}   \Msis^{ \lambda_{\alpha, \beta} , j}$. We conclude from Lemma \ref{2lem: Sis and Msis, R+} that  \eqref{3eq: Hankel transform identity 0, R+} uniquely determines a function  $\Upsilon$ in $ \Ssis^{ \ulambda } (\BR _+)$.
\end{proof}

\subsubsection{The Bessel Function $J(x; \usigma, \ulambda)$}\label{sec: Bessel kernel J(x; sigma, lambda)}
\
\vskip 5 pt 
{\it The Integral Kernel $J(x; \usigma, \ulambda)$ of $\Hsl$.}
Suppose $\upsilon \in \SS (\BR _+)$. By the Mellin inversion, we have
\begin{equation}\label{3eq: Bessel kernel, Mellin inversion, 1}
\Upsilon (x) 
= \frac { 1 } {2 \pi i} \int_{(\sigma)} G (s; \usigma, \ulambda) \EM  \upsilon (1 - s) x ^{ - s} d s, \hskip 10 pt \textstyle \sigma > \max  \left\{ \Re  \lambda_l      \right \}.
\end{equation}
It is an iterated double integral as below
\begin{equation*}
\Upsilon (x) 
= \frac { 1 } {2 \pi i} \int_{(\sigma)} \int_0^\infty \upsilon (y) y^{ - s } d y \cdot G (s; \usigma, \ulambda) x ^{ - s} d s.
\end{equation*}
We now shift the integral contour to $\EC _{\ulambda}$ defined in Definition \ref{3defn: C d lambda}. Using \eqref{1eq: vertical bound, G (s; sigma, lambda)} in Lemma \ref{1lem: vertical bound}, one shows that the above double integral becomes absolutely convergent after this contour shift.
Therefore, on changing the order of integrals, one obtains
\begin{equation}\label{2eq: Psi (x; sigma) as Hankel transform}
\Upsilon (x ) = \int_0^\infty \upsilon (y) J \big( (xy)^{\frac 1 n}; \usigma, \ulambda\big) d y.
\end{equation}
Here $J (x;  \usigma, \ulambda)$ is the 
{Bessel function} defined by the  Mellin-Barnes type integral
\begin{equation}\label{3eq: definition of J (x; sigma)}
J (x ; \usigma, \ulambda) = \frac 1  {2 \pi i}\int_{\EC_{\ulambda}} G(s; \usigma, \ulambda) x^{- n s} d s.
\end{equation}

{\it Shifting the Index of $J (x ; \usigma, \ulambda)$.}
\begin{lem}\label{3lem: normalize J(x; sigma, lambda)}
	Let $(\usigma, \ulambda) \in \{+, -\}^n \times \BC^{n}$ and $\lambda \in \BC$. Recall that  $\ue^n$ denotes the $n$-tuple $ ( {1,..., 1})$. Then
	\begin{equation}\label{4eq: normalize J}
	J (x ; \usigma, \ulambda - \lambda \ue^n) = x^{n \lambda} J (x ; \usigma, \ulambda).
	\end{equation}
\end{lem}

{\it Regularity of $J(x; \usigma, \ulambda)$.}
We now show that  $J (x ; \usigma, \ulambda)$ admits an analytic continuation from $\BR _+$ onto $\BU$. Consider the following Barnes type integral representation,
\begin{equation} \label{3eq: Barnes contour integral 1}
J (\zeta ; \usigma, \ulambda) = \frac 1  {2 \pi i}\int_{\EC_{\ulambda}'} G(s; \usigma, \ulambda) \zeta^{-n s} d s, \hskip 10 pt \zeta = x e^{i \omega} \in \BU, x \in \BR _+, \omega \in \BR,
\end{equation}
with the integral contour  given in Definition \ref{3defn: C ' lambda}.
We first rewrite $G  (s, \pm)$ using Euler's reflection formula,
\begin{equation*} 
G  (s, \pm) = \frac {\pi e \lp \pm \frac 1 4 s   \rp} {\sin (\pi s) \Gamma (1-s)},
\end{equation*}
Then Stirling's asymptotic formula \eqref{1eq: Stirling's formula} yields,
\begin{equation*} 
G(- \rho + it; \usigma, \ulambda) \lll_{\,\ulambda,\, r} {e^{n \rho}} {\rho^{ - n \lp \rho + \frac 1 2 \rp - \Re |\ulambda|}},
\end{equation*}
for all $ - \rho + it \notin \bigcup_{l      = 1}^n \bigcup_{\kappa \in \BN} \BB_r (\lambda_{l     } - \kappa)$ satisfying $\rho \ggg 1$ and $ t \lll \max  \{ |\Im \lambda_{l     }| \} + 1$.
It follows that the contour integral in \eqref{3eq: Barnes contour integral 1} converges absolutely and compactly
in $\zeta$, and hence $J (\zeta ; \usigma, \ulambda)$ is analytic in $\zeta$. 

Moreover, given any bounded open subset of $\BC^{n}$, one may fix a single contour $\EC' = \EC'_{\ulambda}$ for all $\ulambda$ in this set and verify the uniform convergence of the integral in the $\ulambda$ aspect. Then follows the analyticity of $J (\zeta ; \usigma, \ulambda)$  with respect to $\ulambda$.

\begin{lem}\label{3lem: J (x; sigma, lambda) analytic}
	$J (x ; \usigma, \ulambda) $  admits an analytic continuation $J (\zeta ; \usigma, \ulambda)$ from $\BR _+$ onto $\BU$. 
	In particular, $J (x ; \usigma, \ulambda) $ is a real analytic function of $x$ on $\BR _+$. Moreover,  $J (\zeta ; \usigma, \ulambda)$  is an analytic function of $\ulambda$ on $\BC^{n}$.
\end{lem}

\begin{rem}
	In \S \ref{sec: Recurrence relations and differential equations of the Bessel functions} we shall show that   $J (x ; \usigma, \ulambda)$ satisfies a differential equation with analytic coefficients. It then follows from the theory of differential equations that  $J (x ; \usigma, \ulambda)$ has an analytic continuation. This alternative viewpoint will be taken in \S \ref{sec: Analytic continuation of J (x; usigma; ulambda)}.
\end{rem}

\subsection {\texorpdfstring{The Hankel Transforms $\hld$, $\hmum$ and the Bessel Kernels $j_{(\umu, \udelta)} $, $j_{(\umu, \um)}$}{The Hankel Transforms $h_{(\mu, \delta)} $, $h_{(\mu, m)}$ and the Bessel Kernels $j_{(\mu, \delta)} $, $j_{(\mu, m)}$}}

Consider the ordered set $(\BC, \preccurlyeq_2)$ and define $ \lambda _{\alpha, \beta}$, $B_\alpha$, $M_{\alpha, \beta}$ and $N_{\alpha, \beta}$  as in Definition \ref{3defn: ordered set} corresponding to $\ulambda \in \BC^n$. We define the following subspace of $ \Ssis (\BR _+)$
\begin{equation}\label{3eq: Ssis2, R+}
\Ssiss^{ \ulambda } (\BR _+) = \sum_{ \alpha = 1}^A \sum_{\beta = 1}^{B_\alpha}  \sum_{j=0}^{N_{\alpha, \beta} - 1} x^{ - \lambda_{\alpha, \beta}} (\log x)^{j } \SS_{0} (\overline \BR _+).
\end{equation}

\subsubsection{\texorpdfstring{The Definition of  $  \hld$}{The Definition of  $h_{(\mu, \delta)} $}}\label{sec: h (lambda, delta)}


The following proposition provides the definition of the Hankel transform $\hld$, which maps $\Ssiss^{ - \umu - \udelta } (\BR _+)$    onto  $\Ssiss^{ \umu - \udelta } (\BR _+)$ bijectively. 

\begin{prop} \label{3prop: h (lambda, delta)}
	Let $(\umu, \udelta) \in \BC^{n} \times (\BZ/2 \BZ)^n$. Suppose $\upsilon \in \Ssiss^{ - \umu - \udelta } (\BR _+)$. Then there exists a unique function $\Upsilon \in \Ssiss^{ \umu - \udelta } (\BR _+) $ satisfying the following identity,
	\begin{equation}\label{3eq: Hankel transform identity 1, R+}
	\EM \Upsilon (s ) = G_{(\umu, \udelta )} (s) \EM \upsilon ( 1 - s).
	\end{equation}
	We call $\Upsilon$ the Hankel transform of  $\upsilon$ over $\BR _+$ of index $(\umu, \udelta)$ and  write $\hld \upsilon  = \Upsilon$. Furthermore, we have the Hankel inversion formula
	\begin{equation}\label{3eq: Hankel inversion, R 0}
	\hld \upsilon  = \Upsilon, \hskip 10 pt \hmld \Upsilon  = (-)^{|\udelta|} \upsilon.
	\end{equation}
\end{prop}

\begin{proof}
	Recall the definition of $G_{(\umu, \udelta )}$ given by (\ref{1def: G delta}, \ref{1def: G (lambda, delta)}), 
	\begin{equation*}
	G_{(\umu, \udelta )} (s) = i^{|\udelta|} \pi^{n \lp \frac 1 2 - s \rp + |\umu|} 
	\frac { \prod_{l      = 1}^n \Gamma \lp \frac 1 2( {s - \mu_l      + \delta_l     } ) \rp  } 
	{\prod_{l      = 1}^n \Gamma \lp \frac 1 2 ( {1 - s + \mu_l      + \delta_l     } ) \rp }, \end{equation*}
	where $|\udelta| = \sum_{l     } \delta_l      \in \BN$, with each $\delta_l     $ viewed as a number in the set $\{0, 1\} \subset \BN$. 
	
	We write $\umu^\pm = \pm \umu - \udelta$. Since $\mu^+_{  l     } + \mu^-_{ l     } = - 2 \delta_l      \in \{0, - 2\}$, the partition $\left\{ L_{\alpha} \right \}_{\alpha = 1}^A$ of $\{1, ..., n\}$ and $B_\alpha$ in Definition \ref{3defn: ordered set} are the same for both $\umu^+$ and $\umu^-$. Let  $ \mu^\pm _{\alpha, \beta}$, $M^\pm _{\alpha, \beta}$ and $N^\pm _{\alpha, \beta}$ be the notations  in Definition \ref{3defn: ordered set} corresponding to $\umu^\pm$.  Then the Gamma quotient above may be rewritten as follows,
	\begin{equation*}
	\frac { \prod_{ \alpha = 1}^A \prod_{\beta = 1}^{B_\alpha}  \Gamma \left( \frac 1 2 \big( { s - \mu^+_{\alpha, \beta}} \big) \right)^{M^+_{\alpha, \beta} } } 
	{ \prod_{ \alpha = 1}^A \prod_{\beta = 1}^{B_\alpha}  \Gamma \left( \frac 1 2 \big( {1 - s - \mu^-_{\alpha, \beta}} \big) \right)^{M^-_{\alpha, \beta} } }.
	\end{equation*}
	Thus, at each point $\mu \in \mu^+_{\alpha, 1} - 2\BN$ the product in the numerator contributes to $G_{(\umu, \udelta )} (s)$ a pole of pure order  $N_{\alpha, \beta}^+$, with $\beta = \max \left\{ \beta' : \mu \preccurlyeq_2 \mu^+_{\alpha, \beta' } \right \} $, whereas at each point $ \mu \in - \mu^-_{\alpha, 1} + 2\BN + 1$ the denominator contributes a zero of order $N_{\alpha, \beta}^-$, with $\beta = \max \left\{ \beta' : 1 - \mu \preccurlyeq_2 \mu^-_{\alpha, \beta' } \right \} $. 
	Moreover, \eqref{1eq: vertical bound, G (lambda, delta) (s)} in Lemma \ref{1lem: vertical bound} implies that $G_{(\umu, \udelta )} (s)$ is of uniform moderate growth on vertical strips. 

	On the other hand, according to Lemma \ref{2lem: refined decomp, R+}, the Mellin transform $\EM \upsilon $ lies in the space $ \sum_{ \alpha = 1}^A \sum_{\beta = 1}^{B_\alpha}  \sum_{j=0}^{N^-_{\alpha, \beta} - 1} \Nsis^{  \mu^-_{\alpha, \beta}, j}.$
	In particular,  the poles of $\EM \upsilon (1-s)$ are annihilated by the zeros contributed from the denominator of the Gamma quotient. Furthermore, $\EM \upsilon (1-s)$ uniformly rapidly decays on vertical strips. 

	We conclude that the product $ G_{(\umu, \udelta )} (s) \EM \upsilon ( 1 - s)$ on the right hand side of \eqref{3eq: Hankel transform identity 1, R+} lies in the space $\sum_{ \alpha = 1}^A \sum_{\beta = 1}^{B_\alpha}  \sum_{j=0}^{N^+_{\alpha, \beta} - 1} \Nsis^{  \mu^+_{\alpha, \beta}, j} $, and hence  $\Upsilon \in \Ssiss^{ \umu - \udelta } (\BR _+) $, with another application of  Lemma \ref{2lem: refined decomp, R+}.
	
	Finally, the Hankel inversion formula \eqref{3eq: Hankel inversion, R 0} is an immediate consequence of the functional relation \eqref{1eq: G mu delta (1-s) = G - mu delta (s)} of gamma factors.
\end{proof}


\subsubsection{\texorpdfstring{The Definition of  $ \hmum$}{The Definition of  $h_{(\mu, m)}$}}\label{sec: h mu m}

The following proposition provides the definition of the Hankel transform $\hmum$, which maps $\Ssiss^{ - 2 \umu - \| \um \| } (\BR _+)$    onto  $\Ssiss^{ 2 \umu - \| \um \| } (\BR _+)$ bijectively. 

\begin{prop} \label{3prop: h (mu, m)}
	Let $(\umu, \um) \in \BC^{n} \times \BZ ^n$. Suppose $\upsilon \in \Ssiss^{ - 2 \umu - \|\um\| } (\BR _+)$. Then there exists a unique function $\Upsilon \in \Ssiss^{ 2 \umu - \|\um\| } (\BR _+) $ satisfying the following identity,
	\begin{equation}\label{3eq: Hankel transform identity 2, R+}
	\EM \Upsilon (2 s ) = G_{(\umu, \um )} (s) \EM \upsilon ( 2(1-s)).
	\end{equation}
	We call $\Upsilon$ the Hankel transform of  $\upsilon$ over $\BR _+$ of index $(\umu, \um)$ and write $\hmum \upsilon  = \Upsilon$.
	Moreover, we have the Hankel inversion formula
	\begin{equation}\label{3eq: Hankel inversion, C 0}
	\hmum \upsilon  = \Upsilon, \hskip 10 pt \hmmum \Upsilon  = (-)^{|\um|}  \upsilon.
	\end{equation}
\end{prop}

\begin{proof}
	We first rewrite \eqref{3eq: Hankel transform identity 2, R+} as follows,
	\begin{equation*} 
	\EM \Upsilon ( s ) = G_{(\umu, \um )} \lp \frac s 2 \rp \EM \upsilon ( 2 - s ).
	\end{equation*}
	From (\ref{1def: G m (s)}, \ref{1def: G (mu, m)}), we have
	\begin{equation*}
	G_{(\umu, \um)} \lp \frac s 2 \rp = i^{\left|\|\um\|\right|} \pi^{n \lp 1 -  s \rp + 2 |\umu|} 
	\frac { \prod_{l      = 1}^n \Gamma \lp \frac 1 2 ( {s - 2\mu_l      + |m_l     |} ) \rp  } 
	{\prod_{l      = 1}^n \Gamma \lp \frac 1 2 ( {2 - s + 2\mu_l      + |m_l     |} ) \rp },
	\end{equation*}
	where $\left|\|\um\|\right| = \sum_{l     =1}^n |m_l     | $ according to our notation.
	We can now proceed to apply the same arguments in the proof of Proposition \ref{3prop: h (lambda, delta)}. Here we use \eqref{1eq: vertical bound, G (mu, m) (s)} and \eqref{1eq: G mu m (1-s) = G - mu m (s)} instead of \eqref{1eq: vertical bound, G (lambda, delta) (s)} and \eqref{1eq: G mu delta (1-s) = G - mu delta (s)} respectively.
\end{proof}

\subsubsection{\texorpdfstring{The Bessel Kernel  $j_{(\umu, \udelta)} $}{The Bessel Kernel  $j_{(\mu, \delta)} $}} \label{sec: Bessel kernel j lambda delta}
\ 
\vskip 5 pt

{\it The Definition of  $j_{(\umu, \udelta)} $.}
For $(\umu, \udelta) \in \BC^{n} \times (\BZ/2 \BZ)^n$, we define the Bessel kernel $j_{(\umu, \udelta)}$ by the following  Mellin-Barnes type integral,
\begin{equation}\label{2def: Bessel kernel, 1}
j_{(\umu, \udelta)} (x) = \frac 1  {2 \pi i} \int_{\EC_{(\umu, \udelta)}} G_{(\umu, \udelta)} (s ) x^{- s} d s.
\end{equation}
It is clear that
\begin{equation}\label{3eq: normalize j (lambda, delta)}
j_{(\umu - \mu \ue^n, \udelta)} (x) = x^{\mu} j_{(\umu, \udelta)} (x).
\end{equation}
In view of \eqref{1eq: G (lambda, delta) = G (s; sigma, lambda)}, we have
\begin{equation}\label{3eq: j (lambda, delta) and fundamental}
j_{(\umu, \udelta)} (x) = (2\pi)^{|\umu|}  \sum_{\usigma \in \{ +, - \}^n} \usigma^{\udelta} J \big(2 \pi x^{\frac 1 n}; \usigma, \umu \big).
\end{equation}


{\it Regularity of $j_{(\umu, \udelta)} $.}
It follows from \eqref{3eq: j (lambda, delta) and fundamental} and Lemma \ref{3lem: J (x; sigma, lambda) analytic} that
$j_{(\umu, \udelta)} (x)$ admits an analytic continuation $j_{(\umu, \udelta)} (\zeta)$, which is also analytic with respect to $\umu$.
Moreover,  $j_{(\umu, \udelta)} (\zeta)$ has the following Barnes type integral representation,
\begin{equation}\label{2def: Bessel kernel, analytic continuation, 1}
j_{(\umu, \udelta)} (\zeta) = \frac 1  {2 \pi i} \int_{\EC'_{\umu - \udelta}} G_{(\umu, \udelta)} (s ) \zeta^{- s} d s, \hskip  10 pt \zeta \in \BU.
\end{equation}
To see the convergence, the following formula is required
\begin{equation} 
G_\delta (s) = 
\left\{ \begin{split}
& \frac {\pi (2 \pi)^{-s}}  {\sin \left(\frac 1 2 {\pi s}   \right) \Gamma (1-s)}, \hskip 10 pt \text { if } \delta = 0,\\
& \frac {\pi i (2 \pi)^{-s}}  {\cos \left(\frac 1 2 {\pi s}   \right) \Gamma (1-s)}, \hskip 10 pt \text { if } \delta = 1.
\end{split} \right.
\end{equation}

{\it The Integral Kernel of $\hld$.}
Suppose $\upsilon \in \Ssiss^{ - \umu - \udelta } (\BR _+)$. In order to proceed in the same way as in \S \ref{sec: Bessel kernel J(x; sigma, lambda)},
one needs to assume that $(\umu, \udelta) $ satisfies the condition 
\begin{equation}\label{3eq: condition on (lambda, delta), 0}
\textstyle \min  \left\{ \Re \mu_l      + \delta_{l     } \right\} + 1 > \max  \left\{ \Re \mu_l      - \delta_{l     } \right\}.
\end{equation}
Then,
\begin{equation}\label{3eq: integral kernel, R 0}
\hld \upsilon  (x ) =  \int_0^\infty \upsilon (y) j_{(\umu, \udelta)} ( xy ) d y.
\end{equation}
Here, it is required for the convergence of the integral over $d y$ that the contour $\EC _{(\umu, \udelta)}$ in \eqref{2def: Bessel kernel, 1} is chosen   to lie in the left half-plane $\left\{s : \Re s < \min  \left\{ \Re \mu_l      + \delta_{l     } \right\} + 1 \right\}$. According to Definition \ref{3defn: C d lambda}, this choice of  $\EC _{(\umu, \udelta)}$ is permissible due to our assumption \eqref{3eq: condition on (lambda, delta), 0}. However, if one assumes that $\upsilon \in \SS (\BR _+)$, then \eqref{3eq: integral kernel, R 0} remains valid without requiring the condition \eqref{3eq: condition on (lambda, delta), 0}. 

\vskip 5 pt

\subsubsection{\texorpdfstring{The Bessel Kernel $j_{(\umu, \um)} $}{The Bessel Kernel $j_{(\mu, m)} $}} \label{sec: Bessel kernel j mu m}
\
\vskip 5 pt
{\it The Definition of $j_{(\umu, \um)} $.}
For $(\umu, \um) \in \BC^{n} \times  \BZ ^n$ define the Bessel kernel $j_{(\umu, \um)}$ by    the following  Mellin-Barnes type integral,
\begin{equation}\label{3def: Bessel kernel j mu m}
j_{(\umu, \um)} (x) = \frac 1  {2 \pi i} \int_{\EC _{ (\umu, \um)}} G_{(\umu, \um)} (s ) x^{- 2 s} d s.
\end{equation}
We have
\begin{equation}\label{3eq: normalize j (mu, m)}
j_{(\umu - \mu \ue^n, \um)} (x) = x^{2 \mu} j_{(\umu, \um)} (x).
\end{equation}
In view of Lemma \ref{1lem: complex and real gamma factors}, if $(\boldsymbol \eta, \udelta) \in \BC^{2n} \times (\BZT)^{2n}$ is related to $(\umu, \um) \in \BC^{n} \times  \BZ ^n$ via either \eqref{1eq: relation between (mu, m) and (lambda, delta), 1} or  \eqref{1eq: relation between (mu, m) and (lambda, delta), 2}, then 
\begin{equation}\label{3eq: j mu m = j lambda delta}
i^n j_{(\umu, \um)} (x) = j_{(\boldsymbol \eta, \udelta)} \big( x^2 \big).
\end{equation}

{\it Regularity of $j_{(\umu, \um)} $.}
In view of  \eqref{3eq: j mu m = j lambda delta}, the regularity of
$j_{(\umu, \um)} $ follows from that of $j_{(\boldsymbol \eta, \udelta)}$. 
Alternatively, this may be seen from
\begin{equation}\label{3eq: Bessel kernel, analytic continuation, 2}
j_{(\umu, \um)} (\zeta) = \frac 1  {2 \pi i} \int_{\EC'_{ \umu - \frac 1 2 \|\um\|}} G_{(\umu, \um)} (s ) \zeta^{- 2 s} d s, \hskip  10 pt \zeta \in \BU.
\end{equation}
To see the convergence,  the following formula is required
\begin{equation} \label{3eq: rewrite G m}
G_m (s) = 
\frac {\pi i^{|m|} (2\pi )^{1-2s} }  { \sin \lp \pi \lp s + \frac 1 2 {|m|}   \rp \rp \Gamma \lp 1-s-\frac 1 2 {|m|}   \rp  \Gamma \lp 1-s+\frac 1 2 {|m|}   \rp }.
\end{equation}

{\it The Integral Kernel of $\hmum$.}
Suppose $\upsilon \in \Ssiss^{ - 2 \umu - \|\um\| } (\BR _+)$. We assume that $(\umu, \um) $ satisfies the following condition
\begin{equation}\label{3eq: condition on (mu, m), 0}
\textstyle \min  \left\{ \Re \mu_l      + \frac 1 2 { |m_{l     }| }   \right\} + 1 > \max  \left\{ \Re \mu_l      - \frac 1 2 {|m_{l     }|}   \right\}.
\end{equation}
Then
\begin{equation}\label{3eq: h mu m = int of j mu m}
\hmum \upsilon (x)  =  \int_0^\infty \upsilon (y) j_{(\umu, \um)} ( xy ) \cdot 2 y d y,
\end{equation}
It is required for convergence that the integral contour $\EC _{ (\umu, \um)}$ in \eqref{3def: Bessel kernel j mu m} lies in the left half-plane $\left\{s : \Re s < \min  \left\{\Re \mu_l      + \frac 1 2 { |m_{l     }| }   \right\} + 1 \right\}$. This is however guaranteed by \eqref{3eq: condition on (mu, m), 0}. 
Moreover, if one assumes that  $\upsilon \in \SS (\BR _+)$, then \eqref{3eq: h mu m = int of j mu m} holds true for any index $(\umu, \um)$.

\vskip 5 pt

{\it Auxiliary Bounds for $j_{(\umu, \um + m\ue^n)}$.}
\begin{lem}\label{3lem: bound of the Bessel kernel, C}
	Let $(\umu, \um) \in \BC^{n} \times  \BZ ^n$ and $m \in \BZ$. 
	Put 
	\begin{align*}
	& A =\textstyle  n \lp \max  \{ \Re \mu_l      \} + \frac 1 2 \max  \{ |m_l     | \} - \frac 1 2 \rp - \Re |\umu| + \frac 1 2 |\|\um\||, \\
	& B_+ =\textstyle  - 2 \min  \{ \Re \mu_l      \} + \max  \{ |m_l     | \} + \max \left\{ \frac 1 n - \frac 1 2, 0 \right \}, \\
	& B_- =\textstyle  - 2 \max  \{ \Re \mu_l      \} - \max  \{ |m_l     | \}.
	\end{align*} 
	Fix  $ \epsilon > 0$. Denote by $\ue^n$   the $n$-tuple $(1, ..., 1)$. We have the following estimate
	\begin{equation}
	\label{2eq: bound of the Bessel kernel, C}
	\begin{split}
	j_{(\umu, \um + m\ue^n)} (x) \lll_{\,(\umu, \um),\, \epsilon, \, n} &\ \lp \frac { 2 \pi e x^{\frac 1 { n}}} {|m| + 1} \rp^{n |m|}  (|m| + 1)^{A + n \epsilon} \max \left\{ x^{B_+ + 2\epsilon }, x^{ B_- - 2\epsilon}  \right\}.
	\end{split}
	\end{equation}
\end{lem}

\begin{proof} 
	Let 
	\begin{align*}
	\rho_m &= \textstyle \max  \left\{ \Re \mu_l      - \frac 1 2 { |m_l      + m | }   \right  \}, \\
	\sigma_m &= \min \left\{ \textstyle \frac 1 2 + \frac 1 n \lp \Re |\umu| - \frac 1 2 \left| \| \um + m\ue^n\| \right| - 1 \rp, \rho_m \right \}.
	\end{align*}
	Choose the contour $\EC_m = \EC _{(\umu, \um + m\ue^n)}$ (see Definition \ref{3defn: C d lambda}) such that
	\begin{itemize}
		\item[-] if $s \in \EC_m$ and $\Im s $ is sufficiently large, then $\Re s = \sigma_m - \epsilon$, and
		\item[-] $\EC_m$ lies in the vertical strip $\BS[ \sigma_m - \epsilon, \rho_m +\epsilon]$.
	\end{itemize} 
	
	We first assume that $|m|$ is large enough so that
	$$ \textstyle  n \lp \rho_m + \epsilon - \frac 1 2 \rp - \Re |\umu| - \frac 1 2 \left| \| \um + m\ue^n \| \right| < 0.$$
	For the sake of brevity, we write $y = (2\pi)^{ n} x $. 
	We first bound $\left| j_{(\umu, \um + m\ue^n)} (x) \right|$ by
	\begin{equation*}
	\begin{split}
	(2\pi)^{n + \Re |\umu|} \int_{\EC_m}  y^{- 2 \Re s} \prod_{l      = 1}^n \left| \frac { \Gamma \lp s - \mu_l      + \frac 1 2 { |m_l      + m| }   \rp} { \Gamma \lp 1 - s + \mu_l      + \frac 1 2 { |m_l      + m |}   \rp } \right| |d s|.
	\end{split}
	\end{equation*}
	With the observations that for $s \in \EC_m$
	\begin{itemize}
		\item[-] $\Re s \in [\sigma_m - \epsilon, \rho_m +\epsilon]$,
		\vskip 2 pt
		\item[-] $\left| \Re s - \mu_l      + \frac 1 2 { |m_l      + m| }   \right| \lll_{ (\umu, \um)} 1$, 
		\vskip 2 pt
		\item[-] $\left| \lp 1 - \Re s + \mu_l      + \frac 1 2  { |m_l      + m| }   \rp - |m| \right| \lll_{\, (\umu, \um)} 1$,
	\end{itemize}
	in conjunction with Stirling's asymptotic formula \eqref{1eq: Stirling's formula}, we have the following estimate
	\begin{equation*}
	\begin{split}
	 j_{(\umu, \um + m\ue^n)} (x) & \lll_{\, (\umu, \um),\, n,\, \epsilon} \max \left\{ y^{- 2 \sigma_m + 2 \epsilon }, y^{- 2 \rho_m - 2 \epsilon } \right \} \\
	& \hskip 40 pt \int_{\EC_m} \frac {(|\Im s| + 1)^{ n \lp \Re s - \frac 1 2 \rp - \Re |\umu| + \frac 1 2 \left| \| \um + m\ue^n\| \right|}} {  e^{- n |m|} \big( \sqrt {(\Im s)^2 + m^2} + 1 \big)^{n \lp \frac 1 2 - \Re s \rp + \Re |\umu| + \frac 1 2 \left| \| \um + m\ue^n\| \right| } } |d s| \\
	& \leq  \max \left\{ y^{- 2 \sigma_m + 2 \epsilon }, y^{- 2 \rho_m - 2 \epsilon } \right \} e^{n |m|} 
	(|m| + 1)^{n \lp \rho_m + \epsilon - \frac 1 2 \rp - \Re |\umu| - \frac 1 2 \left| \| \um + m\ue^n\| \right|} \\
	& \hskip 111 pt \int_{\EC_m} (|\Im s| + 1)^{ n \lp \Re s - \frac 1 2 \rp - \Re |\umu| + \frac 1 2 \left| \| \um + m\ue^n\| \right|} |d s|.
	\end{split}
	\end{equation*}
	For $s \in \EC_m$, we have  $\Re s = \sigma_m - \epsilon$ if $\Im s$ is sufficiently large, and our choice of $\sigma_m$ implies
	$\textstyle n \lp \sigma_m - \epsilon - \frac 1 2 \rp - \Re |\umu| + \frac 1 2 \left| \| \um + m\ue^n\| \right| \leq - 1 - n \epsilon$, then it follows that the above integral converges and is of size $O_{(\umu, \um),\, \epsilon,\, n} (1)$.
	
	Finally, note that both $ - 2\sigma_m +  2 \epsilon$ and $- 2\rho_m -  2 \epsilon$ are close to $ {|m|} $, whereas the exponent of $(|m| + 1)$, that is $ n \lp  \rho_m + \epsilon - \frac 1 2 \rp - \Re |\umu| - \frac 1 2 \left| \| \um + m\ue^n\| \right| $, is close to $- n |m|$.
	Thus the following bounds yield \eqref{2eq: bound of the Bessel kernel, C},
	\begin{equation*}
	\begin{split}
	& \textstyle {|m|} + B_-  \leq - 2 \rho_m  \leq - 2 \sigma_m  \leq  {|m|}   + B_+, \\
	&  \textstyle n \lp \rho_m - \frac 1 2 \rp - \Re |\umu| - \textstyle \frac 1 2 \left| \| \um + m\ue^n\| \right| \leq - n |m| + A.
	\end{split}
	\end{equation*}
	When $|m|$ is small, we have the following estimate that also implies \eqref{2eq: bound of the Bessel kernel, C},
	\begin{equation*}
	j_{(\umu, \um + m\ue^n)} (x) \lll_{ \, (\umu, \um),\, \epsilon,\, n}
	\max \left\{ y^{- 2 \sigma_m + 2 \epsilon }, y^{- 2 \rho_m - 2 \epsilon } \right \} e^{n |m|}.
	\end{equation*}
\end{proof}

Using the formula \eqref{3eq: rewrite G m} of $G_m (s)$ instead of \eqref{1def: G m (s)} and the  Barnes type integral representation \eqref{3eq: Bessel kernel, analytic continuation, 2} for $j_{(\umu, \um + m\ue^n)} (\zeta)$ instead of the  Mellin-Barnes type integral representation \eqref{3def: Bessel kernel j mu m} for  $j_{(\umu, \um + m\ue^n)} (x)$, similar arguments in the proof of Lemma \ref{3lem: bound of the Bessel kernel, C} imply the following lemma.

\begin{lem}\label{3lem: bound of the Bessel kernel, analytic continuation, C}
	Let $(\umu, \um) \in \BC^{n} \times  \BZ ^n$ and $m \in \BZ$. 
	Put 
	\begin{align*}
	& A =\textstyle  n \lp \max  \{ \Re \mu_l      \} + \frac 1 2 \max  \{ |m_l     | \} - \frac 1 2 \rp - \Re |\umu| + \frac 1 2 |\|\um\||, \\
	& B =\textstyle   - 2 \max  \{ \Re \mu_l      \} - \max  \{ |m_l     | \},  \hskip 10 pt C =\textstyle 2 \max  \{ |\Im \mu_l     | \}.
	\end{align*} 
	Fix $X > 0$ and $ \epsilon > 0$. Then 
	\begin{equation*}
	\begin{split}
	j_{(\umu, \um + m\ue^n)} \lp x e^{i\omega} \rp \lll_{\,(\umu, \um),\, X,\, \epsilon, \, n} &\ \lp \frac { 2 \pi e x^{\frac 1 { n}}} {|m| + 1} \rp^{n |m|}  (|m| + 1)^{A + n \epsilon} x^{B + 2\epsilon } e^{ |\omega| ( C + 2 \epsilon)}
	\end{split}
	\end{equation*}
	for all $x < X$.
\end{lem}

\subsection {\texorpdfstring{The Hankel Transform $\Hld$ and the Bessel Kernel $J_{(\umu, \udelta)} $}{The Hankel Transform $H_{(\mu, \delta)} $ and the Bessel Kernel $J_{(\mu, \delta)} $}}
\

\vskip 5 pt


\subsubsection{\texorpdfstring{The Definition of  $ \Hld$}{The Definition of  $H_{(\mu, \delta)} $}}

Consider the ordered set $(\BC \times \BZT , \preccurlyeq )$ and define $( \mu_{\alpha, \beta}, \delta_{\alpha, \beta} ) = ( \mu, \delta )_{\alpha, \beta}$,  $B_\alpha$, $M_{\alpha, \beta}$ and $N_{\alpha, \beta}$ as in Definition \ref{3defn: ordered set} corresponding to $( \umu, \udelta ) \in ( \BC \times \BZ / 2\BZ)^n$. We define the following subspaces of $ \Ssis (\BRx)$,
\begin{equation}\label{3eq: Ssis (lambda, delta) delta, R}
\Ssis ^{( \umu, \udelta ), \delta} (\BRx) = \sum_{ \alpha = 1}^A \sum_{\beta = 1}^{B_\alpha}  \sum_{j=0}^{N_{\alpha, \beta} - 1} \sgn(x)^{ \delta_{\alpha, \beta}} |x|^{ - \mu_{\alpha, \beta}} (\log |x|)^{j } \SS_{\delta_{\alpha, \beta} + \delta } ( \BR ).
\end{equation}
\begin{equation}\label{3eq: Ssis (lambda, delta), R}
\begin{split}
\Ssis ^{( \umu, \udelta )}  (\BRx) = & \Ssis ^{( \umu, \udelta ), 0} (\BRx) \oplus \Ssis ^{( \umu, \udelta ), 1} (\BRx) \\
= & \sum_{ \alpha = 1}^A \sum_{\beta = 1}^{B_\alpha}  \sum_{j=0}^{N_{\alpha, \beta} - 1} \sgn(x)^{\delta_{\alpha, \beta}} |x|^{ - \mu_{\alpha, \beta}} (\log |x|)^{j } \SS ( \BR ).
\end{split}
\end{equation}

From the definition of $\Ssiss^{\ulambda} (\BR _+)$ in  \eqref{3eq: Ssis2, R+}, together with $\SS_{\delta} (\BR) = \sgn (x)^\delta \SS_{\delta} (\overline \BR _+)$ and $\SS_{\delta} (\overline \BR _+) = x^{\delta} \SS_{0} (\overline \BR _+) $, we have
\begin{equation}\label{3eq: Ssis (lambda, delta) delta = Ssis2}
\Ssis ^{( \umu, \udelta ), \delta} (\BRx) = \sgn(x)^{\delta} \Ssiss^{\umu - (\udelta + \delta \ue^n)} (\BR _+).
\end{equation}

The following theorem gives the definition of the Hankel transform $\Hld$, which maps $\Ssis^{ (- \umu, \udelta) } (\BRx)$    onto  $\Ssis^{( \umu, \udelta) } (\BRx)$ bijectively.

\begin{thm} \label{3prop: H (lambda, delta)}
	Let $(\umu, \udelta) \in \BC^{n} \times (\BZ/2 \BZ)^n$. Suppose $ \upsilon \in \Ssis^{ (- \umu, \udelta) } (\BRx)$.  Then there exists a unique function $\Upsilon \in \Ssis^{( \umu, \udelta) } (\BRx)$ satisfying the following two identities,
	\begin{equation}\label{3eq: Hankel transform identity, R}
	\EM _\delta \Upsilon (s ) = G_{(\umu,  \udelta + \delta \ue^n)} (s) \EM _\delta \upsilon ( 1 - s), \hskip 10 pt \delta \in \BZ/2 \BZ.
	\end{equation}
	We call $\Upsilon$ the Hankel transform of  $\upsilon$ over $\BRx$ of index $(\umu, \udelta)$ and  write $\Hld \upsilon  = \Upsilon$. Moreover, we have the Hankel inversion formula
	\begin{equation}\label{3eq: Hankel inversion, R}
	\Hld \upsilon (x) = \Upsilon (x), \hskip 10 pt \Hmld \Upsilon (x) = (-)^{|\udelta|} \upsilon \left((-)^n x \right).
	\end{equation}
\end{thm}

\begin{proof}
	Recall that 
	$$\EM_\delta \upsilon (s) = 2 \EM \upsilon_\delta (s).$$
	In view of \eqref{3eq: Ssis (lambda, delta) delta = Ssis2}, one has $\upsilon_\delta \in \Ssiss^{- \umu - (\udelta + \delta \ue^n)} (\BR _+)$. Applying Proposition \ref{3prop: h (lambda, delta)}, there is a unique function $\Upsilon_\delta \in \Ssiss^{ \umu - (\udelta + \delta \ue^n)} (\BR _+)$ satisfying 
	\begin{equation*}
	\EM  \Upsilon_\delta (s ) = G_{(\umu, \udelta + \delta \ue^n)} (s) \EM  \upsilon_\delta ( 1 - s).
	\end{equation*}
	According to \eqref{3eq: Ssis (lambda, delta) delta = Ssis2}, $\Upsilon (x) = \Upsilon_0 (|x|) + \sgn (x)\Upsilon_1 (|x|)$ lies in $\Ssis^{( \umu, \udelta), 0} (\BRx) \oplus \Ssis^{( \umu, \udelta), 1} (\BRx) = \Ssis^{( \umu, \udelta) } (\BRx)$. Clearly, $\Upsilon$ satisfies \eqref{3eq: Hankel transform identity, R}. Moreover, \eqref{3eq: Hankel inversion, R} follows immediately from \eqref{3eq: Hankel inversion, R 0} in Proposition  \ref{3prop: h (lambda, delta)}.
\end{proof}

\begin{cor} \label{3cor: H = h, R}
	Let $(\umu, \udelta) \in \BC^{n} \times (\BZT)^n$ and $\delta \in \BZT$. Suppose that $\varphi \in \Ssiss^{- \umu - (\udelta + \delta \ue^n)} (\BR _+)$ and $ \upsilon (x) = \sgn (x)^{\delta} \varphi (|x|)$.  Then
	\begin{equation*}
	\Hld \upsilon (\pm x) = (\pm)^{\delta } \hh_{(\umu, \udelta + \delta \ue^n)} \varphi (x), \hskip 10 pt x \in \BR _+.
	\end{equation*}
\end{cor}

\subsubsection{The Bessel Kernel $J_{(\umu, \udelta)} $}

Let $(\umu, \udelta) \in \BC^{n} \times (\BZ/2 \BZ)^n$.
We define
\begin{equation}\label{3def: Bessel function, R, 0}
\begin{split}
J_{(\umu, \udelta)} \lp \pm x \rp = \frac 1 2 \sum_{\delta \in \BZ/ 2\BZ} (\pm)^{ \delta} j_{(\umu, \udelta + \delta \ue^n)} (x), \hskip 10 pt x \in \BR_+, 
\end{split}
\end{equation}
or equivalently,
\begin{equation}\label{3def: Bessel function, R}
\begin{split}
J_{(\umu, \udelta)} \lp x \rp = \frac 1 2 \sum_{\delta \in \BZ/ 2\BZ} \sgn (x)^{ \delta} j_{(\umu, \udelta + \delta \ue^n)} (|x|), \hskip 10 pt x \in \BRx.
\end{split}
\end{equation}
Some properties of $J_{(\umu, \udelta)}$ are summarized as below.

\begin{prop}\label{3prop: properties of J, R}
	Let $(\umu, \udelta) \in \BC^{n} \times (\BZT)^n$.
	
	{\rm (1).} Let $(\mu, \delta) \in \BC\times \BZT$. We have
	\begin{equation*}
	J_{(\umu - \mu \ue^n, \udelta - \delta \ue^n)} (x) = \sgn (x)^\delta |x|^{\mu} J_{(\umu, \udelta)} (x).
	\end{equation*}
	
	{\rm(2).} $J_{(\umu, \udelta)} (x)$  is a real analytic function of $x$ on $\BRx$ as well as an analytic function of $\umu$ on $\BC^{n}$.

	{\rm(3).} Assume that $\umu$ satisfies the  condition
	\begin{equation}\label{3eq: condition on lambda}
	\textstyle \min  \left\{ \Re \mu_l      \right\} + 1 > \max  \left\{ \Re \mu_l      \right\}.
	\end{equation}
	Then for $ \upsilon \in \Ssis^{ (- \umu, \udelta) } (\BRx)$ 
	\begin{equation}\label{3eq: Hankel transform, with Bessel kernel, R}
	\Hld \upsilon (x) = \int_{\BR ^\times} \upsilon (y) J_{(\umu, \udelta)} (xy ) d y.
	\end{equation}
	Moreover, if $\upsilon \in \SS (\BRx)$, then \eqref{3eq: Hankel transform, with Bessel kernel, R} remains true for any index $ \umu \in \BC^n $.
\end{prop}

\subsection {\texorpdfstring{The Hankel Transform $ \Hmum$ and the Bessel Kernel $J_{(\umu, \um)} $}{The Hankel Transform $ H_{(\mu, m)}$ and the Bessel Kernel $J_{(\mu, m)} $}}

\

\vskip 5 pt

\subsubsection{\texorpdfstring{The Definition of  $ \Hmum$}{The Definition of $ H_{(\mu, m)}$}}

Consider now the ordered set $(\BC \times \BZ , \preccurlyeq )$ and define $( 2 \mu_{\alpha, \beta}, m_{\alpha, \beta} ) = ( 2 \mu, m )_{\alpha, \beta}$,  $B_\alpha$, $M_{\alpha, \beta}$ and $N_{\alpha, \beta}$ as in Definition \ref{3defn: ordered set} corresponding to $( 2 \umu, \um ) \in ( \BC \times \BZ )^n$. We define the following subspace of $ \Ssis (\BCx)$,
\begin{equation}\label{3eq: Ssis (mu, m), C}
\begin{split}
\Ssis ^{( \umu, \um )}  (\BCx) = \sum_{ \alpha = 1}^A \sum_{\beta = 1}^{B_\alpha}  \sum_{j=0}^{N_{\alpha, \beta} - 1} [z]^{- m_{\alpha, \beta}} \|z\|^{ -  \mu_{\alpha, \beta}} (\log |z|)^{j } \SS ( \BC ).
\end{split}
\end{equation}
The projection via the $m$-th Fourier coefficient maps $\Ssis ^{( \umu, \um )}  (\BCx) $ onto the space
\begin{equation}\label{3eq: Ssis (mu, m) m, C}
\Ssis ^{( \umu, \um ), m} (\BCx) = \sum_{ \alpha = 1}^A \sum_{\beta = 1}^{B_\alpha}  \sum_{j=0}^{N_{\alpha, \beta} - 1} [z]^{ - m_{\alpha, \beta}} \|z\|^{ - \mu_{\alpha, \beta}} (\log |z|)^{j } \SS_{m_{\alpha, \beta} + m} ( \BC ).
\end{equation}
From the definition  of $\Ssiss^{\ulambda} (\BR _+)$ in \eqref{3eq: Ssis2, R+}, along with $\SS_{m} (\BC) = [z]^m \SS_{m} (\overline \BR _+)$  and $\SS_{m} (\overline \BR _+) = x^{|m|} \SS_{0} (\overline \BR _+) $, we have
\begin{equation}\label{3eq: Ssis (mu, m) m = Ssis2}
\Ssis ^{( \umu, \um ), m} (\BCx) = [z]^{m} \Ssiss^{2 \umu - \| \um + m\ue^n \|  } (\BR _+).
\end{equation}

The following theorem gives the definition of the Hankel transform $\Hmum$, which maps $\Ssis^{ (- \umu, - \um) } (\BCx)$    onto  $\Ssis^{( \umu, \um) } (\BCx)$ bijectively.

\begin{thm} \label{3prop: H (mu, m)}
	Let $(\umu, \um) \in \BC^{n} \times \BZ ^n$. Suppose $ \upsilon \in \Ssis^{ (- \umu, - \um) } (\BCx)$.  Then there exists a unique function $\Upsilon \in \Ssis^{( \umu, \um) } (\BCx)$ satisfying the following sequence of identities,
	\begin{equation}\label{3eq: Hankel transform identity, C}
	\EM _{-m} \Upsilon (2 s ) = G_{(\umu,  \um + m\ue^n)} (s) \EM _m \upsilon ( 2 (1-s) ), \hskip 10 pt m \in \BZ .
	\end{equation}
	We call $\Upsilon$ the Hankel transform of  $\upsilon$ over $\BCx$ of index $(\umu, \um)$ and  write $\Hmum \upsilon  = \Upsilon$. Moreover, we have the Hankel inversion formula
	\begin{equation}\label{3eq: Hankel inversion, C}
	\Hmum \upsilon (z) = \Upsilon (z), \hskip 10 pt \Hmmum \Upsilon (z) = (-)^{|\um|}\upsilon \lp (-)^n z \rp.
	\end{equation}
\end{thm}

\begin{proof}
	Recall that 
	$$\EM_{m} \upsilon (s) = 4 \pi \EM \upsilon_{-m} (s).$$
	In view of \eqref{3eq: Ssis (mu, m) m = Ssis2}, we have $\upsilon_{- m} \in \Ssiss^{- 2 \umu - \| \um + m\ue^n \|} (\BR _+)$. Applying Proposition \ref{3prop: h (mu, m)}, we infer that there is a unique function $\Upsilon_m \in \Ssiss^{ 2 \umu - \| \um + m\ue^n \|} (\BR _+)$ satisfying 
	\begin{equation*} 
	\EM  \Upsilon_{m} (2 s ) = G_{(\umu,  \um + m\ue^n)} (s) \EM  \upsilon_{-m} ( 2(1-s) ).
	\end{equation*}
	According to Lemma \ref{2lem: Ssis to Msis, C},  in order to show that the Fourier series $\Upsilon \lp x e^{i \phi}\rp = \sum \Upsilon_m (x) e^{im \phi}$ lies in $\Ssis^{( \umu, \um) } (\BCx)$, it suffices to verify that  $G_{(\umu,  \um + m\ue^n)} (s) \EM  \upsilon_{-m} ( 2(1-s) )$   rapidly decays with respect to $m$, uniformly on vertical strips. 
	This however follows from the uniform rapid decay of $\EM  \upsilon_{-m} ( 2(1-s) )$ along with the uniform moderate growth of $G_{(\umu,  \um + m\ue^n)} (s)$ (\eqref{1eq: vertical bound, G (mu, m) (s)} in Lemma \ref{1lem: vertical bound}) in the $m$ aspect on vertical strips. 
	
	Finally,  \eqref{3eq: Hankel inversion, C 0} in Proposition  \ref{3prop: h (mu, m)} implies \eqref{3eq: Hankel inversion, C}.
\end{proof}

\begin{cor} \label{3cor: H = h, C}
	Let $(\umu, \um) \in \BC^{n} \times \BZ^n$ and $m \in \BZ $. Suppose $\varphi \in \Ssiss^{- 2 \umu - \| \um + m\ue^n \|} (\BR _+)$ and $ \upsilon (z) = [z]^{- m} \varphi (|z|)$.  Then
	\begin{equation*}
	\Hmum \upsilon \lp x e^{i\phi} \rp = e^{i m \phi} \hh_{(\umu,  \um + m\ue^n)} \varphi (x), \hskip 10 pt x \in \BR _+, \phi \in \BR / 2 \pi \BZ.
	\end{equation*}
\end{cor}

\subsubsection{\texorpdfstring{The Bessel Kernel $J_{(\umu, \um)} $}{The Bessel Kernel $J_{(\mu, m)} $}}

For $(\umu, \um) \in \BC^{n} \times \BZ^n$, we define
\begin{equation}\label{2eq: Bessel kernel over C, polar}
J_{(\umu, \um)} \lp x e^{i \phi} \rp =  \frac 1 {2 \pi} \sum_{m \in \BZ} j_{(\umu, \um + m\ue^n)} (x ) e^{ i m \phi}, 
\end{equation}
or equivalently,
\begin{equation}\label{2eq: Bessel kernel over C}
J_{(\umu, \um)} \lp z \rp =  \frac 1 {2 \pi} \sum_{m \in \BZ} j_{(\umu, \um + m\ue^n)} (|z|) [z]^m.
\end{equation}
Lemma \ref{3lem: bound of the Bessel kernel, C} secures the absolute convergence of this series.

\begin{prop}\label{3prop: properties of J, C}
	Let $(\umu, \um) \in \BC^{n} \times \BZ^n$.
	
	{\rm (1).} Let $(\mu, m) \in \BC \times \BZ$. We have
	\begin{equation*}
	J_{(\umu - \mu \ue^n, \um - m \ue^n)} (z) = [z]^m \|z\|^{\mu} J_{(\umu, \um)} (z).
	\end{equation*}
	
	{\rm (2).} $ J_{(\umu, \um)} (z)$  is a real analytic function of $z$ on $\BCx$ as well as an analytic function of $\umu$ on $\BC^{n}$.
	
	{\rm (3).} Assume that $\umu$ satisfies the following condition
	\begin{equation}\label{3eq: condition on mu}
	\textstyle  \min  \left\{ \Re \mu_l      \right\} + 1 > \max  \left\{ \Re \mu_l      \right\}.
	\end{equation}
	Suppose $ \upsilon \in \Ssis^{ (- \umu, - \um) } (\BCx)$. Then 
	\begin{equation} \label{3eq: Hankel transform, C, polar}
	\Upsilon \lp x e^{i \phi} \rp  =  \int_0^\infty \int_0^{2\pi} \upsilon \lp y e^{i \theta}\rp  J_{(\umu, \um)} \lp x e^{ i \phi } y e^{ i \theta } \rp \cdot 2 y d \theta dy,
	\end{equation}
	or equivalently,
	\begin{equation} \label{3eq: Hankel transform, with Bessel kernel, C}
	\Upsilon (z)  =  \int_{\BCx} \upsilon (u) J_{(\umu, \um)} ( zu ) d u.
	\end{equation}
	Moreover, \eqref{3eq: Hankel transform, C, polar} and \eqref{3eq: Hankel transform, with Bessel kernel, C} still hold true for any index $ \umu \in \BC$ if $\upsilon \in \SS (\BCx)$.
\end{prop}

\begin{proof}
	
	(1). This is clear.
	
	(2). In \eqref{2eq: Bessel kernel over C, polar}, with abuse of notation, we view $x$ and $\phi$ as complex variables on $\BU$ and $\BC/2 \pi \BZ$ respectively, $j_{(\umu, \um + m\ue^n)} (x )$ and $ e^{ i m \phi}$ as analytic functions. Then Lemma \ref{3lem: bound of the Bessel kernel, analytic continuation, C} implies that the series in \eqref{2eq: Bessel kernel over C, polar} is absolutely convergent, compactly with respect to both $x$ and $\phi$, and therefore $ J_{(\umu, \um)} \lp x e^{i \phi} \rp$ is an analytic function of $x$ and $\phi$. In particular, $ J_{(\umu, \um)} (z)$  is a \textit{real analytic} function of $z$ on $\BCx$.
	
	Moreover, in  Lemma \ref{3lem: bound of the Bessel kernel, C}, we may allow $\umu$ to vary in an $\epsilon$-ball in $\BC^{n}$ and choose the implied constant in the estimate to be uniformly bounded with respect to $\umu$. This implies that the series in \eqref{2eq: Bessel kernel over C, polar} is convergent compactly in the $\umu$ aspect.
	Therefore, $ J_{(\umu, \um)} (z)$ is an analytic function of $\umu$ on $\BC^{n}$.
	
	(3). It follows from  \eqref{3eq: Ssis (mu, m) m = Ssis2} that $\upsilon_{- m} \in \Ssiss^{- 2 \umu - \| \um + m\ue^n\|} (\BR _+)$. Moreover, one observes that  $(\umu, \um + m \ue^n)$ satisfies the condition \eqref{3eq: condition on (mu, m), 0} due to \eqref{3eq: condition on mu}. Therefore, in conjunction with  Proposition \ref{3prop: h (mu, m)}, \eqref{3eq: h mu m = int of j mu m} implies 
	\begin{equation*} 
	\Upsilon_{m} (x ) = 2 \int_0^\infty \upsilon_{- m} (y) j_{(\umu, \um + m\ue^n)} \lp x y \rp y d y.
	\end{equation*}
	Hence
	\begin{equation*} 
	\begin{split}
	\Upsilon \lp x e^{i \phi} \rp =   \sum_{m \in \BZ} \Upsilon_{ m} (x ) e^{i m \phi}  = \sum_{m \in \BZ} \frac 1 { \pi} \int_0^\infty \int_0^{2\pi} \upsilon \lp y e^{i \theta}\rp j_{(\umu, \um + m\ue^n)} \lp x y \rp e^{ i m (\phi + \theta)} y d \theta d y.
	\end{split}
	\end{equation*}
	The estimate of $j_{(\umu, \um + m\ue^n)}$ in Lemma \ref{3lem: bound of the Bessel kernel, C} implies that
	the above series of integrals converges absolutely. On interchanging the order of summation and integration, one obtains \eqref{3eq: Hankel transform, C, polar}  in view of the definition  of $J_{(\umu, \um)}$ in \eqref{2eq: Bessel kernel over C, polar}.
	
	Note that in the case $\upsilon \in \SS (\BCx)$, one has $\upsilon_{- m}  \in \SS (\BR _+)$, and therefore \eqref{3eq: h mu m = int of j mu m} can be applied unconditionally. 
\end{proof}

\subsection{Concluding Remarks}

\ 
\vskip 5 pt

\subsubsection{Connection Formulae}\label{sec: connection formulae}

From \eqref{3eq: j (lambda, delta) and fundamental},  \eqref{3def: Bessel function, R, 0} or \eqref{3def: Bessel function, R}, we deduce the connection formula 
\begin{equation}\label{3eq: Bessel kernel, R, connection}
\begin{split}
J_{(\umu, \udelta)} \lp \pm x \rp = (2\pi)^{|\umu|} \sum_{|\usigma| = \pm } \usigma      ^{\udelta } J \big(2 \pi x^{\frac 1 n}; \usigma, \umu \big).  
\end{split}
\end{equation}
This will  enable us to reduce  the study of $J_{(\umu, \udelta)} (x) $  to that of  $J(x; \usigma, \ulambda)$.

From the various connection formulae (\ref{3eq: j (lambda, delta) and fundamental}, \ref{3eq: j mu m = j lambda delta}, \ref{2eq: Bessel kernel over C, polar}, \ref{2eq: Bessel kernel over C}) which   have been derived so far,  
one can also connect  the Bessel kernel $J_{(\umu, \um)} (z)$   to the Bessel functions  $J(x; \usigma, \ulambda)$ of doubled rank $2n$. However, unlike the expression of $J_{(\umu, \udelta)} (\pm x)$ by a {\it finite  sum} of $J \big(2 \pi x^{\frac 1 n}; \usigma, \umu \big)$ as in \eqref{3eq: Bessel kernel, R, connection},  these connection formulae yield an expression of $J_{(\umu, \um)} \lp x e^{i \phi} \rp$ in terms of an {\it infinite  series} involving the Bessel functions $J \big(2 \pi x^{\frac 1 n}; \usigma, \ulambda \big)$ of rank $2n$,  so a similar reduction for $J_{(\umu, \um)} (z)$ does not  exist and  we need to search for other approaches. 

\vskip 5 pt

\subsubsection{Normalizations of Indices} 
Usually, it is convenient to normalize the indices in $J(x; \usigma, \ulambda)$,  $j_{(\umu, \udelta)} (x)$, $j_{(\umu, \um)} (x)$, $J_{(\umu, \udelta)} (x)$ and $J_{(\umu, \um)} (z)$ so that $\ulambda, \umu \in \BL^{n-1}$. Furthermore, without loss of generality, the assumptions $\delta_n = 0$ and $m_n = 0$ may also be imposed for $J_{(\umu, \udelta)} (x)$ and $J_{(\umu, \um)} (z)$ respectively.
These normalizations are justified by Lemma {\rm \ref{3lem: normalize J(x; sigma, lambda)}}, \eqref{3eq: normalize j (lambda, delta)}, \eqref{3eq: normalize j (mu, m)}, Proposition \ref{3prop: properties of J, R} (1) and \ref{3prop: properties of J, C} (1).

\section{Bessel Functions and Bessel Kernels in Classical Cases}\label{sec: classical cases}

In Introduction, we have seen the occurrence of the oscillatory exponential function $e (x)$ and classical Bessel functions in the Poisson and \Voronoi summation formulae. Using their Mellin-Barnes integral representations, we shall now show that these functions indeed coincide with the Bessel kernels in the real rank-one and rank-two cases defined in \S \ref{sec: Hankel transforms}. Furthermore, we shall also find that the complex Bessel kernel of rank one is exactly $e (\Tr (z)) = e (z + \overline z)$.

\subsection{Rank-One Bessel Functions}

To start with, we have  the   formula {\rm (\cite[{\bf 3.764}]{G-R})} 
\begin{equation}\label{2eq: Gamma and exponential via Mellin}
\Gamma (s) e \lp \pm \frac  s 4 \rp = \int_0^\infty e^{\pm ix} x^{s } d^\times x, \hskip 10 pt 0 < \Re s < 1.
\end{equation}

\begin{prop} \label{prop: n=1}
	Suppose $n = 1$. Choose the contour $\EC = \EC_{0}$  as in Definition {\rm\ref{3defn: C d lambda}};  $\EC$ starts from $\rho - i \infty$ and ends at $\rho + i \infty$, with  $\rho < - \frac 1 2$, and all the nonpositive integers lie on the left side of  $\EC$. We have
	\begin{equation}\label{2eq: n = 1, Mellin inversion}
	e^{\pm i x} = \frac 1 {2\pi i} \int_{\EC} \Gamma (s) e \lp \pm \frac s 4 \rp x^{-s} d s.
	\end{equation}
	Therefore
	\begin{equation*}
	J (x; \pm, 0) = e^{\pm i x}.
	\end{equation*}
\end{prop}
\begin{proof}
	Let $\Re z > 0$. For $\Re s > 0$, we have the formula
	\begin{equation*}
	\Gamma (s) z^{-s} = \int_0^\infty e^{- z x} x^{s } d^\times x,
	\end{equation*}
	where the integral is absolutely convergent. The Mellin inversion formula  yields 
	\begin{equation*}
	e^{- z x} = \frac 1 {2\pi i} \int_{(\sigma)} \Gamma (s) z^{-s} x^{-s} d s, \hskip 10 pt \sigma > 0.
	\end{equation*} 
	Shifting the contour of integration from $(\sigma)$ to $\EC$, one sees that
	\begin{equation*}
	e^{- z x} = \frac 1 {2\pi i} \int_{\EC} \Gamma (s) z^{-s} x^{-s} d s.
	\end{equation*}
	Choose $z = e^{\mp \lp \frac {1 } 2 \pi - \epsilon \rp i}$, $ \pi > \epsilon > 0$. In view of  Stirling's asymptotic formula, the convergence of the integral above is uniform in $\epsilon$. Therefore, we obtain \eqref{2eq: n = 1, Mellin inversion} by letting $\epsilon \ra 0$.
\end{proof}

\begin{rem}
	Observe that the integral in \eqref{2eq: Gamma and exponential via Mellin} is only conditionally convergent, the Mellin inversion formula does not apply in the rigorous sense. Nevertheless, \eqref{2eq: n = 1, Mellin inversion} should be view as the Mellin inversion of \eqref{2eq: Gamma and exponential via Mellin}.
\end{rem}

\begin{rem}
	It follows from the proof of Proposition {\rm \ref{prop: n=1}} that the formula
	\begin{equation}\label{2eq: Mellin n=1}
	e^{- e(a) x } = \frac 1 {2\pi i} \int_{\EC} \Gamma (s) e\lp - a s \rp x^{-s} d s
	\end{equation}
	is valid for any $a \in \left[- \frac 1 4,  \frac 1 4\right]$.
\end{rem}


\subsection{Rank-Two Bessel Functions}

\begin{prop}\label{prop: Classical Bessel functions}
	Let $\lambda \in \BC$. Then
	\begin{align*}
	J (x; \pm, \pm, \lambda, - \lambda) & = \pm \pi i e^{\pm \pi i \lambda} H^{(1, 2)}_{2 \lambda} (2 x), \\
	J (x; \pm, \mp, \lambda, - \lambda) & = 2 e^{\mp \pi i \lambda} K_{2 \lambda} (2 x).
	\end{align*}
	Here $H^{(1)}_{\nu}$ and $H^{(2)}_{\nu}$ are Bessel functions of the third kind, also known as Hankel functions, whereas $K_{\nu}$ is the  modified Bessel function of the second kind, occasionally called the $K$-Bessel function or  MacDonald function.
\end{prop}

\begin{proof}
	The following formulae are derived from \cite[{\bf 6.561} 14-16]{G-R} along with Euler's reflection formula of the Gamma function. 
	\begin{equation*} 
	\pi \int_0^\infty J_{\nu} (2 \sqrt x) x^{s-1} d x = \Gamma \left( s + \frac \nu 2\right) \Gamma \left( s - \frac \nu 2\right) \sin \left(\pi \left(s - \frac \nu 2 \right) \right)
	\end{equation*}
	for $- \frac {1} 2 \Re  \nu< \Re s < \frac 1 4$,
	\begin{equation*} 
	- \pi \int_0^\infty  Y_{\nu} (2 \sqrt x) x^{s-1} d x  = \Gamma \left( s + \frac \nu 2\right) \Gamma \left( s - \frac \nu 2\right) \cos \left(\pi \left(s - \frac \nu 2 \right) \right)
	\end{equation*}
	for $ \frac {1} 2 |\Re  \nu|< \Re s < \frac 1 4$, and
	\begin{equation*} 
	2 \int_0^\infty  K_{\nu} (2 \sqrt x) x^{s-1} d x  = \Gamma \left( s + \frac \nu 2\right) \Gamma \left( s - \frac \nu 2\right)
	\end{equation*}
	for $ \Re s > \frac {1} 2|\Re  \nu|$. For $\Re s$   in the given ranges, these integrals are absolutely convergent.  It follows immediately from the Mellin inversion formula that
	\begin{align*}
	J (x; \pm, \pm, \lambda, - \lambda) & = \pm \pi i e^{\pm \pi i \lambda} (J_{2 \lambda} (2 x) \pm i Y_{2 \lambda} (2 x)), \hskip 10 pt |\Re \lambda| < \frac 1 4, \\
	J (x; \pm, \mp, \lambda, - \lambda) & = 2 e^{\mp \pi i \lambda} K_{2 \lambda} (2 x).
	\end{align*}
	In view of the analyticity in $\lambda$, the first formula remains valid even if $|\Re \lambda| \geq \frac 1 4$ by the theory of analytic continuation. 
	Finally, we conclude the proof by recollecting the connection formula $H_\nu^{(1, 2)} (x) = J_\nu (x) \pm i Y_\nu (x)$ in \eqref{2eq: connection formulae}.
\end{proof}

\subsection{Real Bessel Kernels of Rank One and Two}\label{sec: n=1,2, R}

For $n=1$, using \eqref{3eq: j (lambda, delta) and fundamental} and \eqref{3def: Bessel function, R, 0}, Proposition \ref{prop: n=1} implies that
\begin{equation*}
	j_{(0, 0)} (x ) = 2 \cos (2\pi x), \hskip 10 pt j_{(0, 1)} (x ) = 2 i \sin (2\pi x).
\end{equation*}
and
$$J_{(0, 0)} (x) =  e(x).$$
For $n = 2$, \eqref{3eq: Bessel kernel, R, connection} reads 
\begin{equation*}
J_{(\mu, - \mu, \delta, 0)} (\pm x) = J \lp 2\pi \sqrt x; +, \pm , \mu, - \mu \rp + (-)^{\delta } J \lp 2\pi \sqrt x; -, \mp, \mu, - \mu \rp,
\end{equation*}
for $x \in \BR _+$, $\mu \in \BC$ and $\delta \in \BZT$.
In view of Proposition \ref{prop: Classical Bessel functions}, for $x \in \BR _+$,  we have
\begin{equation*}
J_{(\mu, - \mu, \delta, 0)} (x)  = \left\{ 
\begin{split}
& - \frac {\pi } {\sin (\pi \mu)} \lp J_{2 \mu} (4\pi \sqrt x ) - J_{-2 \mu} (4\pi \sqrt x ) \rp, \hskip 10 pt \text{ if } \delta = 0, \\
&  \frac {\pi i} {\cos (\pi \mu)} \lp J_{2 \mu} (4\pi \sqrt x) + J_{-2 \mu} (4\pi \sqrt x ) \rp , \hskip 21 pt \text{ if } \delta = 1,
\end{split}
\right.
\end{equation*}
where the right hand side is replaced by its limit if  $2 \mu \in  \delta + 2\BZ$, and
\begin{equation*} 
J_{(\mu, - \mu, \delta, 0)} (-x) = 
\left\{ 
\begin{split}
& 4 \cos (\pi \mu) K_{2 \mu} (4 \pi \sqrt x ), \hskip 24 pt \text{ if } \delta = 0,\\
& - 4 i \sin (\pi \mu) K_{2 \mu} (4 \pi  \sqrt x ), \hskip 11 pt \text{ if } \delta = 1.
\end{split}
\right.
\end{equation*}
Observe that for $m \in \BN $
\begin{equation*}
J_{\lp \frac 12 m , - \frac 12 m ,\, \delta (m) + 1, 0 \rp} (x) = 2 \pi i^{m+1} J_m (4\pi \sqrt x), \hskip 10 pt J_{\lp \frac 12  m , - \frac 12 m ,\, \delta (m) + 1, 0 \rp} (-x) = 0.
\end{equation*}

\begin{rem}\label{2rem: n=2}
	Let $\mu = i t$ if $F$ is a Maa\ss  \hskip 3 pt  form of eigenvalue $\frac 1 4 + t^2$ and weight $k$, and let $\mu = \frac 12 ({k - 1})  $ if $F$ is a holomorphic cusp form of weight $k$.  At the real place $F$ may be parametrized by $(\umu, \udelta) = (\mu, - \mu, k (\mod 2), 0)$  so that $J _F = J_{(\umu, \udelta)}$. Hence the formulae of $J_F$ in {\rm (\ref{1eq: n=2, Maass form, Bessel functions, k even}, \ref{1eq: n=2, Maass form, Bessel functions, k odd}, \ref{1eq: n=2, holomorphic form, Bessel functions})} are justified.
\end{rem}

\subsection{Complex Bessel Kernels of Rank One}\label{sec: n=1, C}

Let $n = 1$. By (\ref{3eq: j (lambda, delta) and fundamental}) and (\ref{3eq: j mu m = j lambda delta}), it follows from Proposition \ref{prop: Classical Bessel functions} that 
\begin{equation}\label{3eq: j (0, m)}
j_{(0, m)} (x) = 2 \pi i^{|m| } J_{|m|} (4 \pi x) = 2 \pi i^{m } J_{m} (4 \pi x).
\end{equation}
where the second equality follows from the identity $J_{-m} (x) = (-)^m J_m (x)$.
	In view of \eqref{2eq: Bessel kernel over C, polar}, the following expansions {\rm (\cite[2.22 (3, 4)]{Watson})}
	\begin{align*}
	\cos (x \cos \phi) &= J_0(x) + 2 \sum_{d = 1}^\infty (-)^d J_{2 d} (x) \cos (2 d \phi),\\
	\sin (x \cos \phi) &=   2 \sum_{d = 0}^\infty (-)^d J_{2 d + 1} (x) \cos ((2 d + 1) \phi),
	\end{align*}
	imply
	\begin{equation*}
	J_{(0, 0)} \lp x e^{i \phi} \rp = \cos (4 \pi x \cos \phi) + i \sin (4 \pi x \cos \phi) = e(2 x \cos \phi ),
	\end{equation*}
	or equivalently,
	\begin{equation*}
	J_{(0, 0)} (z) = e(z + \overline z).
	\end{equation*}
	We remark that the two expansions \cite[2.22 (3, 4)]{Watson} can be incorporated into 
	\begin{equation*}
	e^{i x \cos \phi} = \sum_{m = -\infty}^{\infty}  i^{m}  J_m (x) e^{im \phi}.
	\end{equation*}

\section{Fourier Type Integral Transforms} 
\label{sec: Fourier type transforms}

In this section, we shall introduce an alternative perspective of Hankel transforms. 
We shall first show how to construct Hankel transforms from the Fourier transform and   Miller-Schmid transforms. From this, we shall express the Hankel transforms $\Hsl$, $\Hld$ and $\Hmum$ in terms of certain Fourier type integral transforms, assuming that the components of $\Re \ulambda$ or $\Re \umu$ are {strictly} decreasing.


\subsection{The Fourier Transform and Rank-One Hankel Transforms}

By Proposition \ref{prop: n=1}, $J(x; \pm, 0) = e^{\pm i x}$ so that the Hankel transform $\EH_{(\pm, 0)}$ is in essence a real Fourier transform on $\BR_+$. Slightly more generally, we want to consider the transform $ \ES_{(\varsigma, \lambda)} \upsilon (x) = x^{\lambda}  \EH_{(\varsigma, \lambda)} \upsilon (x)$.

\begin{lem}\label{5lem: Hankel R+}
	Let $(\varsigma, \lambda) \in \{+, -\} \times \BC $.
	For $\upsilon \in  \SS (\BR_+) $, define $\ES_{(\varsigma, \lambda)} \upsilon (x) =    x^{ \lambda}  \EH_{(\varsigma, \lambda)} \upsilon (x)$. Then 
	\begin{equation}\label{5eq: Mellin S, R+}
	\EM \ES_{(\varsigma, \lambda)} \upsilon (s) = G  (s, \varsigma) \EM  \upsilon (1-s-\lambda), 
	\end{equation}
	and $\ES_{(\varsigma, \lambda)}$ sends $ \SS (\BR_+)$   into $\SS (\overline \BR_+)$. Furthermore, 
	\begin{equation*}
	\ES_{(\varsigma, \lambda)} \upsilon (x) =  \int_{\BR_+}   y^{-\lambda} \upsilon (y) e^{\varsigma i x y} d y   .
	\end{equation*}
	
\end{lem}

For either $\BF = \BR$ or $\BF = \BC$, we have seen in \S \ref{sec: n=1,2, R} and \ref{sec: n=1, C} that $J_{(0, 0)}$ is exactly the inverse Fourier kernel, namely
$$J_{(0, 0)} (x) = e (\Lambda (x)), \hskip 10 pt x \in \BF,$$ with $\Lambda (x)$ defined by \eqref{1eq: Lambda (x)}. Therefore, in view of Proposition \ref{3prop: properties of J, R} (3) and \ref{3prop: properties of J, C} (3), $\EH_{(0, 0)} $ is precisely the inverse Fourier transform over the Schwartz space $\Ssis^{(0, 0)} (\BFx) = \SS (\BF)$. The following lemma is a consequence of Theorem  \ref{3prop: H (lambda, delta)} and   \ref{3prop: H (mu, m)}.

\begin{lem}\label{5lem: Fourier}
	Let $\upsilon \in \SS (\BF)$. If $\BF = \BR$, then the Fourier transform $\widehat \upsilon $ of $\upsilon$ can be determined by the following two identities
	\begin{equation*}
	\EM_{\delta} \widehat \upsilon (s) = (-)^\delta G_{\delta} (s) \EM_{\delta} \upsilon (1-s), \hskip 10 pt \delta \in \BZT.
	\end{equation*}
	If $\BF = \BC$, then the Fourier transform $\widehat \upsilon $ of $\upsilon$ can be determined by the following sequence of identities
	\begin{equation*}
	\EM_{- m} \widehat \upsilon (2s) = (-)^m G_{m} (s) \EM_{m} \upsilon (2 (1-s)), \hskip 10 pt m \in \BZ.
	\end{equation*}
\end{lem}
\delete{
	More generally, one may express any rank-one Hankel transform in terms of the Fourier transform.
	
	\begin{lem} 
		Let $(\mu, \delta) \in \BC \times \BZT $ and $(\mu, m) \in \BC \times \BZ $.
		
		{\rm (1).} If $\BF = \BR$, then $$J_{(\mu, \delta)} (x) = \sgn(x)^\delta |x|^{-\mu} e (x),$$ and, for  $\upsilon \in \sgn (x)^{\delta} |x|^{\mu} \SS (\BR) $, we have
		\begin{align*}
		\EH_{(\mu, \delta)} \upsilon (x) & =  \sgn(x)^{\delta} |x|^{- \mu} \int_{\BRx} \upsilon (y) \cdot \sgn (y)^\delta |y|^{-\mu} e (x y) d y \\
		& =  \sgn(x)^{\delta} |x|^{- \mu} \EF \lp \sgn^\delta |\ |^{-\mu} \upsilon \rp (-x).
		\end{align*}

		{\rm (2).} If $\BF = \BC$, then
		$$J_{(\mu, m)} (z) = [z]^{-m} \|z\|^{-\mu} e (z + \overline z),$$
		and, for  $\upsilon \in [z]^{m} \|z\|^{\mu} \SS (\BC) $, we have
		\begin{align*}
		\EH_{(\mu, m)} \upsilon (z) & =   [z]^{-m} |z|^{- \mu} \int_{\BCx} \upsilon (u) \cdot [u]^{-m} \|u\|^{-\mu} e (z u + \overline {z u}) d u \\
		& =  [z]^{-m} \|z\|^{- \mu} \EF \lp [\, ]^{-m} \|\, \|^{-\mu} \upsilon \rp (-z).
		\end{align*}

	\end{lem}
	
	It is convenient for our purpose to renormalize the rank-one Hankel transforms as follows.
}

It is convenient for our purpose to also introduce the renormalized rank-one Hankel transforms $\ES_{(\mu, \epsilon)}$ and $\ES_{(\mu, k)}$ as follows.

\begin{lem}\label{5cor: renormalize Hankel}
	Let $(\mu, \epsilon) \in \BC \times \BZT$ and $(\mu, k) \in \BC \times \BZ$. 
	
	{\rm (1).} For $\upsilon (x) \in \sgn (x)^{\epsilon} |x|^{\mu} \SS (\BR) $, define $\ES_{(\mu, \epsilon)} \upsilon (x) =    |x|^{ \mu}  \EH_{(\mu, \epsilon)} \upsilon (x)$. Then 
	\begin{equation}\label{5eq: Mellin S, R}
	\EM_{\delta} \ES_{(\mu, \epsilon)} \upsilon (s) = G_{\epsilon + \delta} (s) \EM_{\delta} \upsilon (1-s-\mu), \hskip 10 pt \delta \in \BZT,
	\end{equation}
	and $\ES_{(\mu, \epsilon)}$ sends $\sgn (x)^{\epsilon} |x|^{\mu} \SS (\BR)$   onto $\sgn (x)^{\epsilon} \SS (\BR)$ bijectively. Furthermore, 
	\begin{equation*}
	\ES_{(\mu, \epsilon)} \upsilon (x) =  \sgn(x)^{\epsilon} \int_{\BRx} \sgn (y)^\epsilon |y|^{-\mu} \upsilon (y) e (x y) d y   
	= \sgn (x)^\epsilon  \EF \varphi \, (-x),
	\end{equation*}
	with $\varphi (x) = \sgn(x)^\epsilon |x |^{-\mu} \upsilon (x) \in \SS (\BR).$
	
	{\rm(2).} For $\upsilon (z) \in [z]^{ k} \|z\|^{\mu} \SS (\BC) $, define $\ES_{(\mu, k)} \upsilon (z) =    \|z\|^{ \mu}  \EH_{(\mu, k)} \upsilon (z)$. Then 
	\begin{equation}\label{5eq: Mellin S, C}
	\EM_{-m} \ES_{(\mu, k)} \upsilon (2 s) = G_{k+m} (s) \EM_{m} \upsilon (2(1-s-\mu)), \hskip 10 pt m \in \BZ, 
	\end{equation} 
	and $\ES_{(\mu, k)}$ sends $[z]^{ k} \|z\|^{\mu} \SS (\BC) $   onto $ [z]^{- k} \SS (\BC)$ bijectively. Furthermore,
	\begin{equation*}
	\begin{split}
	\ES_{(\mu, k)} \upsilon (z) =   [z]^{-k} \int_{\BCx} [u]^{-k} \|u\|^{-\mu} \upsilon (u) e (z u + \overline {z u}) d u
	= [z]^{-k}  \EF \varphi \, (-z),
	\end{split}
	\end{equation*}
	with $ \varphi (z) =  [z]^{-k} \| z \|^{-\mu} \upsilon (z) \in \SS (\BC) $.
	
\end{lem}

\delete{According to \eqref{3def: Bessel function, R, 0} and \eqref{2eq: Bessel kernel over C, polar}, one has the decomposition
	\begin{equation*}
	J_{(0, 0)} \lp \pm x \rp = \frac 1 2 \lp j_{(0, 0)} (x)  \pm j_{(0, 1)} (x) \rp, \hskip 10 pt x\in \BR _+, 
	\end{equation*}
	and the series expansion
	\begin{equation*}
	J_{(0, 0)} \lp x  e^{i \phi} \rp = \frac 1 {2 \pi} \sum_{m \in \BZ} j_{(0, m)} (x)  e^{i m \phi}, \hskip 10 pt x\in \BR _+, \ \phi \in \BR / 2 \pi \BZ.
	\end{equation*}
	Recall the formulae of $j_{(0, \delta)}$, $\delta \in \BZT$, and $j_{(0, m)}$, $m \in \BZ$, that is, \eqref{3eq: j (0, delta)} in Example \ref{ex: Bessel kernel, real, 0, n=1, 2} and \eqref{3eq: j (0, m)} in Example \ref{ex: Bessel kernel j 0 m, complex, n=1}.  Some calculations yield the following lemma.
}

\begin{lem}\label{5lem: S, R and C}
	Let $(\mu, \epsilon) \in \BC \times \BZT$ and $(\mu, k) \in \BC \times \BZ$. 
	
	{\rm (1).} Let $\delta \in \BZT$. Suppose that  $\varphi (x) \in x^{ \mu} \SS_{\delta + \epsilon} (\overline \BR_+)$ and $\upsilon (x) =  \sgn(x)^{\delta}  \varphi (|x|) $. Then
	\begin{equation*}
	\begin{split}
	\ES_{(\mu, \epsilon)} \upsilon (\pm x) &=  (\pm)^{\delta } \int_{\BR_+} y^{- \mu} \varphi (y) j_{(0, \delta + \epsilon)} (x y) d y \\
	& = 
	\begin{cases}
	\ds (\pm)^{\epsilon} 2 \int_{\BR_+} y^{- \mu} \varphi (y) \cos (x y) d y, &  \text { if } \delta = \epsilon,\\
	\ds (\pm)^{\epsilon + 1} 2 i \int_{\BR_+} y^{- \mu} \varphi (y) \sin (x y) d y, &  \text { if } \delta = \epsilon + 1.
	\end{cases}
	\end{split}
	\end{equation*}
	The transform $\ES_{(\mu, \epsilon)}$ is a bijective map from  $\sgn (x)^{\epsilon} |x|^{\mu} \SS_{\delta} (\BR)$ onto $\sgn (x)^{\epsilon} \SS_{\delta} (\BR)$.
	
	{\rm (2).} Let $m \in \BZ$.  Suppose that  $\varphi (x) \in x^{2 \mu} \SS_{-m-k} (\overline \BR_+)$ and $\upsilon (z) =  [z]^{- m }  \varphi (|z|) $. Then
	\begin{equation*}
	\begin{split}
	\ES_{(\mu, k)} \upsilon \lp x e^{i \phi} \rp & = 2 e^{ i m \phi} \int_{\BR_+} y^{1 - 2 \mu} \varphi (y) j_{(0, m + k)} (x y)   d y \\
	& =  4 \pi i^{m+k} e^{ i m \phi} \int_{\BR_+}  y^{1 - 2 \mu} \varphi (y) J_{m+k} (4 \pi x y) d y.
	\end{split}
	\end{equation*}
	The transform  $\ES_{(\mu, k)}$ is a bijective map from  $[z]^{ k} \|z\|^{\mu} \SS_m (\BC) $ onto $ [z]^{- k} \SS_{-m} (\BC)$.
\end{lem}

\subsection{Miller-Schmid Transforms}

Following  \cite[\S 6]{Miller-Schmid-2004}, we now define certain transforms over  $\BR_+$, $\BR$ and $\BC$. Such transforms over $\BR$ have played an important role in the proof of the \Voronoi summation formula for $\GL_n (\BZ)$ in their subsequent work \cite{Miller-Schmid-2006, Miller-Schmid-2009}.

\vskip 5 pt

\subsubsection{The Miller-Schmid Transform $\ET_{(\varsigma, \lambda)}$}

\begin{lem} \label{5lem: Miller-Schmid transforms, R+}
	Let $(\varsigma, \lambda) \in \{+, -\} \times \BC $.
	
	{\rm (1).}
	For any $ \upsilon \in \Ssis (\BR_+)$ there is a unique function $\Upsilon \in \Ssis (\BR_+)$ satisfying the following identity,
	\begin{equation}\label{5eq: Miller-Schmid, R+}
	\EM \Upsilon (s ) = G  (s, \varsigma) \EM \upsilon (s + \lambda).
	\end{equation}
	We  write $ \Upsilon = \ET_{(\varsigma, \lambda)} \upsilon $ and call $\ET_{(\varsigma, \lambda)}$ the Miller-Schmid transform over $\BR_+$ of index $(\varsigma, \lambda)$.
	
	{\rm (2).}
	Let $\nu \in \BC$. Suppose  $\upsilon (x) \in   x^{- \nu} (\log x)^j \SS (\overline \BR_+)$. If  $\Re \nu <  \Re \lambda - \frac 1 2$, then
	\begin{equation}\label{5eq: Miller-Schmid and Fourier, R+}
	\begin{split}
	\ET_{(\varsigma, \lambda)} \upsilon (x)
    =  \int_{\BR_+} y^{-\lambda} \upsilon \lp y\-\rp  e^{\varsigma i x y} d^\times y .
	\end{split}
	\end{equation}
	
	{\rm (3).} Suppose that $\Re \nu < \Re \lambda$. Then the integral in \eqref{5eq: Miller-Schmid and Fourier, R+} is absolutely convergent and \eqref{5eq: Miller-Schmid and Fourier, R+} remains valid for any  $\upsilon (x) \in x^{- \nu} (\log x)^j \SS (\overline \BR_+)$.
	
	{\rm (4).}
	Suppose that $\Re \lambda > 0$. Define the function space
	\begin{equation*}
	\Ssiss (\BR_+) =    \sum_{\Re \nu\, \leq 0 } \ \sum_{j \in \BN}   x^{- \nu} (\log x)^j \SS (\overline \BR_+).
	\end{equation*}
	Then the transform $\ET_{(\varsigma, \lambda)}$ sends $\Ssiss (\BR_+)$ into itself. Moreover, \eqref{5eq: Miller-Schmid and Fourier, R+} also holds true for any $\upsilon \in \Ssiss (\BR_+)$, wherein the integral absolutely converges.
\end{lem}

\begin{proof}
	Following the ideas in the proof of Proposition \ref{3prop:H sigma lambda}, we may prove (1). 
	
	As for (2), we have
	\begin{equation*}
	\begin{split}
	\ET_{(\varsigma, \lambda)} \upsilon (x)
	& = \frac 1 {2 \pi i}  \int_{\BR_+} \upsilon (y) y^{\lambda}  \int_{ \EC _{0} } G  (s, \varsigma) y^{s } x^{-s} d s d^\times y  = \int_{\BR_+} \upsilon (y) y^{\lambda} J ( x y\-; \varsigma, 0) d^\times y,
	\end{split}
	\end{equation*}
	provided that the double integral is absolutely convergent. In order to guarantee the convergence of the integral over $d^\times  y$, the integral contour $\EC _{ 0}$ is required to lie in the right half-plane $\left\{ s : \Re s > \Re (\nu - \lambda) \right \}$. In view of Definition \ref{3defn: C d lambda}, such a choice of  $\EC _{0}$ is permissible since $\Re (\nu - \lambda) <  - \frac 1 2$ according to our assumption. Finally, the change of variables from $y$ to $y\-$, along with the formula $J ( x  ; \varsigma, 0) = e^{\varsigma i x}$, yields \eqref{5eq: Miller-Schmid and Fourier, R+}. 
	
	For the case  $\Re \nu < \Re \lambda$ in (3), the absolute convergence of the integral in \eqref{5eq: Miller-Schmid and Fourier, R+} is obvious. The validity of \eqref{5eq: Miller-Schmid and Fourier, R+} follows from the analyticity with respect to $\lambda$.
	
	Under the isomorphism established by $\EM $ in Lemma \ref{2lem: Sis and Msis, R+}, $\Ssiss (\BR_+)$ corresponds to the subspace of $\Msis $ consisting of  meromorphic functions $H$ such that the poles of   $H $ lie in the left half-plane $\left\{ s : \Re s \leq 0 \right\}$.
	Hence, the first assertion in (4) is clear, since  the map that corresponds to  $\ET _{(\varsigma, \lambda)}$ is given by $ H (s) \mapsto  G  (s, \varsigma) H (s + \lambda) $ and sends the  subspace of $\Msis $ described above into itself. The second assertion  in (4) immediately follows from (3).
\end{proof}

\subsubsection{The Miller-Schmid Transform $\ET _{(\mu, \epsilon)}$}

\begin{lem} \label{5lem: Miller-Schmid transforms, R}
	Let $(\mu, \epsilon) \in \BC \times \BZ/2 \BZ $.
	
	{\rm (1).}
	For any $ \upsilon \in \Ssis (\BRx)$ there is a unique function $\Upsilon \in \Ssis (\BRx)$ satisfying the following two identities,
	\begin{equation}\label{5eq: Miller-Schmid, R}
	\EM _\delta \Upsilon (s ) = G_{ \epsilon + \delta } (s) \EM _\delta \upsilon (s + \mu), \hskip 10 pt \delta \in \BZ/2 \BZ.
	\end{equation}
	We  write $ \Upsilon = \ET_{(\mu, \epsilon)} \upsilon $ and call $\ET_{(\mu, \epsilon)}$ the Miller-Schmid transform over $\BR$ of index $(\mu, \epsilon)$.
	
	{\rm (2).}
	Let $\lambda \in \BC$. Suppose  $\upsilon (x) \in \sgn (x)^{\delta } |x|^{-\lambda} (\log |x|)^j \SS (\BR)$. If  $\Re \lambda <  \Re \mu - \frac 1 2$, then
	\begin{equation}\label{5eq: Miller-Schmid and Fourier, R}
	\begin{split}
	\ET_{(\mu, \epsilon)} \upsilon (x)
	& = \sgn (x)^{\epsilon} \int_{\BRx} \sgn (y)^{\epsilon} |y|^{-\mu} \upsilon \lp y\-\rp  e (x y) d^\times y   = \sgn (x)^{\epsilon} \EF \varphi \, (- x),
	\end{split}
	\end{equation}
	with $\varphi (x) = \sgn(x)^\epsilon |x |^{-\mu-1}  \upsilon \lp x\-\rp$. 
	
	{\rm (3).} Suppose that $\Re \lambda < \Re \mu$. Then the integral in \eqref{5eq: Miller-Schmid and Fourier, R} is absolutely convergent and \eqref{5eq: Miller-Schmid and Fourier, R} remains valid for any  $\upsilon (x) \in \sgn (x)^{\delta } |x|^{-\lambda} (\log |x|)^j \SS (\BR)$.
	
	{\rm (4).}
	Suppose that $\Re \mu > 0$. Define the function space
	\begin{equation*}
	\Ssiss (\BRx) =  \sum_{\delta \in \BZT} \ \sum_{\Re \lambda \, \leq 0 } \ \sum_{j \in \BN}  \sgn(x)^{\delta} |x|^{-\lambda} (\log |x|)^j \SS (\BR).
	\end{equation*}
	Then the transform $\ET_{(\mu, \epsilon)}$ sends $\Ssiss (\BRx)$ into itself. Moreover, \eqref{5eq: Miller-Schmid and Fourier, R} also holds true for any $\upsilon \in \Ssiss (\BRx)$, wherein the integral absolutely converges.
\end{lem}

\begin{proof}
	Following literally the same ideas in the proof of Lemma \ref{5lem: Miller-Schmid transforms, R+}, one may show this lemma without any difficulty.
	Observe that, under the isomorphism established by $\EM_{\BR}$ in Lemma \ref{2lem: Ssis to Msis, R}, $\Ssiss (\BRx)$ corresponds to the subspace of $\Msis^{\BR}$ consisting of pairs of meromorphic functions $(H_0, H_1)$ such that the poles of both $H_0$ and $ H_1$ lie in the left half-plane $\left\{ s : \Re s \leq 0 \right\}$ (see Lemma \ref{2lem: Ssis delta, M delta}). 
\end{proof}

Similar to Lemma \ref{5lem: S, R and C} (1), we have the following lemma.

\begin{lem}\label{5lem: T, R}
	Let $(\mu, \epsilon)  \in \BC \times \BZT$ be such that $\Re \mu > 0$. For $\delta \in \BZT$ define $\Ssiss^{ \delta} (\BRx)$ to be the space of functions in $ \Ssiss (\BRx)$  satisfying the condition \eqref{1eq: delta condition, R}. For  $\upsilon \in \Ssiss^{ \delta} (\BRx)$, we write $\upsilon (x) =  \sgn(x)^{\delta }  \varphi (|x|) $. Then
	\begin{equation*}
	\begin{split}
	\ET_{(\mu, \epsilon)} \upsilon (\pm x) &=  (\pm)^{\delta } \int_{\BR_+} y^{- \mu} \varphi \lp y\- \rp j_{(0, \delta + \epsilon)} (x y) d^\times y \\
	& = 
	\begin{cases}
	\ds (\pm)^{\delta} 2 \int_{\BR_+} y^{- \mu} \varphi \lp y\- \rp \cos (x y) d^\times y, &  \text { if } \delta = \epsilon,\\
	\ds (\pm)^{\delta} 2 i \int_{\BR_+} y^{- \mu} \varphi \lp y\- \rp \sin (x y) d^\times y, &  \text { if } \delta = \epsilon + 1.
	\end{cases}
	\end{split}
	\end{equation*}
	The transform $\ET_{(\mu, \epsilon)}$ sends $\Ssiss^{ \delta} (\BRx)$ into itself.
\end{lem}

\subsubsection{The Miller-Schmid Transform $\ET _{(\mu, k)}$}
In parallel to Lemma \ref{5lem: Miller-Schmid transforms, R}, the following lemma defines the Miller-Schmid transform $\ET _{(\mu, k)}$ over $\BC$ and gives its connection to the Fourier transform over $\BC$.

\begin{lem} \label{5lem: Miller-Schmid transforms, C}
	Let $(\mu, k) \in \BC \times \BZ $.
	
	{\rm (1).}
	For any $ \upsilon \in \Ssis (\BCx)$ there is a unique function $\Upsilon \in \Ssis (\BCx)$ satisfying the following sequence of identities,
	\begin{equation}\label{5eq: Miller-Schmid, C}
	\EM _{-m} \Upsilon (2s ) = G_{ m + k } (s) \EM _{-m} \upsilon (2( s + \mu) ), \hskip 10 pt m \in \BZ.
	\end{equation}
	We  write $ \Upsilon = \ET_{(\mu, k)} \upsilon $ and call $\ET_{(\mu, k)}$ the Miller-Schmid transform  over $\BC$ of index $(\mu, k)$.
	
	{\rm (2).}
	Let $\lambda \in \BC$. If $\Re \lambda < 2\, \Re \mu$,  then for any $\upsilon (z) \in [z]^{m } |z|^{-\lambda} (\log |z|)^j \SS (\BC)$ we have
	\begin{equation}\label{5eq: Miller-Schmid and Fourier, C}
	\begin{split}
	\ET_{(\mu, k)} \upsilon (z)
	& = [z]^{k} \int_{\BCx} [u]^{k} \|u\|^{-\mu} \upsilon \lp u\-\rp  e (z u + \overline {z u}) d^\times u  = [z]^{k} \EF \varphi \, (- z),
	\end{split}
	\end{equation}
	with $\varphi (z ) = [ z ]^k \| z \|^{-\mu-1}    \upsilon \lp z\- \rp$.
	
	{\rm (3).} When $\Re \lambda < 2 \Re \mu$, the integral in \eqref{5eq: Miller-Schmid and Fourier, C} is absolutely convergent for any  $\upsilon (z) \in [z]^{m } |z|^{-\lambda} (\log |z|)^j \SS (\BC)$.

	{\rm (4).}
	Suppose that $\Re \mu > 0$. Define the function space 
	\begin{equation*}
	\Ssiss (\BCx) =  \sum_{m \in \BZ} \ \sum_{\Re \lambda \, \leq 0 } \ \sum_{j \in \BN}  [z]^{m} |z|^{-\lambda} (\log |z|)^j \SS (\BC).
	\end{equation*}
	Then the transform $\ET_{(\mu, k)}$ sends $\Ssiss  (\BCx)$ into itself. Moreover, \eqref{5eq: Miller-Schmid and Fourier, C} also holds true for any $\upsilon \in \Ssiss  (\BCx)$, wherein the integral absolutely converges.
\end{lem}

\begin{proof} Again, the proof is similar as above.
	 We only remark that, via the isomorphism $\EM_{\BC}$ in Lemma \ref{2lem: Ssis to Msis, C}, $\Ssiss  (\BCx)$ corresponds to the subspace of $\Msis^{\BC}$ consisting of sequences $\lpp H_m\rpp $ such that the poles of each $H_{m}$ lie in the left half-plane $\big\{ s : \Re s \leq \min \{ M - |m|, 0 \} \big\}$ for some $M \in \BN$ (see Lemma \ref{2lem: Ssis m, M -m}).
\end{proof}

\begin{lem}\label{5lem: T, C}
	Let $(\mu, k)  \in \BC \times \BZ$ be such that $\Re \mu > 0$. For $m \in \BZ $ define $\Ssiss^{  m} (\BCx)$ to be the space of functions in $ \Ssiss  (\BCx)$  satisfying the condition \eqref{1eq: m condition, C}. For  $\upsilon \in \Ssiss^{   m} (\BCx)$, we write $\upsilon (z) =  [z]^{m }  \varphi (|z|) $. Then
	\begin{align*}
	\ET_{(\mu, k)} \upsilon \lp x e^{i \phi} \rp & = 2 e^{ i m \phi} \int_{\BR_+} y^{- 2 \mu} \varphi \lp y\- \rp j_{(0, m + k)} (x y)  d^\times y \\
	& =  4 \pi i^{m+k} e^{ i m \phi} \int_{\BR_+}  y^{- 2 \mu} \varphi \lp y\- \rp J_{m+k} (4 \pi x y)    d^\times y.
	\end{align*}
	The transform $\ET_{(\mu, k)}$ sends $\Ssiss^{  m} (\BCx)$ into itself.
\end{lem}

\subsection{Fourier Type Integral Transforms}

Following \cite{Miller-Schmid-2006,Miller-Schmid-2009}, we shall now derive the Fourier type integral transform expressions for $\Hsl$, $\Hld$ and $\Hmum$ from the Fourier transform (more precisely, the renormalized rank-one Hankel transforms) and the Miller-Schmid transforms.

\vskip 5 pt

\subsubsection{The Fourier Type Transform Expression for $\Hsl$}\label{sec: Fourier, R+}

Let $(\usigma, \ulambda) \in \{+,-\}^n \times \BC^n$. For $\upsilon \in \SS (\BR_+)$, we put
\begin{equation}\label{5eq: composite T S, R+}
\Upsilon (x) = x^{-\lambda_1} \ET_{(\varsigma_1, \lambda_1-\lambda_2)}\, {\sstyle \ccirc} \, ...\, {\sstyle \ccirc}\, \ET_{(\varsigma_{n-1}, \lambda_{n-1}-\lambda_{n})}\, {\sstyle \ccirc}\, \ES_{(\varsigma_n, \lambda_n)} \upsilon (x).
\end{equation}
According to Lemma \ref{5lem: Hankel R+} and Lemma \ref{5lem: Miller-Schmid transforms, R+} (1), $\ES_{(\varsigma_n, \lambda_n)} \upsilon (x)$ lies in the space $ \SS (\overline \BR_+) $($ \subset \Ssis(\BR_+) $), whereas each Miller-Schmid transform sends $\Ssis(\BR_+)$ into itself.
Thus  one can apply the Mellin transform $\EM$ to both sides of \eqref{5eq: composite T S, R+}. Using \eqref{5eq: Mellin S, R+} and \eqref{5eq: Miller-Schmid, R+}, some calculations show that the application of $\EM $ converts \eqref{5eq: composite T S, R+} exactly  into the identity \eqref{3eq: Hankel transform identity 0, R+} which defines $\Hsl$. Therefore, $\Upsilon = \Hsl \upsilon$. 

\begin{thm}\label{5thm: Fourier type transform, R+}
	Let $(\usigma, \ulambda) \in \{+,-\}^n \times \BC^n$ be such that $\Re \lambda_1  > ... > \Re \lambda_{n-1} > \Re \lambda_n$. Suppose  $\upsilon \in \SS (\BR_+)$. Then
	\begin{equation} \label{5eq: Fourier type integral, R+}
		\Hsl \upsilon (x) = \frac 1 {x} \int_{\BR_+^{  n}} \upsilon \lp \frac {x_1 ... x_n} x \rp \lp \prod_{l      = 1}^{n }    x_{l     }^{ - \lambda_{l     }} e^{\varsigma_l i x_l}  \rp dx_n ... d x_1,
	\end{equation}
	where the integral converges when performed as iterated integral in the indicated order $d x_n d x_{n-1} ... d x_1$, starting from $d x_n$, then $d x_{n-1}$, ..., and finally $d x_1$.
\end{thm}

\begin{proof}
	We first observe that $\ES_{(\varsigma_n, \lambda_n )} \upsilon (x) \in \SS (\overline \BR_+) \subset \Ssiss (\BR_+)$. For each $ l      = 1 ,..., n-1$, since $\Re (\lambda_{l     } -\lambda_{l     +1}) > 0$, Lemma \ref{5lem: Miller-Schmid transforms, R+} (4) implies that the transform $\ET_{(\varsigma_l, \lambda_{l     } -\lambda_{l     +1} )}$ sends the space $\Ssiss (\BR_+)$ into itself. According to Lemma \ref{5lem: Hankel R+} and Lemma \ref{5lem: Miller-Schmid transforms, R+} (3), $\ES_{(\varsigma_n, \lambda_n )} $ and all the $\ET_{(\varsigma_l, \lambda_{l     } -\lambda_{l     +1} )}$ in \eqref{5eq: composite T S, R+}  may be expressed as integral transforms, which are absolutely convergent. From these, the right hand side of \eqref{5eq: composite T S, R+} turns into the integral,
	\begin{equation*}
	\begin{split}
	\int_{\BR_+^{  n}}  x^{-\lambda_1} e^{\varsigma_1 i x y_1} 
	\lp \prod_{l      = 1}^{n-1}   y_{l     } ^{\lambda_{l      + 1} - \lambda_{l     } - 1} e^{ \varsigma_l i y_{l     }\- y_{l     +1}  } \rp  
	 y_n^{-\lambda_n} \upsilon (y_n)   d y_{n}  ...  d y_{1},
	\end{split}
	\end{equation*}
	which converges as iterated integral.
	Our proof is completed upon making the change of variables $x_1 = x y_1$, $x_{l      + 1} = y_l     \- y_{l      + 1}$, $ l      = 1, ..., n-1$.
\end{proof}

\subsubsection{The Fourier Type Transform Expression for $\Hld$}

Let $(\umu, \udelta) \in \BC^n\times (\BZT)^n$. For $\upsilon (x) \in \sgn (x)^{\delta_n} |x|^{\mu_n} \SS (\BR) $, by Lemma \ref{5cor: renormalize Hankel} (1) and Lemma \ref{5lem: Miller-Schmid transforms, R} (1), in particular \eqref{5eq: Mellin S, R} and \eqref{5eq: Miller-Schmid, R}, we may prove that
\begin{equation}\label{5eq: composite T S, R}
 \Hld \upsilon (x) = |x|^{-\mu_1} \ET_{(\mu_1-\mu_2, \delta_1)}\, {\sstyle \ccirc} \, ...\, {\sstyle \ccirc}\, \ET_{(\mu_{n-1}-\mu_{n}, \delta_{n-1})}\, {\sstyle \ccirc}\, \ES_{(\mu_n, \delta_n)} \upsilon (x).
\end{equation}

\begin{thm}\label{5thm: Fourier type transform, R}
	\cite[(1.3)]{Miller-Schmid-2009}. 
	Let $(\umu, \udelta) \in \BC^n\times (\BZT)^n$ be such that $\Re \mu_1  > ... > \Re \mu_{n-1} > \Re \mu_n$. Suppose  $\upsilon (x) \in \sgn (x)^{\delta_n} |x|^{\mu_n} \SS (\BR) $. Then
	\begin{equation} \label{5eq: Fourier type integral, R}
	\Hld \upsilon (x) = \frac 1 {|x|} \int_{\BR^{\times n}} \upsilon \lp \frac {x_1 ... x_n} x \rp \lp \prod_{l      = 1}^{n }  \sgn (x_l     )^{ \delta_{l     }} |x_{l     }|^{ - \mu_{l     }} e \lp x_{l     } \rp \rp dx_n ... d x_1,
	\end{equation}
	where the integral converges when performed as iterated integral in the indicated order $d x_n d x_{n-1} ... d x_1$, starting from $d x_n$, then $d x_{n-1}$, ..., and finally $d x_1$.
\end{thm}

\begin{proof}
	One applies the same arguments in the proof of Theorem \ref{5thm: Fourier type transform, R+} using Lemma \ref{5cor: renormalize Hankel} (1) and Lemma \ref{5lem: Miller-Schmid transforms, R}  (3, 4). 
\end{proof}

We have the following corollary to Theorem \ref{5thm: Fourier type transform, R}, which can also be seen from Lemma \ref{5lem: S, R and C} (1) and Lemma \ref{5lem: T, R}.

\begin{cor}\label{5cor: H, R}
	Let $(\umu, \udelta) \in \BC^n\times (\BZT)^n$ and  $\delta \in \BZT$. Assume that $\Re \mu_1 > ... > \Re \mu_{n-1} > \Re \mu_n$. Let $\varphi (x) \in x^{ \mu_n} \SS_{ \delta + \delta_n} (\overline \BR_+)$ and $\upsilon (x) =  \sgn (x)^{ \delta }  \varphi (|x|) $. Then
	\begin{equation} \label{5eq: Fourier type integral, R, 1}
	\Hmum \upsilon \lp \pm x  \rp = \frac { (\pm)^{\delta}} {x} \int_{\BR _+ ^n} \varphi \lp \frac {x_1 ... x_n} x \rp \lp \prod_{l      = 1}^{n }   x_{l     } ^{ -  \mu_{l     } } j_{(0, \delta_{l     } + \delta)} (x_{l     }) \rp d x_n ... d x_1,
	\end{equation}
	with $x \in \BR _+$. Here the iterated integration is performed in the  indicated order.
\end{cor}

\subsubsection{The Fourier Type Transform Expression for $\Hmum$}

Let $(\umu, \um) \in \BC^n \times \BZ^n$. For $\upsilon (z) \in [z]^{m_n} \|z\|^{\mu_n} \SS (\BC) $, using Lemma \ref{5cor: renormalize Hankel} (2) and Lemma \ref{5lem: Miller-Schmid transforms, C} (1), especially \eqref{5eq: Mellin S, C} and \eqref{5eq: Miller-Schmid, C}, one may show that
\begin{equation}\label{5eq: composite T S, C}
\Hmum \upsilon (z) = \|z\|^{-\mu_1} \ET_{(\mu_1-\mu_2, m_1)} \, {\sstyle \ccirc} \,  ... \, {\sstyle \ccirc} \,  \ET_{(\mu_{n-1}-\mu_{n}, m_{n-1})} \, {\sstyle \ccirc} \,  \ES_{(\mu_n, m_n)} \upsilon (z).
\end{equation}

\begin{thm}\label{5thm: Fourier type transform, C}
	Let $(\umu, \um) \in \BC^n\times \BZ^n$ be such that $\Re \mu_1 > ... > \Re \mu_{n-1} > \Re \mu_n$. Suppose  $\upsilon (z) \in [z]^{m_n} \|z\|^{\mu_n} \SS (\BC) $. Then
	\begin{equation} \label{5eq: Fourier type integral, C}
	\Hmum \upsilon (z) = \frac 1 {\|z\|} \int_{\BC^{\times n}} \upsilon \lp \frac {z_1 ... z_n} z \rp \lp \prod_{l      = 1}^{n }  [z_l     ]^{- m_{l     }} \|z_{l     }\|^{ - \mu_{l     }} e \lp z_{l     } + \overline {z_l     } \rp \rp  dz_n ... d z_1,
	\end{equation}
	where the integral converges when performed as iterated integral in the indicated order. 
\end{thm}
\begin{proof}
	One applies the same arguments in the proof of Theorem \ref{5thm: Fourier type transform, R+} using Lemma \ref{5cor: renormalize Hankel} (2) and Lemma \ref{5lem: Miller-Schmid transforms, C} (3, 4). 
\end{proof}

Theorem \ref{5thm: Fourier type transform, C}, or Lemma \ref{5lem: S, R and C} (2) and Lemma \ref{5lem: T, C}, yields the following corollary.

\begin{cor}\label{5cor: H, C}
	Let $(\umu, \um) \in \BC^n\times \BZ^n$ and  $m \in \BZ$. Assume that $\Re \mu_1 > ... > \Re \mu_{n-1} > \Re \mu_n$. Let  $\varphi (x) \in x^{2 \mu_n} \SS_{-m-m_n} (\overline \BR_+)$ and $\upsilon (z) =  [z]^{- m }  \varphi (|z|) $. Then
	\begin{equation} \label{5eq: Fourier type integral, C, 1}
	\Hmum \upsilon \lp x e^{i \phi} \rp = 2^n \frac  {e^{i m \phi}} {x^{2}} \int_{\BR _+ ^n} \varphi \lp \frac {x_1 ... x_n} x \rp \lp \prod_{l      = 1}^{n }   x_{l     } ^{ - 2 \mu_{l     } + 1} j_{(0, m_{l     } + m)} (x_{l     }) \rp d x_n ... d x_1,
	\end{equation}
	with $x \in \BR _+$ and $\phi \in \BR/ 2\pi \BZ$. Here the iterated integration is performed in the indicated order.
\end{cor}

\section{Integral Representations of Bessel Kernels}\label{sec: integral representations}


First, we shall derive a formal integral representation of the Bessel function $J (x; \usigma, \ulambda)$ from  Theorem \ref{5thm: Fourier type transform, R+} in symbolic manner, ignoring  the assumption $\Re \lambda_1 > ... > \Re \lambda_n$.
This formal integral will be the subject that we study in \S \ref{sec: Rigorous Interpretations} - \ref{sec: Bessel functions of K-type and H-Bessel functions} of Chapter \ref{chap: analytic theory}.


Similarly, Theorem \ref{5thm: Fourier type transform, R} and \ref{5thm: Fourier type transform, C}  also yield  formal integral representations of $J_{(\umu, \udelta)} (x)$ and $J_{(\umu, \um)} (z)$ respectively. It turns out that one can naturally transform the formal integral of $J_{(\umu, \um)} (z)$ into an integral that is absolutely convergent, given that the index $\umu$ satisfies certain conditions. The main reason for the absolute convergence is that $j_{(0, m)} (x) = 2 \pi i^{m } J_{m} (4 \pi x)$ (see  \eqref{3eq: j (0, m)}) decays proportionally to  $ 1 / {\sqrt x}$ at infinity (in comparison, $j_{(0, \delta)} (x)$ is equal to either  $2 \cos (2\pi x)$ or $ 2 i \sin (2\pi x)$). 



\subsection*{Assumptions and Notations} 

Let  $n \geq 2$. Assume that $\ulambda, \umu \in \BL^{n-1}$. 
\begin{notation}\label{not: d, nu}
	Let $d = n-1$. Let the pairs of tuples, $\ulambda$ or $\umu \in \BL^d $ and $\unu \in \BC^d$, $\udelta \in (\BZT)^{d+1}$ and $\boldsymbol \epsilon \in (\BZT)^d$, $\um \in \BZ^{d+1}$  and $\uk \in \BZ^d$, be subjected to the following relations 
	\begin{equation*}
		\nu_{l     } = \lambda_{l     } - \lambda_{d+1}, \hskip 10 pt \nu_{l     } = \mu_{l     } - \mu_{d+1}, \hskip 10 pt \epsilon_{l     } = \delta_{l     } + \delta_{d+1}, \hskip 10 pt  k_{l     } = m_{l     } - m_{d+1},
	\end{equation*}
	for $l      = 1, ..., d.$
\end{notation}

Rather than Hankel transforms, we shall be interested in their Bessel kernels. Therefore, it is convenient to further assume that the weight functions are Schwartz, namely, $\varphi \in \SS (\BR_+)$ and $\upsilon \in \SS (\BFx)$. 
According to (\ref{2eq: Psi (x; sigma) as Hankel transform}, \ref{3eq: integral kernel, R 0}, \ref{3eq: h mu m = int of j mu m}), Proposition \ref{3prop: properties of J, R} (3) and \ref{3prop: properties of J, C} (3), for such Schwartz functions $\varphi$ and $\upsilon$, 
\begin{align} 
    \label{6eq: Hankel = integral of Bessel function} & \Hsl \varphi  (x ) =  \int_{\BR_+} \varphi (y) J \big( (xy)^{\frac 1 n}; \usigma, \ulambda\big) d y,   \\
	\label{6eq: Hankel = integral of j} & \hld \varphi  (x ) =  \int_{\BR_+} \varphi (y) j_{(\umu, \udelta)} ( xy ) d y, \hskip 18 pt \hmum \varphi (x)  = 2 \int_{\BR_+}  \varphi (y) j_{(\umu, \um)} ( xy )  y d y, \\
	\label{6eq: Hankel = integral of Bessel kernel} &  \Hld \upsilon (x) = \int_{\BR ^\times} \upsilon (y) J_{(\umu, \udelta)} (xy ) d y, \hskip 15 pt \Hmum  \upsilon (z) = \int_{\BC^\times} \upsilon (u) J_{(\umu, \um)} (z u ) d u,
\end{align}
with indices $(\usigma, \ulambda) \in \{+, -\}^n \times \BC^n$, $(\umu, \udelta) \in \BC^n \times (\BZT)^n$ and $(\umu, \um) \in \BC^n \times \BZ^n$ being arbitrary.




\subsection{\texorpdfstring{The Formal Integral  $J_{\unu } ( x; \usigma )$}{The Formal Integral  $J_{\nu} ( x; \sigma )$}}\label{sec: formal integral of Bessel functions}

To motivate the definition of  $J_{\unu } ( x; \usigma )$, we shall do certain operations on the Fourier type integral \eqref{5eq: Fourier type integral, R+}  in Theorem \ref{5thm: Fourier type transform, R+}. In the meanwhile, we shall ignore the assumption $\Re\lambda_1 > ... > \Re \lambda_n$.

Upon making the change of variables, $x_{n} = (x_1 ... x_{n-1})\- x y$, $x_{l     } =  y_{l     }\-$, $ l      = 1, ..., n-1$, one converts \eqref{5eq: Fourier type integral, R+}   into
\begin{align*}
\Hsl \upsilon (x) = \int_{\BR_+^n} \upsilon \left( y \right) (x y)^{- \lambda_n }  \left( \prod_{l      = 1}^{n-1} y_l     ^{ \lambda_l      - \lambda_n - 1} \right) e^{ i \lp \varsigma_n  xy y_1 ... y_{n-1} + \sum_{l      = 1}^{n-1} \varsigma_l      y_l     \- \rp } dy d y_{n-1} ... d y_1.
\end{align*}
In symbolic notation, moving the integral over $d y$ to the outermost place and comparing the resulting integral with the right hand side of   \eqref{6eq: Hankel = integral of Bessel function}, the Bessel function $J   (x ; \usigma, \ulambda  )$ is then represented by the following formal integral over $d y_{n-1} ... d y_{1}$,
\begin{align*} 
\int_{\BR_+^{ n-1}} x^{- n \lambda_n }  \left( \prod_{l      = 1}^{n-1} y_l     ^{ \lambda_l      - \lambda_n - 1} \right) e^{ i \lp \varsigma_n  x^n y_1 ... y_{n-1} + \sum_{l      = 1}^{n-1} \varsigma_l      y_l     \- \rp }  d y_{n-1} ... d y_1
\end{align*}
Another change of variables $y_{l     } = t_{l     } x\-$, along with the assumption  $\sum_{l      = 1}^{n} \lambda_l      = 0$,  turns this integral into
\begin{equation*}
\int_{\BR_+^{n-1}} \left(\prod_{l      = 1}^{n-1} t_l     ^{ \lambda_l      - \lambda_n - 1} \right) e^{i x \left(\varsigma_n t_1 ... t_{n-1} + \sum_{l      = 1}^{n-1} \varsigma_l      t_l     \- \right)} d t_{n-1} ... dt_1.
\end{equation*}
For $\unu \in \BC^d$ and $\usigma \in \{+, -\}^{d+1}$, we define the formal integral
\begin{equation}\label{6eq: formal integral, R+}
\begin{split}
J_{\unu } ( x; \usigma ) = \int_{\BR_+^{  d}} 
\lp \prod_{l      = 1}^{d}  t_{l     }^{ \nu_{l     } - 1} \rp e^{ i x \lp \varsigma_{d+1}  t_1 ... t_{d} + \sum_{l      = 1}^{d} \varsigma_l t_{l     }\-  \rp} d t_{d} ... d t_1, \hskip 5 pt x \in \BR_+.
\end{split}
\end{equation}
In view of Notation \ref{not: d, nu}, we have $ J   (x ; \usigma, \ulambda  ) = J_{\unu } ( x; \usigma )$ in symbolic notation.

\begin{rem}
	Before realizing its connection with the Fourier type transform   {\rm (\ref{5eq: Fourier type integral, R+})} in Theorem \ref{5thm: Fourier type transform, R+}, the formal integral representation $J_{\unu } ( x; \usigma )$  of $J ( x ; \usigma, \ulambda)$ was derived by the author from {\rm \eqref{3eq: Hankel transform identity 0, R+}} in Proposition \ref{3prop:H sigma lambda} based on a symbolic application of the product-convolution principle of the Mellin transform together with the  formula
	\begin{equation*}
	\Gamma (s) e \lp \pm \frac  s 4 \rp = \int_0^\infty e^{\pm ix} x^{s } d^\times x, \hskip 10 pt 0 < \Re s < 1.
	\end{equation*}
	Though not specified, this principle is implicitly suggested in Miller and Schmid's work, especially, \cite[Theorem 4.12, Lemma 6.19]{Miller-Schmid-2004} and \cite[(5.22, 5.26)]{Miller-Schmid-2006}.
\end{rem}

\subsection{\texorpdfstring{The Formal Integral  $J_{\unu, \boldsymbol \epsilon} ( x,\pm )$}{The Formal Integral  $J_{\nu, \epsilon} ( x,\pm )$}}\label{sec: formal integral, R}

We may proceed in the same way as in \S \ref{sec: formal integral of Bessel functions}, starting from the Fourier type integral \eqref{5eq: Fourier type integral, R} in Theorem \ref{5thm: Fourier type transform, R}. 
For $\unu \in \BC^d$ and $\boldsymbol \epsilon \in (\BZT)^d$, we define the formal integral
\begin{equation}
	\begin{split}
		J_{\unu, \boldsymbol \epsilon} (x, \pm) = \int_{\BR^{\times d}} 
		\lp \prod_{l      = 1}^{d}  \sgn (y_l     )^{ \epsilon_{l     } } |y_{l     }|^{ \nu_{l     } - 1} \rp e^{ i x \lp \pm  y_1 ... y_{d} + \sum_{l      = 1}^{d} y_{l     }\-  \rp} d y_{d} ... d y_1, \hskip 5 pt x \in \BR_+.
	\end{split}
\end{equation}
Then follows the symbolic identity $ J_{(\umu,\udelta)} (\pm  x  ) = (\pm)^{\delta_{d+1}}  
J_{\unu, \boldsymbol \epsilon} \big( 2 \pi x^{\frac 1 {d+1}}, \pm \big)$.

By splitting $\BRx = \BR_+ \cup (- \BR_+)$ in the domain of integral and letting $y_l = \varsigma_l t_l$, some formal calculations yield the following formula
\begin{equation}\label{3eq: Bessel kernel, R, connection, formal}
\begin{split}
J_{\unu, \boldsymbol \epsilon} \lp x; \pm \rp = \sum_{ \usigma\, \in \{+, -\}^d } \usigma      ^{\boldsymbol \epsilon } J_{\unu}  (x;   \pm |\usigma| ), 
\end{split}
\end{equation}
which corresponds to the formula \eqref{3eq: Bessel kernel, R, connection} for $J_{(\umu,\udelta)} ( x  )$ and $J \big(2 \pi x^{\frac 1 {d+1}}; \usigma, \umu \big)$.

\subsection{\texorpdfstring{The Formal Integral  $j_{\unu, \udelta} (x)$}{The Formal Integral  $j_{\nu, \delta} (x)$}} 
For $\unu \in \BC^d$ and $ \udelta \in (\BZT)^{d+1}$, we define the formal integral
\begin{equation}
	\begin{split}
		j_{\unu, \udelta} (x ) = \int_{ \BR_+ ^{d}} &  j_{(0, \delta_{d+1})} \lp x y_{1} ... y_d \rp \prod_{l      = 1}^{d} y_{l     }^{ \nu_{l     } - 1} j_{(0, \delta_{l     })} \lp x y_{l     } \- \rp   
		d y_{d} ... d y_1, \hskip 10 pt x \in \BR_+.
	\end{split}
\end{equation}
We may derive the symbolic identity  $j_{(\umu, \udelta)} (x) =  
j_{\unu, \udelta} \big( x^{\frac 1 {d+1}} \big)$ from Corollary \ref{3cor: H = h, R} and  \ref{5cor: H, R}.



\subsection{\texorpdfstring{The Integral  $J_{\unu, \uk} ( x, u )$}{The Integral  $J_{\nu, k} ( x, u )$}}

First of all, proceeding in the same way as in \S \ref{sec: formal integral of Bessel functions}, from the Fourier type integral \eqref{5eq: Fourier type integral, C} in Theorem \ref{5thm: Fourier type transform, C}, we can deduce the symbolic equality $ J_{(\umu,\um)} \lp  x  e^{i\phi} \rp = e^{- i m_{d+1} \phi} 
J_{\unu, \uk} \big( 2 \pi x^{\frac 1 {d+1}}, e^{i\phi} \big)$, with the definition of the formal integral,
\begin{equation}\label{5eq: formal integral, C, 0}
	\begin{split}
		J_{\unu, \uk} ( x, u ) =    \int_{\BC^{\times d}}
		\lp \prod_{l      = 1}^{d}  [u_{l     }]^{k_{l     } } \|u_{l     }\|^{ \nu_{l     } - 1} \rp  &  e^{ i x \Lambda \lp u u_1 ... u_{d} + \sum_{l      = 1}^{d} u_{l     }\-  \rp  }  d u_{d} ... d u_1, \\ 
		& \hskip 60 pt x \in \BR_+,\, u \in \BC, |u| = 1.
	\end{split}
\end{equation}
Here, we recall that $\Lambda (z) = z + \overline z$. 

In the polar coordinates, we write $u_{l     } = y_{l     } e^{i \theta_{l     }}$ and $u = e^{i \phi}$. Moving the integral over the torus $(\BR/ 2 \pi \BZ)^d$ inside, in symbolic manner, the integral above turns into
\begin{equation*}
	\begin{split}
		2^d \int_{\BR_+^d} \hskip -3 pt \int_{(\BR/ 2 \pi \BZ)^d} \hskip -3 pt
		\lp \prod_{l      = 1}^{d} y_{l     }^{ 2 \nu_{l     } - 1} \hskip -2 pt \rp  \hskip -2 pt e^{i \sum_{l      = 1}^d k_{l     } \theta_l    + 2 i  x  \lp y_1 ... y_{d} \cos \lp  \sum_{l      = 1}^d \theta_{l     } + \phi \rp + \sum_{l      = 1}^{d} y_{l     }\- \cos \theta_l        \rp } d \theta_{d} ... d \theta_1 d y_{d} ... d y_1.
	\end{split}
\end{equation*}
Let us introduce the following definitions
\begin{align}
	\label{6eq: Theta (theta, y; x, phi)}
	\Theta_{\uk} (\utheta, \uy; x, \phi) = 2 x  y_1 ... y_{d} & \cos \lp \textstyle \sum_{l      = 1}^d \theta_{l     } + \phi \rp + \sum_{l      = 1}^d \lp k_{l     } \theta_l      + 2 x y_{l     }\- \cos \theta_l      \rp ,\\
	\label{6eq: Jk (y;x, phi)}
	J_{\uk}(\uy; x, \phi ) & = 
	\int_{(\BR/ 2 \pi \BZ)^d} e^{i \Theta_{\uk} (\utheta, \uy; x, \phi) } d\utheta, \\
	p_{2 \unu} & (\uy) = \prod_{l      = 1}^d   y_{l     }^{ 2 \nu_l      - 1},
\end{align}
with  $\uy = (y_1, ..., y_d)$, $\utheta = (\theta_1, ..., \theta_d)$.
Then \eqref{5eq: formal integral, C, 0} can be symbolically rewritten as
\begin{equation}
	\label{6eq: formal integral, C}
	J_{\unu, \uk} \lp x, e^{  i \phi} \rp = 2^d \int_{\BR_+^d} p_{2 \unu} (\uy) J_{\uk}(\uy; x, \phi ) d \uy, \hskip 10 pt x\in \BR_+, \, \phi \in \BR/2\pi \BZ.
\end{equation}

\begin{thm}\label{6thm: formal integral, C}
	Let  $(\umu, \um) \in \BL^d \times \BZ^{d+1}$ and $(\unu, \uk) \in \BC^d \times \BZ^d$  satisfy the relations given in Notation {\rm \ref{not: d, nu}}. Suppose  that $\unu$ lies in the union $$   \bigcup_{a \in \left[-\frac 1 2, 0 \right]} \left\{ \unu \in \BC^d : - \frac 1 2 < 2 \Re \nu_{l     } + a < 0 \text{ for all } l      = 1,..., d \right\}.$$

	{\rm(1).}  The integral in \eqref{6eq: formal integral, C} converges absolutely. Subsequently, we shall therefore use \eqref{6eq: formal integral, C} as the definition of $J_{\unu, \uk} \lp x, e^{  i \phi} \rp$. 
	
	{\rm(2).}  We have the {\rm(}genuine{\rm)} identity
	$$ J_{(\umu,\um)} \lp  x  e^{i\phi} \rp = e^{- i m_{d+1} \phi} 
	J_{\unu, \uk} \big( 2 \pi x^{\frac 1 {d+1}}, e^{i\phi} \big).$$
\end{thm}

\subsection{\texorpdfstring{The Integral  $j_{\unu, \um} (x)$}{The Integral  $j_{\nu, m} (x)$}}
Let us consider the integral $j_{\unu, \um} (x)$ defined by
\begin{equation}\label{6def: integral j nu m}
	\begin{split}
		j_{\unu, \um} (x) =	2^d \int_{\BR_+^{ d}} j_{(0, m_{d+1})} \left(  x  y_1 ... y_{d} \right) 
		\prod_{l      = 1}^{d} y_l      ^{2\nu_l      - 1} j_{(0, m_l     )} \left( x y_{l     }\- \right)   
		d y_{d} ... d y_1,
	\end{split}
\end{equation}
with $\unu \in \BC^d $ and $\um \in \BZ^{d+1}$.

\subsubsection{Absolute convergence of  $j_{\unu, \um} (x)$}
In contrast to the real case, where the integral $j_{\unu, \udelta} (x)$ never absolutely converges, $j_{\unu, \um} (x)$ is actually absolutely convergent, if each component of $\unu$ lies in   certain vertical strips of width at least $\frac 1 4$. 

\begin{defn}\label{6def: hyper-strip Sd}
	For $\boldsymbol{a}, \boldsymbol{b} \in \BR^d$ such that $a_{l     } < b_{l     }$ for all $l      = 1, ..., d$, we define the open hyper-strip $\BS^d (\boldsymbol a,  \boldsymbol{b}) = \left\{ \unu \in \BC^d : \Re \nu_{l     } \in (a_{l     }, b_{l     }) \right\}$. We write $ \BS^d (a ,  b ) = \BS^d (a \ue^d,  b \ue^d)$ for simplicity.
\end{defn}

\begin{prop}\label{6prop: convergence of j nu m}
	Let $(\unu, \um) \in \BC^d \times \BZ^{d+1}$. The integral $j_{\unu, \um} (x)$ defined above by \eqref{6def: integral j nu m} absolutely converges if $\unu \in \bigcup_{a \in \left[- \frac 1 2 , |m_{d+1}| \right]} \BS^d \left(\frac 1 2 \left(- \frac 1 2 - a\right) \ue^d, \frac 1 2 \lp \left\| \um^d \right\| - a \ue^d \rp \right)$, with $\um^d = (m_1, ..., m_d)$ and $\left\| \um^d \right\| = (|m_1|, ..., |m_d|)$.
\end{prop}

To show this, we first recollect some well-known facts concerning $J_m (x)$, as $j_{(0, m)} (x) = 2 \pi i^{m } J_{m} (4 \pi x)$ in view of  \eqref{3eq: j (0, m)}. 

Firstly, for $m \in \BN$, we have the Poisson-Lommel integral representation (see \cite[3.3 (1)]{Watson})
\begin{equation*}
	J_{m} (x) = \frac {\lp \frac 1 2x \rp^m } {\Gamma \left( m + \frac 1 2\right) \Gamma \left(\frac 1 2\right)} \int_0^{\pi} \cos (x \cos \theta) \sin^{2 m} \theta d \theta.
\end{equation*}
This yields the bound 
\begin{equation}\label{6eq: bound < x m}
	|J_{m} (x) | \leq \frac {\sqrt \pi \lp \frac 1 2x \rp^{|m|} } {\Gamma \left( |m| + \frac 1 2\right) },
\end{equation}
for $m \in \BZ$.
Secondly, the asymptotic expansion of $J_m (x)$ (see \cite[7.21 (1)]{Watson}) 
provides the estimate
\begin{equation}\label{6eq: bound < x -1/2}
	J_m (x) \lll_{\, m} x^{- \frac 1 2}.
\end{equation}
Combining these, we then arrive at the following lemma. 
\begin{lem}\label{6lem: crude bound for j m}
	Let $m$ be an integer.

	{\rm(1).} We have 
	the estimates 
	$$j_{(0, m)} (x) \lll_{\, m} x^{|m|}, \hskip 10 pt j_{(0, m)} (x) \lll_{\, m} x^{- \frac 1 2}.$$
	
	{\rm (2).} More generally, for any $a \in \left[- \frac 1 2, |m|\right]$, we have the estimate
	$$j_{(0, m)} (x) \lll_{\, m} x^{a}.$$
\end{lem}

\begin{proof}[Proof of Proposition \ref{6prop: convergence of j nu m}]
	
	We divide $\BR_+ = (0, \infty)$ into the union of two intervals, $ I_- \cup I_+ = (0, 1] \cup [1, \infty)$. Accordingly,  the integral in \eqref{6def: integral j nu m}  is partitioned into $2^d$ many integrals, each of which is supported on some hyper-cube $I_{\urho} = I_{\varrho_1 } \times ... \times I_{\varrho_d }$ for $\urho 
	\in \{+ , -\}^d$. For each such integral, we estimate  $j_{(0, m_l     )} \left(x y_{l     }\-\right)$ using the first or the second estimate in Lemma \ref{6lem: crude bound for j m} (1) according as $\varrho_{l     } = +$ or  $\varrho_{l     } = -$ and apply  the bound in Lemma \ref{6lem: crude bound for j m} (2) for $j_{(0, m_{d+1})} (x y_1 ... y_d) $. In this way, for any $a \in \left[- \frac 1 2, |m_{d+1} | \right]$, we have
	\begin{align*}
		&  2^d \int_{\BR_+^{ d}} 
		\left|j_{(0, m_{d+1})} \left(  x  y_1 ... y_{d} \right) \right| 
		\prod_{l      = 1}^{d} \left|y_l      ^{2\nu_l      - 1} j_{(0, m_l     )} \left( x y_{l     }\- \right) \right| 
		d y_{d} ... d y_1 
		\\
		\lll &\ \sum_{ \urho \in \{+, -\}^d } x^{ \sum_{ l       \in L_+ (\urho)} |m_l     | - \frac 1 2 \left|L_-(\urho)\right| + a} 
		I_{2 \unu +   a \ue^d, \um^d} ( \urho ),
	\end{align*}
	with the auxiliary definition
	\begin{equation*} 
		I_{\ulambda, \uk} ( \urho ) = \int_{ I_{\, \urho} } \Bigg( \prod_{ l       \in L_+ (\urho)}  y_l      ^{ \Re \lambda_l      - |k_{l     } | - 1} \Bigg) \Bigg(   \prod_{ l      \in L_- (\urho)}  y_l      ^{  \Re \lambda_l      - \frac 1 2} \Bigg) d y_{d} ... d y_1, \hskip 5 pt (\ulambda, \uk) \in \BC^d \times \BZ^d,
	\end{equation*}
	and $L_\pm (\urho) = \left\{ l      : \varrho_{l     } = \pm \right\}$. The implied constant depends only on $\um$ and $ d$. It is clear that all the integrals $I_{2 \unu +   a \ue^d, \um^d} ( \urho )$ absolutely converge if  $ - \frac 1 2   < 2 \, \Re \nu_{l     } + a < |m_l      | $ for all $l      = 1,..., d$. The proof is then completed. 
\end{proof}

\begin{rem}\label{6rem: d=1, j}
	When $d = 1$, one may apply the two estimates in Lemma  {\rm \ref{6lem: crude bound for j m} (1)} to $j_{(0, m_{2})} (x y )$ in the similar fashion as $j_{(0, m_1)} \left( x y \- \right)$. Then
	\begin{equation*}
	\begin{split}
	   2 \int_{0 }^{\infty} 
	\left|y ^{2\nu - 1} j_{(0, m_1)} \left( x y \- \right)   j_{(0, m_{2})}  \left(  x  y \right) \right|  & d y  \\
	 \lll_{\, m_1,\, m_2} \, x^{|m_1| -\frac 1 2}  
	\int_{1}^\infty  & y ^{2\Re \nu - |m_1| - \frac 3 2} d y  + x^{|m_2| -\frac 1 2} \int_{0}^1  y ^{2 \Re \nu + |m_2| - \frac 1 2} d y .
	\end{split}
	\end{equation*}
	Since both integrals above absolutely converge if $\, - |m_2| - \frac 1 2 < 2 \Re \nu <  |m_1| + \frac 1 2$, this also proves Proposition {\rm \ref{6prop: convergence of j nu m}} in the case $d=1$.
\end{rem}

\subsubsection{Equality between $j_{(\umu, \um)} (x) $ and $ j_{\unu, \um} \big (x^{\frac 1 {d+1}} \big)$}

\begin{prop}\label{6prop: j mu = j nu}
	Let $(\unu, \um) \in \BC^d \times \BZ^{d+1}$ be as in Proposition {\rm \ref{6prop: convergence of j nu m}} so that the integral $j_{\unu, \um} (x)$ absolutely converges. Suppose that $\umu \in \BL^{d}$ and $\unu \in \BC^d$ satisfy the relations given in Notation {\rm \ref{not: d, nu}}.  Then we have the identity
	$$j_{(\umu, \um)} (x) = j_{\unu, \um} \big (x^{\frac 1 {d+1}} \big).$$
\end{prop}

\begin{proof}
	Some change of variables turns the integral in Corollary \ref{5cor: H, C} into 
	\begin{equation*}
		2^{d+1} e^{i m \phi} \int_{\BR_+^{ d+1}} \varphi (y)  \,  
		j_{(0, m_{d+1})} \big(  (xy)^{\frac 1 {d+1}} y_1 ... y_{d} \big)
		\prod_{l      = 1}^{d} y_l      ^{2\nu_l        - 1} j_{(0, m_l     )} \big( (xy)^{\frac 1 {d+1}} y_{l     }\- \big)   
		y  d y d y_{d} ... d y_1. 
	\end{equation*}
	Corollary   \ref{3cor: H = h, C} and \ref{5cor: H, C}, along with the second formula in \eqref{6eq: Hankel = integral of j}, yield 
	\begin{equation*}
		\begin{split} 
			&2 \int_{\BR_+} \varphi (y) j_{(\umu, \um)} ( xy )   y d y = \\
			& 2^{d+1} \int_{\BR_+^{ d+1}} \varphi (y)  
			j_{(0, m_{d+1})} \big(  (xy)^{\frac 1 {d+1}}  y_1 ... y_{d} \big)
			\prod_{l      = 1}^{d} y_l      ^{2\nu_l        - 1} j_{(0, m_l     )} \big( (xy)^{\frac 1 {d+1}} y_{l     }\- \big)   
			y d y d y_{d} ... d y_1,
		\end{split}
	\end{equation*}
	for any $\varphi \in \SS (\BR_+)$, provided that $\Re \mu_1 > ... > \Re \mu_d > \Re \mu_{d+1} $ or equivalently $ \Re \nu_1  > ... > \Re \nu_d > 0$. In view of Proposition {\rm \ref{6prop: convergence of j nu m}}, the integral on the right hand side is absolute convergent at least when $\frac 1 4 > \Re \nu_1  > ... > \Re \nu_d > 0$. Therefore, the asserted equality holds on the domain $\left\{ \unu\in \BC^d : \frac 1 4 > \Re \nu_1  > ... > \Re \nu_d > 0 \right\}$ and remains valid on the whole domain of convergence for $j_{\unu, \um} (x) $ given in Proposition {\rm \ref{6prop: convergence of j nu m}} due to the principle of analytic continuation.
\end{proof}

\subsubsection{An Auxiliary Lemma}

\begin{lem}\label{6lem: bound for j nu m + m e}
	Let $(\unu, \um) \in \BC^d \times \BZ^{d+1}$ and $m \in \BZ$. Set $A = \max_{l     =1,..., d+1} \left\{  |m_{l     } | \right\}$.
	Suppose $\unu \in \bigcup_{a \in \left[-\frac 1 2, 0 \right]} \BS^d \left(- \frac 1 4 -\frac 1 2 a, - \frac 1 2 a \right)$. We have the estimate 
	\begin{align*}
		& \hskip 27 pt 2^d \int_{\BR_+^{ d}} 
		\left|j_{(0, m_{d+1} + m)} \left(  x  y_1 ... y_{d} \right) \right| 
		\prod_{l      = 1}^{d} \left|y_l      ^{2\nu_l      - 1} j_{(0, m_l      + m)} \left( x y_{l     }\- \right) \right|   
		d y_{d} ... d y_1  \\
		& \lll_{\, \um,\, d}   \sum_{  {\scriptstyle \urho \,\neq\, \urho^- } } \lp \frac {2 \pi e x} {|m| + 1} \rp^{\left| L _+ (\urho)\right|  |m|} (|m| + 1)^{2 \left|L_-(\urho)\right| + A  \left|L_+ (\urho)\right| } \\
		& \hskip  126 pt x^{ - \frac 1 2 \left|L_-(\urho)\right|} \max \left\{ x^{\left|L_+ (\urho)\right| A}, x^{- \left|L_+ (\urho)\right| A - \frac 1 2} \right\} \\
		& \hskip 27 pt + \lp \frac {2 \pi e x} {|m| + 1} \rp^{ |m|} (|m| + 1)^{ A } x^{ - \frac d 2} \max \left\{ x^{ A}, x^{-A} \right\},
	\end{align*}
	where $\urho \in \{+, -\}^d$, $\urho_- = (-,  ..., -)$ and $L_\pm (\urho) = \left\{ l      : \varrho_{l     } = \pm \right\}$.  
\end{lem}

Firstly, we require the bound \eqref{6eq: bound < x m}  for $J_m (x)$.
Secondly, we observe that when $x \geq (|m| + 1)^2$ the bound \eqref{6eq: bound < x -1/2} for $J_m (x)$ can be improved  so that the implied constant becomes absolute. This follows from the asymptotic expansion of $J_m (x)$ given in \cite[\S 7.13.1]{Olver}.
Moreover, we have Bessel's integral representation (see \cite[2.2 (1)]{Watson})
\begin{equation*}
	J_m (x) = \frac 1 {2 \pi} \int_0^{2 \pi} \cos \lp m \theta - x \sin \theta \rp d \theta,
\end{equation*}
which yields the bound
\begin{equation}\label{6eq: bound < 1}
	|J_m (x)| \leq 1.
\end{equation}
We then have the following lemma (compare \cite[Proposition 8]{Harcos-Michel}).

\begin{lem}\label{6lem: crude bound for j m, with m}
	Let $m$ be an integer.

	{\rm (1).} The following two estimates hold
	$$j_{(0, m)} (x) \lll \frac { \lp  2 \pi x \rp^{|m|} } {\Gamma \left( |m| + \frac 1 2\right) }, \hskip 10 pt j_{(0, m)} (x) \lll \frac {|m| + 1} {\sqrt x}, $$
	with absolute implied constants.
	
	{\rm (2).} For any $a \in \left[- \frac 1 2, 0 \right]$ we have the estimate
	$$j_{(0, m)} (x) \lll   \lp {(|m|+1)^{- 2 } } {x} \rp^{ a},$$
	with absolute implied constant.
\end{lem}

\begin{proof}[Proof of Lemma \ref{6lem: bound for j nu m + m e}] Our proof here is similar to that of Proposition  \ref{6prop: convergence of j nu m}, except that
	\begin{itemize}
		\item [-]   Lemma \ref{6lem: crude bound for j m, with m} (1) and (2) are applied in place of Lemma  \ref{6lem: crude bound for j m}  (1) and (2) respectively to bound $j_{(0, m_l      + m)} \left( x y_{l     }\- \right) $ and $j_{(0, m_{d+1} + m)} \left(  x  y_1 ... y_{d} \right) $, and
		\item [-]      the first estimate in Lemma \ref{6lem: crude bound for j m, with m} (1) is used for $j_{(0, m_{d+1} + m)} \left(  x  y_1 ... y_{d} \right)$  in the case $\urho = \urho_- $. 
	\end{itemize}
	In this way, one obtains the following estimate 
	\begin{align*}
		& \ \, 2^d \int_{\BR_+^{ d}} 
		\left|j_{(0, m_{d+1} + m)} \left(  x  y_1 ... y_{d} \right) \right|    
		\prod_{l      = 1}^{d} \left|y_l      ^{2\nu_l      - 1} j_{(0, m_l      + m)} \left( x y_{l     }\- \right) \right|
		d y_{d} ... d y_1 
		\\
		\lll \  & \sum_{ {\scriptstyle \urho \neq \urho_- } } \frac {\prod_{l       \in L_- (\urho)}  (|m_l      + m| + 1)^{ 1 - 2 a } } {\prod_{l       \in L_+ (\urho)} \Gamma \left( |m_l      + m| + \frac 1 2\right) (|m_l      + m| + 1)^{2a} } \\
		& \hskip 77 pt 
		(2 \pi x)^{  \sum_{ l       \in L_+ (\urho)} |m_l      + m | - \frac 1 2 \left|L_-(\urho)\right| + a } I_{2 \unu +   a \ue^d, \um^{d } + m \ue^{d }} (\urho) \\
		& + \frac {\prod_{l     =1}^d  (|m_{l     } + m| + 1) } { \Gamma \left( |m_{d+1} + m| + \frac 1 2\right) } 
		(2 \pi x)^{ - \frac {d} 2 +  |m_{d+1} + m | } I_{2 \unu +   |m_{d+1} + m| \ue^d } (\urho_-),
	\end{align*}
	\footnote{When $\urho = \urho_-$, $\uk$ does not occur in the definition of $I_{\ulambda, \uk} ( \urho_- )$ and   is therefore suppressed from the subscript.}with $a \in \left[- \frac 1 2, 0\right] $. 
	Suppose that $- \frac 1 2 - a < 2 \, \Re \nu_{l     } < - a $ for all $l      = 1, ..., d$, then   the integrals $I_{2 \unu +  a \ue^d, \um^{d } + m \ue^{d }} (\urho)$ and $I_{2 \unu +   |m_{d+1} + m| \ue^d } (\urho_-)$  are absolutely convergent and   of size $O_d \lp \prod_{l      \in L_+ (\urho)}( |m_{l     } + m| +1 )\- \rp$ and $O_d \lp  (|m_{d+1} + m| +1 )^{- d} \rp$ respectively. A final estimation using  Stirling's asymptotic formula yields our asserted bound.
\end{proof}

\begin{rem}\label{6rem: d=1, j m + m e}
	In the case $d=1$, modifying over the ideas in Remark {\rm \ref{6rem: d=1, j}}, one may show the  slightly improved estimate
	\begin{align*}
		2 \int_{0 }^{\infty} 
		\left|y ^{2\nu - 1} j_{(0, m_1 + m)} \left( x y \- \right) j_{(0, m_{2} + m)} \left(  x  y \right) \right| &    d y   \\
		\lll_{\, m_1,\, m_2} \lp \frac {2 \pi e x} {|m| + 1} \rp^{ |m|} & (|m| + 1)^{ A }  x^{ - \frac 1 2} \max \left\{ x^{A}, x^{- A} \right\},
	\end{align*}
	given that $|\Re \nu | < \frac 1 4$, with $A = \max \left\{ |m_1|, |m_2| \right\}$.
\end{rem}

\subsection{\texorpdfstring{The Series of Integrals $J_{\unu, \um} (x, u)$}{The Series of Integrals $J_{\nu, m} (x, u)$}}
We define the following series of integrals,
\begin{equation}\label{6def: J nu m (x, u)}
	\begin{split}
		& J_{\unu, \um} (x, u) = \frac 1 {2 \pi} \sum_{m \in \BZ}  u^m  j_{ \unu, \um + m\ue^n } (x)  \\
		= \ & \frac {2^{d-1}} { \pi} \sum_{m \in \BZ}  u^m \int_{\BR_+^{ d}}  
		j_{(0, m_{d+1} + m)} \left(  x  y_1 ... y_{d} \right)    
		\prod_{l      = 1}^{d} y_l      ^{2\nu_l      - 1} j_{(0, m_l      + m)} \left( x y_{l     }\- \right)
		d y_{d} ... d y_1,
	\end{split}
\end{equation}
with $x \in \BR_+$ and $u \in \BC$, $|u| = 1$.

\subsubsection{Absolute Convergence of $J_{\unu, \um} (x, u)$}

We have the following  direct consequence of Lemma \ref{6lem: bound for j nu m + m e}.

\begin{prop}\label{6prop: J nu m (x, u)}
	Let $(\unu, \um) \in \BC^d \times \BZ^{d+1}$. The series of integrals $J_{\unu, \um} (x, u)$ defined by \eqref{6def: J nu m (x, u)} is absolutely convergent if  $\unu \in \bigcup_{a \in \left[-\frac 1 2, 0 \right]} \BS^d \left(- \frac 1 4 -\frac 1 2 a, - \frac 1 2 a \right)$. 
\end{prop}

\subsubsection{Equality between $J_{(\umu, \um)} \lp x e^{i \phi} \rp $ and $J_{\unu, \um} \big(x^{\frac 1 {d+1}}, e^{i \phi} \big)$}

In view of Proposition \ref{6prop: j mu = j nu} along with \eqref{2eq: Bessel kernel over C, polar} and \eqref{6def: J nu m (x, u)}, the following proposition is readily established.

\begin{prop}\label{6prop: J mu m = J nu m}
	Let $(\unu, \um) \in \BC^d \times \BZ^{d+1}$. Suppose that $ \unu $ satisfies the condition in Proposition {\rm \ref{6prop: J nu m (x, u)}} so that $J_{\unu, \um} (x, u)$ is absolutely convergent. Then, given that $\umu $ and $\unu$ satisfy the relations   in Notation {\rm \ref{not: d, nu}}, we have the identity
	$$J_{(\umu, \um)} \lp x e^{i \phi} \rp = J_{\unu, \um} \big(x^{\frac 1 {d+1}}, e^{i \phi} \big), $$
	with $x \in \BR_+$ and $\phi \in \BR/2\pi \BZ$.
\end{prop}

\subsection{Proof of Theorem \ref{6thm: formal integral, C}}

\begin{lem}\label{6lem: series expansion of Jk}
	Let $ \uk \in \BZ^d $ and recall the integral $J_{\uk}(\uy; x, \phi )$ defined by {\rm (\ref{6eq: Theta (theta, y; x, phi)}, \ref{6eq: Jk (y;x, phi)})}. We have  the following absolutely convergent  series expansion of $J_{\uk}(\uy; x, \phi )$
	\begin{equation} \label{6eq: series expansion of J k}
		J_{\uk}(\uy; 2 \pi x, \phi ) =  \frac {1} { 2 \pi} \sum_{m \in \BZ}  e^{i m\phi} 
		j_{(0,  m)} \left(  x  y_1 ... y_{d} \right) \prod_{l      = 1}^{d}  j_{(0, k_l      + m)} \left( x y_{l     }\- \right).
	\end{equation}
\end{lem}

\begin{proof}
	We find in \S \ref{sec: n=1, C} the integral representation 
	$$j_{(0, m)} (x) = \int_{\BR/2\pi\BZ} e^{i m \theta + 4\pi i x \cos \theta } d \theta $$
	as well as the Fourier series expansion
	$$e^{4 \pi i x \cos \phi} = \frac 1 {2 \pi} \sum_{m \in \BZ} j_{(0, m)} (x) e^{i m \phi}.$$
	Therefore
	\begin{align*}
		& \frac {1} { 2 \pi} \sum_{m \in \BZ}  e^{i m\phi} 
		j_{(0,  m)} \left(  x  y_1 ... y_{d} \right) \prod_{l      = 1}^{d}  j_{(0, k_l      + m)} \left( x y_{l     }\- \right)\\
		= \, & \frac {1} { 2 \pi} \sum_{m \in \BZ}  e^{i m\phi} j_{(0,  m)} \left(  x  y_1 ... y_{d} \right)
		\int_{(\BR/2\pi\BZ)^d} e^{i m \sum_{l     =1}^d \theta_{l     } } e^{i \sum_{l      = 1}^d \lp i k_{l     } \theta + 4 \pi i x y_{l     }\- \cos \theta_{l     } \rp  }  d \theta_d ... d \theta_1 \\
		= \ & \int_{(\BR/2\pi\BZ)^d} \lp \frac {1} { 2 \pi} \sum_{m \in \BZ}  e^{i m \lp \sum_{l     =1}^d \theta_{l     } + \phi \rp} j_{(0,  m)} \left(  x  y_1 ... y_{d} \right) \rp e^{i \sum_{l      = 1}^d \lp i k_{l     } \theta + 4 \pi i x y_{l     }\- \cos \theta_{l     } \rp  }  d \theta_d ... d \theta_1\\
		= \ & \int_{(\BR/2\pi\BZ)^d} e^{4 \pi i x  y_1 ... y_{d} \cos \lp \sum_{l     =1}^d \theta_{l     } + \phi \rp } e^{i \sum_{l      = 1}^d \lp i k_{l     } \theta + 4 \pi i x y_{l     }\- \cos \theta_{l     } \rp  }  d \theta_d ... d \theta_1.
	\end{align*}
	The absolute convergence required for the validity of each equality above is justified by the first estimate for $j_{(0, m)} (x)$ in Lemma \ref{6lem: crude bound for j m, with m} (1). The proof is completed, since the last line is exactly the definition of  $J_{\uk}(\uy; 2 \pi x, \phi )$.
\end{proof}

Inserting the series expansion of $J_{\uk}(\uy; 2 \pi x, \phi ) $ in Lemma \ref{6lem: series expansion of Jk} into the integral in \eqref{6eq: formal integral, C} and interchanging the order of integration and summation, one arrives exactly at the series of integrals $J_{\unu, (\uk, 0)} \lp x, e^{i \phi} \rp = e^{ - i m_{d+1} \phi}  J_{\unu, \um} \lp x, e^{i \phi} \rp$.  The first assertion on   absolute convergence in Theorem \ref{6thm: formal integral, C}   follows immediately from Proposition \ref{6prop: J nu m (x, u)}, whereas the  identity in the second assertion is a direct consequence of Proposition \ref{6prop: J mu m = J nu m}.

\subsection{The Rank-Two Case ($d=1$)}
\ 
\vskip 5 pt

\subsubsection{The Case of Bessel Functions}\label{sec: case d=1, Bessel functions}
When  $x > 0$ and $| \Re  \nu| < 1$, we have the following integral representations of Bessel functions due to Mehler and Sonine (\cite[6.21 (10, 11), 6.22 (13)]{Watson})
\begin{align*}
H^{(1, 2)}_\nu (x) & = \pm \frac {2 e^{\mp \frac 1 2 \pi i \nu }} {\pi i} \int_0^\infty e^{\pm i x \cosh r} \cosh (\nu r) d r,\\
K_{\nu} (x) & = \frac 1 {\cos \left( \frac {1 } 2 \pi \nu \right)} \int_0^\infty \cos (x \sinh r) \cosh (\nu r) d r.
\end{align*}
The change of variables $t = e^r$ yields
\begin{align*}
\pm \pi i e^{\pm \frac 1 2 \pi i \nu } H^{(1, 2)}_\nu (2 x)  & = \int_0^\infty t^{\nu - 1} e^{\pm i x (t + t\-)} d t,\\
2 e^{\pm \frac 1 2 \pi i \nu } K_{\nu} (2 x) & =  \int_0^\infty t^{\nu - 1} e^{\pm i x (t - t\-)} d t.
\end{align*}
The integrals in these formulae  are exactly the formal integrals in \eqref{6eq: formal integral, R+} in the case $d=1$. They   \textit{conditionally} converge  if  $| \Re  \nu| < 1$, but diverge if otherwise. 

\vskip 5 pt
\subsubsection{The Real Case}
By the formula \eqref{3eq: Bessel kernel, R, connection, formal}, the formal integral $J_{\nu, \epsilon} ( 2 \pi \sqrt x, \pm )$, which represents the Bessel kernel $J_{\left(\frac 1 2 \nu, -\frac 1 2 \nu \right), (\epsilon, 0)} (\pm x)$, is reduced to the integrals in \S \ref{sec: case d=1, Bessel functions}.

\vskip 5 pt

\subsubsection{The Complex Case}

\begin{lem}\label{6lem: d = 1, Jk}
	Let  $k \in \BZ$. Recall from {\rm (\ref{6eq: Theta (theta, y; x, phi)}, \ref{6eq: Jk (y;x, phi)})} the definition $$J_k (y; x, \phi) = \int_0^{2 \pi}  e^{i k \theta + 2 i x y\- \cos \theta + 2 i x y \cos (\theta + \phi)} d \theta, \hskip 10 pt x, y \in (0, \infty), \, \phi \in [0, 2 \pi) .$$
	Define $Y(y, \phi) = \left|y\-  +  y e^{i \phi}\right| = \sqrt {y^{-2} + 2 \cos \phi + y^2 } $, $\Phi (y, \phi) = \arg (y\-  +  y e^{i \phi})$ and $E(y, \phi) = e^{i \Phi(y, \phi)}$. 
	Then  
	\begin{equation}\label{6eq: d=1, Jk}
		J_k (y; x, \phi) = 2 \pi i^k E(y, \phi)^{-k} J_k \lp 2 x Y(y, \phi) \rp .
	\end{equation}
\end{lem}

\begin{proof}
	\eqref{6eq: d=1, Jk} follows immediately from the identity
	$$2 \pi i^k J_k ( x) = \int_{0}^{2 \pi} e^{i k \theta + i x \cos \theta } d \theta, $$
	along with the observation $$y\- \cos \theta + y \cos (\theta + \phi) = \Re \big(  y\- e^{i\theta} + y e^{i (\theta + \phi)} \big) = Y(y, \phi) \cos \lp \theta + \Phi (y, \phi) \rp.$$
\end{proof}

\begin{prop}\label{6prop: d=1, J v k}
	Let $\nu \in \BC$ and $k \in \BZ$. Recall the definition of $J_{\nu, k} \lp x, e^{  i \phi} \rp$ given by \eqref{6eq: formal integral, C}. Then 
	\begin{equation}\label{6eq: d = 1 J nu k (x, e i phi)}
		J_{\nu, k} \lp x, e^{  i \phi} \rp = 4 \pi i^k \int_0^\infty y^{2 \nu - 1} \left[ y\-  +  y e^{i \phi} \right]^{-k} J_k \lp 2 x \left|y\-  +  y e^{i \phi}\right| \rp d y,
	\end{equation}
	with $x \in (0, \infty)$ and $\phi \in [0, 2 \pi)$. Here, we recall the notation $[z] =   {z} / {|z|}$. The integral in \eqref{6eq: d = 1 J nu k (x, e i phi)} converges when $|\Re \nu| < \frac 3 4$ and the  convergence is absolute if and only if $|\Re \nu| < \frac 1 4$. Moreover, it is analytic with respect to $\nu $ on the open vertical strip $\BS\lp - \frac 3 4, \frac 3 4 \rp$.
\end{prop}

\begin{proof}
	\eqref{6eq: d = 1 J nu k (x, e i phi)} follows immediately from Lemma \ref{6lem: d = 1, Jk}. 
	
	As for the convergence, since one arrives at an integral of the same form with $\nu, \phi$ replaced by $- \nu, - \phi$ if    the variable is changed from $y$ to $y\-$, it suffices to consider the integral 
	$$\int_2^\infty y^{2 \nu - 1} e^{-i k \Phi (y, \phi)} J_k \lp 2 x Y(y, \phi) \rp d y,$$ 
	for $\Re \nu < \frac 3 4$. 
	We have the following  asymptotic of $J_k (x)$ (see \cite[7.21 (1)]{Watson})
	$$ J_k (x) = \lp \frac  2 {\pi x } \rp^{\frac 1 2} \cos \lp x - \tfrac 12 k \pi - \tfrac 1 4 \pi \rp + O_k \big( {x^{- \frac 3 2} } \big).$$
	The error term contributes an absolutely convergent integral when $\Re \nu < \frac 3 4$, whereas the integral coming from the main term absolutely converges if and only if  $\Re \nu < \frac 1 4$. We are now reduced to the integral
	\begin{align*}
		\int_2^\infty  y^{2 \nu - 1} e^{-i k \Phi (y, \phi)} \lp x Y(y, \phi) \rp^{- \frac 1 2} e^{\pm  2 i x Y(y, \phi) } d y.
	\end{align*}
	In order to see the convergence, we split out $e^{\pm 2 i x y}$ from $e^{\pm  2 i x Y(y, \phi) }$ and put $f_{\nu,\, k} (y; x, \phi) = y^{2 \nu - 1} e^{-i k \Phi (y, \phi)} \lp xY(y, \phi) \rp^{- \frac 1 2} e^{\pm  2 i x (Y(y, \phi) - y) }$. Partial integration turns the above  integral into
	$$\mp \frac{1} {2 i x^{\frac 3 2}} \lp 2^{2 \nu - 1} e^{-i k \Phi (2, \phi)} Y(2, \phi)^{- \frac 1 2} e^{\pm  2 i x  Y(2, \phi)  } + \int_2^\infty \lp \partial f_{\nu,\, k}/ \partial y \rp (y; x, \phi) e^{\pm 2 i x y}	d y \rp.$$ 
	Some calculations show that $\lp \partial f_{\nu,\, k}/ \partial y \rp (y; x, \phi) \lll_{\,\nu,\, k,\, x} y^{2\Re \nu - \frac 5 2}$ for $y \geq 2$, and hence the integral in the second term is absolutely convergent when $\Re \nu < \frac 3 4$. With the above arguments, the analyticity with respect to $\nu$ is  obvious.
\end{proof}
\begin{cor}\label{6cor: d=1, J mu m}
	Let $\mu \in \BS \lp - \frac 3 8, \frac 3 8 \rp$ and $m \in \BZ$. We have
	\begin{equation}\label{6eq: d = 1 J mu m (x, e i phi)}
		J_{\lp  \mu, - \mu, m, 0\rp } \lp x e^{  i \phi} \rp = 4 \pi i^m \int_0^\infty y^{4 \mu - 1} \left[ y\-  +  y e^{i \phi} \right]^{-m} J_m \lp 4 \pi \sqrt x \left|y\-  +  y e^{i \phi}\right| \rp d y,
	\end{equation}
	with $x \in (0, \infty)$ and $\phi \in [0, 2 \pi)$. The integral in \eqref{6eq: d = 1 J mu m (x, e i phi)} converges if $|\Re \mu| < \frac 3 8$ and absolutely converges if and only if $|\Re \mu| < \frac 1 8$. 
\end{cor}
\begin{proof}
	From Theorem \ref{6thm: formal integral, C}, we see that \eqref{6eq: d = 1 J mu m (x, e i phi)} holds for $\BS \lp - \frac 1 8, \frac 1 8 \rp$. In view of Proposition \ref{6prop: d=1, J v k}, the right hand side of \eqref{6eq: d = 1 J mu m (x, e i phi)} is analytic in $\mu$ on $\BS \lp - \frac 3 8, \frac 3 8 \rp$, and therefore it is allowed to extend the domain of equality from $\BS \lp - \frac 1 8, \frac 1 8 \rp$ onto  $\BS \lp - \frac 3 8, \frac 3 8 \rp$.
\end{proof}

\begin{subappendices}

\renewcommand{\thesection}{A}

\renewcommand{\thesection}{A}

\section{A Prototypical Example of Bessel Functions}\label{sec: special example}

Before delving into the general  analytic theory of Bessel functions in Chapter \ref{chap: analytic theory}, we shall give here a very typical example of $J (x; \usigma, \ulambda)$ which illustrates their  asymptotic nature with dependence on $\usigma$.

\vskip 5 pt

When $n = 1$, we simply have two oscillatory exponential functions $$J (x; \pm, 0) = e^{\pm i x}.$$

We then consider the case when $n = 2$. According to \cite[3.4 (3, 6), 3.71 (13)]{Watson},
\begin{equation*}
	J_{\frac 1 2} (x) = \lp \frac 2 {\pi x} \rp^{\frac 1 2} \sin x, \hskip  10 pt
	J_{- \frac 1 2} (x) = \lp \frac 2 {\pi x} \rp^{\frac 1 2} \cos x.
\end{equation*}
The connection formulae in \eqref{2eq: connection formulae} (\cite[3.61 (5, 6)]{Watson})
 imply that
\begin{equation*} 
	H^{(1)}_{\frac 1 2} (x) = - i \lp \frac 2 {\pi x} \rp^{\frac 1 2} e^{i x}, \hskip 10 pt H^{(2)}_{\frac 1 2} (x) = i  \lp \frac 2 {\pi x} \rp^{\frac 1 2} e^{- i x}.
\end{equation*}
Moreover, \cite[3.71 (13)]{Watson} reads
\begin{equation*}
	K_{\frac 1 2} (x) = \lp \frac \pi {2 x} \rp^{\frac 1 2} e^{- x}.
\end{equation*}
Therefore, the formulae in Proposition \ref{prop: Classical Bessel functions} yield
\begin{equation*}
	J \lp x; \pm, \pm, \tfrac 1 4 , - \tfrac 1 4\rp = \lp \frac {\pi  } x \rp^{\frac 1 2} e^{\pm 2 i  x \pm \frac 1 4 \pi i}, \hskip 10 pt J \lp x; \pm, \mp, \tfrac 1 4 , - \tfrac 1 4 \rp = \lp \frac {\pi } x \rp^{\frac 1 2} e^{- 2 x \mp \frac 1 4 \pi i}.
\end{equation*}

We now prove that these formulae admit generalizations to arbitrary rank.
\begin{prop}
	\label{prop: special example}
	For $\usigma \in \{ +, -\}^n$ we define $L_\pm (\usigma) = \{ l      : \varsigma_l      = \pm \}$ and $n_\pm (\usigma) = \left| L_\pm (\usigma) \right|$.
	Put $\xi (\usigma) = i e^{ {\pi} i \frac {n_- (\usigma) - n_+ (\usigma) } {2n} } =  \mp e^{\mp \pi i \frac { n_{\pm} (\usigma)} n}$. Suppose $\ulambda = \frac 1 n \lp  \frac {n-1} 2, ..., - \frac {n-1} 2 \rp$. 
	Then  
	\begin{equation}\label{2eq: special example}
	J (x; \usigma, \ulambda) = \frac { c(\usigma)} {\sqrt n} \lp \frac {2\pi} x \rp^{\frac {n-1} 2} e^{ i n \xi (\usigma) x},
	\end{equation}
	with $c(\usigma) = e \lp \mp \frac {n-1} 8 \mp \frac {n_{\pm} (\usigma) } {2 n} \pm \frac 1 {2 n}  \sum_{l      \in L_{\pm} (\usigma)}   {l     } \rp$. 
\end{prop}
\begin{proof}
	Using the multiplication formula of the Gamma function
	\begin{equation}\label{2eq: multiplication theorem}
	\prod_{k = 0}^{n-1} \Gamma \lp s + \frac {k} n \rp = 
	(2 \pi)^{\frac {n-1} 2} n^{\frac 1 2 - n s} \Gamma (n s),
	\end{equation}
	straightforward calculations yield
	\begin{equation*}
		G (s; \usigma, \ulambda) = c_1 (\usigma) (2 \pi)^{\frac {n-1} 2} n^{\frac 1 2 - n \lp s - \frac {n-1} {2n} \rp} \Gamma  \lp n \lp s - \frac {n-1} {2n} \rp\rp e \lp \frac {n_+ (\usigma) - n_- (\usigma) } {4  } \cdot s \rp,
	\end{equation*}
	with $c_1 (\usigma) = e \lp \mp \frac {(n+1) n_{\pm } (\usigma) } {4 n} \pm \frac 1 {2 n}  \sum_{l      \in L_{\pm} (\usigma)}   {l     } \rp$. Inserting this into the contour integral in \eqref{3eq: definition of J (x; sigma)} and making the change of variables from $s$ to $\frac 1 n \lp s + \frac {n-1} 2 \rp$, one arrives at
	\begin{equation*}
		J (x; \usigma, \ulambda) = \frac {c_1 (\usigma) c_2 (\usigma)} {\sqrt n} \lp \frac {2 \pi} x \rp^{\frac {n-1} 2 } \frac 1 {2\pi i} \int_{ n \EC - \frac {n-1} 2 }  \Gamma (s) e \lp \frac {n_+ (\usigma) - n_- (\usigma) } {4 n} \cdot s \rp  
		(n x)^{- s} ds,
	\end{equation*}
	with $c_2 (\usigma) = e \lp \mp \frac {n-1} 8 \pm \frac {(n-1) n_{\pm} (\usigma) } {4 n}  \rp$. \eqref{2eq: special example} now follows from \eqref{2eq: Mellin n=1} if the contour $\EC$ is suitably chosen.
\end{proof}

\end{subappendices}

%
%
%

\chapter{Analytic Theory of Bessel Functions}\label{chap: analytic theory}

\setcounter{section}{6}

This chapter is on the analytic theory of Bessel functions. There are two parts of this theory, formal integrals and Bessel differential equations. Combining these, this chapter culminates at the connection formulae for Bessel functions.

In \S \ref{sec: Rigorous Interpretations}, \ref{sec: Relating Bessel functions} and \ref{sec: Bessel functions of K-type and H-Bessel functions} we shall first rigorously justify the formal integral $J_{\unu} (x; \usigma)$ as a representation of the Bessel function $J (x; \ulambda, \usigma ) $ and then use the stationary phase method of  H\"ormander for the oscillatory multiple integral $J_{\unu} (x; \usigma)$ to study the asymptotic of   $J (x; \ulambda, \usigma ) $ on $\BR_+$ or its analytic continuation onto the upper or lower half plane $\BH^{\pm}$.

In \S \ref{sec: Recurrence relations and differential equations of the Bessel functions}  we shall derive the   differential equation, namely  {Bessel equation}, satisfied by the Bessel function $J (x; \usigma, \ulambda)$.

In \S \ref{sec: Bessel equations} we shall study two types of Bessel functions,  $J_{l     } (z; \varsigma, \ulambda) $ and $ J (z; \ulambda; \xi) $, arising as solutions of Bessel equations in accordance with their asymptotics at zero and infinity respectively.

In \S \ref{sec: H-Bessel functions and K-Bessel functions revisited}, using the asymptotic theory we have developed in previous sections, we shall present various connection formulae that relate  Bessel functions $J (z; \usigma, \ulambda)$, $J_{l     } (z; \varsigma, \ulambda) $ and $ J (z; \ulambda; \xi) $.

In \S \ref{sec: H-Bessel functions and K-Bessel functions, II}, using results from \S \ref{sec: Bessel equations}, we shall make improvements on the asymptotic of $J (x; \ulambda, \usigma ) $ obtained in \S \ref{sec: Bessel functions of K-type and H-Bessel functions}.

In Appendix \ref{appendix: asymptotic}, for the purpose of comparison, we shall include the approach to the asymptotics of Bessel functions from the literature using   Stirling's asymptotic formula for the Gamma function.

\section{\texorpdfstring{The Rigorous Interpretation of the Formal Integral $J_{\unu}(x; \usigma)$}{The Rigorous Interpretation of the Formal Integral $J_{\nu}(x; \varsigma)$}}
\label{sec: Rigorous Interpretations}


We first introduce some notations. Let $d $ be a positive integer, $\ut = (t_1, ... , t_d) \in \BR_+^d$, $\unu = (\nu_1, ..., \nu_d) \in \BC^{d}$ and  $\usigma = (\varsigma_1, ..., \varsigma_d, \varsigma_{d+1}) \in \{+, - \}^{d+1}$.  
For $a > 0$ define $\BS^d_a = \big\{ \unu \in \BC^d : |\Re  \nu_l     | < a \text { for all } l      = 1, ..., d \big\}$.  
Denote by $p_{\unu} $ the power function
\begin{equation*} 
p_{\unu} (\ut) = \prod_{l      = 1}^d t_{l     }^{\nu_l      - 1},
\end{equation*}
and let
\begin{equation*} 
\theta ( \ut; \usigma) = \varsigma_{d+1} t_1 ... t_d + \sum_{l      = 1}^d \varsigma_l      t_{l     }\-,
\end{equation*}
then   the formal integral $J_{\unu} ( x ; \usigma)$ given in \eqref{6eq: formal integral, R+}, which symbolically represents  $J ( x ; \usigma, \ulambda)$, may be written as 
\begin{equation}\label{3eq: definition J nu (x; sigma)}
J_{\unu} ( x ; \usigma) = \int_{\BR_+^d} p_{\unu} (\ut)e^{i x \theta ( \ut; \usigma)} d \ut.
\end{equation}


For $d = 1$, it is seen in \S \ref{sec: case d=1, Bessel functions} that $J_{\nu}(x; \usigma)$ is conditionally convergent if and only if $|\Re \nu | < 1$ but fails to be absolutely convergent. 
When $d \geq 2$, we are in a worse scenario. The notion of convergence for multiple integrals is always in the absolute sense. Thus, the $d$-dimensional multiple integral in \eqref{3eq: definition J nu (x; sigma)} alone does not make any sense, 
since it is clearly not absolutely convergent. 


In the following, we shall address this fundamental convergence issue of the formal integral $J_{\unu} ( x ; \usigma)$, relying on its structural simplicity, so that it will be provided with mathematically rigorous meanings\footnote{It turns out that our rigorous interpretation actually coincides with the {\it Hadamard partie finie} of the formal integral.}. Moreover, it will be shown that our rigorous interpretation of $J_{\nu}(x; \usigma)$ is a smooth function of $x$ on $\BR_+$ as well as an analytic function of $\unu$ on $\BC^d$. 


\subsection{Formal Partial Integration Operators}

The most crucial observation is that there are \textit{two} kinds of formal partial integrations. 
The first kind arises from 
$$\partial \big(e^{\varsigma_l      i x t_{l     }\- } \big) = - \varsigma_l      i  x t_{l     }^{-2} e^{\varsigma_l      i x t_{l     }\-} \partial t_{l     },$$ 
and the second kind from
$$\partial \left(e^{\varsigma_{d+1} i x t_1 ... t_d} \right) = \varsigma_{d+1} i x t_1 ... \widehat {t_{l     }} ... t_d e^{\varsigma_{d+1} i x t_1 ... t_d} \partial t_{l     },$$ where $\widehat {t_{l     }}$ means that  $t_{l     }$ is omitted from the product.
\begin{defn}\label{2defn: formal partial integration}
	Let $$\mathscr T (\BR_+) = \left\{ h \in C^\infty (\BR_+) : t^\alpha h^{(\alpha)}(t) \lll_{\, \alpha} 1 \text { for all } \alpha \in \BN \right \}.$$ For $h(\ut) \in \bigotimes^d \mathscr T (\BR_+)$, in the sense that $h(\ut) $ is a linear combination of functions of the form $ \prod_{l      = 1}^d h_l     (t_{l     })$, define the   integral
	\begin{equation*} 
	J_{\unu} (x; \usigma; h) = \int_{\BR_+^d} h (\ut) p_{\unu} (\ut) e^{ i x \theta (\ut; \usigma)} d \ut. 
	\end{equation*} 
	We call $J_{\unu} (x; \usigma; h)$ a $J$-integral of index $\unu$.
	Let us introduce an auxiliary space
	$$\mathscr J_{\unu} (\usigma) = \mathrm {Span }_{\BC[x\-]} \left\{ J_{\boldsymbol \nu'} (x; \usigma; h) : \boldsymbol \nu' \in \unu + \BZ^d, h \in \textstyle \bigotimes^d \mathscr T (\BR_+)\right \}. $$
	Here $\BC[x\-]$ is the ring of polynomials of variable $x\-$ and complex coefficients.
	Finally, we define $\EuScript P_{+,\, l     }$ and $\EuScript P_{-,\, l     }$  to be the two $\BC[x\-]$-linear operators on the space $\mathscr J_{\unu} (\usigma)$, in symbolic notion, as follows,
	\begin{align*} 
	\EuScript P_{+,\, l     } (J_{\unu} & (x ; \usigma; h)) = \varsigma_l      \varsigma_{d+1} J_{\unu + \ue^d + \ue_l      } \left(x; \usigma; h \right) \\
	& - \varsigma_l      i (\nu_l      + 1) x \- J_{\unu + \ue_l     } \left(x; \usigma; h \right) - \varsigma_{l     } i x\- J_{\unu + \ue_l     } \left(x; \usigma; t_{l     } \partial_{l     } h \right), 
	\end{align*}
	\begin{align*} 
	\EuScript P_{-,\, l     } & (J_{\unu} (x ; \usigma; h)) = \varsigma_l      \varsigma_{d+1} J_{\unu - \ue^d - \ue_l      } \left(x; \usigma; h \right)\\
	& + \varsigma_{d+1} i (\nu_l      - 1) x \- J_{\unu - \ue^d} \left(x; \usigma; h \right) + \varsigma_{d+1} i x\- J_{\unu -\ue^d } \left(x; \usigma; t_{l     } \partial_{l     } h \right), 
	\end{align*}
	where $\ue_l      = (\underbrace {0,..., 0, 1}_{l     }, 0 ..., 0)$ and $\ue^d = (1, ..., 1)$, and $\partial_{l     } h$ is the abbreviated $\partial h / \partial t_{l     }$.
\end{defn}

The formulations of $\EuScript P_{+,\, l     }$ and $\EuScript P_{-,\, l     }$ are quite involved at a first glance. However, the most essential feature of these operators is simply \textit{index shifts}! 

\begin{figure}
	\begin{center}
		\begin{tikzpicture}
		
		\node [left] at (0, 0) { \small $\unu$};

		\draw[->] (0,0) to [out=20,in=180] (1, 0.5);
		\draw[->] (0,0) to [out=- 20,in=180] (1, - 0.3);
		\draw[->] (0,0) to [out= -5,in=120] (1, - 0.2);
		
		\node [right] at (1, 0.5) { \small $\unu + \ue^d + \ue_l     $};
		\node [right] at (1, -0.3) { \small $\unu + \ue_l     $};
		\node [left] at (-0.3, 0) { \small $\EuScript P_{+,\, l     } :$};

		\node [left] at (5, 0) { \small $\unu$};

		\draw[->] (5,0) to [out=-20,in=180] (6, -0.5);
		\draw[->] (5,0) to [out= 5,in=-120] (6, 0.2);
		\draw[->] (5,0) to [out= 20,in=180] (6,  0.3);

		\node [right] at (6, 0.3) { \small $\unu - \ue^d$};
		\node [right] at (6, -0.5) { \small $\unu - \ue^d - \ue_l     $};
		\node [left] at (5-0.3, 0) { \small $\EuScript P_{-,\, l     } :$};
		
		\end{tikzpicture}
	\end{center}
	\caption{Index Shifts} 
\end{figure}

\begin{observation} \nonumber
	After the operation of $\EuScript P_{+,\, l     }$ on a $J$-integral, all the indices of the three resulting $J$-integrals are nondecreasing and the increment of the $l     $-th index is one greater than the others. The operator $\EuScript P_{-,\, l     }$ has the effect of decreasing all indices by one except possibly two for  the $l     $-th index.
\end{observation}

\begin{lem}\label{lem: I pm ell}
	Let notations be as above.
	
	{\rm(1).} Let  $h(\ut) = \prod_{l      = 1}^d h_l     (t_{l     })$. Suppose that the set $\{1, 2, ..., d\}$ splits into two subsets $L_+$ and $L_-$ such that 
	\begin{itemize}
		\item [-]	$h_l      $ vanishes at infinity if $l      \in L_-$, and 
		\item[-] $h_l      $ vanishes in a neighbourhood of zero if $l      \in L_+$.
	\end{itemize} If $\Re \nu_l     >  0$ for all  $l      \in L_-$ and 
	$\Re \nu_l      < 0$ for all  $l      \in L_+$, then the $J$-integral $J_{\unu} (x ; \usigma; h)$ absolutely converges.
	
	{\rm(2).} Assume the same conditions in {\rm(1)}.  Moreover, suppose that $\Re \nu_l     > 1$ for all  $l      \in L_-$ and  $\Re \nu_l      < -1$ for all  $l      \in L_+$. Then, for $l      \in L_-$, all the three $J$-integrals in the definition of $\EuScript P_{+,\, l     } (J_{\unu} (x ; \usigma; h) )$ are absolutely convergent and the operation of $\EuScript P_{+,\, l     }$ on $J_{\unu} (x ; \usigma; h)$ is the actual partial integration of the first kind on the integral over $d t_{l     }$. Similarly, for  $l      \in L_+$, the operation of $\EuScript P_{-,\, l     }$ preserves absolute convergence and is the actual partial integration of the second kind on the integral over $d t_{l     }$.
	
	
	{\rm(3).}  $\EuScript P_{+,\, l     }$ and $\EuScript P_{-,\, l     }$ commute with $\EuScript P_{+,\, k}$ and  $\EuScript P_{-,\, k}$ if $l      \neq k$.
	
	{\rm(4).} Let  $\alpha \in \BN$. $\EuScript P^{\alpha}_{+,\, l     } (J_{\unu}  (x; \usigma; h) )$ is a linear combination of 
	\begin{align*}  
	[\nu_l  - 1]_{\alpha_3} & x^{- \alpha + \alpha_1} J_{\unu +  \alpha_1 \ue^d + \alpha \ue_l     } (x; \usigma; \textstyle  t^{\alpha_2}_l      \partial_{l     }^{\alpha_2} h ), 
	\end{align*} 
	and $\EuScript P^{\alpha}_{-,\, l     }  (J_{\unu} (x;  \usigma;  h)  )$ is a linear combination of
	\begin{align*}  
	 [ \nu_l      - 1 ]_{\alpha_3} x^{- \alpha + \alpha_1} J_{\unu - \alpha \ue^d - \alpha_1 \ue_l     } (x; \usigma; \textstyle  t^{\alpha_2}_l      \partial_l     ^{\alpha_2}  h ),
	\end{align*}
	for $\alpha_1 + \alpha_2 + \alpha_3 \leqslant \alpha$.
	The coefficients of these linear combinations may be uniformly bounded by a constant depending only on $\alpha$.
\end{lem}
\begin{proof}
	(1-3) are obvious. The two statements in (4) follow from calculating 
	\begin{equation*}
	 x^{-\alpha} t_l^{\alpha + \alpha_0} \partial ^{\alpha_0}_l   \left(   h(\ut) p_{\unu       } (\ut) e^{\varsigma_{d+1} i x    t_1 ... t_d     }\right) e^{i x \sum_{k = 1}^d \varsigma_{k} t_{k}\- }, \hskip 10 pt \alpha_0 \leqslant \alpha,
	\end{equation*}
	and
	\begin{equation*}
	x^{-\alpha} \partial ^\alpha_l      \left(   h(\ut) p_{\unu - \alpha \ue^d + \alpha \ue_l     } (\ut) e^{i x \sum_{k = 1}^d \varsigma_{k} t_{k}\- }\right) e^{\varsigma_{d+1} i x t_1 ... t_d }.
	\end{equation*}
	For the latter, one applies the following formula
	\begin{equation*}
	\frac {d^\alpha \big( e^{ a t\- } \big)} {d t^\alpha} = (- )^\alpha \sum_{\beta = 1}^\alpha \frac {\alpha! (\alpha-1)!} {(\alpha-\beta)! \beta! (\beta-1)!} a^\beta t^{- \alpha - \beta} e^{ a t\- },  \hskip 10 pt \alpha \in \BN_+, \, a \in \BC.
	\end{equation*} 
\end{proof}

\subsection{\texorpdfstring{Partitioning the Integral $J_{\unu}(x; \usigma)$}{Partitioning the Integral $J_{\nu}(x; \varsigma)$}}

Let $I$ be a finite set that includes $\{+, -\}$ and let
$$\sum_{\varrho \, \in \, I} h_\varrho (t) \equiv 1, \hskip 10 pt t \in \BR_+,$$ be a partition of unity on $\BR_+$ such that each $h_\varrho $ is a function in $ \mathscr T (\BR_+)$, 
$h_-(t) \equiv 1$ on a neighbourhood of zero and $h_+(t) \equiv 1$ for large $t$. Put $h_{\urho} (\ut) =  \prod_{l      = 1}^d h_{\varrho_l     } (t_{l     })$ for $\urho =  (\varrho_1, ..., \varrho_d) \in I^d$. We partition the integral $J_{\unu}(x; \usigma)$ into a finite sum of $J$-integrals 
\begin{equation*} 
J_{\unu}(x; \usigma) = \sum_{\urho\, \in \, I^d } J_{\unu}(x; \usigma; \urho),
\end{equation*}
with
\begin{equation*} 
J_{\unu}(x; \usigma; \urho) = J_{\unu}(x; \usigma; h_{\urho})
= \int_{\BR_+^d} h_{\urho} (\ut) p_{\unu} (\ut) e^{ i x \theta (\ut; \usigma)} d \ut.
\end{equation*}

\subsection{\texorpdfstring{The Definition of $\BJ_{\unu}(x; \usigma)$}{The Definition of $J_{\nu}(x; \varsigma)$}} \label{sec: Definition of BJ nu (x; sigma)}

Let $a > 0$ and assume $\unu \in \BS_{a}$. Let $A \geq a + 2$ be an integer. For $\urho \in I^d$ denote $L_\pm (\urho) = \left\{ l      : \varrho_l      = \pm \right \}$. 

We first treat $J_{\unu}(x; \usigma; \urho)$ in the case when both $L_+(\urho)$ and $L_-(\urho) $ are nonempty. Define $\EuScript P_{+,\, \urho} = \prod_{l      \in L_-(\urho)} \EuScript P_{+,\, l     }$. This is well-defined due to commutativity (Lemma \ref{lem: I pm ell} (3)).
By Lemma \ref{lem: I pm ell} (4) we find that $\EuScript P^{2 A}_{+, \, \urho} (J_{\unu}(x; \usigma; \urho) )$ is a linear combination of
\begin{equation} \label{3eq: after I - rho}
\begin{split}
&  \prod_{l      \in L_- (\urho)} [ \nu_l     - 1 ]_{\alpha_{3  l     }} \cdot x^{- 2 A |L_-(\urho)| + \sum_{l      \in L_-(\urho)} \alpha_{1 l     }} \cdot \\
& \hskip 40pt J_{\unu + (\sum_{l      \in L_-(\urho)} \alpha_{1 l     } ) \ue^d + 2 A \sum_{l      \in L_-(\urho)} \ue_l     } \big(x; \usigma; \big( \textstyle \prod_{l      \in L_-(\urho)} t^{\alpha_{2 l     }}_l      \partial_l     ^{\alpha_{2 l     }}\big) h_{\urho}\big),
\end{split}
\end{equation}
with 
$\alpha_{1 l     } + \alpha_{2 l     } + \alpha_{3 l     } \leqslant 2 A$ for each $l      \in L_-(\urho)$.
After this, we choose $l     _+ \in L_+(\urho)$ and apply  $\EuScript P_{-,\, l     _+}^{ A + \sum_{l      \in L_-(\urho)} \alpha_{1 l     }  }$ on the $J$-integral in (\ref{3eq: after I - rho}). 
By Lemma \ref{lem: I pm ell} (4) we obtain a linear combination of
\begin{equation} \label{3eq: linear combination first case}
\begin{split}
& [\nu_{l     _+} - 1]_{\alpha_3}  \prod_{l      \in L_- (\urho)} [\nu_l     - 1]_{\alpha_{3 l     }}  \cdot x^{- A (2 |L_-(\urho)| + 1) + \alpha_1} \cdot \\
& \hskip 40pt J_{\unu - A \ue^d + 2A \sum_{l      \in L_-(\urho)} \ue_l      - \alpha_1 \ue_{l_{ + }} } \big(x; \usigma; \big( t^{\alpha_2}_{l     _+} 
\partial^{\alpha_2}_{l     _+} \textstyle \prod_{l      \in L_-(\urho)} t^{\alpha_{2 l     }}_l      \partial^{\alpha_{2 l     }} _{l     } \big) h_{\urho}\big),
\end{split}
\end{equation}
with 
$\alpha_1 + \alpha_2 + \alpha_3 \leq \sum_{l      \in L_-(\urho)} \alpha_{1 l     } + A$.
It is easy to verify that the real part of the $l     $-th index of the $J$-integral in (\ref{3eq: linear combination first case}) is positive if    $l      \in L_-(\urho)$ and 
negative if $l      \in L_+(\urho)$. 
Therefore, the $J$-integral in (\ref{3eq: linear combination first case}) is absolutely convergent according to Lemma \ref{lem: I pm ell} (1). We define  $\BJ_{\unu}(x; \usigma; \urho)$ to be the total linear combination of all the resulting $J$-integrals. 

When $L_- (\urho) \neq \O$ but $L_+ (\urho) = \O$, we define $\BJ_{\unu}(x; \usigma; \urho) = \EuScript P^A_{+, \, \urho} ( J_{\unu}(x; \usigma; \urho) )$. It is a linear combination of 
\begin{equation} \label{3eq: linear combination second case}
\begin{split}
&   \prod_{l      \in L_- (\urho)} [\nu_l     - 1]_{\alpha_{3 l     }} \cdot x^{- A |L_-(\urho)| + \sum_{l      \in L_-(\urho)} \alpha_{1 l     }} \cdot \\
& \hskip 40pt J_{\unu + (\sum_{l      \in L_-(\urho)} \alpha_{1 l     } ) \ue^d + A \sum_{l      \in L_-(\urho)} \ue_l     } \big(x; \usigma; \big(\textstyle \prod_{l      \in L_-(\urho)} t^{\alpha_{2 l     }}_l      \partial^{\alpha_{2 l     }} _{l     } \big) h_{\urho}\big),
\end{split}
\end{equation}
with $\alpha_{1 l     } + \alpha_{2 l     } + \alpha_{3 l     } \leqslant A$. The $J$-integral in (\ref{3eq: linear combination second case}) is absolutely convergent.

When  $L_+ (\urho) \neq \O$ but $L_- (\urho) = \O$, we choose $l     _+ \in L_+ (\urho) $ and define $\BJ_{\unu}(x; \usigma; \urho) = \EuScript P^A_{-, \, l     _+} ( J_{\unu}(x; \usigma; \urho) )$. This is a linear combination of 
\begin{equation} \label{3eq: linear combination third case}
\begin{split}
[\nu_{l     _+} - 1]_{\alpha_3} x^{- A + \alpha_1} J_{\unu - A \ue^d - \alpha_1 \ue_{l     _{ +}} } \big(x; \usigma; t^{\alpha_2}_{l     _+} \partial_{l     _+}^{\alpha_2} h_{\urho} \big),
\end{split}
\end{equation}
with $\alpha_1 + \alpha_2 + \alpha_3 \leq A$. The $J$-integral in (\ref{3eq: linear combination third case}) is again absolutely convergent.

Finally,  when both  $L_- (\urho) $ and $L_+ (\urho) $ are empty, we put $\BJ_{\unu}(x; \usigma;  \urho) = J_{\unu}(x; \usigma;  \urho)$.

\begin{lem}\label{lem: BJ nu (x; sigma; rho) is well-deifned}
	The definition of  $\BJ_{\unu}(x; \usigma; \urho)$ is independent on $A$ and the choice of $l     _+\in L_+ (\urho)$.
\end{lem}
\begin{proof}
	We shall only consider the case when both $L_+(\urho)$ and $L_-(\urho) $ are nonempty. The other cases are similar and simpler.

	Starting from the $\BJ_{\unu}(x; \usigma; \urho)$ defined with $A$, we conduct the following operations in succession for all $l      \in L_-(\urho)$: $\EuScript P_{+,\, l     }$ twice and then $\EuScript P_{-,\, l     _+}$ once, twice or three times on each resulting $J$-integral so that the increment of the $l     $-th index is exactly one.  In this way, one arrives at the $\BJ_{\unu}(x; \usigma; \urho)$ defined with $A+1$. 
	In view of the assumption $A \geq a + 2 $, 
	absolute convergence is maintained at each step due to Lemma \ref{lem: I pm ell} (1). Moreover, in our settings, the operations $\EuScript P_{+,\, l     }$ and  $\EuScript P_{-,\, l     _+}$ are actual partial integrations (Lemma \ref{lem: I pm ell} (2)), so the value  is preserved in the process. In conclusion, $\BJ_{\unu}(x; \usigma; \urho)$  is independent on $A$.
	
	Suppose $l     _+, k_+ \in L_+(\urho)$. Repeating the process described in the last paragraph $A$ times, but with $l     _+$ replaced by $k_+$,  the $\BJ_{\unu}(x; \usigma; \urho)$ defined with $l     _+$ turns into a sum of integrals of an expression symmetric about $l     _+$ and $k_+$. 
	Interchanging $l     _+$ and $k_+$ throughout the arguments above, the $\BJ_{\unu}(x; \usigma; \urho)$ defined with $k_+$ is transformed into the same  sum of integrals. Thus we conclude that $\BJ_{\unu}(x; \usigma; \urho)$ is independent on the choice of $l     _+$.
\end{proof}

Putting these together, we define
\begin{equation*} 
\BJ_{\unu}(x; \usigma) = \sum_{\urho \, \in \, I^d } \BJ_{\unu}(x; \usigma; \urho),
\end{equation*}
and call $\BJ_{\unu} (x; \usigma)$ the \textit{rigorous interpretation of $J_{\unu} (x; \usigma)$}.  
The definition of $\BJ_{\unu} (x; \usigma)$ is clearly independent on the partition of unity $\{ h_{\varrho} \}_{\varrho \in I }$ on $\BR_+$.

Uniform convergence of the $J$-integrals in (\ref{3eq: linear combination first case}, \ref{3eq: linear combination second case}, \ref{3eq: linear combination third case}) with respect to $\unu$ implies that $\BJ_{\unu}(x; \usigma)$ is an analytic function of $\unu$ on $\BS_a^d$ and hence on the whole $\BC^d$ since $a$ was arbitrary. Moreover, for any nonnegative integer $j $, if one chooses $A \geq a + j + 2$, differentiating  $j$ times under the integral sign for  the $J$-integrals in (\ref{3eq: linear combination first case}, \ref{3eq: linear combination second case}, \ref{3eq: linear combination third case}) is legitimate. Therefore,  $\BJ_{\unu}(x; \usigma)$ is a smooth function of $x$. 

Henceforth, with ambiguity, we shall write $\BJ_{\unu} (x; \usigma)$ and  $\BJ_{\unu} (x; \usigma; \urho)$ as $J_{\unu} (x; \usigma)$ and $J_{\unu} (x; \usigma; \urho)$ respectively.

\section{\texorpdfstring{Equality between $J_{\unu}(x; \usigma)$ and $J (x ; \usigma, \ulambda)$}{Equality between $J_{\nu}(x; \varsigma)$ and $J (x ; \varsigma, \lambda)$}} \label{sec: Relating Bessel functions} 
The goal of this section is to prove  that the Bessel function $J (x ; \usigma, \ulambda)$ is indeed equal to the rigorous interpretation of its formal integral representation  $J_{\unu}(x; \usigma)$.

\begin{prop}\label{prop: J (x ; sigma, lambda) = J nu (x; sigma)} Suppose that $\ulambda \in \BL^{d}$ and $\unu \in \BC^d$ satisfy $\nu_l      = \lambda_l      - \lambda_{d+1}$, $l      = 1, ..., d$. Then
	\begin{equation*}
	J (x ; \usigma, \ulambda) = J_{\unu} (x; \usigma).
	\end{equation*}
\end{prop}

To prove this proposition, we first recall the expression of the Hankel transform $\Upsilon = \Hsl \upsilon $ as a Fourier type integral transform given in \S \ref{sec: Fourier, R+}.

Suppose that $\Re \lambda_1 > ... > \Re \lambda_d > \Re \lambda_{d+1}$. 
Let $\upsilon \in \mathscr S (\BR_+)$ be a Schwartz function on $\BR_+$. Define
\begin{equation}\label{4eq: Psi d+1}
\Upsilon_{d+1} (x ) = \ES_{(\varsigma_{d+1}, \lambda_{d+1})} \upsilon (x) = \int_{\BR_+} \upsilon (y) y^{-\lambda_{d+1}} e^{ \varsigma_{d+1} i x y} d y, \hskip 10 pt x \in \BR _+,
\end{equation}
and, for each $l      = 1, ..., d $, recursively define
\begin{equation}\label{4eq: Psi ell}
\begin{split}
\Upsilon_{l     } \left(x \right)  = \ET_{(\varsigma_{l}, \lambda_{l}-\lambda_{l+1})} \Upsilon_{l      + 1} (x) = \int_{\BR_+} \Upsilon_{l      + 1} \lp y  \rp y^{\lambda_l      - \lambda_{l      + 1} - 1} e^{ \varsigma_{l     } i x y\- } d y,
\hskip 10 pt x \in \BR _+.
\end{split}
\end{equation}
It is known from Lemma \ref{5lem: Hankel R+} and Lemma \ref{5lem: Miller-Schmid transforms, R+} (4) that $\Upsilon_{d+1}   \in \SS (\overline \BR_+) \subset \Ssiss (\BR_+)$ and $\Upsilon_{l} \in  \Ssiss (\BR_+)$ for all $l = 1, ..., d$. Moreover, we have $\Upsilon(x ) = x^{ - \lambda_1 } \Upsilon_1 (x  )$. See \S \ref{sec: Fourier, R+}.

The change of variables from $y $ to $ x y$ in (\ref{4eq: Psi ell}) yields
\begin{equation*} 
\Upsilon_{l     } \lp x \rp = \int_{\BR_+} \Upsilon_{l      + 1} \lp x y \rp x^{\lambda_l      - \lambda_{l      + 1}} y^{\lambda_l      - \lambda_{l      + 1} - 1} e^{ \varsigma_{l     } i y\-} d y.
\end{equation*}
Some calculations then show that $ \Upsilon_{1} \lp x \rp$ is equal to the iterated integral
\begin{equation}\label{4eq: Psi 1 (x; sigma)}
\begin{split}
x^{\nu_1 } \int_{\BR_+^{d+1}} \upsilon \left( y \right) y^{- \lambda_{d+1} }   \left( \prod_{l      = 1}^{d} y_l     ^{ \nu_l       - 1} \right)  
e^{i \lp \varsigma_{d+1}  xy y_1 ... y_{d} + \sum_{l      = 1}^d \varsigma_l      y_l     \- \rp } dy d y_d ... dy_1.
\end{split}
\end{equation}
The (actual) partial integration $\EuScript P_{l     }$ on the integral over $d y_{l     }$ 
is in correspondence with $\EuScript P_{+,\, l     }$, whereas the partial integration $\EuScript P_{d+1}$ on the integral over $d y$ 
has the similar effect as $\EuScript P_{-,\, l     _+}$ of decreasing the powers of all the $y_{l     }$ by one. 
These observations are crucial to our proof of Proposition \ref{prop: J (x ; sigma, lambda) = J nu (x; sigma)} as follows.

\begin{proof}[Proof of Proposition {\rm \ref{prop: J (x ; sigma, lambda) = J nu (x; sigma)}}]
	
	Suppose that  $\Re \lambda_1 > ... > \Re \lambda_d > \Re \lambda_{d+1}$. 
	We first partition the integral over $d y_{l     }$ in \eqref{4eq: Psi 1 (x; sigma)}, for each $l      = 1, ..., d$, into a sum of integrals according to a partition of unity $\{h^{\mathrm{o}}_{\varrho} \}_{\varrho \in I}$ of $\BR_+$. These partitions result in a partition of the integral (\ref{4eq: Psi 1 (x; sigma)}) into the sum
	\begin{equation*}
	\Upsilon_{1} \lp x \rp = \sum_{\urho\, \in\, I^d } \Upsilon_1 (x;   \urho),
	\end{equation*}
	with
	\begin{equation*}
	\begin{split}
	\Upsilon_1 (x;  \urho) = x^{\nu_1 } \int_{\BR_+^{d+1}} \upsilon \left( y \right) y^{- \lambda_{d+1} } & \left( \prod_{l      = 1}^{d} h^{\mathrm{o}}_{\varrho_l     } (y_l     ) y_l     ^{ \nu_l      - 1} \right)   e^{i \left( \varsigma_{d+1}  xy y_1 ... y_{d} + \sum_{l      = 1}^d \varsigma_l      y_l     \- \right)} dy d y_d ... dy_1.
	\end{split}
	\end{equation*}
	We now conduct the operations in \S \ref{sec: Definition of BJ nu (x; sigma)} with $\EuScript P_{+,\, l     }$ replaced by $\EuScript P_{l     }$ and $\EuScript P_{-,\, l     _+}$ by $\EuScript P_{d+1}$ to each integral $ \Upsilon_1 (x;  \urho)$ defined above. While preserving the value, these   partial integrations turn the iterated integral $ \Upsilon_1 (x;   \urho)$ into an absolutely convergent multiple integral. We are then able to move the innermost integral over $d y$ to the outermost place. The integral over  $dy_d ... dy_1$ now becomes   the {inner} integral. Making the change of variables $y_l      = t_{l     } (xy)^{-\frac 1 { d+1} }$ to the inner integral  over  $dy_d ... dy_1$,
	each partial integration $\EuScript P_{l     }$ that we did turns into $\EuScript P_{+,\, l     }$. By the same arguments in the proof of Lemma \ref{lem: BJ nu (x; sigma; rho) is well-deifned} showing that $J_{\unu}(x; \usigma)$ is independent on the choice of $l     _+\in L_+(\urho)$, the operations of $\EuScript P_{d+1}$ that we conducted at the beginning may be reversed and substituted by those of $\EuScript P_{-,\, l     _+}$. It follows that the inner integral over $dy_d ... dy_1$ is equal to $x^{\lambda_1 } \upsilon (y) J_{\unu} \big(  (xy)^{\frac 1 { d+1 } }; \usigma; \urho \big)$, with $ h_{\varrho}(t) = h^{\mathrm{o}}_{\varrho}\big(  t (xy)^{-\frac 1 { d+1 }} \big)$. Summing over $\urho \in I^d$, we conclude that
	\begin{equation*}
	\Upsilon (x ) = x^{ - \lambda_1} \Upsilon_1 (x ) =  \int_{\BR_+}  \upsilon (y) J_{\unu} \big( 
	 (xy)^{\frac 1 { d+1 }}; \usigma \big) d y.
	\end{equation*}
	Therefore, in view of (\ref{2eq: Psi (x; sigma) as Hankel transform}), we have
	$J (x ; \usigma, \ulambda) = J_{\unu} (x; \usigma).$ This equality  holds true universally due to the principle of analytic continuation. 
\end{proof}


In view of Proposition \ref{prop: J (x ; sigma, lambda) = J nu (x; sigma)}, we shall subsequently assume that $\ulambda \in \BL^{d}$ and $\unu \in \BC^d$ satisfy the relations $\nu_l      = \lambda_l      - \lambda_{d+1}$, $l      = 1, ..., d$.

\section{$H$-Bessel Functions and $K$-Bessel Functions}\label{sec: Bessel functions of K-type and H-Bessel functions}

According to Proposition \ref{prop: Classical Bessel functions}, $ J_{2\lambda} (x; \pm, \pm) = J (x; \pm, \pm, \lambda, - \lambda) $ 
is a Hankel function, and $ J_{2\lambda} (x; \pm, \mp) = J (x; \pm, \mp, \lambda, - \lambda)$
is a $K$-Bessel function. There is a remarkable difference between the behaviours of   Hankel functions and the $K$-Bessel function for large argument.   Hankel functions oscillate and decay proportionally to $  1 / {\sqrt x}$, whereas the $K$-Bessel function exponentially decays. 
On the other hand, this phenomena also arises in higher rank for the  prototypical example shown in Proposition \ref{prop: special example}.  

In the following, we shall  show that such a categorization stands in general for the Bessel functions $J_{\unu} (x; \usigma)$ of an arbitrary index $\unu$.
For this, we shall analyze each integral $J_{\unu}(x; \usigma; \urho)$ in the rigorous interpretation of  $J_{\unu} (x; \usigma)$ by \textit{the method of stationary phase}.

\vskip 5 pt

First of all, the asymptotic behaviour of $J_{\unu} (x; \usigma)$  for large argument should rely on the existence of a stationary point of the phase function $ \theta (\ut ; \usigma )$ on $\BR_+^d$.
We have $$ \theta' (\ut; \usigma) = \left( \varsigma_{d+1} t_1 ... \widehat {t_{l     }} ... t_d -  \varsigma_l      t_{l     }^{-2} \right)_{l      = 1}^d.$$ 
A stationary point  of  $ \theta (\ut ; \usigma )$  exists in $\BR^d_+$ if and only if $\varsigma_1 = ... = \varsigma_d = \varsigma_{d+1}$, in which case  it is equal to $ \ut_0 = (1, ..., 1)$.

\begin{term}\label{term: Bessel functions of K-type and H-Bessel functions}
	We write $H^{\pm}_{\unu}(x ) = J_{\unu} (x; \pm, ..., \pm)$, $H^{\pm} (x; \ulambda ) = J (x; \pm, ... \pm, \ulambda)$ 
	and call them $H$-Bessel functions. If two of the signs $\varsigma_1, ..., \varsigma_d, \varsigma_{d+1}$ are different, then $J_{\unu} (x; \usigma)$, or $J (x; \usigma, \ulambda)$, is called a $K$-Bessel function.
\end{term}

\subsection*{Preparations}
We shall retain the notations in \S \ref{sec: Rigorous Interpretations}.
Moreover, for our purpose we choose a partition of unity $\left\{ h_{\varrho} \right \}_{\varrho \, \in \{-, 0, +\}}$ on $\BR_+$ such that
$h_- $, $h_0 $ and $h_+ $ are  functions in $\mathscr T (\BR_+)$ supported on $K_- = \left(0, \frac 1 2\right]$, $K_0 = \left[\frac 1 4, 4\right]$ and $K_+ = \left[2, \infty \right)$ respectively.  
Put $K_{\urho} = \prod_{l      = 1}^d K_{\varrho_l     } $ and $h_{\urho} (\ut) =  \prod_{l      = 1}^d h_{\varrho_l     } (t_{l     })$ for $\urho \in \{-, 0 , + \}^d$.
Note that $\ut_0$ is enclosed in the central hypercube $K_{\boldsymbol 0}$.
According to this partition of unity, $J_{\unu}(x; \usigma)$ is partitioned into the sum of $3^d$ integrals $J_{\unu}(x; \usigma; \urho)$. In view of (\ref{3eq: linear combination first case}, \ref{3eq: linear combination second case}, \ref{3eq: linear combination third case}), $J_{\unu}(x; \usigma; \urho)$ is a $\BC [x\-]$-linear combination of absolutely convergent $J$-integrals of the form
\begin{equation}\label{5eq: J-integral}
J_{\boldsymbol \nu'} (x; \usigma; h) = \int_{\BR_+^d} h (\ut) p_{\boldsymbol \nu'} (\ut) e^{ i x \theta (\ut; \usigma)} d \ut.
\end{equation}
Here $h \in \bigotimes^d \mathscr T (\BR_+)$ is supported in $K_{\urho}$, and $\boldsymbol \nu' \in \unu + \BZ^d$ satisfies
\begin{equation}\label{5eq: conditions on nu'}
\Re  \nu'_{l     } - \Re  \nu_{l     } \geq A \text { if } l      \in L_-(\urho), \text { and } \Re  \nu'_{l     } - \Re  \nu_{l     } \leq - A \text{ if }l      \in L_+(\urho),
\end{equation} 
with $A > \max  \left\{ |\Re \nu_l      | \right \} + 2$.

\subsection{\texorpdfstring{Estimates for $J_{\unu}(x; \usigma; \urho)$ with $\urho \neq \boldsymbol 0$}{Estimates for  $J_{\nu}(x; \varsigma; \varrho)$ with $\varrho \neq 0$}} \label{sec: Bound for rho neq 0}


Let 
\begin{equation}\label{5eq: Theta (t; sigma)}
\Theta (\ut; \usigma) = \sum_{l      = 1}^d \lp t_{l     } \partial_{l     } \theta (\ut; \usigma) \rp^2 = \sum_{l      = 1}^d \lp \varsigma_{d+1} t_1 ... t_d - \varsigma_{l     } t_{l     }\- \rp^2.
\end{equation}

\begin{lem}\label{lem: lower bound for Theta}
	Let $\urho \neq \boldsymbol 0$. We have for all $\ut \in K_{\urho}$
	$$\Theta (\ut; \usigma) \geq \frac 1 {16}.$$
\end{lem}
\begin{proof}
	Instead, we shall prove
	\begin{equation*}
	\max \Big\{\left|\varsigma_{d+1} t_1 ... t_d - \varsigma_{l     } t_{l     }\- \right| :  \ut \in \BR_+^d \smallsetminus  K_{\boldsymbol 0} \text { and } l      = 1, ..., d \Big \} \geq \frac 1 4.
	\end{equation*}
	Firstly, if $t_1 ... t_d < \frac 3 4$, then there exists $t_l      < 1$ and hence $\left|\varsigma_{d+1} t_1 ... t_d - \varsigma_{l     } t_{l     }\- \right| > 1 - \frac 3 4 = \frac 1 4$. Similarly, if $t_1 ... t_d > \frac 7 4$, then there exists $t_l      > 1$ and hence $\left|\varsigma_{d+1} t_1 ... t_d - \varsigma_{l     } t_{l     }\- \right| >\frac 7 4 - 1 > \frac 1 4$. Finally, suppose  that $  t_1 ... t_d \in \left[\frac 3 4, \frac 7 4 \right]$, then for our choice of $\ut $ there exists  $t_{l     } \notin \lp \frac 1 2, 2 \rp$, and therefore we still have $\left|\varsigma_{d+1} t_1 ... t_d - \varsigma_{l     } t_{l     }\- \right| \geq \frac 1 4$.
\end{proof}

Using (\ref{5eq: Theta (t; sigma)}), we rewrite the $J$-integral $J_{\boldsymbol \nu'} (x; \usigma; h)$ in \eqref{5eq: J-integral} as below,
\begin{equation} \label{5eq: rewrite J nv' (x; sigma; h)}
\begin{split}
\sum_{l      = 1}^d \int_{\BR_+^d} h (\ut) \left( \varsigma_{d+1} p_{\boldsymbol \nu' + \ue^d + \ue_{l     }} (\ut) - \varsigma _{l     } p_{\boldsymbol \nu'} (\ut) \right) \Theta (\ut; \usigma)\- \cdot \partial_{l     } \theta (\ut; \usigma) e^{i x \theta (\ut; \usigma)} d \ut.
\end{split}
\end{equation}
We now make use of  the    \textit{third} kind of partial integrations arising from
$$\partial \big(e^{i x \theta (\ut; \usigma) } \big) = i x \cdot \partial_{l     } \theta (\ut; \usigma) e^{ i x \theta (\ut; \usigma) } \partial t_l     . $$
For the $l     $-th integral in \eqref{5eq: rewrite J nv' (x; sigma; h)}, we apply the corresponding partial integration of the third kind. In this way, \eqref{5eq: rewrite J nv' (x; sigma; h)} turns into
\begin{align*}
& - (ix)\- \sum_{l      = 1}^d \int_{\BR_+^d} t_l      \partial_l      h \lp \varsigma _{d+1} p_{\boldsymbol \nu' + \ue^d} - \varsigma _{l     } p_{\boldsymbol \nu' -  \ue_{l     }} \rp  \Theta \- e^{i x \theta} d\ut\\
& - (ix)\- \sum_{l      = 1}^d \int_{\BR_+^d} h \lp \varsigma _{d+1} (\nu'_{l     } + 1) p_{\boldsymbol \nu' + \ue^d } - \varsigma _{l     } (\nu'_l      - 1) p_{\boldsymbol \nu' - \ue_{l     } } \rp \Theta \-  e^{i x \theta} d\ut\\
& + \varsigma _{d+1} 2 d^2 (ix)\- \int_{\BR_+^d} h  p_{\boldsymbol \nu' + 3 \ue^d} \Theta^{-2}  e^{i x \theta} d \ut \\
& + 2 (ix)\-  \sum_{l      = 1}^d \int_{\BR_+^d} h \big( \varsigma _{l     } (1 - 2 d) p_{\boldsymbol \nu' + 2 \ue^d -\ue_l     }    - \varsigma _{d+1} p_{\boldsymbol \nu' + \ue^d - 2 \ue_l     } + \varsigma _{l     } p_{\boldsymbol \nu' - 3 \ue_l     } \big) \Theta^{-2}  e^{i x \theta} d \ut\\
& + 4 (ix)\- \sum_{1\leq l      < k \leq d} \varsigma _{d+1} \varsigma _{l      } \varsigma _{k} \int_{\BR_+^d} h p_{\boldsymbol \nu' + \ue^d - \ue_l      - \ue_{k}}  \Theta^{-2}  e^{i x \theta} d \ut,
\end{align*}
where $\Theta$ and $\theta$ are the shorthand notations for $\Theta (\ut; \usigma)$ and $\theta (\ut; \usigma)$. Since the shifts of indices do not exceed $3$, it follows from the condition (\ref{5eq: conditions on nu'}), combined with  Lemma \ref{lem: lower bound for Theta}, that all the integrals above absolutely converge provided $A > \max  \left\{ |\Re \nu_l      | \right \} + 3$. 

Repeating the above manipulations, we obtain the following lemma by a straightforward inductive argument.
\begin{lem}\label{lem: J(x; sigma; h) for rho neq 0}
 	Let $B$ be a nonnegative integer, and choose $A =  \lfloor \max  \left\{ |\Re \nu_l      | \right \} \rfloor + 3 B + 3$. Then $J_{\boldsymbol \nu'} (x; \usigma; h)$ is equal to a linear combination of $\left(\frac 1 2 {(d^2 - d)}  + 7 d + 1\right)^B$ many absolutely convergent integrals of the following form
	\begin{equation*}
	(ix)^{-B} P(\boldsymbol \nu') \int_{\BR_+^d} \ut^{\ualpha}\partial^{\ualpha} h (\ut) p_{\boldsymbol \nu''}(\ut) \Theta(\ut; \usigma)^{-B-B_2} e^{i x \theta(\ut; \usigma)} d\ut,
	\end{equation*}
	where  $|\ualpha| + B_1 + B_2 = B$ {\rm($\ualpha \in \BN^d $)}, $P $ is a polynomial  of degree $B_1$ and integer coefficients of size $O_{B,\, d} (1)$, and $\boldsymbol \nu'' \in \boldsymbol \nu' + \BZ^d$ satisfies $|\nu''_{l     }  - \nu'_{l     }| \leq  B + 2 B_2$ for all $l      = 1, ..., d$. Recall that in the multi-index notation $|\ualpha| = \sum_{l      = 1}^d \alpha_l     $, $\ut^{\ualpha} = \prod_{l      = 1}^d t_{l     }^{\alpha_l     }$ and $\partial^{\ualpha} = \prod_{l      = 1}^d \partial_{l     }^{\alpha_l     } $.
\end{lem}
Define $\mathfrak c = \max  \left\{ | \nu_l      | \right \} + 1$ and  $\mathfrak r = \max  \left\{ |\Re \nu_l      | \right \}$.  Suppose that  $x \geq \mathfrak c$.  Applying Lemma \ref{lem: J(x; sigma; h) for rho neq 0} and \ref{lem: lower bound for Theta} to the $J$-integrals in (\ref{3eq: linear combination first case}, \ref{3eq: linear combination second case}, \ref{3eq: linear combination third case}), one obtains the estimate
\begin{equation*} 
J_{\unu}(x; \usigma; \urho) \lll_{\, \mathfrak r, \, M, \, d} \lp \frac {\mathfrak c} x \rp^M ,
\end{equation*} 
for   any given nonnegative integer $M $. 
Slight modifications of the above arguments yield a similar estimate for the derivative
\begin{equation}\label{5eq: bound for J nu (j) (x; sigma; rho) rho neq 0}
J^{(j)}_{\unu}(x; \usigma; \urho) \lll_{\,\mathfrak r,\, M,  j, \, d} \lp \frac {\mathfrak c} x \rp^M .
\end{equation} 

\begin{rem}
	Our proof of \eqref{5eq: bound for J nu (j) (x; sigma; rho) rho neq 0} is similar to that of \cite[Theorem 7.7.1]{Hormander}. Indeed, $\Theta (\ut; \usigma) $ plays the same role as $|f'|^2 + \Im f$ in the proof of \cite[Theorem 7.7.1]{Hormander}, where $f$ is the phase function there. The non-compactness of $K_{\urho}$ however prohibits the application of \cite[Theorem 7.7.1]{Hormander} to the $J$-integral in \eqref{5eq: J-integral} in our case.

\end{rem}

\subsection{Rapid Decay of $K$-Bessel Functions} Suppose that   there exists $ l  $ such that  $ \varsigma_{l} \neq \varsigma_{d+1}$. Then for any $\ut \in K_{\boldsymbol 0}$  
$$  
\left|\varsigma _{d+1} t_1 ... t_d - \varsigma _{l} t_{l}^{-1} \right| > t_{l}\- \geq \frac 1 4.$$
Similar to the arguments in \S \ref{sec: Bound for rho neq 0}, repeating the $l$-th partial integration of the third kind yields
the same bound (\ref{5eq: bound for J nu (j) (x; sigma; rho) rho neq 0}) in the case $\urho = \boldsymbol 0$.

\begin{rem}
	For this, we may also directly apply \cite[Theorem 7.7.1]{Hormander}.
\end{rem}

\begin{thm}\label{thm: Bessel functions of K-type}
	Let $\mathfrak c = \max  \left\{ | \nu_l      | \right \} + 1$ and  $\mathfrak r = \max  \left\{ |\Re \nu_l      | \right \}$. Let $j$ and $M$ be nonnegative integers. Suppose that one of the  signs $\varsigma_1, ..., \varsigma_d$ is different from $ \varsigma_{d+1}$. 
	Then  
	\begin{equation*} 
	J_{\unu}^{(j)} (x;  \usigma) \lll_{\, \mathfrak r,\, M, j,\, d} \lp \frac {\mathfrak c} x \rp^M
	\end{equation*}
	for any $x \geq \fc$. In particular, $J_{\unu} (x;  \usigma)$ is a Schwartz function at infinity, namely, all the derivatives $ J_{\unu}^{(j)} (x;  \usigma)$ rapidly decay at infinity.
\end{thm}

\subsection{Asymptotic Expansions of $H$-Bessel Functions}\label{sec: asymptotic expansions of H Bessel functions}

In the following, we shall adopt the convention $(\pm i)^{a} = e^{ \pm \frac 1 2 i \pi a}$, $a \in \BC$.

We first introduce the function $W^{\pm}_{\unu} (x )$, which is closely related to the   Whittaker  function of imaginary argument if $d=1$ (see \cite[\S 17.5, 17.6]{Whittaker-Watson}), defined by
\begin{equation*} 
W^{\pm}_{\unu} (x ) = (d+1)^{\frac 1 2} (\pm 2 \pi i)^{- \frac d 2} e^{\mp i (d+1) x} H^{\pm}_{\unu} (x).
\end{equation*}
Write
$H^\pm_{\unu} (x; \urho) = J_{\unu} (x; \pm,..., \pm; \urho)$ and define
\begin{equation*} 
W^{\pm}_{\unu} (x;  \urho) = (d+1)^{\frac 1 2} (\pm 2 \pi i)^{- \frac d 2} e^{\mp i (d+1) x} H^{\pm}_{\unu} (x; \urho).
\end{equation*}
For $\urho \neq \boldsymbol 0$, the bound (\ref{5eq: bound for J nu (j) (x; sigma; rho) rho neq 0}) for $H^\pm_{\unu} (x; \urho)$ is also valid for $ W^{\pm}_{\unu} (x;  \urho)$. Therefore, we are left with analyzing $W^{\pm}_{\unu} (x; \boldsymbol 0)$. We have
\begin{equation}\label{5eq: W nu (j) (x; pm; 0) integral}
\begin{split}
W^{\pm, (j)}_{\unu} (x; \boldsymbol 0) =\ & (d+1)^{\frac 1 2} (\pm 2 \pi i)^{- \frac d 2} (\pm i)^j \\
& \int_{K_{\boldsymbol 0}} \lp \theta (\ut) - d - 1 \rp^j h_{\boldsymbol 0} (\ut) p_{\unu} (\ut) e^{\pm i x \lp \theta (\ut) - d - 1 \rp} d\ut,
\end{split}
\end{equation}
with \begin{equation}\label{5eq: theta (t)}
\theta (\ut) = \theta (\ut; +, ..., +) = t_1 ... t_d + \sum_{l      = 1}^d t_l     \-.
\end{equation}
\begin{prop} \label{prop: Stationary phase}
	\cite[Theorem 7.7.5]{Hormander}.
	Let $K \subset \BR^d$ be a compact set, $X$ an open neighbourhood of $K$ and $M$ a nonnegative integer. If $u(\ut) \in C^{2M}_0 (K)$, $f(\ut) \in C^{3M + 1} (X)$ and $\Im f \geq 0$ in $X$, $\Im f (\ut_0) = 0$, $ f' (\ut_0) = 0$, $\det f'' (\ut_0) \neq 0$ and $ f'  \neq 0$ in $K \smallsetminus \{\ut_0 \}$, then for $x > 0$
	\begin{equation*}
	\begin{split}
	\hskip -1 pt \left| \int_{K} \hskip -3 pt  u(\ut) e^{i x f(\ut)} d\ut   -    e^{i x f(\ut_0)} \hskip -1 pt \left((2\pi i)^{- d } \det f'' (\ut_0) \right)^{ - \frac 1 2} \hskip -3 pt \sum_{ m = 0 }^{M-1} x^{- m - \frac d 2} \EuScript L_m u \right|   \lll  \hskip -1 pt  x^{-M} \hskip -3 pt \sum_{|\ualpha| \leq 2M} \sup\left| D^{\ualpha} u \right|.
	\end{split}
	\end{equation*}
	Here the implied constant depends only on $M$, $f$, $K$ and  $d$.
	With
	$$g  (\ut) = f(\ut) - f(\ut_0) - \frac 1 2 \left\langle f'' (\ut_0) (\ut - \ut_0), \ut - \ut_0 \right\rangle $$
	which vanishes of third order at $\ut_0$, we have
	\begin{equation*}
	\EuScript L_m u =  i^{- m}\sum_{r = 0}^{2 m} \frac 1 { 2^{m+r} (m+r) !r! } \left \langle  f'' (\ut_0)\- D, D \right \rangle^{m+r} \lp g^r  u \rp (\ut_0).\footnote{According to H\"ormander, $D = - i (\partial_1, ..., \partial_d)$. }
	\end{equation*}
	This is a differential operator of order $2 m$ acting on $u$ at $\ut_0$. The coefficients are rational homogeneous functions of degree $- m$ in
	$ f'' (\ut_0)$, ..., $ f^{(2m+2)} (\ut_0)$ with denominator $(\det f'' (\ut_0))^{3m}$. In every term the total number of derivatives of $u$ and of $f''$ is at most $2 m$.
\end{prop}

We now apply Proposition \ref{prop: Stationary phase} to the integral in \eqref{5eq: W nu (j) (x; pm; 0) integral}. For this, we let
\begin{align*} 
& K = K_{\boldsymbol 0} = \textstyle  \left[ \frac 1 4, 4\right ]^{d}, \hskip 10 pt X = \textstyle  \left( \frac 1 5, 5\right )^{d},\\
& f (\ut) =  \pm \left( \theta (\ut) - d - 1 \right), \hskip 10 pt f' (\ut) = \pm \left( t_1  ...\widehat{t_{l     }} ... t_{d} - t_{l     }^{-2} \right)_{l      = 1}^{d}, \hskip 10 pt \ut_0 = (1, ..., 1),\\
& f'' (\ut_0) = \pm \begin{pmatrix}
2 & 1 & \cdots & 1\\
1 & 2 & \cdots & 1\\
\vdots & \vdots & \ddots &  \vdots\\
1 & 1 & \cdots & 2
\end{pmatrix}, \hskip 10 pt \det  f'' (\ut_0) = (\pm)^d ( d+1 ), \hskip 10 pt g (\ut) = \pm G(\ut),\\
& f'' (\ut_0)\- = \pm \frac 1 {d+1} \begin{pmatrix}
d & -1 & \cdots & -1\\
-1 & d & \cdots & -1\\
\vdots & \vdots & \ddots &  \vdots\\
-1 & -1 & \cdots & d
\end{pmatrix},\\
& u(\ut) = (d+1)^{\frac 1 2} (\pm 2 \pi i)^{- \frac d 2} (\pm i)^j \left( \theta (\ut) - d - 1 \right)^j p_{\unu} (\ut) h_{\boldsymbol 0} (\ut),
\end{align*}
with
\begin{equation}\label{5eq: G(t)}
\begin{split}
G(\ut) =   t_1 ... t_{d} + \sum_{l      = 1}^{d} \left( - t_{l     } ^2 + (d+1) t_{l     } + t_{l     }\- \right)   - \sum_{ 1 \leq l      < k \leq d } t_{l     } t_{k} - \frac {(d+1)(d+2)} 2 .
\end{split}
\end{equation}
Proposition \ref{prop: Stationary phase} yields the following asymptotic expansion of $W^{\pm, (j)}_{\unu} (x; \boldsymbol 0)$,
\begin{equation*} 
W^{\pm, (j)}_{\unu} (x; \boldsymbol 0) = \sum_{m = 0 }^{M-1} ( \pm i )^{j - m} B_{m, j} (\unu) x^{- m - \frac d 2} + O_{\,\fr,\, M, j,\, d} \left( \frc^{2 M} x^{-M}\right), \hskip 10 pt x > 0,
\end{equation*}
with 
\begin{equation}\label{5eq: B mj}
B_{m, j} (\unu) =  \sum_{r=0}^{2m} \frac {(-)^{m+r} \EuScript L^{m+r} \left( G^r (\theta - d - 1)^j p_{\unu} \right) (\ut_0)} { (2 (d+1))^{m+r} (m+r)! r! } ,
\end{equation}
in  which $\EuScript L $ is the second-order differential operator given by
\begin{equation}\label{5eq: differential operator D}
\EuScript L = d \sum_{l      = 1}^{d} \partial_{l     }^2 - 2 \sum_{ 1 \leq l      < k \leq d } \partial_{l     } \partial_{k} .
\end{equation} 

\begin{lem}\label{5lem: coefficient}  We have $B_{m, j} (\unu)  = 0$ if $  m < j$.
	Otherwise,  $B_{m, j} (\unu) \in \BQ[\unu]$ is a symmetric polynomial  of degree $2m-2j$. 
	In particular,
	$B_{m, j} (\unu) \lll_{\,m, j, \, d} \frc^{2 m - 2 j}$ for $m \geq j$.
\end{lem}

\begin{proof}
	The symmetry of  $ B_{m, j} (\unu) $ is clear from definition. 
	Since $ \theta - d - 1 $ vanishes of second order at $\ut_0$, $2 j$ many differentiations are required to remove the zero of $ \left( \theta - d - 1 \right)^j $ at $\ut_0$. From this, along with  the descriptions of the differential operator $\EuScript L_m$ in Proposition \ref{prop: Stationary phase}, one proves the lemma. 
\end{proof}

Furthermore, in view of the bound (\ref{5eq: bound for J nu (j) (x; sigma; rho) rho neq 0}), the total contribution to $W_{\unu}^{\pm, (j)}(x )$ from  all the $W_{\unu}^{\pm, (j)}(x;  \urho)$ with $\urho \neq \boldsymbol 0$ is of size $O_{\,\mathfrak r,\, M, j,\, d} \lp   {\mathfrak c}^{ M} x^{- M} \rp$ and hence may be absorbed into the error term in the asymptotic expansion of $W^{\pm, (j)}_{\unu} (x; \boldsymbol 0)$. 

In conclusion, the following proposition is established. 

\begin{prop}\label{prop: asymptotics for W nu (j) (x; pm)} Let $M$, $j$ be nonnegative integers such that $M \geq j$.
	Then for $x \geq \frc$ 
	\begin{equation*} 
	W_{\unu}^{\pm, (j)}(x ) = \sum_{ m = j }^{M-1} ( \pm i )^{j - m} B_{m, j} (\unu) x^{- m - \frac d 2} + O_{\,\fr,\, M, j,\, d} \left( \frc^{2 M } x^{-M}\right).
	\end{equation*}
\end{prop}


\begin{cor}\label{cor: W(j)}
	Let $N$, $j$ be nonnegative integers such that $N \geq j$, and let $\epsilon > 0$.
	
	{ \rm (1).}  We have
	$ W_{\unu}^{\pm, (j)}(x ) \lll_{\,\fr, j,\, d} \fc^{2j} x^{-j}$ for $x \geq \frc $.
	
	{ \rm (2).} If $x \geq \frc^{2 + \epsilon}$, then
	\begin{equation*} 
	W_{\unu}^{\pm, (j)}(x ) = \sum_{ m = j }^{N-1} ( \pm i )^{j - m} B_{m, j} (\unu) x^{- m - \frac d 2} + O_{\,\fr,\, N, j,\, \epsilon,\, d} \left( \frc^{2 N } x^{-N - \frac d 2}\right).
	\end{equation*} 
\end{cor}
\begin{proof}
	On letting $M = j$, Proposition \ref{prop: asymptotics for W nu (j) (x; pm)} implies (1). On  choosing $M$ sufficiently large so that $(2 + \epsilon) \lp M-N + \frac d 2\rp \geq 2 (M-N) $, Proposition \ref{prop: asymptotics for W nu (j) (x; pm)} and Lemma \ref{5lem: coefficient} yield
	\begin{align*} 
	& W_{\unu}^{\pm, (j)}(x ) - \sum_{ m = j }^{N-1} ( \pm i )^{j - m} B_{m, j} (\unu) x^{- m - \frac d 2} \\
	= \ & \sum_{ m = N }^{M-1} ( \pm i )^{j - m} B_{m, j} (\unu) x^{- m - \frac d 2} + O_{\,\fr, j,\, M,\, d} \left( \frc^{2 M } x^{-M}\right) = O_{\,\fr, j,\, N,\, \epsilon,\, d} \left( \frc^{2 N } x^{-N - \frac d 2}\right).
	\end{align*}
\end{proof}

Finally,   the asymptotic expansion of $H^{\pm} (x; \ulambda)$($ = H^{\pm}_{\unu} (x)$) is formulated as below. 

\begin{thm}\label{thm: asymptotic expansion}
	Let $\mathfrak C =  \max  \left\{ | \lambda_l      | \right \} + 1$ and $\mathfrak R = \max  \left\{ |\Re \lambda_l      | \right \}$. Let $M$ be a nonnegative integer.
	
	{\rm (1).} Define $W^{\pm} (x; \ulambda) = \sqrt n (\pm 2 \pi i)^{- \frac { n-1} 2} e^{\mp i n x} H^{\pm}  (x; \ulambda)$. Let $M \geq j \geq 0$. Then \begin{equation*} 
	W^{\pm, (j)} (x; \ulambda) = \sum_{m = j }^{M-1} ( \pm i )^{j - m} B_{m, j} (\ulambda) x^{- m - \frac { n-1} 2} + O_{\,\mathfrak R,\, M, j ,\, n} \left( \fC^{2 M } x^{-M}\right)
	\end{equation*} 
	for all $x \geq \mathfrak C $.
	Here $B_{m, j} (\ulambda) \in \BQ[\ulambda]$ is a symmetric polynomial in $\ulambda$ of degree $2m$, with $B_{0, 0} (\ulambda) = 1$. The coefficients of $B_{m, j} (\ulambda)$ depends only on $m$, $j$ and $d$.

	{\rm (2).} Let $B_{m } (\ulambda) = B_{m, 0} (\ulambda)$. Then for $x \geq \mathfrak C $  
	\begin{equation*}
	\begin{split}
	H^{\pm}  (x; \ulambda) = n^{- \frac 1 2}   (\pm 2 \pi i)^{ \frac {n-1}  2} e^{ \pm i n x} x^{ - \frac {n-1} 2}  \hskip - 3 pt \lp  \sum_{m=0}^{M-1} ( \pm i )^{ - m} B_{m} (\ulambda) x^{- m} \hskip -1 pt + \hskip -1 pt O_{\,\mathfrak R,\, M,\, d} \left( \mathfrak C^{2 M} x^{-M + \frac {n-1} 2} \hskip -1 pt \right) \hskip -2 pt \rp \hskip -3 pt.
	\end{split}
	\end{equation*}
\end{thm}

\begin{proof}
	This theorem is a direct consequence of Proposition \ref{prop: asymptotics for W nu (j) (x; pm)} and Lemma \ref{5lem: coefficient}.
	It is only left to verify the symmetry of $B_{m, j}(\ulambda) = B_{m, j} (\unu)$ with respect to $\ulambda$. Indeed, in view of the definition of $H^{\pm} (x; \ulambda)  $ by \eqref{1def: G(s; sigma; lambda)} and \eqref{3eq: definition of J (x; sigma)}, $H^{\pm} (x; \ulambda)  $ is symmetric with respect to $\ulambda$, so $B_{m, j}(\ulambda) $ must be represented by a symmetric polynomial in $\ulambda$ modulo $\sum_{l     =1}^{d+1} \lambda_l     $.
\end{proof}

\begin{cor}\label{cor: H pm}
	Let $M$ be a nonnegative integer, and let $\epsilon > 0$.
	Then for $x \geq  \mathfrak C^{2 + \epsilon}$  
	\begin{equation*}
	\begin{split}
	H^{\pm}  (x; \ulambda) =   n^{- \frac 1 2} (\pm 2 \pi i)^{ \frac {n-1} 2} e^{ \pm i n x} x^{ - \frac {n-1} 2} 
	\lp \sum_{m=0}^{M-1} ( \pm i )^{ - m} B_{m} (\ulambda) x^{- m} + O_{\,\mathfrak R,\, M,\, \epsilon,\, n} \left( \mathfrak C^{2 M} x^{-M  } \right) \rp.
	\end{split}
	\end{equation*}
\end{cor}

\subsection{Concluding Remarks}

\
\vskip 5 pt

\subsubsection{On the Analytic Continuations of $H$-Bessel Functions}\label{sec: Analytic continuations of the H-Bessel functions}
Our observation is that the phase function $\theta $ defined by (\ref{5eq: theta (t)}) is always positive on $\BR^{d}_+$. It follows that if one replaces $x$  by $z = x e^{i\omega}$, with $x > 0$ and $0 \leq \pm \omega \leq \pi$, then the various $J$-integrals in the rigorous interpretation of $H^{\pm}_{\unu} (z)$ remain absolutely convergent, uniformly with respect to $z$, since $\left| e^{\pm i z \theta(\ut)} \right| = e^{\mp x \sin \omega \, \theta(\ut)} \leq 1.$ Therefore, the resulting integral $H^{\pm}_{\unu} (z)$ gives rise to an analytic continuation of $H^{\pm}_{\unu} (x)$ onto the half-plane $\BH^{\pm} = \left\{ z \in \BC \smallsetminus \{0\} : 0 \leq \pm \arg z \leq \pi \right\} $. In view  of Proposition \ref{prop: J (x ; sigma, lambda) = J nu (x; sigma)}, one may define $H^{\pm} (z; \ulambda) = H^{\pm}_{\unu} (z)$ and regard it as the analytic continuation of  $H^{\pm} (x; \ulambda)$ from $\BR _+$ onto $\BH^{\pm}$. Furthermore, with slight modifications of the arguments above, where the phase function $f $ is now chosen to be $\pm e^{i \omega} (\theta - d - 1)$ in the application of Proposition \ref{prop: Stationary phase},  the domain of validity for the asymptotic expansions in Theorem \ref{thm: asymptotic expansion} may be extended from $\BR_+$ onto $  \BH^{\pm}$. For example, we have 
\begin{equation}\label{5eq: asymptotic expansion 1}
\begin{split}
H^{\pm}  (z; \ulambda) = n^{- \frac 1 2}  &(\pm 2 \pi i)^{ \frac { n-1} 2} e^{ \pm i n z} z^{ - \frac { n-1} 2} \\
& \lp \sum_{m=0}^{M-1} ( \pm i )^{ - m} B_{m} (\ulambda) z^{- m} + O_{\,\mathfrak R,\, M,\, n} \left( \mathfrak C^{2 M} |z|^{-M + \frac { n-1} 2} \right) \rp,
\end{split}
\end{equation}
for all $z \in \BH^\pm$ such that $ |z| \geq \fC$.

Obviously, the above method of obtaining the analytic continuation of $H^{\pm}_{\unu} $ does not apply to $K$-Bessel functions.

\vskip 5 pt

\subsubsection{On the Implied Constants of Estimates} All the implied constants that occur in this section are of exponential dependence on the real parts of the indices. 
If one considers  the $d$-th symmetric lift of a holomorphic Hecke cusp form of weight $k$, the estimates are particularly awful in the $k$ aspect. 


In \S \ref{sec: Recurrence relations and differential equations of the Bessel functions} and \S \ref{sec: Bessel equations}, we shall further explore the theory of Bessel functions from the perspective of differential equations.  Consequently,  if the argument is sufficiently large, then all the estimates in this section can be improved so that the dependence on the index can be completely eliminated.

\vskip 5 pt

\subsubsection{On the Coefficients in the Asymptotics} \label{sec: remark on the coefficients in the asymptotics}
One feature of the method of stationary phase is the  explicit formula of the coefficients in the asymptotic expansion in terms of  certain partial differential operators.  In the present case of $H^{\pm}  (x; \ulambda) = H^{\pm}_{\unu} (x)$,
\eqref{5eq: B mj} provides an explicit formula of $B_{m} (\ulambda) = B_{m, 0} (\unu)$. 
To compute $\EuScript L^{m+r} \left( G^r p_{\unu} \right) (\ut_0)$ appearing in \eqref{5eq: B mj},  we observe that the function $G$ defined in (\ref{5eq: G(t)}) does not only vanish of third order at $\ut_0$. Actually, $\partial^{\ualpha} G (\ut_0) $ vanishes except for $ \ualpha = (0, ..., 0, \alpha , 0 ..., 0)$, with $\alpha \geq 3$. In the exceptional case we have $\partial^{\ualpha} G (\ut_0) = (- )^{\alpha} \alpha !$. However, the resulting expression is  considerably complicated. 

When $d = 1$, by the definition \eqref{5eq: differential operator D}, we have $\EuScript L = (d/dt)^2$. For $2 m \geq r \geq 1$,
\begin{align*}
& \lp d/ dt \rp^{2m+2r} \lp G^r p_{\nu}\rp (1) \\
= &\, (2m+2r)! \sum_{\alpha = 0}^{2m-r} \left| \left\{ (\alpha_1, ..., \alpha_r) : \sum_{q = 1}^r \alpha_q = 2m+2r- \alpha, \alpha_q \geq 3 \right \} \right| \frac { (- )^{\alpha}[ \nu - 1 ]_\alpha} {\alpha !} \\
= &\, (2m+2r)! \sum_{\alpha = 0}^{2m-r} {2m-\alpha-1 \choose r-1} \frac {  (1 - \nu)_{\alpha} } {\alpha !}.
\end{align*}
Therefore \eqref{5eq: B mj} yields 
\begin{equation*}
\begin{split}
B_{m, 0} (\nu) = \lp -\frac {1} {4}\rp^m \lp \frac { (1 - \nu)_{2m}} {m!} + \sum_{r=1}^{2m} \frac {(-)^r (2m+2r)!} {4^r (m+r)! r!} \sum_{\alpha = 0}^{2m-r} {2m-\alpha-1 \choose r-1} \frac { (1 - \nu)_{\alpha}} { \alpha !} \rp.
\end{split}
\end{equation*}
However, this expression of  $ B_{m, 0} (\nu)$ is more involved than what is known in the literature. 
Indeed, we have the  asymptotic expansions of $H^{(1)}_{\nu} $ and  $H^{(2)}_{\nu} $ (\cite[7.2 (1, 2)]{Watson})
\begin{align*}
H_{\nu}^{(1, 2)} (x)  \sim \lp \frac 2 {\pi x} \rp^{\frac 1 2} e^{\pm i \lp x - \frac 1 2 \nu \pi - \frac 1 4\pi \rp} \lp \sum_{m=0}^\infty \frac { (\pm)^m \lp \frac 1 2 - \nu \rp_m  \lp \frac 1 2 + \nu \rp_m} { m! (2 i x)^m} \rp,  
\end{align*}
which are deducible from Hankel's integral representations (\cite[6.12 (3, 4)]{Watson}). 
In view of Proposition \ref{prop: Classical Bessel functions} and Theorem \ref{thm: asymptotic expansion}, we have 
\begin{equation*}
B_{m, 0} (\nu) =  \frac {\lp \frac 1 2 - \nu \rp_m  \lp \frac 1 2 + \nu \rp_m } {4^m m!}.
\end{equation*}
Therefore, we deduce the following  combinatorial identity 
\begin{equation}
\begin{split}
  \frac { (- )^m  \lp \frac 1 2 - \nu \rp_m  \lp \frac 1 2 + \nu \rp_m } {m!}  
& =   \\
  \frac { (1 - \nu)_{2m}} {m!} & + \sum_{r=1}^{2m} \frac {(- )^r  (2m+2r)!} {4^r (m+r)! r!} \sum_{\alpha = 0}^{2m-r}{2m-\alpha-1 \choose r-1} \frac {(1 - \nu)_{\alpha}} {\alpha !}.
\end{split}
\end{equation}
It seems however hard to find an elementary proof of this identity\footnote{Added in proof: This problem was solved by Zhaolin Li, an undergraduate student at Zhejiang University.}.

\section{Recurrence Formulae and the Differential Equations for Bessel Functions}
\label{sec: Recurrence relations and differential equations of the Bessel functions}

Making use of certain recurrence formulae for  $J_{\unu} (x; \usigma)$, 
we shall derive the differential equation satisfied by $J (x; \usigma, \ulambda)$.  

\subsection{Recurrence Formulae}

Applying the formal partial integrations of either the first or the second kind and the differentiation under the integral sign on the formal integral expression of $J_{\unu} (x; \usigma)$ in (\ref{3eq: definition J nu (x; sigma)}), one obtains the recurrence formulae
\begin{equation}\label{6eq: recurrence from partial integration}
{\nu_l     } ({i x})\-  J_{\unu} (x; \usigma) = \varsigma_l      J_{\unu  - \ue_l     } (x; \usigma) - \varsigma_{d + 1} J_{\unu + \ue^d} (x; \usigma)
\end{equation}
for $l      = 1, ..., d$, and
\begin{equation}\label{6eq: reccurence from differentiation}
J'_{\unu} (x; \usigma) = \varsigma_{d + 1} i J_{\unu + \ue^d} (x; \usigma) + i \sum_{l      = 1}^{d} \varsigma_l      J_{\unu - \ue_l     } (x; \usigma).
\end{equation}
It is easy to verify (\ref{6eq: recurrence from partial integration}) and (\ref{6eq: reccurence from differentiation}) using the rigorous interpretation of $J_{\unu} (x; \usigma)$ established in \S \ref{sec: Definition of BJ nu (x; sigma)}.
Alternatively, one may prove these by the Mellin-Barnes integral representation \eqref{3eq: definition of J (x; sigma)} of  $J (x; \usigma, \ulambda)$. Moreover, using (\ref{6eq: recurrence from partial integration}), one may reformulate  (\ref{6eq: reccurence from differentiation})  as below,
\begin{equation}
J'_{\unu} (x; \usigma) \label{6eq: reccurence second form}
= \varsigma_{d + 1} i (d + 1) J_{\unu + \ue^d} (x; \usigma) + \frac {\sum_{l      = 1}^{d} \nu_l     } { x} J_{\unu} (x; \usigma).
\end{equation}

\subsection{The Differential Equations}


\begin{lem}\label{lem: J (j) nu (x; sigma) derivatives}
	Define $\ue^{l     } = (\underbrace {1,..., 1}_{l      }, 0 ..., 0)$,  $l      = 1, ..., d$, and denote $\ue^0 = \ue^{d+1} = (0, ..., 0)$ for convenience. Let $\nu_{d+1} = 0$.
	
	{ \rm (1).} For $l      = 1, ..., d+1$ we have
	\begin{equation}\label{6eq: J' nu+e ell (x; sigma)}
	J'_{\unu + \ue^{l     }} (x; \usigma) = \varsigma_{l     } i (d + 1) J_{\unu + \ue^{l     -1}} (x; \usigma) - \frac {\varLambda_{d-l     +1} (\unu) + d - l      +1} { x} J_{\unu + \ue^{l     }} (x; \usigma),
	\end{equation} 
	with
	\begin{equation*} 
	\varLambda_{m } (\unu) = - \sum_{k = 1}^{d}\nu_{k} + (d+1) \nu_{d-m+1}, \hskip 10 pt   m = 0, ..., d.
	\end{equation*}
	
	{ \rm (2).} For $0 \leq j \leq k \leq d+1$ define
	\begin{equation*}
	U_{k, j} (\unu) = 
	\begin{cases}
	1, & \text { if } j=k,\\
	- \lp \varLambda_{j} (\unu) + k - 1 \rp U_{k - 1, j} (\unu) + U_{k-1, j-1} (\unu), & \text { if } 0 \leq j \leq k-1,
	\end{cases}
	\end{equation*}
	with the notation $U_{k, -1} (\unu) = 0$, and
	\begin{equation*}
	S_0 (\usigma) = +, \hskip 10 pt S_j (\usigma) = \prod_{m = 0}^{j - 1}\varsigma_{d-m+1}\ \text{ for } j = 1, ..., d+1.
	\end{equation*}
	Then
	\begin{equation} \label{6eq: J (j) nu (x; sigma) derivatives}
	J^{(k)}_{\unu}  (x; \usigma) = \sum_{j = 0}^k  S_j (\usigma) (i (d+1))^j U_{k, j} (\unu) x^{j - k} J_{\unu + \ue^{d - j + 1}}  (x; \usigma).
	\end{equation}
\end{lem}

\begin{proof}
	By (\ref{6eq: reccurence second form}) and (\ref{6eq: recurrence from partial integration}),   
	\begin{equation*}
	\begin{split}
	& J'_{\unu + \ue^{l     }} (x; \usigma) = \varsigma_{d + 1} i (d + 1) J_{\unu + \ue^{l     } + \ue^d} (x; \usigma) + \frac {\sum_{k = 1}^{d}\nu_{k} + l     } { x} J_{\unu + \ue^{l     }} (x; \usigma) \\
	= \ &  i (d + 1) \left(- \frac {\nu_{l     } + 1} {i x} J_{\unu + \ue^{l     }} (x; \usigma) + \varsigma_{l     } J_{\unu + \ue^{l      - 1}} (x; \usigma) \right) + \frac {\sum_{k = 1}^{d}\nu_{k} + l     } { x} J_{\unu + \ue^{l     }} (x; \usigma)\\
	= \ &  \varsigma_{l     } i (d + 1)  J_{\unu + \ue^{l     -1}} (x; \usigma) + \frac {\sum_{k = 1}^{d}\nu_{k} - (d+1) \nu_l      + l      - d - 1} { x} J_{\unu + \ue^{l     }} (x; \usigma).
	\end{split}
	\end{equation*}
	This proves \eqref{6eq: J' nu+e ell (x; sigma)}.
	
	(\ref{6eq: J (j) nu (x; sigma) derivatives}) is trivial when $k=0$. Suppose that $k \geq 1$ and that  (\ref{6eq: J (j) nu (x; sigma) derivatives}) is already proven for $k-1$. The inductive hypothesis and  (\ref{6eq: J' nu+e ell (x; sigma)}) imply
	\begin{align*}
	J^{(k)}_{\unu}  (x; \usigma) = &  \sum_{j=0}^{k-1}  S_j(\usigma) (i(d+1))^j U_{k-1, j} (\unu) x^{j-k+1} \\
	& \lp (j-k+1) x^{-1} J_{\unu + \ue^{d-j+1}} (x; \usigma) \right. \\
	& \left. \quad + \varsigma _{d-j+1} i (d+1) J_{\unu + \ue^{d-j}}  (x; \usigma) - (\varLambda_j(\unu) + j) x\- J_{\unu + \ue^{d-j+1}} (x; \usigma)  \rp\\
	= & - \sum_{j=0}^{k-1} S_j(\usigma) (i(d+1))^j U_{k-1, j} (\unu) (\varLambda_j(\unu) + k-1) x^{j-k} J_{\unu + \ue^{d-j+1}} (x; \usigma)\\
	& + \sum_{j=1}^{k } S_{j-1}(\usigma) \varsigma _{d-j+2} (i(d+1))^j U_{k-1, j-1} (\unu) x^{j-k} J_{\unu + \ue^{d-j+1}} (x; \usigma).
	\end{align*}
	Then (\ref{6eq: J (j) nu (x; sigma) derivatives}) follows from the definitions of $ U_{k, j} (\unu)$ and  $ S_j (\usigma) $.
\end{proof}

Lemma \ref{lem: J (j) nu (x; sigma) derivatives} (2) may be recapitulated as
\begin{equation}\label{6eq: differentials, matrix form}
X_{\unu} (x; \usigma)
= 
D(x)\- U ( \unu) D(x) S(\usigma) Y_{\unu}  (x; \usigma),
\end{equation}
where $X_{\unu} (x; \usigma) = \big( J^{(k)}_{\unu} (x; \usigma) \big)_{k=0}^{ d+1 }$ and $Y_{\unu}  (x; \usigma) = \big(  J_{\unu + \ue^{d - j + 1}}  (x; \usigma) \big)_{j=0}^{d+1}$ are column vectors of functions, 
$S(\usigma) = \mathrm{diag} \left(S_j(\usigma)(i(d+1))^j \right)_{j=0}^{d+1}$ and $D(x) = \mathrm{diag} \left(x^j \right)_{j=0}^{d+1}$ are diagonal matrices, and $U ( \unu) $ is the lower triangular unipotent $(d+2) \times (d+2)$ matrix whose $(k+1, j+1)$-th entry is equal to $U_{k, j} (\unu)$. 
The inverse matrix $U ( \unu)\-$ is again a lower triangular unipotent matrix. Let  $ V_{k, j} (\unu) $ denote the $(k+1, j+1)$-th entry of  $U ( \unu)\-$. It is evident that $ V_{k, j} (\unu) $ is a polynomial in $\unu$ of  degree $k-j$ and integral coefficients.

Observe that $J_{\unu + \ue^{d+1}} (x; \usigma) = J_{\unu + \ue^0} (x; \usigma) = J_{\unu} (x; \usigma)$. Therefore,  (\ref{6eq: differentials, matrix form}) implies that $J_{\unu} (x; \usigma)$ satisfies the following linear differential equation of order $d+1$
\begin{equation} \label{6eq: differential equation of J}
\begin{split}
\sum_{j = 1}^{d+1} V_{d+1, j} (\unu)   x^{j - d - 1} w^{(j)}  
+ \left( V_{d+1, 0} (\unu) x^{- d - 1} - S_{d+1} (\usigma) (i(d+1))^{d+1} \right) w = 0.
\end{split}
\end{equation}

\subsection{Calculations of the Coefficients in the Differential Equations}\label{sec: Calculation of the coefficients}

\begin{defn}\label{6defn: coefficients of Bessel equation}
	Let $\uLambda = \{\varLambda_m \}_{m=0}^\infty$ be a sequence of complex numbers. 
	
	{ \rm (1).} For $k, j \geq -1$ inductively define a double sequence of polynomials $ U_{k, j} (\uLambda) $ in $\uLambda$ by the initial conditions
	\begin{equation*} 
	U_{-1, -1}(\uLambda) = 1, \hskip 10 pt U_{k, - 1} (\uLambda) =  U_{-1, j} (\uLambda) = 0 \ \text{ if } k , j \geq 0,
	\end{equation*} 
	and the recurrence relation
	\begin{equation}\label{6eq: definition of U k j (mu)}
	U_{k, j} (\uLambda) = - \lp \varLambda_j + k - 1\rp U_{k-1, j} (\uLambda)  + U_{k-1, j-1} (\uLambda), \hskip 10 pt k, j \geq 0.
	\end{equation}
	
	{ \rm (2).} For $j, m \geq -1$ with $(j, m) \neq (-1, -1)$ define a double sequence of integers $A_{j, m}$ by the initial conditions
	\begin{equation*}
	A_{-1, 0} = 1, \hskip 10 pt A_{-1, m} = A_{j, -1} = 0 \  \text{ if } m \geq 1, j \geq 0,
	\end{equation*}
	and the recurrence relation
	\begin{equation}\label{6eq: definition of a k m}
	A_{j, m}  = j A_{j, m-1} + A_{j-1, m}, \hskip 10 pt j,  m \geq 0.
	\end{equation}

	{ \rm (3).} For $ k, m \geq 0 $ we define $\sigma_{k, m} (\uLambda)$ to be the elementary symmetric polynomial in $\varLambda_0, ..., \varLambda_{k}$ of degree $m$, with the convention that
	$\sigma_{k, m} (\uLambda) = 0 $ if $m \geq k+2$. Moreover, we denote
	\begin{equation*}
	\sigma_{-1, 0} (\uLambda) = 1, \hskip 10 pt \sigma_{k, -1} (\uLambda) = \sigma_{-1, m} (\uLambda) = 0 \ \text{ if } k \geq -1, m \geq 1.
	\end{equation*} Observe that, with the above notations as initial conditions, $\sigma_{k, m} (\uLambda)$ may also be inductively defined by the recurrence relation
	\begin{equation}\label{6eq: recurrence relation of elementary symmetric polynomials}
	\sigma_{k, m}(\uLambda) = \varLambda_k \sigma_{k-1, m-1} (\uLambda) + \sigma_{k-1, m} (\uLambda),  \hskip 10 pt k,  m \geq 0.
	\end{equation}
	
	{ \rm (4).} For $k \geq 0, j \geq -1$ define 
	\begin{equation}\label{6eq: definition of V k j (mu)}
	V_{k, j} (\uLambda) = 
	\begin{cases}
	0,  & \text { if } j > k,\\
	\ds \sum_{m=0}^{k-j} A_{j, k-j-m} \sigma_{k-1, m} (\uLambda),  & \text{ if } k \geq j.
	\end{cases}
	\end{equation}
\end{defn}

\begin{lem} \label{lem: simple facts on U, a, V} Let notations be as above.
	
	{ \rm (1).} $U_{k, j}(\uLambda)$ is a polynomial in $\varLambda_0, ..., \varLambda_j$. $U_{k, j} (\uLambda) = 0$ if $j > k $, and $U_{k, k} (\uLambda) = 1$. $U_{k, 0}(\uLambda) = [ - \varLambda_0]_k$ for $k \geq 0$.
	
	{ \rm (2).} $A_{j, 0} = 1$, and $A_{j, 1} = \frac 1 2 {j(j+1)} $.
	
	{ \rm (3).} $V_{k, j} (\uLambda) $ is a symmetric polynomial in $\varLambda_0, ..., \varLambda_{k-1}$. $V_{k, k} (\uLambda) = 1$. $V_{k, -1} (\uLambda) = 0$ and $V_{k, k-1} (\uLambda) = \sigma_{k-1, 1} (\uLambda) + \frac 12 {k(k-1)} $ for $k \geq 0$.
	
	{ \rm (4).} $V_{k, j} (\uLambda) $ satisfies the following recurrence relation
	\begin{equation}\label{6eq: Recurrence relation of V k j}
	V_{k, j} (\uLambda) =  ( \varLambda_{k-1} + j) V_{k-1, j} (\uLambda) + V_{k-1, j-1} (\uLambda), \hskip 10 pt k \geq 1, j \geq 0.
	\end{equation}
\end{lem}

\begin{proof}
	(1-3) are  evident from the definitions.
	
	(4). (\ref{6eq: Recurrence relation of V k j}) is obvious if $j \geq k$. If $k > j$, then the recurrence relations (\ref{6eq: recurrence relation of elementary symmetric polynomials}, \ref{6eq: definition of a k m}) for $\sigma_{k, m} (\uLambda)$ and $A_{j, m}$, in conjunction with the definition (\ref{6eq: definition of V k j (mu)}) of $V_{k, j} (\uLambda)$,  yield
	\begin{align*}
	V_{k, j} (\uLambda) = & \sum_{m=0}^{k - j } A_{j, k - j - m } \sigma_{k - 1, m} (\uLambda)\\
	= &\, \varLambda_{k - 1} \sum_{m = 1}^{k - j }  A_{j, k - j - m } \sigma_{k-2, m-1} (\uLambda) + \sum_{m = 0}^{k - j }   A_{j, k - j - m } \sigma_{k-2, m} (\uLambda)\\
	= &\, \varLambda_{k - 1} \sum_{m = 0}^{k - j - 1}  A_{j, k - j - m - 1} \sigma_{k-2, m} (\uLambda) \\
	& + j \sum_{m = 0}^{k - j - 1} A_{j, k - j - m - 1} \sigma_{k-2, m} (\uLambda) + \sum_{m = 0}^{k - j } A_{j-1, k - j - m } \sigma_{k-2, m} (\uLambda)\\
	= &\,  ( \varLambda_{k-1} + j) V_{k-1, j} (\uLambda) + V_{k-1, j-1} (\uLambda).
	\end{align*}
\end{proof}

\begin{lem}\label{lem: Orthogonality}
	For $k \geq 0$ and $j \geq - 1$ such that $k \geq j$, we have
	\begin{equation}\label{6eq: Orthogonality}
	\sum_{l      = j}^k U_{k, l     } (\uLambda) V_{l     , j} (\uLambda) = \delta_{k, j},
	\end{equation}
	where $ \delta_{k, j}$ denotes Kronecker's delta symbol.
	Consequently, 
	\begin{equation}\label{6eq: Orthogonality, 1}
	\sum_{l      = j}^k V_{k, l     } (\uLambda) U_{l     , j} (\uLambda) = \delta_{k, j}.
	\end{equation}
\end{lem}
\begin{proof}
	(\ref{6eq: Orthogonality}) is obvious if either $k = j$ or $j = -1$. In the proof we may therefore assume that $k - 1 \geq j \geq 0$ and that (\ref{6eq: Orthogonality}) is already proven for smaller values of $k - j$ as well as for smaller values of $j$ and the same $k-j$.
	
	By the recurrence relations (\ref{6eq: definition of U k j (mu)}, \ref{6eq: Recurrence relation of V k j}) for $U_{k, j} (\uLambda)$ and  $V_{ k, j} (\uLambda)$ and the induction hypothesis,
	\begin{align*}
	& \sum_{l      = j}^k U_{k, l     } (\uLambda) V_{l     , j} (\uLambda) \\
	=  & - \sum_{l     =j}^{k-1} ( k-1 +  \varLambda_{l     }) U_{k-1, l     } (\uLambda) V_{l     , j} (\uLambda) + \sum_{l     =j}^{k } U_{k-1, l     -1} (\uLambda) V_{l     , j} (\uLambda)\\
	= & - ( k-1) \delta_{k-1, j} - \sum_{l     =j}^{k-1} \varLambda_{l     } U_{k-1, l     }  (\uLambda) V_{l     , j} (\uLambda) + \sum_{l     =j +1}^{k } \varLambda_{l     -1} U_{k-1, l     -1} (\uLambda) V_{l     -1, j} (\uLambda) \\
	& + j \sum_{l     =j + 1}^{k } U_{k-1, l     -1} (\uLambda) V_{l     -1, j} (\uLambda) +  \sum_{l     =j }^{k } U_{k-1, l     -1} (\uLambda) V_{l     -1, j-1} (\uLambda) \\
	= & -( k-1) \delta_{k-1, j} + 0 + j \delta_{k-1, j} + \delta_{k-1, j-1} = 0.
	\end{align*}
	This completes the proof of (\ref{6eq: Orthogonality}).
\end{proof}

Finally, we have the following explicit formulae for $A_{j, m}$.
\begin{lem}\label{lem: formula of a jm}
	We have $A_{0,0} = 1$, $A_{0, m} = 0$ if $m \geq 1$, and
	\begin{equation}\label{6eq: formula for a j m}
	A_{j, m} = \sum_{r=1}^j \frac {(- )^{j-r} r^{m+j}}{r! (j-r)!} \  \text { if }  j \geq 1, m \geq 0.
	\end{equation}
\end{lem}
\begin{proof}
	It is easily seen that $A_{0,0} = 1$ and $A_{0, m} = 0$ if $m \geq 1$.
	
	It is straightforward to verify that the sequence given by (\ref{6eq: formula for a j m}) satisfies the recurrence relation (\ref{6eq: definition of a k m}), so it is left to show that \eqref{6eq: formula for a j m} holds true for $m = 0$. Initially, $A_{j, 0} = 1$, and hence one must verify
	\begin{equation*}
	\sum_{r=1}^j \frac {(- )^{j-r}
		r^{j}}{r! (j-r)!} = 1.
	\end{equation*}
	This however follows from considering all the identities obtained by differentiating the following binomial identity up to $j$ times and then evaluating at $x = 1$,
	\begin{equation*}
	(x - 1)^j - (-1)^j = j! \sum_{r=1}^j \frac {(-1)^{j-r}} {r! (j-r)!} x^r.
	\end{equation*}
\end{proof}

\subsection {Conclusion}

We first observe that, when $0\leq j \leq k \leq d+1$, both $U_{k, j} (\uLambda)$ and $V_{k, j} (\uLambda)$ are polynomials in $\varLambda_0, ..., \varLambda_{d}$  according to Lemma \ref{lem: simple facts on U, a, V} (1, 3). 
If one puts $\varLambda_{m} = \varLambda_m (\unu)$ for $m = 0, ..., d$, then   $U_{k, j} (\unu) = U_{k, j} (\uLambda)$. It follows from Lemma \ref{lem: Orthogonality} that  $V_{k, j} (\unu) = V_{k, j} (\uLambda)$. Moreover, the relations $\nu_{l     } = \lambda_{l     } - \lambda_{d+1}$, $l      = 1, ..., d$, along with the assumption $\sum_{l      = 1}^{d+1} \lambda_l      = 0$, yields
\begin{equation*} 
\varLambda_{m} (\unu) = (d+1)  \lambda_{d- m+1 }.
\end{equation*}
Now we can reformulate  \eqref{6eq: differential equation of J} in the following theorem.
\begin{thm}\label{thm: Bessel equations}
	The Bessel function $J (x; \usigma, \ulambda)$ satisfies the following linear differential equation of order $d+1$
	\begin{equation} \label{6eq: differential equation of J 2}
	\sum_{j = 1}^{d+1} V_{d+1, j} (\ulambda) x^{j} w^{(j)} + \left( V_{d+1, 0} (\ulambda) - S_{d+1} (\usigma) (i(d+1))^{d+1}  x^{ d + 1} \right) w  = 0,
	\end{equation}
	where
	\begin{equation*} 
	S_{d+1} (\usigma) = \prod_{l      = 1}^{d+1} \varsigma_{l     }, \hskip 10 pt V_{d+1, j} (\ulambda) = \sum_{m=0}^{d-j+1} A_{j, d-j-m+1} (d+1)^m \sigma_{m} (\ulambda),
	\end{equation*}
	$\sigma_m (\ulambda)$ denotes the elementary symmetric polynomial in $\ulambda$ of degree $m$, with $\sigma_1 (\ulambda) = 0$, and $A_{j, m}$ is recurrently defined in Definition {\rm \ref{6defn: coefficients of Bessel equation} (3)} and explicitly given in Lemma {\rm \ref{lem: formula of a jm}}.  We shall call the equation \eqref{6eq: differential equation of J 2} a Bessel equation of index $\ulambda$, or simply a Bessel equation if the index $\ulambda$ is given.
\end{thm}


For a given index  $\ulambda$, (\ref{6eq: differential equation of J 2}) only provides two Bessel equations.  The sign $S_{d+1} (\usigma)$ determines which one of the two   Bessel equations   a Bessel function  $ J (x; \usigma, \ulambda)$ satisfies.
\begin{defn}
	We call $S_{d+1} (\usigma) = \prod_{l      = 1}^{d+1} \varsigma_{l     }$ the sign of the Bessel function $ J (x; \usigma, \ulambda)$ as well as   the Bessel equation satisfied by $ J (x; \usigma, \ulambda)$. 
\end{defn}

Finally, we collect some simple facts on $V_{d+1, j} (\ulambda)$ in the following lemma, which will play   important roles later in the study of   Bessel equations. See \eqref{6eq: Orthogonality, 1} in Lemma \ref{lem: Orthogonality} and Lemma \ref{lem: simple facts on U, a, V} (3).

\begin{lem}\label{lem: simple facts on V (lambda)} We have
	
	{ \rm (1).}
	$\sum_{j = 0}^{d+1} V_{d+1, j} (\ulambda) [ - (d+1) \lambda_{d+1}]_{j} = 0.$
	
	{ \rm (2).} $V_{d+1, d} (\ulambda) = \frac 12 {d (d+1)} $.
	
\end{lem}

\begin{rem}
	If we define
	\begin{equation*}
	\boldsymbol J (x; \usigma, \ulambda) = J \lp(d+1)\-x; \usigma, (d+1)\- \ulambda \rp,
	\end{equation*} 
	then this normalized Bessel function satisfies a differential equation with coefficients free of powers of $(d+1)$, that is,
	\begin{equation*} 
	\sum_{j = 1}^{d+1} \boldsymbol V_{d+1, j} (\ulambda) x^{j} w^{(j)} + \left(\boldsymbol V_{d+1, 0} (\ulambda) - S_{d+1} (\usigma) i^{d+1}  x^{ d + 1} \right) w  = 0,
	\end{equation*}
	with
	\begin{equation*} 
	\boldsymbol V_{d+1, j} (\ulambda) = \sum_{m=0}^{d-j+1} A_{j, d-j-m+1} \sigma_{m} (\ulambda).
	\end{equation*}
	In particular, if $d=1$, $\ulambda = (\lambda, -\lambda)$, then the two normalized Bessel equations are
	\begin{equation*} 
	x^2 \frac {d^2 w} {dx^2} + x \frac {d w} {dx } + \left( - \lambda^2 \pm  x^{2} \right) w = 0.
	\end{equation*}
	These are exactly the Bessel equation and the modified Bessel equation of index $\lambda$. 
\end{rem}


\section{Bessel Equations} \label{sec: Bessel equations}

The theory of linear ordinary differential equations with analytic coefficients\footnote{\cite[Chapter 4, 5]{Coddington-Levinson} and \cite[Chapter II-V]{Wasow} are the main references that we follow, and the reader is referred to these books for terminologies and definitions.} will be employed in this section to study Bessel equations. 

\vskip 5 pt

Subsequently, we shall use $z$ instead of $x$ to indicate complex variable.
For $\varsigma \in \{+, - \}$ and $\ulambda \in \BL^{n-1}$, we introduce the Bessel differential operator
\begin{equation}
\nabla_{\varsigma, \ulambda} = \sum_{j = 1}^{n} V_{n, j} (\ulambda) z^{j} \frac {d^j} {d z^j} + V_{n, 0} (\ulambda) - \varsigma (i n)^{n}  z^{n} .
\end{equation}
The Bessel equation of index $\ulambda$ and sign $\varsigma$  may be written as
\begin{equation} \label{7eq: differential equation, variable z}
\nabla_{\varsigma, \ulambda} (w)  = 0.
\end{equation}
Its corresponding system of differential equations is given by
\begin{equation}\label{7eq: differential equation, matrix form}
w' = B (z; \varsigma, \ulambda) w,
\end{equation}
with
\begin{equation*} 
\begin{split}
B   (z; \varsigma, \ulambda)  = 
\begin{pmatrix}
0 & 1 & 0 & \cdots & 0\\
0 & 0 & 1 & \cdots & 0\\
\vdots & \vdots & \vdots & \ddots & \vdots \\
0 & 0 &  \cdots & \cdots & 1 \\
- V_{n, 0} (\ulambda) z^{-n} + \varsigma (in)^n & -  V_{n, 1} (\ulambda) z^{-n+1} & \cdots & \cdots & -  V_{n, n-1} (\ulambda) z^{-1}
\end{pmatrix}.
\end{split}
\end{equation*}

We shall study Bessel equations on the Riemann surface $\BU$ associated with $\log z$, that is, the universal cover of $\BC \smallsetminus \{ 0 \}$. Each element in $\BU$ is represented by a pair $(x, \omega)$ with modulus $x \in \BR _+$ and argument $\omega \in \BR$, and will be denoted by $z = x e^{i\omega} = e^{\log x + i \omega}$ with some ambiguity. Conventionally, define $z^\lambda = e^{\lambda \log z }$ for $z \in \BU, \lambda \in \BC$, $\overline z = e^{- \log z}$, and moreover let $1 = e^{ 0}$, $-1 = e^{\pi i}$ and $\pm i = e^{\pm \frac 1 2 \pi i}$.

First of all, since Bessel equations are nonsingular on $\BU$, all the solutions of  Bessel equations are analytic on $\BU$.

Each Bessel equation has only two singularities at $z = 0$ and $z = \infty$. According to the classification of singularities, $0$ is a \textit{regular singularity}, so the Frobenius method gives rise to solutions of Bessel equations developed in series of ascending powers of $z$, or possibly logarithmic sums of this kind of series, whereas $\infty$ is an \textit{irregular singularity of rank one}, and therefore one may find certain formal solutions that are the asymptotic expansions of some actual solutions of Bessel equations.
Accordingly, there are two kinds of Bessel functions arising as solutions of Bessel equations. 

Finally, a simple but important observation is as follows. 

\begin{lem}\label{lem: phi (e pi i m/n z) is also a solution}
	Let $\varsigma \in \{+, - \}$ and $a$ be an integer. If $\varphi (z)$ is a solution of the Bessel equation of sign $\varsigma$, then
	$\varphi \big(e^{\pi i \frac a n} z\big)$ satisfies the  Bessel equation of sign $(-)^a \varsigma$.
\end{lem}

Variants of Lemma \ref{lem: phi (e pi i m/n z) is also a solution}, Lemma \ref{lem: J ell (z; sigma; lambda) relations},  \ref{lem: J(z ; usigma; lambda) relations to H +} and \ref{lem: J(z; lambda; xi) relations to J(xi z; lambda; 1)}, will play important roles later in \S \ref{sec: H-Bessel functions and K-Bessel functions revisited} when we establish the connection formulae for various kinds of Bessel functions. 

\subsection{Bessel Functions of the First Kind}\label{sec: Bessel functions of the first kind}

The indicial equation associated with $\nabla_{\varsigma, \ulambda}$ is given as below,
\begin{equation*}
\sum_{j = 0}^n [\rho]_{j} V_{n, j} (\ulambda) = 0.
\end{equation*}
Let $P_{\ulambda} (\rho) $ denote the polynomial on the left of this equation. Lemma \ref{lem: simple facts on V (lambda)} (1) along with the symmetry of  $V_{n, j} (\ulambda)$ yields the following identity,
\begin{equation*}
\sum_{j = 0}^n [- n \lambda_{l     }]_{j} V_{n, j} (\ulambda) = 0,
\end{equation*}
for each $l      = 1, ..., n$. Therefore,
\begin{equation*} 
P_{\ulambda} (\rho) = \prod_{l      = 1}^n (\rho + n \lambda_{l     }).
\end{equation*}

Consider the formal series
\begin{equation*}
\sum_{m=0}^\infty c_m z^{\, \rho + m},
\end{equation*}
where the index $\rho$ and the coefficients $c_m$, with $c_0 \neq 0$, are to be determined.
It is easy to see that
\begin{equation*}
\nabla_{\varsigma, \ulambda} \sum_{m=0}^\infty c_m z^{\,\rho + m} = \sum_{m=0}^\infty c_m P_{\ulambda} (\rho + m)  z^{\,\rho + m} - \varsigma (in)^n \sum_{m=0}^\infty c_m z^{\,\rho + m + n}.
\end{equation*}
If the following equations are satisfied
\begin{equation}\label{7eq: series solution, system of equations}
\begin{split}
& c_m P_{\ulambda} (\rho + m) =0, \hskip 10 pt n > m  \geq 1,\\
& c_m  P_{\ulambda} (\rho + m) - \varsigma (in)^n c_{m-n} = 0,  \hskip 10 pt  m \geq n,
\end{split}
\end{equation}
then 
\begin{equation*}
\nabla_{\varsigma, \ulambda} \sum_{m=0}^\infty c_m z^{\,\rho + m} = c_0  P_{\ulambda} (\rho ) z^{\,\rho}.
\end{equation*}
Given $l      \in\{ 1, ..., n \}$. Choose $\rho = - n \lambda_{l     }$ and let $c_0 = \prod_{k = 1}^n \Gamma \lp  { \lambda_{k} - \lambda_{l     }} + 1 \rp \- $. Suppose, for the moment, that no two components of $n \ulambda$ differ by an integer. Then  $P_{\ulambda} (- n \lambda_l      + m) \neq 0$   for any $m \geq 1$ and  $c_0\neq 0$, and hence the system of equations (\ref{7eq: series solution, system of equations}) is uniquely solvable. It follows that
\begin{equation}\label{7eq: series solution J ell (z; sigma)}
\sum_{m=0}^\infty \frac { (\varsigma i^n)^m  z^{ n (- \lambda_{l      } + m)} } { \prod_{k = 1}^n \Gamma \lp  { \lambda_{ k } - \lambda_{l     }}  + m + 1 \rp} 
\end{equation}
is a formal solution of the differential equation (\ref{7eq: differential equation, variable z}).

Now suppose that $\ulambda \in \BL^{n-1}$ is unrestricted. The series in (\ref{7eq: series solution J ell (z; sigma)}) is absolutely convergent, compactly convergent with respect to $\ulambda $, and hence gives rise to an analytic function of $z$ on  
the Riemann surface $\BU$, as well as an analytic function of $\ulambda$. 
We denote by $J_{l     } (z; \varsigma, \ulambda)$ the analytic function given by the series (\ref{7eq: series solution J ell (z; sigma)}) and call it a  {Bessel function of the first kind}. It is evident that  $J_{l     } (z; \varsigma, \ulambda)$ is an actual solution of (\ref{7eq: differential equation, variable z}).

\begin{defn}\label{defn: D, generic lambda}
	Let $\BD^{n-1}$ denote the set of $\ulambda \in \BL^{n-1}$ such that no two components of $\ulambda$ differ by an integer. We call an index $\ulambda $ { generic} if $\ulambda \in \BD^{n-1}$.
\end{defn}

When $\ulambda \in \BD^{n-1}$,  all the $J_{l     } (z; \varsigma, \ulambda)$ constitute a fundamental set of solutions, since the leading term in the expression (\ref{7eq: series solution J ell (z; sigma)}) of $J_{l     } (z; \varsigma, \ulambda)$ does not vanish.
However, this is no longer the case if $ \ulambda \notin  \BD^{n-1}$. Indeed, if $\lambda_{l      } - \lambda_{k}$ is an integer, $k \neq l     $, then $J_{l      } (z; \varsigma, \ulambda) = (\varsigma i^n)^{\lambda_{l      } - \lambda_{k}} J_{k} (z; \varsigma, \ulambda)$. There are other solutions arising as certain logarithmic sums of series of ascending powers of $z$. Roughly speaking, powers of $\log z$ may occur in some solutions. 
For more details the reader may consult \cite[\S 4.8]{Coddington-Levinson}.

\begin{lem}\label{lem: J ell (z; sigma; lambda) relations}
	Let $a$ be an integer. We have
	\begin{equation*}
	J_{l     } \big(e^{\pi i \frac a n} z; \varsigma, \ulambda \big) = e^{- \pi i a \lambda_l     }J_{l     } (z; (-)^a \varsigma, \ulambda).
	\end{equation*}
\end{lem}

\begin{rem}\label{rem: n=2 J-Bessel function}
	If $n = 2$, then we have the following formulae according to \cite[3.1 (8), 3.7 (2)]{Watson},
	\begin{equation*}
	\begin{split}
	J_1 (z; +, \lambda, - \lambda) = J_{- 2 \lambda} (2 z), \hskip 10 pt &  J_2 (z; +, \lambda, - \lambda) = J_{2 \lambda} (2 z),\\
	J_1 (z; -, \lambda, - \lambda) = I_{- 2 \lambda} (2 z), \hskip 10 pt &  J_2 (z; -, \lambda, - \lambda) = I_{2 \lambda} (2 z).
	\end{split}
	\end{equation*}
\end{rem}


\begin{rem}\label{rem: GHF}
	Recall the definition of the generalized hypergeometric functions given by the series {\rm \cite[\S 4.4]{Watson}}
	\begin{equation*}
	{_pF_q} (\alpha_1, ..., \alpha_p ; \rho_1, ..., \rho_q ; z) = \sum_{m=0}^\infty \frac {(\alpha_1)_m ... (\alpha_p)_m} {m! (\rho_1)_m ... (\rho_q)_m} z^m.
	\end{equation*}
	It is evident that each Bessel function $J_{l     } (z; \varsigma, \ulambda)$ is closely related to a certain generalized hypergeometric function $ {_0F_{n-1}} $ as follows
	\begin{equation*} 
	J_{l     } (z; \varsigma, \ulambda) = \lp \prod_{k \neq l}  \frac {z^{ \lambda_k - \lambda_l }} {\Gamma \lp \lambda_k - \lambda_l + 1 \rp} \rp \cdot
	{_0F_{n-1}} \lp \left\{  \lambda_k - \lambda_l + 1 \right\}_{k \neq l} ; \varsigma i^n z^n \rp.
	\end{equation*}
\end{rem}

\subsection{\texorpdfstring{The Analytic Continuation of $J (x; \usigma, \ulambda)$}{The Analytic Continuation of $J (x; \varsigma, \lambda)$}}\label{sec: Analytic continuation of J (x; usigma; ulambda)}
In \S \ref{sec: Bessel kernel J(x; sigma, lambda)}, from its Barnes integral representation, we have seen that  $J (x; \usigma, \ulambda)$ admits a unique analytic continuation from $\BR_+$ onto $\BU$. For any given $\ulambda \in \BL^{n-1}$, since $J (x; \usigma, \ulambda)$ satisfies the Bessel equation of sign $S_n (\usigma)$, such an analytic continuation  $J (z; \usigma, \ulambda)$  on $\BU$ is also clear from the theory of differential equations. In the following, we shall make  $J (z; \usigma, \ulambda)$ more precise in terms of Bessel functions of the first kind.

Recall the definition
\begin{equation*} 
J (x ; \usigma, \ulambda) = \frac 1  {2 \pi i}\int_{\EuScript C} G(s; \usigma, \ulambda) x^{-n s} d s, \hskip 10 pt x \in \BR _+,
\end{equation*}
where $G (s; \usigma, \ulambda) = \prod_{k = 1}^{n} \Gamma (s - \lambda_{k} ) e \left(\frac 1 4 {\varsigma_{k} (s - \lambda_{k} )}   \right)$ and $\EuScript C$ is a suitable contour. 

Let $\varsigma = S_n (\usigma)$. For the moment, let us assume that $ \ulambda $ is generic.
For $l      = 1, ..., n$ and  $m = 0, 1, 2, ...$, $G (s; \usigma, \ulambda)$ has a simple pole at $\lambda_l      - m$ with residue
\begin{align*}
& (- )^m\frac {1} {m !}  e \lp \frac {\sum_{k = 1}^n \varsigma_{k} ( \lambda_{l     } - \lambda_{k} - m)} 4 \rp  \prod_{k \neq l      } \Gamma (\lambda_{l     } - \lambda_{k} - m) 
=   \\
& \pi^{n-1} e \hskip -1 pt \lp - \frac {\sum_{k = 1}^n \varsigma_{k}\lambda_{k}} 4 \rp \hskip -1 pt e \hskip -1 pt \lp   \frac {\sum_{k = 1}^n \varsigma_{k}\lambda_{l     }} 4 \hskip -1 pt \rp  \hskip -1 pt \lp \prod_{k \neq l      } \frac 1 {\sin \lp \pi (\lambda_{l     } - \lambda_{k} ) \hskip -1 pt \rp} \hskip -2 pt \rp   \hskip -2 pt \frac { (\varsigma i^n )^{m} } {\prod_{k=1}^n  \Gamma (\lambda_{k} - \lambda_{l      } + m + 1)} .
\end{align*}
Here we have used Euler's reflection formula for the Gamma function. 
Applying Cauchy's residue theorem, $J (x ; \usigma, \ulambda)$ is developed into an absolutely convergent series on shifting the contour $\EuScript C$ far left, and, in view of (\ref{7eq: series solution J ell (z; sigma)}), we obtain 
\begin{equation} \label{7eq: J(x; sigma; lambda) as a sum of J ell (x; sigma; lambda)}
\begin{split}
J (z ; \usigma, \ulambda) =  \pi^{n-1} E (\usigma, \ulambda) \sum_{l      = 1}^n E_{l     } (\usigma, \ulambda)   S_{l     } ( \ulambda) J_{l     } (z; \varsigma, \ulambda), \hskip 10 pt z \in \BU,
\end{split}
\end{equation}
with 
\begin{align*}
& E (\usigma, \ulambda) = e \lp - \frac {\sum_{k = 1}^n \varsigma_{k}\lambda_{k}} 4 \rp, \hskip 3 pt E_{l     } (\usigma, \ulambda) = e \lp \frac {\sum_{k = 1}^n \varsigma_{k}\lambda_{l     }} 4 \rp, \hskip 3 pt  S_{l     } ( \ulambda) =     \prod_{k \neq l      } \frac 1 {\sin \lp \pi (\lambda_{l     } - \lambda_{k} ) \rp }  .
\end{align*}
Because of the possible vanishing of $\sin \lp \pi (\lambda_{l     } - \lambda_{k} ) \rp$, the definition of $S_{l     } ( \ulambda) $ may fail  to make sense if $\ulambda $ is not generic.  In order to properly interpret \eqref{7eq: J(x; sigma; lambda) as a sum of J ell (x; sigma; lambda)} in the nongeneric case, one has to pass to the limit, that is, 
\begin{equation} \label{7eq: J(x; sigma; lambda) as a sum of J ell (x; sigma; lambda), limit}
J   (z ; \usigma, \ulambda) = \pi^{n-1} E (\usigma, \ulambda) \cdot  \lim_{\sstyle \boldsymbol \lambda' \ra \ulambda \atop \sstyle  \boldsymbol \lambda' \, \in \, \BD^{n-1}} \sum_{l      = 1}^n E_{l     } (\usigma, \ulambda')  S_l      (\boldsymbol \lambda')   J_{l     } (z; \varsigma, \boldsymbol \lambda' ). 
\end{equation}

We recollect the definitions of $L_{\pm} (\usigma)$ and $n_{\pm} (\usigma)$  introduced in Proposition \ref{prop: special example}.
\begin{defn}\label{defn: signature}
	Let $\usigma \in \{+, - \}^n$. We define $L_\pm (\usigma) = \{ l      : \varsigma_l      = \pm \}$ and $n_\pm (\usigma) = \left| L_\pm (\usigma) \right|$.
	The pair of integers $(n_+ (\usigma), n_- (\usigma) )$ is called the signature of $\usigma$, as well as the signature of the Bessel function $J(z ; \usigma, \ulambda)$.
\end{defn}

With Definition \ref{defn: signature}, we reformulate (\ref{7eq: J(x; sigma; lambda) as a sum of J ell (x; sigma; lambda)}, \ref{7eq: J(x; sigma; lambda) as a sum of J ell (x; sigma; lambda), limit}) in the following lemma.

\begin{lem}\label{lem: J(x; sigma; lambda) as a sum of J ell (x; sigma; lambda)}
	We have
	\begin{equation*} 
	\begin{split}
	J (z ;  \usigma , \ulambda) = \pi^{n-1} E (\usigma, \ulambda)  \sum_{l      = 1}^n  E_{l     } (\usigma, \ulambda) S_{l     } ( \ulambda) J_{l     } \big(z; (-)^{n_- (\usigma)}, \ulambda \big),
	\end{split}
	\end{equation*}
	with  $E (\usigma, \ulambda) = e   \lp  - \frac 1 4 \sum_{k\in L_+ (\usigma)} + \frac 1 4 \sum_{k\in L_- (\usigma)} \lambda_{k} \rp$, $E_{l     } (\usigma, \ulambda) = e \lp \frac 1 4 { (n_+ (\usigma) - n_- (\usigma)) }  \lambda_{l     } \rp$ and $S_{l     } ( \ulambda) =   1/  \prod_{k \neq l      } {\sin \lp \pi (\lambda_{l     } - \lambda_{k} ) \rp } $. When $\ulambda $ is not generic, the right hand side is to be replaced by its limit.
\end{lem}

\begin{rem}\label{7rem: connections to the first kind}
	In view of Proposition {\rm \ref{prop: Classical Bessel functions}} and Remark {\rm \ref{rem: n=2 J-Bessel function}}, Lemma {\rm \ref{lem: J(x; sigma; lambda) as a sum of J ell (x; sigma; lambda)}} is equivalent to the connection formulae in {\rm (\ref{2eq: connection formulae}, \ref{1eq: connection formula K})} {\rm(}see \cite[3.61(5, 6), 3.7 (6)]{Watson}{\rm)}.
\end{rem}

 {
\begin{rem}
	In the  case when  $\ulambda = \frac 1 n \lp  \frac {n-1} 2, ..., - \frac {n-1} 2 \rp$, the formula in  Lemma {\rm \ref{lem: J(x; sigma; lambda) as a sum of J ell (x; sigma; lambda)}} amounts to splitting the Taylor series expansion of $e^{in \xi (\usigma) x }$ in \eqref{2eq: special example} 
	according to the congruence classes of indices modulo $n$. To see this, one requires the multiplicative formula  of the Gamma function \eqref{2eq: multiplication theorem} as well as the trigonometric identity
	\begin{equation*}
	\prod_{k=1}^{n-1} \sin \lp \frac {k \pi} n  \rp = \frac n {2^{n-1}}.
	\end{equation*}
\end{rem}
}

Using Lemma \ref{lem: J ell (z; sigma; lambda) relations} and \ref{lem: J(x; sigma; lambda) as a sum of J ell (x; sigma; lambda)}, one proves the following lemma, which implies that the Bessel function $J(z ; \usigma, \ulambda)$ is determined by its signature up to a constant multiple. 

\begin{lem}\label{lem: J(z ; usigma; lambda) relations to H +}
	Define $H^{\pm} (z; \ulambda) = J(z; \pm, ..., \pm, \ulambda)$. Then
	\begin{equation*}
	J(z ; \usigma, \ulambda) = e \lp \pm \frac {\sum_{l      \in L_ \mp (\usigma)} \lambda_{l      }} 2 \rp H^\pm \Big(e^{\pm \pi i \frac { n_\mp (\usigma)} n} z; \ulambda \Big).
	\end{equation*}
\end{lem}

\subsection{Asymptotics for Bessel Equations and Bessel Functions of the Second Kind}

Subsequently,  we   proceed to investigate the asymptotics at infinity for Bessel equations. 

\begin{defn}
	For  $\varsigma \in \{ +, -\}$ and a positive integer $N$, we let $\BX_{N} (\varsigma) $ denote the set of $N$-th roots of $\varsigma 1$.\footnote{Under certain circumstances, it is suitable to view an element $\xi$ of $\BX_{N} (\varsigma)$  as a point in $\BU$ instead of $\BC \smallsetminus \{0\}$. 
		This however should be clear from the context.}
\end{defn}

Before   delving into our general study, let us first consider the prototypical example given in Proposition \ref{prop: special example}. 

\begin{prop}
	For any $\xi \in \BX_{2n} (+)$, the function $ z^{- \frac {n-1} 2} e^{ i n \xi z}$ is a solution of the Bessel equation of index $ \frac 1 n \lp  \frac {n-1} 2, ..., - \frac {n-1} 2 \rp$ and sign $\xi^{\, n}$. 
\end{prop}

\begin{proof}
	When $ \Im \xi \geq 0$, this can be seen from Proposition \ref{prop: special example} and Theorem \ref{thm: Bessel equations}. For arbitrary $ \xi $, one makes use of Lemma \ref{lem: phi (e pi i m/n z) is also a solution}.
\end{proof}

\subsubsection{Formal Solutions of Bessel Equations at Infinity}\label{sec: formal solutions at infinity}
Following \cite[Chapter 5]{Coddington-Levinson}, we shall consider the system of differential equations (\ref{7eq: differential equation, matrix form}).
We have 
\begin{equation*}
B (\infty; \varsigma, \ulambda) = \begin{pmatrix}
0 & 1 & 0 & \cdots & 0\\
0 & 0 & 1 & \cdots & 0\\
\vdots & \vdots & \vdots & \ddots & \vdots \\
0 & 0 & 0 & \cdots & 1 \\
\varsigma (in)^n & 0 & \cdots & \cdots & 0
\end{pmatrix}.
\end{equation*}
If one let  $\BX_n (\varsigma) = \left \{\xi_1 ,..., \xi_n \right \}$, then the eigenvalues of $B (\infty; \varsigma, \ulambda)$ are $in\xi_1 ,..., in\xi_n$.
The conjugation by the following matrix diagonalizes $B (\infty; \varsigma, \ulambda)$,
\begin{align*}
& T = \frac 1 n \begin{pmatrix}
1 & (in\xi_1)^{-1} &  \cdots & (in\xi_1)^{-n+1}\\
1 & (in\xi_2)^{-1} &  \cdots & (in\xi_2)^{-n+1}\\
\cdots & \cdots &  \cdots & \cdots \\
1 & (in\xi_n)^{-1} &  \cdots & (in\xi_n)^{-n+1}\\
\end{pmatrix},\\
& T\- = \begin{pmatrix}
1 & 1 & \cdots  & 1\\
in\xi_1 & in\xi_2   & \cdots & in\xi_n\\
\cdots & \cdots & \cdots & \cdots \\
(in\xi_1)^{n-1} & (in\xi_2)^{n-1} & \cdots & (in\xi_n)^{n-1}\\
\end{pmatrix}.
\end{align*}
Thus, the substitution $u = T w$ converts the system of differential equations (\ref{7eq: differential equation, matrix form}) into
\begin{equation}\label{7eq: differential equation, matrix form 2}
u' = A(z) u,
\end{equation}
where $A(z) = T B (z; \varsigma, \ulambda) T\-$ is a matrix of polynomials in $z\-$ of degree $n$,
\begin{equation*}
A(z) = \sum_{j=0}^n  z^{-j} A_j,
\end{equation*}
with 
\begin{equation}\label{7eq: A0, Aj}
\begin{split}
A_0 &= \varDelta = \mathrm {diag} \lp i n \xi_{l     } \rp_{l      = 1}^n, \\
A_j &= - i^{\, -j+1} n^{-j} V_{n, n-j} (\ulambda) \lp \xi_{k} \xi_{l     }^{\, -j} \rp_{k, l      = 1}^n, \ j = 1, ..., n.
\end{split}
\end{equation}
It is convenient to put $A_j = 0$ if $j > n$. The dependence on $ \varsigma, \ulambda$ and the ordering of the eigenvalues have been suppressed in our notation for the interest of brevity.

Suppose $\widehat \Phi$ is a formal solution matrix for (\ref{7eq: differential equation, matrix form 2}) of the form
\begin{equation*} 
\widehat \Phi (z) = P(z) z^R e^{Q z},
\end{equation*}
where $P $ is a formal power series in $z\-$,
\begin{equation*}
P(z) = \sum_{m=0}^\infty z^{-m} P_m,
\end{equation*}
and $R$, $Q$ are constant diagonal matrices. Since
\begin{equation*}
\widehat \Phi ' = P' z^R e^{Qz} + z\- P R z^R e^{Qz} + P z^R Q e^{Qz} = \lp P' + z\- P R + P Q \rp  z^R e^{Qz},
\end{equation*}
the differential equation (\ref{7eq: differential equation, matrix form 2}) yields
\begin{equation*}
\sum_{m=0}^\infty z^{-m-1} P_m (R - m I) + \sum_{m=0}^\infty z^{-m } P_m Q = \lp  \sum_{j=0}^\infty z^{-j } A_j \rp \lp \sum_{m=0}^\infty z^{- m } P_m \rp,
\end{equation*}
where $I$ denotes the identity matrix.
Comparing the coefficients of various powers of $z\-$, it follows that $\widehat \Phi$ is a formal solution matrix for (\ref{7eq: differential equation, matrix form 2}) if and only if $R$, $Q$ and $P_m$ satisfy the following equations
\begin{equation}\label{7eq: relations of R Q Pm}
\begin{split}
P_0 Q - \varDelta P_0 &= 0\\
P_{m+1} Q - \varDelta P_{m+1} & = \sum_{j = 1}^{m+1} A_j P_{m-j+1} + P_m (m I - R ), \hskip 10 pt  m \geq 0.
\end{split}
\end{equation}
A solution of the first equation in (\ref{7eq: relations of R Q Pm}) is given by
\begin{equation}\label{7eq: Q=A0, P0=I}
Q = \varDelta, \hskip 10 pt P_0 = I.
\end{equation}
Using (\ref{7eq: Q=A0, P0=I}), the second equation in (\ref{7eq: relations of R Q Pm}) for $m=0$ becomes
\begin{equation}\label{7eq: relation m=0}
P_1 \varDelta - \varDelta P_1 =  A_1 - R.
\end{equation}
Since $\varDelta$ is diagonal, the diagonal entries of the left side of (\ref{7eq: relation m=0}) are zero, and hence the diagonal entries of $R$ must be identical with those of $A_1$. In view of (\ref{7eq: A0, Aj}) and Lemma \ref{lem: simple facts on V (lambda)} (2), we have $$A_1 = - \frac 1 n V_{n, n-1} (\ulambda) \cdot \lp \xi_{k} \xi_{l     }^{-1} \rp_{k, l      = 1}^n = - \frac {n-1} 2 \lp \xi_{k} \xi_{l     }^{-1} \rp_{k, l      = 1}^n,$$ and therefore 
\begin{equation}\label{7eq: R = - (n-2)/2}
R = - \frac {n-1} 2 I.
\end{equation}
Let $ p _{1,\, k l     }$ denote the $(k, l     )$-th entry of $P_1$. It follows from (\ref{7eq: A0, Aj}, \ref{7eq: relation m=0}) that
\begin{equation}\label{7eq: off-diagonal P1}
in (\xi_l      - \xi_k) p _{1,\, k l     } = - \frac {n-1} 2 \xi_k \xi_l     \-, \hskip 10 pt k \neq l     .
\end{equation}
The off-diagonal entries of $P_1$ are uniquely determined by \eqref{7eq: off-diagonal P1}. Therefore, a solution of (\ref{7eq: relation m=0}) is 
\begin{equation}\label{7eq: P1 = D1 + tilde P1}
P_1 = D_1 + P^o_1,
\end{equation}
where $D_1$ is any diagonal matrix and $ P^o_1$ is the matrix with diagonal entries zero and $(k, l     )$-th entry $p _{1,\, k l     }$, $ k \neq l     $. To determine $D_1$, one resorts to the second equation in (\ref{7eq: relations of R Q Pm}) for $m=1$, which, in view of (\ref{7eq: Q=A0, P0=I}, \ref{7eq: R = - (n-2)/2}, \ref{7eq: P1 = D1 + tilde P1}), may be written as
\begin{equation*}
P_2\varDelta - \varDelta P_2 -  \lp A_1 + \frac {n-1} 2 \rp D_1 - \frac {n+1} 2 P^o_1 = A_1 P^o_1 + A_2 + D_1.
\end{equation*}
The matrix on the left side has  zero diagonal entries. It follows that $D_1$ must be equal to the diagonal part of $- A_1 P^o_1 - A_2.$

In general, using (\ref{7eq: Q=A0, P0=I}, \ref{7eq: R = - (n-2)/2}), the second equation in (\ref{7eq: relations of R Q Pm}) may be written as
\begin{equation}\label{7eq: relations of Pm}
P_{m+1} \varDelta - \varDelta P_{m+1} = \sum_{j = 1}^{m+1} A_j P_{m-j+1} + \lp m + \frac {n-1} 2 \rp P_m, \hskip 10 pt  m \geq 0.
\end{equation}
Applying (\ref{7eq: relations of Pm}), an induction on $m$ implies that
\begin{equation*} 
P_{m } = D_m + P^o_m, \hskip 10 pt m\geq 1, 
\end{equation*}
where $D_m$ and $P^o_m$ are inductively defined as follows. Put $D_0 = I$. Let $m D_m$ be the diagonal part of
\begin{equation*}
- \sum_{j=2}^{m+1} A_{j } D_{m-j+1 } - \sum_{j=1}^{m} A_j P^o_{m-j+1}, 
\end{equation*} 
and let $P^o_{m+1}$ be the matrix with diagonal entries zero such that $P^o_{m+1} \varDelta - \varDelta P^o_{m+1}$ is the off-diagonal part of
\begin{equation*}
\begin{split}
\sum_{j=1}^{m+1} A_{j } D_{m-j+1 } + \sum_{j=1}^{m} A_j P^o_{m-j+1} 
+ \lp m + \frac {n-1} 2 \rp  P^o_m.   
\end{split}
\end{equation*}
In this way, an inductive construction of the formal solution matrix of (\ref{7eq: differential equation, matrix form 2}) is completed for the given initial choices $Q = \varDelta$, $P_0 = I$.

With the observations that $A_j$ is of degree $j$ in $\ulambda$ for $j \geq 2$  and that $A_1$ is constant, we may show the following lemma using an inductive argument.

\begin{lem} \label{lem: entries of Pm} The entries of $P_m$ are symmetric polynomials in $\ulambda$. If $m \geq 1$, then the off-diagonal entries of $P_m$ have degree  at most $2m-2$, whereas the degree of each diagonal entry is exactly $2m$.
\end{lem}


The first row of $T\- \widehat \Phi$ constitutes a fundamental system of formal solutions of the Bessel equation  (\ref{7eq: differential equation, variable z}). Some calculations yield the following proposition, where for   derivatives of order higher than $n-1$ the differential equation (\ref{7eq: differential equation, variable z}) is applied.

\begin{prop}\label{prop: formal solution}
	Let $\varsigma \in \{+, - \}$ and $\xi \in \BX_n (\varsigma)$. There exists a unique sequence of symmetric polynomials $B_m (\ulambda; \xi)$ in $\ulambda$ of degree $2 m$ 
	and coefficients depending only on $m$, $\xi$ and $n$, normalized so that $B_0 (\ulambda; \xi) = 1$, such that
	\begin{equation}\label{7eq: formal solution, asymptotic}
	e^{i n \xi z} z^{- \frac { n-1} 2} \sum_{m=0}^\infty B_m (\ulambda; \xi) z^{-m}
	\end{equation}
	is a formal solution of the Bessel equation  of sign $\varsigma$  {\rm(\ref{7eq: differential equation, variable z})}. We shall denote the formal series in {\rm (\ref{7eq: formal solution, asymptotic})} by $\widehat J (z; \ulambda; \xi)$. Moreover, the $j$-th formal derivative $\widehat J^{(j)} (z; \ulambda; \xi)$ is also of the form as {\rm (\ref{7eq: formal solution, asymptotic})}, but with coefficients depending on $j$ as well. 
\end{prop}

\begin{rem}
	The above arguments are essentially adapted from the proof of \cite[Chapter 5, Theorem 2.1]{Coddington-Levinson}. This construction of the formal solution and Lemma {\rm \ref{lem: entries of Pm}} will be required later in \S { \rm \ref{sec: Error Bounds for the asymptotic expansions} } for the error analysis. 
	
	However, This method is \textit{not} the best for the actual computation of the coefficients $B_m (\ulambda; \xi)$. We may derive the recurrent relations for $B_m (\ulambda; \xi)$ by a more direct but less suggestive approach as follows.
	
	The substitution $w = e^{i n \xi z} z^{- \frac { n-1} 2} u $ transforms the Bessel equation {\rm (\ref{7eq: differential equation, variable z})} into
	\begin{equation*}
	\sum_{j = 0}^{n} W_{j} (z; \ulambda) u^{(j)} = 0,
	\end{equation*}
	where $W_j  (z; \ulambda)$ is a polynomial in $z\-$ of degree $n-j$,
	\begin{equation*}
	W_j (z; \ulambda) = \sum_{k = 0}^{n-j} W_{j, k} (\ulambda) z^{-k},
	\end{equation*}
	with
	\begin{align*}
	W_{0, 0} (\ulambda) &=  (in \xi)^{n} - \varsigma (in)^n = 0,\\
	W_{j, k} (\ulambda) &= \frac { (in \xi)^{n-j-k} } {j! (n-j-k)!} \sum_{r=0}^k \frac {(n-r)!} { (k-r)!} \left[- \frac {n-1} 2\right]_{k-r} V_{n, n-r} (\ulambda), \hskip 10 pt (j,k) \neq (0,0).
	\end{align*}
	We have
	\begin{equation*}
	W_{0, 1} (\ulambda) = (in \xi)^{n-1} \lp n \lp - \frac {n-1} 2 \rp V_{n, n}  (\ulambda) + V_{n, n-1} (\ulambda) \rp = 0,
	\end{equation*}
	but $W_{1, 0} (\ulambda) = n( in \xi)^{n-1}$ is nonzero. Some calculations show that $ B_m (\ulambda; \xi)$ satisfy the following recurrence relations
	\begin{align*}
	(m-1) &  W_{1, 0} (\ulambda) B_{m-1} (\ulambda; \xi) =   \sum_{k=2}^{\min\{m , n\} } B_{m- k} (\ulambda; \xi) \sum_{j=0}^{k} W_{j, k-j} (\ulambda) [k-m]_j  , \hskip 10 pt m \geq 2.
	\end{align*}
	
	If $n = 2$, for a fourth root of unity $\xi = \pm 1, \pm i$ one may calculate in this way to obtain
	\begin{equation*}
	B_m (\lambda, - \lambda; \xi) = \frac {  \lp \frac 1 2 - 2\lambda \rp_m  \lp \frac 1 2 + 2\lambda\rp_m} { (4i\xi)^m m!}.
	\end{equation*}
\end{rem}

\subsubsection{Bessel Functions of the Second Kind}\label{sec: asymptotic expansions for the Bessel equations}

{\it Bessel functions of the second kind} are solutions of Bessel equations defined according to their asymptotic expansions at infinity. 
We shall apply several results in the asymptotic theory of ordinary differential equations  from \cite[Chapter IV]{Wasow}. 

Firstly, \cite[Theorem 12.3]{Wasow} implies the following lemma.

\begin{lem} [Existence of solutions] \label{lem: asymptotic, existence of a solution}
	Let $\varsigma \in \{+, - \}$, $\xi \in \BX_n (\varsigma)$, and $\BS \subset \BU$ be an open sector with vertex at the origin and a positive central angle not exceeding $\pi$. Then there exists a solution of the Bessel equation  of sign $\varsigma$ {\rm (\ref{7eq: differential equation, variable z})} that has the asymptotic expansion $\widehat J (z; \ulambda; \xi )$ defined in {\rm (\ref{7eq: formal solution, asymptotic})} on $\BS$. Moreover, each derivative of this solution has the formal derivative of  $\widehat J (z; \ulambda; \xi )$ of the same order as its asymptotic expansion.
\end{lem}


For two distinct $\xi, \xi' \in \BX_n(\varsigma)$, the ray emitted from the origin on which
$$\Re \lp (i \xi - i\xi') z \rp = - \Im \lp (\xi - \xi') z \rp = 0 $$ is called a  {separation ray}. 

We first consider the case $n=2$. It is clear that the separation rays constitute either the real or the imaginary axis and thus separate $ \BC \smallsetminus \{0 \}$ into two half-planes. 
Accordingly, we define  $\BS_{\pm 1} = \left\{ z : \pm \Im z > 0 \right \}$ and $\BS_{\pm i} = \left\{ z : \pm \Re z > 0 \right \}$.

In the case $n \geq 3$, there are $2n$ distinct separation rays in $ \BC \smallsetminus \{0 \}$ given by the equations 
$$\arg z = \arg ( i \xi' ),  \hskip 10 pt \xi' \in \BX_{2n} (+).$$
These separation rays divide $\BC \smallsetminus \{0\}$ into $2 n$  many open sectors  \begin{equation}\label{7def: S pm xi}
\BS^{\pm}_\xi = \left\{ z : 0 < \pm \left( \arg z - \arg ( i \xoverline \xi ) \right) < \frac {\pi} n \right \}, \hskip 10 pt \xi \in \BX_{ n} (\varsigma).
\end{equation}
In both sectors $\BS^{+}_{ \xi }$ and $\BS^{-}_{ \xi }$ we have
\begin{equation}\label{7eq: condition for S xi}
\Re (i \xi  z) < \Re (i \xi' z) \ \text{ for all } \xi' \in \BX_n(\varsigma),\ \xi' \neq \xi.
\end{equation}
Let $\BS_{ \xi }$ be the sector on which \eqref{7eq: condition for S xi} is satisfied. It is evident that
\begin{equation}\label{7def: S xi}
\BS_\xi = \left\{ z : \left| \arg z - \arg ( i \xoverline \xi ) \right| < \frac {\pi} n \right \}.
\end{equation}


\begin{lem}\label{lem: asymptotic, uniqueness}
	Let $\varsigma \in \{+, - \}$ and $\xi \in \BX_n( \varsigma)$.

	{\rm (1. Existence of asymptotics).} If $n \geq 3$,  on the sector $\BS^{\pm}_{\xi }$, all the solutions of the Bessel equation  of sign $\varsigma$ have asymptotic representation a multiple of $\widehat J (z; \ulambda; \xi')$ for some $\xi' \in \BX_n( \varsigma)$. If $ n = 2$, the same assertion is true with  $\BS^{\pm}_{\xi }$ replaced by  $\BS_{\xi }$.
	
	{\rm (2. Uniqueness of the solution).}
	There is a unique solution of the Bessel equation of sign $\varsigma$ that possesses $\widehat J (z; \ulambda; \xi )$ as its asymptotic expansion on $\BS_{\xi }$ or any of its open subsector, and we shall denote this solution by $J (z; \ulambda; \xi )$. Moreover, $J^{(j)} (z; \ulambda; \xi ) \sim \widehat J^{(j)} (z; \ulambda; \xi )$ on $\BS_{\xi }$ for any $j \geq 0$.
\end{lem}

\begin{proof}
	(1) follows directly from \cite[Theorem 15.1]{Wasow}.
	
	For $n = 2$, since \eqref{7eq: condition for S xi} holds for the sector $\BS_{\xi }$, (2) is true according to \cite[Corollary to Theorem 15.3]{Wasow}. Similarly, if $n \geq 3$, (2) is true with  $\BS_{\xi }$ replaced by  $\BS^{\pm}_{\xi }$. Thus there exists a unique solution of the Bessel equation  of sign $\varsigma$ possessing $\widehat J (z; \ulambda; \xi )$ as its asymptotic expansion on $\BS^{\pm}_{\xi }$ or any of its open subsector. For the moment, we denote this solution   by  $J^{\pm} (z; \ulambda; \xi )$.
	On the other hand, because $\BS_{\xi }$ has central angle $\frac 2 {\,n\,} \pi < \pi$, there exists a solution $J (z; \ulambda; \xi )$ with asymptotic $\widehat J (z; \ulambda; \xi )$ on a given open subsector $\BS \subset \BS_{\xi }$ due to Lemma \ref{lem: asymptotic, existence of a solution}. Observe that at least one of $\BS \cap \BS^+_{\xi }$ and $\BS \cap \BS^-_{\xi }$ is a nonempty open sector, say $\BS \cap \BS^+_{\xi } \neq \O $, then the uniqueness of $J (z; \ulambda; \xi )$ follows from that of $J^{+} (z; \ulambda; \xi )$ along with the principle of analytic continuation.
\end{proof}

\begin{prop}\label{prop: J (z; lambda; xi) on a large sector}
	Let $\varsigma \in \{+, - \}$, $\xi \in \BX_n( \varsigma)$, $ \vartheta $ be a small positive constant, say $0 < \vartheta < \frac 1 2 \pi $, and define
	\begin{equation}\label{7eq: S' xi (delta)}
	\BS'_{\xi } (\vartheta) = \left\{ z : \left| \arg z - \arg ( i \xoverline \xi) \right|  < \pi + \frac {\pi} n - \vartheta \right \}.
	\end{equation} Then $J (z; \ulambda; \xi )$ is the unique solution of the Bessel equation  of sign $\varsigma$ that has the asymptotic expansion $\widehat J (z; \ulambda; \xi )$ on $\BS'_{\xi } (\vartheta)$. Moreover, $ J^{(j)} (z; \ulambda; \xi ) \sim \widehat J^{(j)} (z; \ulambda; \xi )$ on $ \BS'_{\xi }(\vartheta)$ for any   $j \geqslant 0$.
\end{prop}

\begin{proof}
	Following from Lemma \ref{lem: asymptotic, existence of a solution}, there exists a solution of the Bessel equation  of sign $\varsigma$ that has the asymptotic expansion $\widehat J (z; \ulambda; \xi )$ on the open sector
	\begin{equation*}
	\begin{split}
	\BS^{\pm}_{\xi } (\vartheta) = \left\{ z : \frac {\pi} n - \vartheta < \pm \lp \arg z - \arg ( i \xoverline \xi ) \rp < \pi + \frac {\pi} n - \vartheta \right \}.
	\end{split}
	\end{equation*}
	On the nonempty open sector $\BS_{\xi } \cap \BS^{\pm}_{\xi } (\vartheta)$ this solution must be identical with  $J (z; \ulambda; \xi )$ by Lemma \ref{lem: asymptotic, uniqueness} (2) and hence is equal to  $J (z; \ulambda; \xi )$ on $\BS_{\xi } \cup \BS^{\pm}_{\xi } (\vartheta)$ due to the principle of analytic continuation. Therefore, the region of validity of the asymptotic  $ J (z; \ulambda; \xi ) \sim \widehat J (z; \ulambda; \xi )$ may be widened from $\BS_{\xi }$ onto $\BS'_{\xi } (\vartheta) = \BS_{\xi } \cup \BS^{+}_{\xi } (\vartheta) \cup \BS^{-}_{\xi } (\vartheta)$. 
	In the same way, Lemma \ref{lem: asymptotic, existence of a solution} and \ref{lem: asymptotic, uniqueness} (2) also imply  that $ J^{(j)} (z; \ulambda; \xi ) \sim \widehat J^{(j)} (z; \ulambda; \xi )$ on $  \BS'_{\xi }(\vartheta)$. 
\end{proof}

\begin{cor}\label{prop: J (z; lambda; xi) form a fundamental set of solutions}
	Let $\varsigma \in \{+, - \}$.
	All the $J (z; \ulambda; \xi )$, with $\xi \in \BX_n(\varsigma)$, form a fundamental set of solutions of the Bessel equation  of sign $\varsigma$.
\end{cor}

\begin{rem}\label{rem: n=2, asymptotics}
	If $n = 2$, by \cite[3.7 (8), 3.71 (18), 7.2 (1, 2), 7.23 (1, 2)]{Watson} we have the following formula of $J(z; \lambda, - \lambda; \xi)$, with $\xi = \pm 1, \pm i$, and the corresponding sector on which its asymptotic expansion is valid
	\begin{equation*}
	\begin{split}
	& J(z; \lambda, - \lambda; 1) = \sqrt {\pi i} e^{\pi i \lambda} H^{(1)}_{2 \lambda} (2 z), \hskip 10 pt 
	\BS'_{1} (\vartheta) = \left\{ z : - \pi + \vartheta < \arg z < 2 \pi - \vartheta \right\} ;\\
	& J(z; \lambda, - \lambda; - 1) = \sqrt {- \pi i} e^{- \pi i \lambda} H^{(2)}_{2 \lambda} (2 z), \hskip 10 pt 
	\BS'_{-1} (\vartheta) = \left\{z : - 2 \pi + \vartheta < \arg z < \pi - \vartheta \right\};\\
	& J(z; \lambda, - \lambda; i) = \frac 2  {\sqrt \pi}K_{2 \lambda} (2 z), \hskip 10 pt 
	\BS'_{i} (\vartheta) = \left\{ z :  |\arg z| < \frac 3 2 \pi - \vartheta \right\}; \\
	& J(z; \lambda, - \lambda; - i) = 2 \sqrt \pi I_{2 \lambda} (2 z) - \frac {2 i}  {\sqrt \pi} e^{2\pi i \lambda} K_{2 \lambda} ( 2 z), \\
	& \hskip 130 pt
	\BS'_{- i} (\vartheta) = \left\{ z : - \frac 1 2 \pi + \vartheta < \arg z < \frac 5 2 \pi - \vartheta \right\} .
	\end{split}
	\end{equation*}
\end{rem}


\begin{lem}\label{lem: J(z; lambda; xi) relations to J(xi z; lambda; 1)}
	Let $\xi\in \BX_{2n} (+)$. We have $$J( z; \ulambda; \xi) = (\pm \xi)^{ \frac {n-1} 2} J(\pm \xi z; \ulambda; \pm 1),$$
	and $B_m (\ulambda; \xi) = (\pm \xi)^{- m} B_m (\ulambda; \pm 1)$. 
	
\end{lem}

\begin{proof}
	By Lemma \ref{lem: phi (e pi i m/n z) is also a solution}, $(\pm \xi)^{ \frac {n-1} 2} J(\pm \xi z; \ulambda; \pm 1)$ is a solution of one of the two Bessel equations of index $\ulambda$. In view of Proposition \ref{prop: formal solution} and Lemma \ref{lem: asymptotic, uniqueness} (2), it  possesses $ \widehat J (z; \ulambda; \xi )$ as its asymptotic expansion on $\BS _{\xi } $ and hence must be identical with $ J( z; \ulambda; \xi)$. 
\end{proof}

\begin{term}\label{term: Bessel functions of the second kind}
	For $\xi \in \BX_{2n} (+)$, $J (z; \ulambda; \xi)$ is called a Bessel function of the second kind.
\end{term}

\begin{rem}
	The results in this section do not provide any information on the asymptotics near zero of Bessel functions of the second kind, and therefore their connections with Bessel functions of the first kind can not be clarified here. We shall nevertheless derive the connection formulae between the two kinds of Bessel functions later in \S {\rm \ref{sec: H-Bessel functions and K-Bessel functions revisited}}, appealing to the asymptotic expansion of the $H$-Bessel function $H^{\pm}  (z; \ulambda)$ on the half-plane $\BH^{\pm}$ that we showed earlier in \S {\rm \ref{sec: Bessel functions of K-type and H-Bessel functions}}. 
\end{rem}

\subsection{Error Analysis for Asymptotic Expansions}\label{sec: Error Bounds for the asymptotic expansions}

The error bound for the asymptotic expansion of $J (z; \ulambda; \xi ) $ with dependence on  $\ulambda$ is always desirable for potential applications in analytic number theory. However, the author does not find any general results on the error analysis for differential equations of order higher than two.
We shall nevertheless combine and generalize the ideas from \cite[\S 5.4]{Coddington-Levinson} and \cite[\S 7.2]{Olver} to obtain an almost optimal error estimate for the asymptotic expansion of the Bessel function $J (z; \ulambda; \xi )$.
Observe that both of their methods have drawbacks for generalizations. \cite{Olver} hardly uses the viewpoint from differential systems as only the second-order case is treated, whereas \cite[\S 5.4]{Coddington-Levinson} is restricted to the positive real axis for more clarified expositions.


\vskip 5 pt

\subsubsection{Preparations}

We retain the notations from \S \ref{sec: formal solutions at infinity}. For  a positive integer $M $ denote by $P_{(M)}$ the polynomial in $z\-$,
\begin{equation*}
P_{(M)} (z) = \sum_{m=0}^M z^{-m} P_m,
\end{equation*} 
and by $\widehat \Phi_{(M)}$ the truncation of $\widehat \Phi$,
\begin{equation*}
\widehat \Phi_{(M)} (z) = P_{(M)} (z) z^{-\frac { n-1} 2} e^{\varDelta z}.
\end{equation*}
By Lemma \ref{lem: entries of Pm}, we have $\left|z^{-m} P_m \right| \lll_{\, m,\, n} \fC^{2m} |z|^{-m}$, so $P_{(M)}\-$ exists as an analytic function for $|z| > c_1 \fC^2$, where $c_1$ is some constant depending only on $M$ and $n$. Moreover,
\begin{equation}
\label{7eq: bound of P (M) P (M) inverse}
\left|  P_{(M)} (z) \right|, \ \big| P_{(M)}\- (z) \big| = O_{\,M,\, n} (1), \hskip 10 pt |z| > c_1 \fC^2.
\end{equation} 
Let $A_{(M)}$ and $E_{(M)}$ be defined by 
\begin{equation*}
A_{(M)} = \widehat \Phi_{(M)}' \widehat \Phi\-_{(M)}, \hskip 10 pt E_{(M)} = A - A_{(M)}.
\end{equation*}
$A_{(M)}$ and $E_{(M)}$ are clearly analytic for $|z| > c_1 \fC^2$.
Since $$E_{(M)} P_{(M)} = A P_{(M)} - \lp P'_{(M)} - \frac {n-1} 2 z\- P_{(M)} + P_{(M)} \varDelta \rp,$$ 
it follows from the construction of $\widehat \Phi$ in \S \ref{sec: formal solutions at infinity} that $E_{(M)} P_{(M)}$ is a polynomial in $z\-$ of the form $\sum_{m=M+1}^{M+n} z^{-m} E_{m}$
so that
\begin{equation*}
\begin{split}
E_{M+1} & = P^o_{M+1} \varDelta - \varDelta P^o_{M+1},\\
E_m & = \sum_{j = m-M}^{\min\{m, n\}} A_j P_{m-j}, \hskip 10 pt M+1 < m \leq M+n.
\end{split}
\end{equation*} 
Therefore, in view of Lemma \ref{lem: entries of Pm}, $| E_{M+1}| \lll_{\,M,\, n} \fC^{2M} $ and $| E_{m}| \lll_{\,m,\, n} \fC^{m + M} $ for $M+1 < m \leq M+n$. It follows that $| E_{(M)}(z) P_{(M)}(z) | \lll_{\, M,\, n} \fC^{2M} z^{-M-1}$ for $|z| > c_1 \fC^2$, and this, combined with (\ref{7eq: bound of P (M) P (M) inverse}), yields
\begin{equation}\label{7eq: bound of E (M)}
| E_{(M)} (z) | = O_{\,M,\, n} \lp \fC^{2M} |z|^{-M-1} \rp.
\end{equation}


By the definition of $A_{(M)}$, for $|z| > c_1 \fC^2$,  $\widehat \Phi_{(M)}$ is a fundamental matrix of the system
\begin{equation}\label{7eq: differential system truncated}
u' = A_{(M)} u.
\end{equation}
We shall regard the differential system (\ref{7eq: differential equation, matrix form 2}), that is,
\begin{equation} \label{7eq: differential system, nonhomogeneous}
u' = A u = A_{(M)} u + E_{(M)} u,
\end{equation}
as a nonhomogeneous system with (\ref{7eq: differential system truncated}) as the corresponding homogeneous system.

\vskip 5 pt

\subsubsection{Construction of a Solution}

Given $l      \in \{ 1, ..., n \}$,   let $$ \widehat \varphi_{(M),\, l     } (z) = p_{(M),\, l     } (z) z^{- \frac { n-1} 2} e^{in \xi_{l     } z}$$ be the $l     $-th column vector of the matrix $\widehat \Phi_{(M)}$, where $ p_{(M),\, l     }$ is the $l     $-th  column vector of $P_{(M)}$. Using a version of the variation-of-constants formula and the method of successive approximations, we shall construct a solution $\varphi_{(M),\, l     } $ of (\ref{7eq: differential equation, matrix form 2}), for $z$ in some suitable domain, satisfying
\begin{equation} \label{7eq: error bounds 1}
\left| \varphi_{(M),\, l     } (z) \right| = O_{\,M,\, n} \lp |z|^{- \frac { n-1} 2} e^{\Re \lp i n \xi_{l     } z \rp} \rp,
\end{equation}
and
\begin{equation} \label{7eq: error bounds 2}
\left| \varphi_{(M),\, l     } (z) - \widehat \varphi_{(M),\, l     } (z) \right| = O_{\,M,\, n} \lp \fC^{2 M} |z|^{-M- \frac { n-1} 2} e^{\Re \lp i n \xi_{l     } z \rp} \rp,
\end{equation}
with the implied constant in \eqref{7eq: error bounds 2} also depending on the chosen domain.

\vskip 5 pt

{\it Step} 1. {\it Constructing the domain and the contours for the integral equation.} For $C \geq c_1 \fC^2$ and $0 < \vartheta < \frac 12 \pi $, define the domain $\BD(C; \vartheta) \subset \BU$ by 
\begin{equation*}
\BD(C; \vartheta) = \left\{ z : \left| \arg z \right| \leq  \pi, |z| > C \right \} \cup \left\{ z : \pi < \left| \arg z \right| < \frac {3} 2 \pi - \vartheta,\ \Re\,  z < - C \right \}.
\end{equation*}
For $k \neq l     $ let $\omega(l     , k) = \arg (i \xoverline \xi_{l     } - i \xoverline \xi_{k}) = \arg (i \xoverline \xi_{l     }) + \arg ( 1 - \xi_{k} \xoverline \xi_{l      })$, and define
\begin{equation*}
\BD_{\xi_l     } (C; \vartheta) = \bigcap_{k \neq l     }  e^{i \omega(l     , k)} \cdot \BD(C; \vartheta).
\end{equation*}
With the observation that
$$
\left\{ \arg ( 1 - \xi_{k} \xoverline \xi_{l      }) : k \neq l      \right \} = \left \{ \lp \frac 1 2 - \frac a n\rp \pi :  a =1, ..., n-1 \right \},
$$
it is straightforward to show that $\BD_{\xi_l     } (C; \vartheta) = i \xoverline \xi_l      \BD' (C; \vartheta)$, where $\BD' (C; \vartheta)$ is defined to be the union of the sector 
\begin{equation*}
\left\{ z : \left| \arg z \right| \leq \frac \pi 2 + \frac {\pi} n,\ |z| > C \right \}
\end{equation*}
and the following two domains
\begin{equation*}
\begin{split}
& \left\{ z : \frac \pi 2 + \frac {\pi} n < \arg z < \pi + \frac {\pi} n - \vartheta,\ \Im \big(e^{- \frac 1 n \pi i } z\big) > C  \right \},\\
& \left\{ z : - \pi  - \frac {\pi} n + \vartheta < \arg z < - \frac \pi 2 - \frac {\pi} n,\ \Im\big(e^{\frac 1 n \pi i } z\big) < - C  \right \}.
\end{split}
\end{equation*}


\begin{figure}
	\begin{center}
		\begin{tikzpicture}
		\draw [dotted, fill=lightgray](0,0) circle [radius=0.75];
		
		\draw [->] (-2,0) -- (3,0);
		\draw [->] (0, -1.75) -- (0,2.5);
		\draw [thick](1.6,0) arc [radius=1.6, start angle=0, end angle= 140];
		\draw [-, thick] (1.6, 0) -- (2.97, 0);
		
		\draw [->] (2.5,0) -- (2.3,0);
		\draw [->](1.6,0) arc [radius=1.6, start angle=0, end angle= 80];
		
		\draw[fill=white] (0.75,0) circle [radius=0.04];
		\node [below right] at (0.75,0) { \footnotesize $C$};
		\draw[fill=white] (-1.2257, 1.0285) circle [radius=0.04];
		\node [below] at (-1.2257, 1.0285) { \footnotesize $z$};
		\draw[fill=white] (1.6, 0) circle [radius=0.04];
		\node [above right] at (1.6, 0) { \footnotesize $|z|$};
		\node [above right] at (1.06, 1.06) { \footnotesize $\EuScript {C}(z)$};
		\node [below] at (0, -1.85) { \small $|\arg z| \leq \pi$};

		\draw [dotted, fill=lightgray](7.75,0) arc [radius=0.75, start angle=0, end angle= 180];
		\draw [fill=lightgray, lightgray] (6.25,0) to (6.25,-1) to (7, -1) to (7, 0) to (6.25,0);
		\draw [fill=lightgray, lightgray] (6.25,0) to (6.25,-0.75) to (5.25, -1.75) to (7, -1.75) to (7, 0) to (6.25,0);
		\draw [dotted](6.25,0) -- (6.25,-0.75);
		\draw [dotted](7 ,0) -- (5.25 ,-1.75);
		
		\draw [->] (4.5,0) -- (10,0);
		\draw [->] (7, -1.75) -- (7,2.5);
		\draw [thick](8.6,0) arc [radius=1.6, start angle=0, end angle= 180];
		
		\draw [-, thick] (5.4, 0) -- (5.4, -0.75);
		\draw [-, thick] (8.6, 0) -- (9.97, 0);
		\draw [->] (9.5,0) -- (9.3,0);
		\draw [->](8.6,0) arc [radius=1.6, start angle=0, end angle= 80];
		\draw [->] (5.4, 0) -- (5.4, -0.5);
		
		\draw [-](7,-0.25) arc [radius=0.25, start angle=270, end angle= 225];
		\node [below] at (6.85,-0.15) {\scriptsize $\vartheta$};
		
		\draw[fill=white] (7.75,0) circle [radius=0.04];
		\node [below] at (7.75,0) {\footnotesize $C$};
		\draw[fill=white] (5.4, -0.75) circle [radius=0.04];
		\node [ left] at (5.4, -0.75) {\footnotesize $z$};
		\draw[fill=white] (5.4, 0) circle [radius=0.04];
		\node [above left] at (5.4, 0) {\footnotesize $\Re  z$};
		\draw[fill=white] (8.6, 0) circle [radius=0.04];
		\node [below] at (8.6, 0) {\footnotesize $- \Re  z$};
		\node [above right] at (8.06, 1.06) {\footnotesize $\EuScript {C}(z)$};
		\node [below] at (7, -1.85) {\small $\pi < \arg z < \frac 3 2\pi - \vartheta$};
		\end{tikzpicture}
	\end{center}
	\caption{$\EuScript C (z) \subset \BD (C; \vartheta)$}\label{fig: C (z)}
\end{figure}

\begin{figure}
	\begin{center}
		\begin{tikzpicture}

		\draw [dotted, fill=lightgray](0.75,0) arc [radius=0.75, start angle=0, end angle= 120];
		\draw [dotted, fill=lightgray](0.75,0) arc [radius=0.75, start angle=0, end angle= -120];
		
		\draw [fill=lightgray, lightgray] (0,0) to (-0.75/2,1.7320508*0.75/2) to (0.75, 0);
		\draw [fill=lightgray, lightgray] (0,0) to (-0.75/2,-1.7320508*0.75/2) to (0.75, 0);
		
		\draw [-](- 1.4,1.7320508*1.4) -- (0,0);
		\draw [-](- 0.8 , -1.7320508*0.8 ) -- (0,0);

		\draw [->] (-1.4,0) -- (3,0);
		\draw [->] (0, -1.7320508*0.8) -- (0,2.5);

		\draw [-](0, 0.25) arc [radius=0.25, start angle=90, end angle= 120];
		\node [above] at (- 0.1, 0.2) {\footnotesize $\frac \pi n$};

		\draw [thick](1.6,0) arc [radius=1.6, start angle=0, end angle= 75];
		\draw [-, thick] (1.6, 0) -- (2.97, 0);

		\draw [->] (2.5,0) -- (2.3,0);
		\draw [->](1.6,0) arc [radius=1.6, start angle=0, end angle= 55];

		\draw[fill=white] (1.6, 0) circle [radius=0.04];
		\draw[fill=white] (0.2588190451*1.6, 0.96592582628*1.6) circle [radius=0.04];
		\node [above] at (0.2588190451*1.6, 0.96592582628*1.6) {\footnotesize $ z$};
		
		\node [above right] at (1.36, 0.66) { \footnotesize $\EuScript {C}'(z)$};
		
		\node [below] at (0, -1.5) { \small $|\arg z| \leq   \frac 1 2 \pi + \frac 1 n   \pi  $};

		\draw [dotted, fill=lightgray](7.75,0) arc [radius=0.75, start angle=0, end angle= 120];
		\draw [fill=lightgray, lightgray] (7,0) to (7 - 1.7320508*1, -1) to (7 -0.75/2- 1.7320508*1,1.7320508*0.75/2-1) to (7 -0.75/2,1.7320508*0.75/2);
		\draw [fill=lightgray, lightgray] (7,0) to (7 -0.75/2,1.7320508*0.75/2) to (7.75, 0);
		\draw [dotted](7 -0.75/2,1.7320508*0.75/2) -- (7 -0.75/2- 1.7320508*1,1.7320508*0.75/2-1);

		\draw [fill=lightgray, lightgray] (7,0) to (7 -0.75/2,1.7320508*0.75/2) to (7 - 1.7320508*1, -1) to (7 -0.96592582628*2.82842712475, 0.258819*2.82842712475);
		\draw [dotted](7 -0.96592582628*2.82842712475, 0.258819*2.82842712475) -- (7,0);
		\draw [dotted](7, 0) -- (7 - 1.7320508*1, -1);
		
		\draw [fill=lightgray, lightgray] (7 - 1.7320508*1, -1) to (7 -0.96592582628*2.82842712475, 0.258819*2.82842712475) to (7 -0.96592582628*2.82842712475, -1);

		\draw [-](7 - 1.4,1.7320508*1.4) -- (7,0);

		\draw [->] (7 -0.96592582628*2.82842712475,0) -- (7+3,0);
		\draw [->] (7 , -1) -- (7,2.5);
		
		\draw [-](7 - 1.7320508*0.125,-0.125) arc [radius=0.25, start angle=210, end angle= 165];
		\draw [fill=lightgray, lightgray](7 - 0.3,-0.05) arc [radius=0.15, start angle=0, end angle= 360];
		\node [left] at (7 - 1.7320508*0.125-0.015, -0.125+0.08) {\scriptsize $\vartheta$};
		
		\draw [-](7, 0.25) arc [radius=0.25, start angle=90, end angle= 120];
		\node [above] at (7 - 0.1, 0.2) {\footnotesize $\frac \pi n$};

		\draw [thick](7+ 1.6,0) arc [radius=1.6, start angle=0, end angle= 120];
		\draw [-, thick] (7+ 1.6, 0) -- (7+ 2.97, 0);
		\draw [->] (7 -0.8-1.7320508/6, 1.7320508*0.8-1/6) -- (7 -0.8-1.7320508/5, 1.7320508*0.8-0.2);
		\draw [-, thick] (7 -0.8, 1.7320508*0.8) -- (7 -0.8-1.7320508/2, 1.7320508*0.8-1/2);
		
		\draw [->] (7+ 2.5,0) -- (7+ 2.3,0);
		\draw [->](7+ 1.6,0) arc [radius=1.6, start angle=0, end angle= 80];

		\draw[fill=white] (7 -0.8, 1.38564) circle [radius=0.04];
		\draw[fill=white] (7+ 1.6, 0) circle [radius=0.04];
		\draw[fill=white] (7 -0.8-1.7320508/2, 1.38564-1/2) circle [radius=0.04];
		\node [above left] at (7 -0.8-1.7320508/2 + 0.2, 1.38564-1/2) {\footnotesize $ z$};
		
		\node [above right] at (8.06, 1.06) {\footnotesize $\EuScript {C}'(z)$};

		\node [below] at (7, -1.5) {\small $\frac 1 2 \pi   + \frac 1 n {\pi}   < \arg z < \pi + \frac 1 n {\pi}   - \vartheta$};
		
		\end{tikzpicture}
	\end{center}
	\caption{$ \EuScript C' (z) \subset \BD' (C; \vartheta)$}\label{fig: D xi (C)}
\end{figure}

For $z \in \BD (C; \vartheta)$ we define a contour $\EuScript C (z) \subset \BD (C; \vartheta)$ that starts from $\infty$ and ends at $z$; see Figure \ref{fig: C (z)}. For $z \in \BD (C; \vartheta)$ with $ |\arg z| \leq \pi$, the contour $\EuScript C (z)$ consists of the part of the positive axis where the magnitude exceeds $ |z|$ and an arc of the circle centered at the origin of radius $|z|$, with angle not exceeding $\pi$ and endpoint $z$. For $z \in \BD (C; \vartheta)$ with $\pi < |\arg z| < \frac {3 } 2 \pi - \vartheta$, the definition of the contour $\EuScript C (z)$ is modified so that the circular arc has radius $- \Re  z$ instead of $|z|$ and ends at $\Re  z$ on the negative real axis, and that $\EuScript C (z)$ also consists of a vertical line segment joining $\Re  z$ and $z$. 
The most crucial property that $\EuScript C (z)$ satisfies is the \textit{nonincreasing} of $\Re  \zeta$ along $\EuScript C (z)$.

We also define a contour $\EuScript C' (z)$ for $z \in \BD' (C; \vartheta) $ of a similar shape as $ \EuScript C \left( z \right)$ illustrated in Figure \ref{fig: D xi (C)}.

\vskip 5 pt

{\it Step} 2. {\it Solving the integral equation via successive approximations.}
We first split $\widehat \Phi_{(M)}\-$ into $n$ parts
\begin{equation*}
\widehat \Phi_{(M)}\- = \sum_{k = 1}^n \Psi^{(k)}_{(M)},
\end{equation*}
where the $j$-th row of $\Psi^{(k)}_{(M)}$ is identical with the $k$-th row of $\widehat \Phi_{(M)}\-$, or identically zero, according as $j = k$ or not.

The integral equation to be considered is the following
\begin{equation}\label{7eq: differential equation to integral equation}
u(z) = \widehat \varphi_{(M),\, l     } (z) + \sum_{k \neq l     } \int_{\infty e^{i \omega(l     , k)}}^z K_{k} (z, \zeta) u(\zeta) d \zeta + \int_{\infty i \overline \xi_l     }^z K_{l     } (z, \zeta) u(\zeta) d \zeta,
\end{equation}
where
\begin{equation*}
K_{k} (z, \zeta) = \widehat \Phi_{(M)} (z) \Psi_{(M)}^{(k)} (\zeta) E_{(M)}(\zeta), \hskip 10 pt z, \zeta \in \BD_{\xi_l     } (C; \vartheta), k = 1, ..., n,
\end{equation*}
the integral in the sum is integrated on the contour $ e^{i \omega(l     , k)} \EuScript C \big(e^{ - i \omega(l     , k)} z \big)$, whereas the last integral is on the contour $ i \xoverline \xi_l      \EuScript C' \left(- i  \xi_l      z \right)$. Clearly, all these contours lie in $\BD_{\xi_l     } (C; \vartheta)$. Most importantly, we note that $\Re ((i \xi_{l     } - i \xi_{k}) \zeta )$ is a negative multiple of $\Re \big( e^{- i \omega(l     , k)} \zeta \big)$ and hence is \textit{nondecreasing} along the contour $ e^{i \omega(l     , k)} \EuScript C \big(e^{ - i \omega(l     , k)} z\big)$.

By direct verification, it follows that if $u (z) = \varphi (z)$ satisfies \eqref{7eq: differential equation to integral equation}, with the integrals convergent, then $\varphi$ satisfies \eqref{7eq: differential system, nonhomogeneous}.

In order to solve (\ref{7eq: differential equation to integral equation}), define the successive approximations
\begin{equation}\label{7eq: successive approximation}
\begin{split}
& \varphi^0 (z) \equiv 0,\\
\varphi^{\alpha+1} (z) = \widehat \varphi_{(M),\, l     } (z) + & \sum_{k \neq l     } \int_{\infty e^{i \omega(l     , k)}}^z K_{k} (z, \zeta) \varphi^\alpha(\zeta) d \zeta + \int_{\infty i \overline \xi_l     }^z K_{l     } (z, \zeta) \varphi^\alpha(\zeta) d \zeta.
\end{split}
\end{equation}

The $(j, r)$-th entry of the matrix $\widehat \Phi_{(M)} (z) \Psi_{(M)}^{(k)} (\zeta)$ is given by
\begin{equation*}
\lp \widehat \Phi_{(M)} (z) \Psi_{(M)}^{(k)} (\zeta) \rp_{j r} = \lp P_{(M)} (z) \rp_{j k} \big( P\-_{(M)} (\zeta) \big)_{k r} \lp\frac z {\zeta} \rp^{-\frac { n-1} 2} e^{in \xi_{k} (z - \zeta)}.
\end{equation*}
It follows from (\ref{7eq: bound of P (M) P (M) inverse}, \ref{7eq: bound of E (M)}) that
\begin{equation}\label{7eq: bound of K ell' (z, zeta)}
| K_{k} (z, \zeta) | \leq c_2 \fC^{2M} |z|^{- \frac { n-1} 2} |\zeta|^{-M-1+ \frac { n-1} 2} e^{\Re (in \xi_{k} (z - \zeta))},
\end{equation}
for some constant $c_2$ depending only on $M$ and $n$. Furthermore, we may appropriately choose $c_2$ such that
\begin{equation}\label{7eq: bound of integral zeta -M-1}
\int_{\infty i \overline \xi_l     }^z \left| \zeta \right|^{-M-1} \left|d \zeta \right|, \ \int_{\infty e^{i \omega(l     , k)}}^z \left| \zeta \right|^{-M-1} \left|d \zeta \right| \leq c_2 C^{-M}, \hskip 10 pt  k \neq l     .
\end{equation}

According to \eqref{7eq: successive approximation}, $\varphi^1 (z) = \widehat \varphi_{(M),\, l     } (z) = p_{(M),\, l     } (z) z^{- \frac { n-1} 2} e^{in \xi_{l     } z}$, so
\begin{equation*}
\left|\varphi^1 (z) - \varphi^0 (z) \right| = \left|\varphi^1 (z) \right| \leq c_2 |z|^{- \frac { n-1} 2} e^{\Re (in \xi_l      z)}, \hskip 10 pt z \in \BD_{\xi_l     } (C; \vartheta).
\end{equation*}
We shall show by induction that for all $z \in \BD_{\xi_l     } (C; \vartheta)$
\begin{equation}\label{7eq: successive approximation induction hypothesis}
\left|\varphi^\alpha (z) - \varphi^{\alpha -1} (z) \right| \leq c_2 \lp n c_2^2  \fC^{2M} C^{-M} \rp^{\alpha -1} |z|^{- \frac { n-1} 2} e^{\Re (in \xi_l      z)}.
\end{equation}
Let $z \in \BD_{\xi_l     } (C; \vartheta)$. Assume that (\ref{7eq: successive approximation induction hypothesis}) holds.
From (\ref{7eq: successive approximation}) we have
\begin{equation*}
\left|\varphi^{\alpha+1} (z) - \varphi^\alpha (z) \right|  
\leq \sum_{k \neq l     } R_{k} + R_{l     },
\end{equation*}
with 
\begin{align*}
R_{k} & = \int_{\infty e^{i \omega(l     , k)}}^z | K_{k} (z, \zeta)| \left|\varphi^\alpha (\zeta) - \varphi^{\alpha-1} (\zeta) \right| \left|d \zeta\right|,\\
R_{l     } & = \int_{\infty i \overline \xi_l     }^z | K_{l     } (z, \zeta)| \left|\varphi^\alpha (\zeta) - \varphi^{\alpha-1} (\zeta) \right| \left|d \zeta\right|.
\end{align*}
It follows from (\ref{7eq: bound of K ell' (z, zeta)}, \ref{7eq: successive approximation induction hypothesis}) that $R_{k} $ has bound
\begin{equation*}
c_2^2 \fC^{2M} \lp n c_2^2  \fC^{2M} C^{-M} \rp^{\alpha-1}  |z|^{- \frac { n-1} 2} e^{\Re (in \xi_l      z)} \int_{\infty e^{i \omega(l     , k)}}^z |\zeta|^{-M-1} e^{\Re (in (\xi_{l     } - \xi_{k}) (\zeta - z))} \left|d \zeta\right|.
\end{equation*}
Since $\Re ((i \xi_{l     } - i \xi_{k}) \zeta)$ is nondecreasing on the integral contour,
\begin{equation*}
R_{k} \leq c_2^2 \fC^{2M} \lp n c_2^2 \fC^{2M} C^{-M} \rp^{\alpha-1} |z|^{- \frac { n-1} 2} e^{\Re (in \xi_l      z)} \int_{\infty e^{i \omega(l     , k)}}^z |\zeta|^{-M-1} \left|d \zeta\right|,
\end{equation*}
and (\ref{7eq: bound of integral zeta -M-1}) further yields
\begin{equation*}
R_{k} \leq c_2 n^{\alpha -1} \lp c_2^2  \fC^{2M} C^{-M} \rp^{\alpha} |z|^{- \frac { n-1} 2} e^{\Re (in \xi_l      z)}.
\end{equation*}
Similar arguments  show that $R_{l     }$ has the same bound as $R_{k}$. Thus (\ref{7eq: successive approximation induction hypothesis}) is true with $\alpha$ replaced by $\alpha+1$.

Set the constant $C = c \fC^2$ such that $c^M \geq 2 n c_2^2$. Then $ n c_2^2 \fC^{2M} C^{-M} \leq \frac 1 2$, and therefore the series $\sum_{\alpha = 1}^\infty (\varphi^\alpha(z) - \varphi^{\alpha-1}(z))$ absolutely and compactly converges. The limit function $\varphi_{(M),\, l     } (z)$ satisfies (\ref{7eq: error bounds 1}) for all $z \in \BD_{\xi_l     } (C; \vartheta)$. More precisely,
\begin{equation}\label{7eq: precise error bound, 1}
\left| \varphi_{(M),\, l     } (z) \right| \leq 2 c_2 |z|^{- \frac { n-1} 2} e^{\Re \lp i n \xi_{l     } z \rp}, \hskip 10 pt z \in \BD_{\xi_l     } (C; \vartheta).
\end{equation} 
Using a standard argument for successive approximations, it follows that $\varphi_{(M),\, l     }$ satisfies the integral equation \eqref{7eq: differential equation to integral equation} and hence the differential system \eqref{7eq: differential system, nonhomogeneous}.

The proof of the error bound (\ref{7eq: error bounds 2}) is similar. Since $\varphi_{(M),\, l     } (z)$ is a solution of the integral equation \eqref{7eq: differential equation to integral equation}, we have
\begin{equation*}
\left| \varphi_{(M),\, l     } (z) - \widehat \varphi_{(M),\, l     } (z) \right| \leq \sum_{k \neq l     } S_{k} + S_l     ,
\end{equation*}
where
\begin{align*}
S_{k}  = \int_{\infty e^{i \omega(l     , k)}}^z | K_{k} (z, \zeta)| \left|\varphi_{(M),\, l     } (\zeta) \right| \left|d \zeta\right|, \hskip 5 pt
S_{l     } = \int_{\infty i \overline \xi_l     }^z | K_{l     } (z, \zeta)| \left|\varphi_{(M),\, l     } (\zeta) \right| \left|d \zeta\right|.
\end{align*}
With the observation that $|\zeta| \geq \sin \vartheta \cdot |z|$ for $z \in \BD_{\xi_l     } (C; \vartheta)$ and $\zeta$ on the integral contours given above, we may replace \eqref{7eq: bound of integral zeta -M-1} by the following
\begin{equation}\label{7eq: bound of integral zeta -M-1, 2}
\int_{\infty i \overline \xi_l     }^z \left| \zeta \right|^{-M-1} \left|d \zeta \right|, \ \int_{\infty e^{i \omega(l     , k)}}^z \left| \zeta \right|^{-M-1} \left|d \zeta \right| \leq c_2 |z|^{-M}, \hskip 10 pt  k \neq l     ,
\end{equation}
with  $c_2$ now also depending on $\vartheta$.

The bounds (\ref{7eq: bound of K ell' (z, zeta)}, \ref{7eq: precise error bound, 1}) of
$ K_{k} (z, \zeta) $ and $\varphi_{(M),\, l     } (z)$ along with \eqref{7eq: bound of integral zeta -M-1, 2} yield
\begin{align*}
S_{k} & \leq 2 c_2^2 \fC^{2M} |z|^{- \frac { n-1} 2} e^{\Re (in \xi_l      z)} \int_{\infty e^{i \omega(l     , k)}}^z |\zeta|^{-M-1} e^{\Re (in (\xi_{l     } - \xi_{k}) (\zeta - z))} \left|d \zeta\right|\\
& \leq 2 c_2^2 \fC^{2M} |z|^{- \frac { n-1} 2} e^{\Re (in \xi_l      z)} \int_{\infty e^{i \omega(l     , k)}}^z |\zeta|^{-M-1} \left|d \zeta\right|\\
& \leq 2 c_2^3 \fC^{2M} |z|^{- M - \frac { n-1} 2} e^{\Re (in \xi_l      z)}.
\end{align*}
Again, the second inequality follows from the fact that $\Re ((i \xi_{l     } - i \xi_{k}) \zeta)$ is nondecreasing on the integral contour.
Similarly, $S_{l     }$ has the same bound as $S_{k}$. Thus \eqref{7eq: error bounds 2} is proven and can be made precise as below
\begin{equation}\label{7eq: precise error bound, 2}
\left| \varphi_{(M),\, l     } (z) - \widehat \varphi_{(M),\, l     } (z) \right| \leq 2 n c_2^3 \fC^{2 M} |z|^{-M- \frac { n-1} 2} e^{\Re \lp i n \xi_{l     } z \rp}, \hskip 10 pt z \in \BD_{\xi_l     }(C; \vartheta).
\end{equation}

\vskip 5 pt

\subsubsection{Conclusion}

Restricting to the sector $\BS^{\pm}_{\xi_{l     }} \cap \{z : |z| > C \} \subset \BD_{\xi_l     } (C; \vartheta)$, with 
$\BS^{\pm}_{\xi_{l     }}$   replaced by $\BS_{\xi_{l     }}$ if $n = 2$, each $\varphi_{(M),\, l     }$ has an asymptotic representation a multiple of $\widehat \varphi_{k}$ for some $k$ according to Lemma \ref{lem: asymptotic, uniqueness} (1). Since $\Re (i \xi_{l     } z) < \Re (i \xi_{j} z)$ for all $ j \neq l     $, the bound (\ref{7eq: error bounds 1}) forces $k = l     $. Therefore, for any positive integer $M  $, $\varphi_{(M),\, l     }$ is identical with the unique solution $\varphi_{l     }$ of the differential system (\ref{7eq: differential equation, matrix form 2}) with asymptotic expansion $\widehat \varphi_{l     }$ on $\BS^{\pm}_{\xi_l     }$ (see Lemma \ref{lem: asymptotic, uniqueness}). Replacing $\varphi_{(M),\, l     }$ by $\varphi_{ l     }$ and absorbing the $M$-th term of $\widehat \varphi_{(M),\, l     }$ into the error bound,  
we may reformulate (\ref{7eq: precise error bound, 2}) as the following error bound for $\varphi_{l     }$
\begin{equation}\label{7eq: error bound 3}
\left| \varphi_{ l     } (z) - \widehat \varphi_{(M-1),\, l     } (z) \right| = O_{\,M,\, \vartheta,\, n} \lp \fC^{2 M} |z|^{-M- \frac { n-1} 2} e^{\Re \lp i n \xi_{l     } z \rp} \rp, \hskip 10 pt z \in \BD_{\xi_l     } (C; \vartheta).
\end{equation}
Moreover, in view of the definition of the sector $\BS'_{\xi_l     }(\vartheta)$ given in (\ref{7eq: S' xi (delta)}), we have
\begin{equation}\label{7eq: S'(delta) subset of D(C)}
\BS'_{\xi_l     }(\vartheta) \cap \left\{z : |z| >   \frac C {\sin \vartheta} \right \} \subset \BD_{\xi_l     } (C; \vartheta).
\end{equation}
Thus the following theorem is finally established by (\ref{7eq: error bound 3}) and  (\ref{7eq: S'(delta) subset of D(C)}).

\begin{thm}\label{thm: error bound}
	Let $\varsigma \in \{+, - \}$, $\xi \in \BX_n ( \varsigma)$, $0 < \vartheta < \frac 12 \pi $, $\BS'_{\xi} (\vartheta)$ be the sector defined as in {\rm (\ref{7eq: S' xi (delta)})}, and $M $ be a positive integer. Then there exists a constant $c$, depending only on $M$, $\vartheta$ and $n$, such that
	\begin{equation}\label{7eq: asymptotic of J(z; lambda, xi)}
	J (z; \ulambda; \xi) = e^{i n \xi z} z^{- \frac { n-1} 2} \lp \sum_{m=0}^{M-1} B_m (\ulambda; \xi) z^{-m} + O_{\,M,\, \vartheta,\, n} \lp \fC^{2 M} |z|^{- M } \rp \rp
	\end{equation}
	for all $z \in \BS'_{\xi} (\vartheta)$ such that $ |z| > c \fC^2 $. Similar asymptotic is valid for all the derivatives of $J (z; \ulambda; \xi)$, where the implied constant of the error estimate is allowed to depend on the order of the derivative.
\end{thm}

Finally, we remark that, since $B_m (\ulambda; \xi) z^{-m} $ is of size $O_{m,\, n} \lp \fC^{2 m} |z|^{- m } \rp$, the error bound in \eqref{7eq: asymptotic of J(z; lambda, xi)} is optimal, given that $\vartheta$ is fixed.

\section{Connections between Various Types of Bessel Functions} \label{sec: H-Bessel functions and K-Bessel functions revisited}

Recall from \S \ref{sec: Analytic continuations of the H-Bessel functions} that the asymptotic expansion in Theorem \ref{thm: asymptotic expansion} remains valid for the $H$-Bessel function $H^{\pm} (z; \ulambda)$ on the half-plane $\BH^{\pm} = \left\{ z : 0 \leq \pm \arg z \leq \pi \right\}$ (see \eqref{5eq: asymptotic expansion 1}).
With the observations that $H^{\pm} (z; \ulambda)$ satisfies the Bessel equation of sign $( \pm)^n$, that the asymptotic expansions of $\sqrt n (\pm 2 \pi i)^{- \frac { n-1} 2} H^{\pm} (z; \ulambda)$ and $J (z; \ulambda; \pm 1)$ have exactly the same form and the same leading term due to Theorem \ref{thm: asymptotic expansion} and Proposition \ref{prop: formal solution}, and that $\BS_{\pm 1} = \left\{ z : \lp \frac 1 2    - \frac 1 n \rp \pi   < \pm \arg z < \lp \frac 1 2     + \frac 1 n \rp \pi  \right \} \subset \BH^{\pm} $, 
Lemma \ref{lem: asymptotic, uniqueness} (2) implies the following theorem.

\begin{thm}\label{thm: bridge between H + (z; lambda) and J( z; lambda; 1)}
	We have
	$$ H^{\pm} (z; \ulambda) = n^{-\frac 1 2} (\pm 2 \pi i)^{ \frac { n-1} 2} J (z; \ulambda; \pm 1),$$
	and $B_m(\ulambda; \pm 1) = (\pm i)^{-m} B_m (\ulambda)$.
\end{thm}

\begin{rem}\label{rem: Stirling not working}
	The reader should observe that $\BS_{\pm 1} \cap \BR_+ = \O $, so Theorem {\rm \ref{thm: bridge between H + (z; lambda) and J( z; lambda; 1)}}  can not be obtained by the asymptotic expansion of $ H^{\pm} (x; \ulambda) $ on $\BR_+$ derived from Stirling's  asymptotic formula in Appendix {\rm \ref{appendix: asymptotic}} {\rm(}see Remark {\rm\ref{rem: appendix only for R+}}{\rm)}.

\end{rem}

\begin{rem}
	$B_m(\ulambda; \pm 1)$ can only be obtained from certain recurrence relations in \S {\rm \ref{sec: formal solutions at infinity}} from the differential equation aspect. On the other hand,  using the stationary phase method, \eqref{5eq: B mj} in \S {\rm \ref{sec: asymptotic expansions of H Bessel functions}} yields an explicit formula of $B_m (\ulambda)$. Thus Theorem {\rm \ref{thm: bridge between H + (z; lambda) and J( z; lambda; 1)}} indicates that the recurrence relations for $B_m(\ulambda; \pm 1)$ are actually solvable!
\end{rem}

As consequences of Theorem \ref{thm: bridge between H + (z; lambda) and J( z; lambda; 1)}, we can establish the connections between various Bessel functions, that is, $J  (z; \usigma, \ulambda)$, $J_l      ( z; \varsigma, \ulambda)$ and $J(z; \ulambda; \xi )$. Recall that $J  (z; \usigma, \ulambda)$ has already been expressed in terms of   $J_l      ( z; \varsigma, \ulambda)$ in Lemma \ref{lem: J(x; sigma; lambda) as a sum of J ell (x; sigma; lambda)}.

\subsection{\texorpdfstring{Relations between $J (z; \usigma, \ulambda)$ and $J(z; \ulambda; \xi )$}{Relations  between $J (z; \varsigma, \lambda)$ and $J(z; \lambda; \xi )$}} \label{sec: J (z; sigma; lambda) and Bessel functions of the second kind}

$J (z; \usigma, \ulambda)$ is equal to a multiple of $H^{\pm} \Big( e^{\pm \pi i \frac { n_{\mp}(\usigma)} n} z; \ulambda \Big)$ due to Lemma \ref{lem: J(z ; usigma; lambda) relations to H +}, whereas $J(z; \ulambda; \xi)$ is a multiple of  $J(\pm \xi z; \ulambda; \pm 1)$ in view of Lemma \ref{lem: J(z; lambda; xi) relations to J(xi z; lambda; 1)}. Furthermore, the equality, up to  constant, between $H^{\pm} (z; \ulambda)$ and  $J( z; \ulambda; \pm 1)$ has just been established in Theorem \ref{thm: bridge between H + (z; lambda) and J( z; lambda; 1)}. We then arrive at the following corollary.

\begin{cor}\label{cor: connetions, 1}
	Let $L_\pm (\usigma) = \{ l      : \varsigma_l      = \pm \}$ and $n_\pm (\usigma) = \left| L_\pm (\usigma) \right|$ be as in  Definition {\rm \ref{defn: signature}}. Let $c(\usigma, \ulambda) = e \lp \mp \frac {n-1} 8 \pm \frac {(n-1) n_{\pm}(\usigma)} {4 n} \mp \frac 1 2 {\sum_{l      \in L_{\pm}(\usigma)} \lambda_{l      }}   \rp$ and  $\xi (\usigma) = \mp e^{\mp \pi i \frac { n_{\pm} (\usigma)} n}$. Then 
	\begin{equation*} 
	J (z; \usigma, \ulambda) = \frac { ( 2\pi  )^{\frac { n-1} 2} c(\usigma, \ulambda)  } {\sqrt n}    J(z; \ulambda;  \xi (\usigma) ).
	\end{equation*}
	Here, it is understood that $\arg \xi (\usigma) = \frac {n_- (\usigma)} n \pi  = \pi -  \frac {n_+ (\usigma)} n \pi $.
\end{cor}


Corollary \ref{cor: connetions, 1} shows that $J (z; \usigma, \ulambda)$ should really be categorized in the class of Bessel functions of the second kind. 
Moreover, the asymptotic behaviours of the Bessel functions $J (z; \usigma, \ulambda)$ are classified by their signatures $(n_+ (\usigma), n_- (\usigma) )$. Therefore, $J (z; \usigma, \ulambda)$ is \textit{uniquely} determined by its  signature up to a constant multiple.

\subsection{Relations Connecting the Two Kinds of Bessel Functions}\label{sec: connection two kinds of Bessel functions}


From Lemma \ref{lem: J(z; lambda; xi) relations to J(xi z; lambda; 1)} and Theorem \ref{thm: bridge between H + (z; lambda) and J( z; lambda; 1)}, one sees that $J(z; \ulambda ; \xi)$ is a constant multiple of $H^+ (\xi z; \ulambda)$.
On the other hand, $H^+ (z; \ulambda)$ can be expressed in terms of Bessel functions of the first kind in view of Lemma \ref{lem: J(x; sigma; lambda) as a sum of J ell (x; sigma; lambda)}.
Finally, using Lemma \ref{lem: J ell (z; sigma; lambda) relations}, the following corollary is readily established.
\begin{cor}\label{cor: J(z; lambda; xi) and the J-Bessel functions}
	Let $\varsigma \in \{+,-\}$. If $\xi \in \BX_n ( \varsigma )$, then
	\begin{equation*} 
	\begin{split}
	J(z; \ulambda ; \xi) =  \sqrt n \lp - \frac {\pi i \xi } {2 } \rp^{\frac {n-1} 2}  
	\sum_{l      = 1}^{n}  \big( i  \xoverline \xi \big)^{n \lambda_l     } S_l      (\ulambda) J_l      ( z; \varsigma, \ulambda),
	\end{split}
	\end{equation*}
	with $S_l      (\ulambda) = 1 / \prod_{k \neq l     } \sin \lp \pi (\lambda_l      - \lambda_{k} ) \rp  $. According to our convention, we have $\lp   - i \xi  \rp^{  \frac {n-1} 2 } = e^{ \frac {n-1} 2 \lp - \frac 1 2 \pi i + i \arg \xi \rp }  $ and $\big( i  \xoverline \xi \big)^{n \lambda_l     } = e^{ \frac 1 2 \pi i n \lambda_l      - i n \lambda_l      \arg \xi }$. When $\ulambda $ is not generic, the right hand side should be replaced by its limit.
\end{cor}

We now fix an integer $a$ and let $\xi _{  j} = e^{ \pi i \frac { {  2 j + a - 2}  } n} \in \BX_n \lp (-)^a \rp $, with   $j = 1, ..., n$. 
It follows from Corollary \ref{cor: J(z; lambda; xi) and the J-Bessel functions} that
\begin{equation*}
X  (z; \ulambda) 
= \sqrt n \lp \frac  {\pi } {2  } \rp^{\frac {n-1} 2} e^{- \frac 1 4 \pi i (n-1)}  \cdot  D  V(\ulambda) S (\ulambda) E  (\ulambda) 
Y (z; \ulambda),
\end{equation*}
with
\begin{align*}
& X  (z; \ulambda)   = \big(J  \lp  z; \ulambda ; \xi _{  j}  \rp \big)_{j=1}^n, \hskip 10 pt   Y  (z; \ulambda) = \big(J_l       \lp  z; (-)^a , \ulambda \rp \big)_{l      =1}^n, \\ 
& D    = \mathrm{diag} \Big( \xi_{  j} ^{\frac {n-1} 2} \Big)_{j = 1}^n, \hskip 10 pt   E  (\ulambda)  = \mathrm{diag} \left( e^{ \pi i \lp \frac 1 2 n -  a \rp \lambda_l     } \right)_{l      = 1}^n,  \hskip 10 pt
S (\ulambda)   = \mathrm{diag}  \textstyle \big(    S_l      (\ulambda)  \big)_{l      = 1}^n, \\
& V(\ulambda)  = \lp e^{- 2 \pi i (j - 1) \lambda_{l     }} \rp_{j,\, l      = 1}^n.
\end{align*}
Observe that $V(\ulambda)$ is a {\it Vandermonde matrix}. 

\begin{lem}
	\label{8lem: Vandermonde}
	For an $n$-tuple $\ux = (x_1, ..., x_n) \in \BC^n$ we define the  Vandermonde matrix $V = \big(  x_{l     }^{ j - 1 } \big)_{j,\, l      = 1}^n$.
	For $d = 0, 1, ..., n-1$ and $m = 1, ..., n$, let $\sigma_{m, d} $ denote the elementary symmetric polynomial in $x_1, ..., \widehat {x_m}, ..., x_n$ of degree $d$, and let $\tau_m = \prod_{k \neq m} (x_m - x_k) $.
	If $\ux$ is generic in the sense that all the components of $\ux$ are distinct, then $V$ is invertible, and furthermore, the inverse of $V$ is $\lp (-)^{n-j } \sigma_{m, n-j}  \tau_m\- \rp_{m,\, j = 1}^n $.
\end{lem}
\begin{proof}[Proof of Lemma \ref{8lem: Vandermonde}]
	It is a well-known fact that $V$ is invertible whenever $\ux$ is generic.
	If one denotes by $w_{m, j}$ the $(m, j)$-th entry of $V\-$, then
	$$\sum_{j=1}^n w_{m, j} x_l     ^{j-1} = \delta_{m, l     }.$$
	The Lagrange interpolation formula implies the following identity of polynomials
	$$\sum_{j=1}^n w_{m, j} x^{j-1} = \prod_{k \neq m} \frac {x - x_k} {x_m - x_k}.$$
	Identifying the coefficients of $x^{j-1}$ on both sides yields the desired formula of $w_{m, j}$.
\end{proof}

\begin{cor}\label{8cor: inverse connection}
	Let $a$ be a given integer. For $j = 1, ..., n$ define $\xi _{  j} = e^{ \pi i \frac { {2 j  + a - 2}  } n}$. For $d = 0, 1, ..., n-1$ and $ l      = 1, ..., n$, let $\sigma_{l     , d} (\ulambda) $ denote the elementary symmetric polynomial in $e^{- 2 \pi i \lambda_1}, ..., \widehat {e^{- 2 \pi i \lambda_l     }}, ..., e^{- 2 \pi i \lambda_n}$ of degree $d$. Then 
	\begin{equation*}
	J_l      \big( z; (-)^a , \ulambda \big) = \frac {e^{\frac 3 4 \pi i (n-1)}} {\sqrt n (2\pi)^{\frac {n-1} 2}}   e^{\pi i \lp \frac 1 2 n +  a   - 2 \rp \lambda_l     } \sum_{j=1}^n (-)^{n-j} \xi_{  j} ^{- \frac {n-1} 2}  \sigma_{l     , n-j} (\ulambda) J \lp z; \ulambda; \xi_{  j} \rp.
	\end{equation*} 
\end{cor}

\begin{proof}
	Choosing $x_{l     } = e^{-2\pi i \lambda_l     }$ in Lemma \ref{8lem: Vandermonde}, one sees that if  $\ulambda$ is generic then the matrix $V(\ulambda)$ is invertible and its inverse is given by $$ \lp (-2i)^{1-n} \cdot (-)^{n-j }  \sigma_{l     , n-j} (\ulambda) e^{\pi i (n-2) \lambda_l     } S_{l     } (\ulambda) \rp_{l     ,\, j = 1}^n .$$
	Some straightforward calculations then complete the proof.
\end{proof}

\begin{rem}
	In view of Proposition {\rm \ref{prop: Classical Bessel functions}}, Remark {\rm \ref{rem: n=2 J-Bessel function}} and {\rm \ref{rem: n=2, asymptotics}}, when $n=2$, Corollary {\rm \ref{cor: J(z; lambda; xi) and the J-Bessel functions}} corresponds to the connection formulae {\rm (}\cite[3.61(5, 6), 3.7 (6)]{Watson}{\rm )}, 
	\begin{align*} 
	&  	H^{(1)}_\nu (z) = \frac {J_{-\nu} (z) - e^{- \pi i \nu} J_\nu (z) }{i \sin ( \pi \nu)}, \hskip 10 pt 
	H^{(2)}_\nu (z) = \frac { e^{\pi i \nu} J_\nu (z) - J_{-\nu} (z)}{i \sin ( \pi \nu)},\\
	&  K_{\nu} (z) =  \frac {\pi \lp I_{-\nu} (z) - I_\nu (z) \rp} {2 \sin (\pi \nu)}, \hskip 27 pt 
	\pi I_\nu (z) - i e^{\pi i \nu} K_\nu (z) = 
	\frac {\pi i \lp e^{-\pi i \nu} I_{ \nu} (z) - e^{\pi i \nu} I_{-\nu} (z) \rp} {2 \sin (\pi \nu)},
	\end{align*}
	whereas Corollary {\rm \ref{8cor: inverse connection}}, with $a = 0$ or $1$, amounts to the formulae  {\rm (}see \cite[3.61(1, 2), 3.7 (6)]{Watson}{\rm )}
	\begin{align*}
	& J_\nu (z) = \frac {H_\nu^{(1)} (z) + H_\nu^{(2)} (z)} 2, \hskip 61 pt 
	J_{-\nu} (z) =  \frac {e^{\pi i \nu} H_\nu^{(1)} (z) + e^{-\pi i \nu} H_\nu^{(2)} (z) } 2, \\
	& I_{\nu} (z) = \frac {i e^{ \pi i \nu} K_{\nu} (z) + \lp \pi I_{\nu} (z) - i e^{\pi i \nu} K_{\nu} (z)  \rp } {\pi},  
	I_{- \nu} (z) = \frac {i e^{-\pi i \nu} K_{\nu} (z) + \lp \pi I_{\nu} (z) - i e^{\pi i \nu} K_{\nu} (z)  \rp } {\pi} .
	\end{align*}
\end{rem}


\section{$H$-Bessel Functions and $K$-Bessel Functions, II}\label{sec: H-Bessel functions and K-Bessel functions, II}

In this concluding section, we apply Theorem  \ref{thm: error bound}  to  improve the results in \S \ref{sec: Bessel functions of K-type and H-Bessel functions} on the asymptotics of Bessel functions $J (x; \usigma, \ulambda)$ 
for $x \ggg \mathfrak C^2$. 

\subsection{Asymptotic Expansions of $H$-Bessel Functions} \label{8sec: Asymptotics of H-Bessel functions}

The following proposition is a direct  consequence of Theorem  \ref{thm: error bound} and \ref{thm: bridge between H + (z; lambda) and J( z; lambda; 1)}.

\begin{prop}\label{prop: improved asymptotic}
	Let $0 < \vartheta < \frac 1 2 \pi  $.
	
	{ \rm (1).} Let $M$ be a positive integer. We have 
	\begin{equation} \label{7eq: asymptotic expansion H pm, improved}
	\begin{split}
	H^{\pm}  (z; \ulambda) = n^{- \frac 1 2}  &  (\pm 2 \pi i)^{ \frac { n-1} 2} e^{ \pm i n z} z^{ - \frac { n-1} 2} \\
	& \lp \sum_{m=0}^{M-1} ( \pm i )^{ - m} B_{m} (\ulambda) z^{- m} + O_{\,M, \,\vartheta,\, n} \left( \mathfrak C^{2 M} |z|^{-M} \right) \rp,
	\end{split}
	\end{equation}
	for all $z \in  \BS'_{\pm 1} (\vartheta)$ such that $|z| \ggg_{\,M,\, \vartheta,\, n} \fC^2$.
	
	{ \rm (2).} Define $W^{\pm} (z;  \ulambda) = \sqrt n (\pm 2 \pi i)^{- \frac { n-1} 2} e^{\mp i n z} H^{\pm}  (z; \ulambda)$. Let $M - 1 \geq j \geq 0$. We have
	\begin{equation} 
	W^{\pm, (j)} (z;  \ulambda) = z^{- \frac { n-1} 2} \lp \sum_{m = j }^{M-1} ( \pm i )^{j - m} B_{m, j} (\ulambda) z^{- m } + O_{\,M,\, \vartheta, j ,\, n} \left( \fC^{2 M - 2j} |z|^{-M}\right) \rp,
	\end{equation}
	for all $z \in  \BS'_{\pm 1} (\vartheta)$  such that $|z| \ggg_{\,M,\, \vartheta,\, n} \fC^2$.
\end{prop}

Observe that \begin{equation*}
\begin{split}
& \BH^{\pm} = \left\{ z \in \BC : 0 \leq \pm \arg z \leq \pi \right\} \\
\subset &\ \BS'_{\pm 1} (\vartheta) = \left\{ z \in \BU : - \lp \frac 1 2 - \frac 1 n \rp \pi - \vartheta < \pm \arg z < \lp \frac 3 2 + \frac 1 n \rp \pi + \vartheta \right\}.
\end{split}
\end{equation*} 
Fixing $\vartheta$ and restricting to the domain $\left \{ z \in \BH^\pm : |z| \ggg_{\,M,\, n} \fC^2 \right\}$, Proposition \ref{prop: improved asymptotic} improves Theorem \ref{thm: asymptotic expansion}.

\subsection{Exponential Decay of $K$-Bessel Functions}

Now suppose that $J (z; \usigma, \ulambda)$ is a $K$-Bessel function so that $0 < n_\pm (\usigma) < n$. Since $\BR_+ \subset \BS'_{\xi (\usigma)} (\vartheta)$, Corollary \ref{cor: connetions, 1} and Theorem \ref{thm: error bound} imply that $J (x; \usigma, \ulambda)$, as well as  all its derivatives, is not only a Schwartz function at infinity, which was shown in Theorem \ref{thm: Bessel functions of K-type}, but also a function of exponential decay on $\BR_+$. 

\begin{prop}\label{8prop: K}
	If $J (x; \usigma, \ulambda)$ is a $K$-Bessel function, then for all $ x \ggg_{\, n} \fC^2$
	\begin{equation*}
	J^{(j)} (x; \usigma, \ulambda) \lll_{j,\, n} x^{-\frac { n-1} 2} e^ { {- \pi  \Im \Lambda (\usigma, \ulambda)  - n I (\usigma) x}  }, 
	\end{equation*}
	where  $\Lambda (\usigma, \ulambda) = \mp \sum_{l      \in L_\pm (\usigma)} \lambda_{l     }$ and $I (\usigma) = \Im \xi (\usigma) = \sin \big( \frac { n_\pm (\usigma)} n \pi \big) > 0$. In particular, we have  
	\begin{equation*}
	J^{(j)} (x; \usigma, \ulambda) \lll_{j,\, n} x^{-\frac { n-1} 2} e^ { { \pi  \mathfrak I    - n \sin \lp \frac 1 n \pi \rp x}  }, 
	\end{equation*}
	for all $K $-Bessel functions $J  (x; \usigma, \ulambda)$ with given $\ulambda$, where $ \mathfrak I   = \max \left\{ \left| \Im \lambda_l      \right| \right\} $.

\end{prop}


\begin{subappendices}
	
\renewcommand{\thesection}{B}

\section{An Alternative Approach to Asymptotic Expansions}\label{appendix: asymptotic} 

When $n=3$, the application of Stirling's asymptotic formula in deriving the asymptotic expansion of  a Hankel transform was first found in \cite[\S 4]{Miller-Wilton}. The asymptotic was later formulated more explicitly in \cite[Lemma 6.1]{XLi}, where the author attributed the arguments in her proof to \cite{Ivic}. Furthermore, using similar ideas as in \cite{Miller-Wilton},  \cite{Blomer} simplified the proof of \cite[Lemma 6.1]{XLi} (see the proof of \cite[Lemma 6]{Blomer}).  
This method using Stirling's asymptotic formula is however the only known approach so far in the literature.

Closely following \cite{Blomer}, we shall  prove the asymptotic expansions of $H$-Bessel functions $H^{\pm} (x; \ulambda)$ of any rank $n$ by means of  Stirling's asymptotic formula.

\vskip 5 pt

From   (\ref{1def: G pm (s)}, \ref{1def: G(s; sigma; lambda)}, \ref{3eq: definition of J (x; sigma)}) we have
\begin{equation}\label{10def: H pm via Mellin inversion}
H^{\pm} (x; \ulambda) = \frac 1  {2 \pi i}\int_{\EuScript C} \lp \prod_{l      = 1}^{n} \Gamma (s - \lambda_l      ) \rp e \left( \pm \frac {n s } 4 \right) x^{- n s} d s.
\end{equation}
In view of the condition $\sum_{l     = 1}^n \lambda_l     = 0$,  Stirling's asymptotic formula yields
\begin{equation*}
\prod_{l      = 1}^{n} \Gamma (s - \lambda_l      ) = n^{-ns}   \Gamma \lp n s - \frac {n-1} 2 \rp   \exp \lp \sum_{m=0}^{M } C_m(\ulambda) s^{-m} \rp \lp 1+ R_{M+1} (s) \rp 
\end{equation*}
for some constants $C_m(\ulambda)$ and remainder term $R_{M+1} (s) = O_{\ulambda,\, M,\, n} \lp |s|^{-M-1} \rp$. Using the Taylor expansion for the exponential function and some straightforward algebraic manipulations, the right hand side can be written as
\begin{equation*}
n^{-ns}   \sum_{m=0}^{M } \widetilde C_m(\ulambda) \Gamma \lp n s - \frac {n-1} 2 - m \rp \lp 1+ \widetilde R_{M+1, \,m} (s) \rp
\end{equation*}
for certain constants $\widetilde C_m(\ulambda)$ and   similar functions $\widetilde R_{M+1,\, m}(s) = O_{\ulambda,\, M,\, n} \lp |s|^{-M-1} \rp$. Suitably choosing the contour $\EC$, it follows from \eqref{2eq: n = 1, Mellin inversion} that
\begin{equation*}
\begin{split}
&\frac 1  {2 \pi i}\int_{\EuScript C} \Gamma \lp n s - \frac {n-1} 2 - m \rp e \left( \pm \frac {n s } 4 \right) (n x)^{- n s} d s \\
= &\, \frac {e \lp \pm \lp \frac {n-1} 8  + \frac 14 m   \rp \rp} {n (n x)^{\frac {n-1} 2 + m}} \cdot \frac 1  {2 \pi i}\int_{ n \EuScript C - \frac {n-1} 2 - m} \Gamma (s) e \lp \pm \frac s 4 \rp (nx)^{-s} ds \\
= & \, \frac {  (\pm i)^{ \frac {n-1} 2 + m} } {  n^{\frac {n+1} 2 + m} } \cdot \frac { e^{\pm i n x} } {  x^{ \frac {n-1} 2 + m} }.
\end{split}
\end{equation*}
As for the error estimate,  let us assume $x \geqslant 1$. Insert the part containing $\widetilde R_{M+1,\, m}(s)$ into \eqref{10def: H pm via Mellin inversion} and shift the contour to the vertical line of real part $ \frac 1 n (M - \frac 1 2) + \frac 1 2 $. By Stirling's asymptotic formula, the integral remains absolutely convergent   and is of size  $O_{\ulambda,\, M,\, n} \big( x^{- M  - \frac {n-1} 2 }\big)$. Absorbing the last main term into the error, 
we arrive at the following asymptotic expansion
\begin{equation}\label{10eq: asymptotic expansion of H pm (x; lambda)}
\begin{split}
H^{\pm}  (x; \ulambda)  = e^{ \pm i n x} x^{ - \frac { n-1} 2} 
\lp \sum_{m=0}^{M-1} C^{\pm}_{m} (\ulambda) x^{- m} + O_{ \ulambda,\, M,\, n} \left( x^{- M } \right) \rp, \hskip 10 pt x \geqslant 1,
\end{split}
\end{equation}
where $ C^{\pm}_{m} (\ulambda)$ is some constant depending on $\ulambda$.

\begin{rem}\label{rem: appendix only for R+}
	For the analytic continuation  $H^{\pm}  (z; \ulambda)$, we have the Barnes type  integral representation   as in \S \ref{sec: Bessel kernel J(x; sigma, lambda)}. This however does not yield an asymptotic expansion of $H^{\pm}  (z; \ulambda)$ by the above method. The obvious issue is with the error estimate, as $\left| z^{-ns} \right|$ is unbounded on the integral contour if $|z| \ra \infty$. 
\end{rem}

Finally, we make some comparisons between the three asymptotic expansions \eqref{10eq: asymptotic expansion of H pm (x; lambda)}, \eqref{5eq: asymptotic expansion 1} and \eqref{7eq: asymptotic expansion H pm, improved} obtained from 
\begin{itemize}
	\item[-] Stirling's asymptotic formula,
	\item[-] the method of stationary phase,
	\item[-] the asymptotic method of ordinary differential equations.
\end{itemize}
Recall that  $\mathfrak C =  \max  \left\{ | \lambda_l      | \right \} + 1$, $\mathfrak R = \max  \left\{ |\Re \lambda_l      | \right \}$. Firstly, the admissible domains of these asymptotic expansions are
\begin{equation*}
\begin{split}
& \{ x \in \BR_+ : x \geq 1 \},\\
& \left\{ z \in \BC : |z| \geq \mathfrak C, \ 0 \leq \pm \arg z \leq \pi \right \}, \\
& \left\{ z \in \BU : |z| \ggg_{\,M,\, \vartheta,\, n} \mathfrak C^2, \ - \lp \frac 1 2 - \frac 1 n \rp \pi - \vartheta < \pm \arg z < \lp \frac 3 2 + \frac 1 n \rp \pi + \vartheta \right \},
\end{split}
\end{equation*}
respectively. The range of argument is extending while that of modulus is reducing. Secondly, the  error estimates are
\begin{equation*}
O_{\ulambda,\, M, \,  n} \left( x^{- M - \frac { n-1} 2} \right),\
O_{\,\mathfrak R,\, M,\, n} \left( \mathfrak C^{2 M} |z|^{-M} \right),\
O_{M,\, \vartheta,\, n} \left( \mathfrak C^{2 M} |z|^{- M - \frac { n-1} 2 } \right),
\end{equation*}
respectively. Thus, in the error estimate, the dependence of the implied constant  on $\ulambda$ is improving in all aspects.

\end{subappendices}

%
%
%

\setcounter{section}{13}

\chapter{Bessel Kernels}\label{chap: Bessel Kernels}

In this chapter, we shall return to the study of Bessel kernels, with emphasis on two connection formulae and the asymptotic expansion for the complex Bessel kernel $J_{(\umu, \um)} (z) $. 

\section{\texorpdfstring{The Asymptotic of $J_{(\ulambda, \udelta)} (x)$}{The Asymptotic of $J_{(\lambda, \delta)} (x)$}}

According to \eqref{3eq: Bessel kernel, R, connection}, $ J_{(\ulambda, \udelta)} (\pm x) $ is a combination of $J \big(2 \pi x^{\frac 1 n}; \usigma, \ulambda \big) $, and hence its asymptotic   follows immediately from Theorem \ref{thm: Bessel functions of K-type} and \ref{thm: asymptotic expansion} in \S \ref{sec: Bessel functions of K-type and H-Bessel functions} as well as Proposition \ref{prop: improved asymptotic} and \ref{8prop: K} in \S \ref{sec: H-Bessel functions and K-Bessel functions, II}. For convenience of reference, we record the asyptotic  of $J_{(\ulambda, \udelta)} (\pm x)$ in the following theorem.

\begin{thm}\label{thm: asymptotic Bessel kernel, 1}
	Let $(\ulambda, \udelta) \in \BL^{n-1} \times (\BZ/2 \BZ)^{n}$.  Put $\mathfrak C (\ulambda) = \max \left\{|  \lambda_l     |  \right\} + 1$, $\mathfrak R (\ulambda) = \max \left\{|\Re \lambda_l     | \right\}$ and $\mathfrak I (\ulambda) = \max \left\{|\Im \lambda_l     | \right\}$.  Let $M \geq   0 $. 
	Then, for $x > 0$,  we may write  
	\begin{align*}
	&J_{(\ulambda, \udelta)} \left(x^n \right)  =  \sum_{\pm}   \frac  { { (\pm)^{|\udelta|     }   e \lp   \pm   \big(n x  +   \frac {n-1} 8 \big)  \rp }} {n^{\frac 1 2} x^{  \frac {n-1} 2} }  W_{ \ulambda }^{\pm} (x) + E^+_{(\ulambda, \udelta)} (x), \\
	& J_{(\ulambda, \udelta)} \left( - x^n \right)   = E^-_{(\ulambda, \udelta)} (x),  
	\end{align*}
	if $n$ is even, and
	\begin{align*}
	J_{(\ulambda, \udelta)} \left(\pm x^n \right) & =   \frac  { { (\pm)^{|\udelta|     }   e \lp   \pm   \big(n x  +   \frac {n-1} 8 \big)  \rp }} {n^{\frac 1 2} x^{  \frac {n-1} 2} }  W_{ \ulambda }^{\pm} (x) + E^\pm_{(\ulambda, \udelta)} (x),   
	\end{align*}
	if $n$ is odd, such that
	\begin{align*}
	W_{ \ulambda }^{\pm } (x) =    \sum_{m= 0}^{M-1}  B^{\pm}_{m } (\ulambda) x^{-   m  }  
	+ O_{\,\mathfrak R (\ulambda), \, M , \, n} \left( \fC (\ulambda)^{2 M } x^{-  M + \frac {n-1} 2  }\right), 
	\end{align*} 
	and
	\begin{equation*}
	E^{\pm }_{(\ulambda, \udelta)} \left( x \right) =  O_{\,\mathfrak R (\ulambda),\, M,\,  n}  \lp    \fC (\ulambda)^{M} x^{-   M } \rp,
	\end{equation*}
	for $x \geq \mathfrak C (\ulambda)$. Furthermore,  for $x \ggg_{\,M,\, n} \mathfrak C (\ulambda)^{2 }$, we have
	\begin{equation*}
	W_{ \ulambda }^{\pm } (x) =    \sum_{m= 0}^{M-1}  B^{\pm}_{m } (\ulambda) x^{-   m   }  
	+ O_{  M , \, n} \left( \fC (\ulambda)^{2 M } x^{- M   }\right), 
	\end{equation*} 
	and
	\begin{equation*}
	E^{\pm }_{(\ulambda, \udelta)} \left( x \right) =  O_{ n} \lp x^{- \frac {n-1} {2 } } \exp \lp { \pi   \mathfrak I (\ulambda) - 2 \pi n \sin \lp \tfrac 1 n \pi \rp  x } \rp \rp.
	\end{equation*} 
		With the notations in Theorem {\rm \ref{thm: asymptotic expansion}}, we have $W_{ \ulambda }^{\pm } (x) = (2 \pi x)^{\frac {n-1} 2 } W^{\pm} (2 \pi x; \ulambda)$ and $B^{\pm}_{m } (\ulambda) = (\pm 2 \pi i)^{ -m} B_{m } (\ulambda)$.

\end{thm}

\section{\texorpdfstring{Two Connection Formulae for $J_{(\umu, \um)} (z) $}{Two Connection Formulae for $J_{(\mu, m)} (z)$}}\label{sec: two connection formulae for J mu m}

In this section, we shall prove two formulae for $J_{(\umu, \um)} (z)$ in connection with the two kinds of Bessel functions of the same rank and {positive sign}. These Bessel functions arise as solutions of Bessel differential equations in \S \ref{sec: Bessel equations} and their relations have been unraveled in \S \ref{sec: connection two kinds of Bessel functions}. Our motivation is based on the following self-evident identity for the rank-one example
$$ e(z + \overline z) = e(z) e(\overline z). $$


\subsection{The First Connection Formula}

For  $\varsigma \in \{+, -\}$, $\ulambda \in \BC^n$ and $l      = 1,..., n$, we recollect the definition of the Bessel function of the first kind  $J_{l     } (z; \varsigma, \ulambda)$ by the following series of ascending powers of $z$ (see \S \ref{sec: Bessel functions of the first kind})
\begin{equation}\label{4eq: Bessel of the first kind}
J_{l     } (z; \varsigma, \ulambda) = \sum_{m=0}^\infty \frac { (\varsigma i^n)^m  z^{ n (- \lambda_{l      } + m)} } { \prod_{k = 1}^n \Gamma \lp  { \lambda_{ k } - \lambda_{l     }}  + m + 1 \rp}, \hskip 10 pt z \in \BU.
\end{equation}
Since the definition \eqref{4eq: Bessel of the first kind} is valid for any $\ulambda \in \BC^n$, the assumption $\ulambda \in \BL^{n-1}$ that we imposed in \S \ref{sec: Bessel equations} is rather superfluous. Also, we have the following formula in the same fashion as \eqref{4eq: normalize J} in Lemma \ref{3lem: normalize J(x; sigma, lambda)},
\begin{equation}
J_{l     } \left(z; \varsigma, \ulambda - \lambda \ue^n \right) = z^{  n \lambda} J_{l     } (z; \varsigma, \ulambda).
\end{equation}

\begin{thm}\label{4thm: connection formula}
	Let $(\umu, \um) \in \BL^{n-1} \times \BZ^n$. We have
	\begin{equation}\label{4eq: connection formula}
	\begin{split}
	J_{(\umu, \um)} (z) 
	= \left(2 \pi^2\right)^{n-1}   \sum_{l      = 1}^n S_{l     } (\umu, \um) { J_l      \big( 2\pi z^{\frac 1 n}  ; +, \umu + \tfrac 1 2 \um \big) 
		J_l      \big( 2\pi \overline z^{\frac 1 n}  ; +, \umu - \tfrac 1 2 \um \big)  },
	\end{split}
	\end{equation}
	with $S_{l     } (\umu, \um) = \prod_{k \neq l     } (\pm i)^{m_l      - m_k} / \sin \lp \pi \lp   \mu_{l     } - \mu_k \pm \frac 1 2 (m_l      - m_k) \rp \rp $.
	Here, $z^{\frac 1 n}$ is the principal  $n$-th root of $z$, that is $\lp {x e^{i \phi}} \rp^{\frac 1 n} = x^{\frac 1 n} e^{\frac 1 n i \phi}$. The expression on the right hand side of \eqref{4eq: connection formula} is independent on the choice of the argument of $z$ modulo $2 \pi$. It is understood that the right hand side should be replaced by its limit if $ (\umu, \um)$ is not generic with respect to the order $\preccurlyeq$ on $\BC \times \BZ$ in the sense of Definition {\rm \ref{3defn: ordered set}}.
\end{thm}

\begin{proof}
	Recall from (\ref{1def: G m (s)}, \ref{1def: G (mu, m)}, \ref{3def: Bessel kernel j mu m}, \ref{2eq: Bessel kernel over C, polar}) that 
	\begin{equation*}
	\begin{split}
	J_{(\umu, \um)} \lp x e^{i \phi} \rp = &  \, (2 \pi)^{n-1} \sum_{m = - \infty}^\infty i^{\sum_{k = 1}^n |m_k + m| }  e^{i m \phi} \\
	& \frac 1 {2\pi i} \int_{\EC_{\left(\umu, \um + m \ue^n\right)}} \lp \prod_{l     =1}^n  \frac {\Gamma \lp s - \mu_l      + \frac 1 2 {|m_l     +m|}   \rp} {\Gamma \lp 1 - s + \mu_l      + \frac 1 2 {|m_l     +m|}   \rp}  \rp 
	\lp (2 \pi)^n x \rp^{-2s} d s.
	\end{split}
	\end{equation*}
	Assume first that $(\umu, \um)$ is generic  with respect to the order $\preccurlyeq$ on $\BC \times \BZ$. The sets  of poles of the gamma factors in the above integral are $\left\{ \mu_l      - \frac 1 2 {|m_l     +m|}  - \alpha \right \}_{\alpha \in \BN}$, $l      = 1,..., n$. With the generic assumption, the integrand has only \textit{simple} poles. We left shift the integral contour of each integral in the series and collect the residues from these poles. The contribution from the residues at the poles of the $l     $-th gamma factor is the following absolutely convergent double series,
	\begin{equation*}
	\begin{split}
	(2 \pi)^{n-1} \sum_{m = -\infty}^\infty  i^{\sum_{k = 1}^n |m_k + m| } e^{i m \phi}  \sum_{\alpha = 0}^\infty &  \frac { (-)^\alpha  \lp (2 \pi)^n x \rp^{- 2 \mu_l      + |m_l     +m| + 2 \alpha}} {\alpha ! (\alpha + |m_l     +m|)!  } \\
	&   \prod_{k \neq l     }  \frac {\Gamma \lp \mu_l      - \mu_k - \frac 1 2 ( {|m_l     +m| - |m_k + m|} ) - \alpha \rp} {\Gamma \lp 1 - \mu_l      + \mu_k + \frac 1 2 ( {|m_l     +m| + |m_k + m| } ) + \alpha \rp}    .
	\end{split}
	\end{equation*}
	Euler's reflection formula of the Gamma function converts this into
	\begin{align*}
	& \frac {\left(2 \pi^2\right)^{n-1}  } {\prod_{k \neq l     } i^{  m_l      - m_k} \sin \lp \pi \lp   \mu_{l     } - \mu_k - \frac 1 2 (m_l      - m_k) \rp \rp} \\
	&   \sum_{m = -\infty}^\infty  i^{n |m_l     +m|}  e^{i m \phi} \sum_{\alpha = 0}^\infty  \frac { (-)^{n \alpha}  \lp (2 \pi)^n x \rp^{- 2 \mu_l      + |m_l     +m| + 2 \alpha}}   {\prod_{k = 1}^n \prod_{\pm} \Gamma \lp 1 - \mu_l      + \mu_k + \frac 1 2 ( {|m_l     +m| \pm |m_k + m| } ) + \alpha \rp } .
	\end{align*}
	We now interchange the order of summations, truncate the sum over $m$ between $- m_l     $ and $- m_l      + 1$ and make the change of indices $\beta = \alpha + |m_l     +m|$. Observe that, regardless of the value of $m_k$, 
	up to ordering, $\frac 1 2 ( {|m_l     +m| + |m_k + m| } )$ and $\frac 1 2 ( {|m_l     +m| - |m_k + m| } )$ equal to $ \frac 1 2 (m_l      - m_k)$ and  $|m_l      + m | - \frac 1 2 (m_l      - m_k)$ when $m \geqslant - m_l      + 1$, or equal to $ - \frac 1 2 (m_l      - m_k)$ and  $|m_l      + m | + \frac 1 2 (m_l      - m_k)$ when  $m \leqslant - m_l      $. Hence
	the double series in the expression above turns into
	{  \begin{align*}
		\sum_{\alpha = 0}^\infty \sum_{\beta = \alpha + 1}^\infty  \frac
		{ i^{n (\alpha + \beta)} e^{i (\beta - \alpha -m_l     ) \phi} \lp (2 \pi)^n x \rp^{- 2 \mu_l      + \alpha + \beta}} 
		{ \prod_{k =1}^n   \Gamma \lp 1 -  \mu_l      + \mu_k + \frac 1 2 (m_l      - m_k) + \alpha   \rp  \Gamma \lp 1 -  \mu_l      + \mu_k - \frac 1 2 (m_l      - m_k) + \beta   \rp   } &  \\
		+ \sum_{\alpha = 0}^\infty \sum_{\beta = \alpha }^\infty  \frac
		{ i^{n (\alpha + \beta)} e^{i (\alpha - \beta - m_l     ) \phi} \lp (2 \pi)^n x \rp^{- 2 \mu_l      + \alpha + \beta}} 
		{ \prod_{k = 1}^n   \Gamma \lp 1 -  \mu_l      + \mu_k - \frac 1 2 (m_l      - m_k) + \alpha   \rp  \Gamma \lp 1 -  \mu_l      + \mu_k + \frac 1 2 (m_l      - m_k) + \beta   \rp } &, 
		\end{align*} }
	which is then equal to
	$$\sum_{\alpha = 0}^\infty \sum_{\beta = 0}^\infty  \frac
	{ i^{n (\alpha + \beta)} e^{i (\beta - \alpha -m_l     ) \phi} \lp (2 \pi)^n x \rp^{- 2 \mu_l      + \alpha + \beta}} 
	{ \prod_{k  = 1}^n   \Gamma \lp 1 -  \mu_l      + \mu_k + \frac 1 2 (m_l      - m_k) + \alpha   \rp  \Gamma \lp 1 -  \mu_l      + \mu_k - \frac 1 2 (m_l      - m_k) + \beta   \rp   }. $$
	This double series is clearly independent on the choice of $\phi$ modulo $2 \pi$, and splits exactly as the product  $$J_l      \big( 2\pi x^{\frac 1 n} e^{\frac 1 n i \phi}; +, \umu + \tfrac 1 2 \um \big) 
	J_l      \big( 2\pi x^{\frac 1 n} e^{- \frac 1 n i \phi}; +, \umu - \tfrac 1 2 \um \big) .$$ 
	This proves \eqref{4eq: connection formula} in the case when $(\umu, \um)$ is generic. As for the nongeneric case, one just passes  to the limit.
\end{proof}

\subsection{The Second Connection Formula} \label{sec: second connection}

According to \S \ref{sec: asymptotic expansions for the Bessel equations}, {Bessel functions of the second kind} are solutions of Bessel equations defined according to their asymptotics at infinity. 
To remove the restriction $\ulambda \in \BL^{n-1}$ on the definition of $J(z; \ulambda ; \xi) $, with $ \xi $ a $2 n$-th root of unity, we simply impose the additional condition
\begin{equation}
J \left(z;  \ulambda - \lambda \ue^n ; \xi \right) = z^{  n \lambda} J(z; \ulambda ; \xi).  
\end{equation}

\begin{rem}
	Let $ \xi $ be an $n$-th root of $\varsigma 1$.  We may also use the following formula as an alternative definition of $J \left(z;  \ulambda  ; \xi \right)$ {\rm(}compare  Corollary {\rm \ref{cor: J(z; lambda; xi) and the J-Bessel functions})}
	\begin{equation} \label{5eq: connection 1st and 2nd, 1}
	J(z; \ulambda ; \xi) =  \sqrt n \lp \frac {\pi  } {2 }  \rp^{\frac {n-1} 2}   (- i\xi)^{\frac {n-1} 2 + |\ulambda|} \\
	\sum_{l      = 1}^{n}  \big( i  \overline \xi \big)^{n \lambda_l     } S_l      (\ulambda) J_l      ( z; \varsigma , \ulambda).
	\end{equation}
	where   $\lp   - i \xi  \rp^{\frac {n-1} 2 + |\ulambda|} = e^{ \lp \frac {n-1} 2 + |\ulambda| \rp \lp - \frac 1 2 \pi i + i \arg \xi \rp }  $ and $\big( i  \overline \xi \big)^{n \lambda_l     } = e^{ \frac 1 2 \pi i n \lambda_l      - i n \lambda_l      \arg \xi }$ by convention, and $S_l      (\ulambda) = 1/ \prod_{k \neq l     } \sin \lp \pi (\lambda_l      - \lambda_{k} ) \rp $.
\end{rem}

Given an integer $a$, define $\xi _{a, j} = e^{ 2 \pi i \frac { {  j  + a  -  1}  } n}$, $j = 1, ..., n$. For $d = 0, 1, ..., n-1$ and $ l      = 1, ..., n$, let $\sigma_{l     , d} (\ulambda) $ denote the elementary symmetric polynomial in $e^{- 2 \pi i \lambda_1}, ..., \widehat {e^{- 2 \pi i \lambda_l     }}, ...$, $e^{- 2 \pi i \lambda_n}$ of degree $d$. It follows from  Corollary \ref{8cor: inverse connection} that
\begin{equation}\label{5eq: connection 1st and 2nd, 2}
\begin{split}
J_l      ( z; + , \ulambda ) = \frac {e^{\frac 3 4 \pi i \lp  (n-1) +  2 |\ulambda| \rp}} {\sqrt n (2\pi)^{\frac {n-1} 2}}  & e^{\pi i \lp \frac 1 2 n + 2 a - 2 \rp \lambda_l     } \\
& \sum_{j=1}^n (-)^{n-j} \xi_{a, j} ^{- \frac {n-1} 2 - |\ulambda|}  \sigma_{l     , n-j} (\ulambda) J \lp z; \ulambda; \xi_{a, j} \rp.
\end{split}
\end{equation} 
In addition, 
we shall require the definition 
$$\tau_{l     } (\ulambda) = \prod_{k \neq l     }    \lp e^{-2\pi i \lambda_m} - e^{-2\pi i \lambda_k} \rp = (- 2i)^{n-1} e^{-\pi i |\ulambda|}  e^{- \pi i (n-2) \lambda_l      }  \prod_{k\neq l     }  \sin \lp \pi (\lambda_l      - \lambda_{k} ) \rp.$$ 
Note that the product of sines is  equal to $1/ S_l (\ulambda)$.

We introduce the column vectors of the two kinds of Bessel functions 
\begin{align*}
X (z; \ulambda) =  \big(  J_l      (z; +, \ulambda)  \big)_{l     =1}^n, \hskip 10 pt   Y_a (z; \ulambda) =  \big(  J (z;   \ulambda; \xi_{a, j})  \big)_{j=1}^n,
\end{align*}
and the matrices
\begin{align*}
\varSigma (\ulambda) & = \big( \sigma_{l     , n-j} (\ulambda) \big)_{l     , j=1}^n, \\
E_a (\ulambda) & = \mathrm{diag}\lp e^{\pi i \lp \frac 1 2 n + 2 a - 2 \rp \lambda_l     } \rp_{l      = 1}^n, \hskip 10 pt D_a (\ulambda) = \mathrm{diag}\Big(  (-)^{n-j} \xi_{a, j}^{- \frac {n-1} 2 - |\ulambda| } \Big)_{j = 1}^n.
\end{align*}
Then the formula \eqref{5eq: connection 1st and 2nd, 2} may be written as
\begin{equation}\label{5eq: connection 1st and 2nd, matrix form}
X  (z; \ulambda) = \frac {e^{\frac 3 4 \pi i \lp  (n-1) +  2 |\ulambda| \rp}} {\sqrt n (2\pi)^{\frac {n-1} 2}} \cdot E_a(\ulambda) \varSigma (\ulambda) D_a(\ulambda) Y_a  (z; \ulambda).
\end{equation}

We now formulate \eqref{4eq: connection formula} as  
\begin{align} \label{4eq: connection formula, matrix}
J_{(\umu, \um)} (z)  = (-)^{ |\um|} e^{- \frac 1 2 \pi i (n-1)} \left(  4 \pi^2\right)^{n-1} \cdot 
{^t X\left( \hskip -1 pt 2 \pi z^{\frac 1 n}  ; \ulambda_{(\umu, \um)}^+ \hskip -1 pt \right)}  S_{(\umu, \um)}
X \left( \hskip -1 pt 2\pi \overline z^{\frac 1 n}  ; \ulambda_{(\umu, \um)}^- \hskip -1 pt \right),
\end{align}
with $\ulambda^\pm_{(\umu, \um)} = \umu \pm \tfrac 1 2 \um$ and
$$S_{(\umu, \um)} =  \mathrm{diag} \lp  \tau_{l     } \lp  \ulambda_{(\umu, \um)}^\pm \rp \- e^{- \pi i \lp  (n-2) \mu_l      \mp m_l     \rp } \rp_{l     =1}^n  .$$
We insert into \eqref{4eq: connection formula, matrix} the formulae of  $X\left( 2 \pi z^{\frac 1 n}  ; \ulambda_{(\umu, \um)}^+ \right)$ and $X \left( 2\pi \overline z^{\frac 1 n}  ; \ulambda_{(\umu, \um)}^- \right)$ given by \eqref{5eq: connection 1st and 2nd, matrix form}, with $\ulambda = \ulambda_{(\umu, \um)}^+$, $a=0$ in the former and $\ulambda = \ulambda_{(\umu, \um)}^-$, $a= 1 - r $, for $r = 0, 1, ..., n$, in the latter. 
Then follows the formula
\begin{equation}\label{7eq: second matrix form}
\begin{split}
& J_{(\umu, \um)}  (z) = (-)^{(n - 1) + |\um|} \frac  {\lp 2 \pi  \rp^{n-1}}  n    \\
& {^t Y_0 \left(\hskip -1 pt 2 \pi z^{\frac 1 n}  ; \ulambda_{(\umu, \um)}^+ \hskip -1 pt  \right) } \hskip -1 pt  D_0\left( \hskip -1 pt  \ulambda_{(\umu, \um)}^+ \hskip -1 pt  \right) \hskip -1 pt 
{^{t \hskip -1 pt } \varSigma_{(\umu, \um)}  } R_{(\umu, \um)}  \varSigma_{(\umu, \um)} 
D_{1-r} \hskip -1 pt  
\left( \hskip -1 pt  \ulambda_{(\umu, \um)}^- \hskip -1 pt  \right) Y_{1-r} \left( \hskip -1 pt  2 \pi \overline z^{\frac 1 n}  ; \ulambda_{(\umu, \um)}^- \hskip -1 pt  \right) \hskip -2 pt ,
\end{split}
\end{equation}
where 
\begin{align*}
\varSigma_{(\umu, \um)} & = \varSigma \lp \ulambda_{(\umu, \um)}^+ \rp = \varSigma \lp \ulambda_{(\umu, \um)}^- \rp, \\
R_{(\umu, \um)} & =  E_0 \lp \ulambda_{(\umu, \um)}^+  \rp S_{(\umu, \um)} E_{1-r} \lp \ulambda_{(\umu, \um)}^- \rp  = \mathrm{diag} \lp  \tau_{l     } \lp \ulambda_{(\umu, \um)}^\pm  \rp\- e^{- 2 \pi i r \lambda_{(\umu, \um),\, l     }^\pm } \rp_{l     =1}^n.
\end{align*}
We are therefore reduced to computing the matrix ${^t \varSigma_{(\umu, \um)}  } R_{(\umu, \um)}  \varSigma_{(\umu, \um)} $. For this, we have the following lemma.

\begin{lem}\label{7lem: matrix}
	Let  $\ux = (x_1, ..., x_n) \in \BC^n$ be a generic $n$-tuple  in the sense that all its components are distinct.
	Let $\sigma_{l     , d} $, respectively $\sigma_d$, denote the elementary symmetric polynomial in $x_1, ..., \widehat {x_l     }, ..., x_n$,  respectively  $x_1, ... , x_n$, of degree $d$, and let $\tau_l      = \prod_{h \neq l     } (x_l      - x_h) $. Define the matrices $\varSigma = \big( \sigma_{l     , n-j} \big)_{l     , j=1}^n$, $X = \mathrm{diag} \lp  x_l      \rp_{l     =1}^n$ and $T = \mathrm{diag} \lp  \tau_l     \-  \rp_{l     =1}^n$. 
	Then, for any $r = 0, 1, ..., n$, the matrix $ {^t \varSigma  } X^r T \varSigma$ can be written as
	\begin{equation*}
	\begin{pmatrix}
	(-)^{n - r } A & 0 \\
	0 &  (-)^{n-r+1} B
	\end{pmatrix},
	\end{equation*}
	where
	\begin{equation*}
	A = \begin{pmatrix}
	0 & \cdots & 0 & \sigma_{n} \\
	\vdots & \iddots & \iddots & \vdots\\
	0 & \iddots & \iddots & \sigma_{n-r+2} \\
	\sigma_{n} & \cdots & \sigma_{n-r+2} & \sigma_{n-r+1 }
	\end{pmatrix}, \hskip 10 pt
	B = \begin{pmatrix}
	\sigma_{n-r-1} &  \sigma_{n-r-2} & \cdots & \sigma_0 \\
	\sigma_{n-r-2} & \iddots & \iddots & 0  \\
	\vdots & \iddots &  \iddots & \vdots   \\
	\sigma_0 & 0 & \cdots & 0
	\end{pmatrix}.
	\end{equation*}
	More precisely, the $(k, j)$-th entry $a_{k, j}$, $k, j = 1,..., r$, of $A$ is given by
	\begin{equation*}
	a_{k, j} = \left\{
	\begin{split}
	& \sigma_{n+r - k - j +1} \hskip 5 pt \text{ if } k + j \geq  r + 1, \\ 
	& 0 \hskip 51 pt \text{if otherwise},
	\end{split} \right.
	\end{equation*}
	whereas the $(k, j)$-th entry $b_{k, j}$, $k, j = 1,..., n-r$, of $B$ is given by
	\begin{equation*}
	b_{k, j} = \left\{
	\begin{split}
	& \sigma_{ n- r - k - j + 1 }  \hskip 5 pt \text{ if } k + j \leq n-r+1, \\ 
	& 0  \hskip 51 pt \text{if otherwise}.
	\end{split}\right.
	\end{equation*}

\end{lem}

\begin{proof} 
	
	Appealing to the Lagrange interpolation formula, we find in  Lemma \ref{8lem: Vandermonde}  that the inverse of $T \varSigma$ is equal to the matrix $U = \lp (-)^{n - j} x_{l     }^{j-1} \rp_{j, l      = 1}^n $. Therefore, it suffices to show that
	\begin{equation*}
	{^t \varSigma  } X^r = \begin{pmatrix}
	(-)^{n-r } A & 0 \\
	0 &  (-)^{n-r+1} B
	\end{pmatrix} U.
	\end{equation*}
	This is equivalent to the following two collections of identities,
	\begin{align*}
	\sum_{j=r-k+1}^{r} (-)^{r+j} \sigma_{n+r - k - j +1}  x_l     ^{j-1}  &= \sigma_{l     , n-k } x_l     ^r, \hskip 10 pt k = 1, ..., r, \\
	\sum_{j=1}^{n - r - k +1} (-)^{j-1} \sigma_{ n- r - k - j + 1}  x_l     ^{r+j-1}  &= \sigma_{l     , n-r-k} x_l     ^r , 
	\hskip 10 pt k = 1, ..., n-r,
	\end{align*} 
	which are further equivalent to
	\begin{align*}
	\sum_{j=1}^{k } (-)^{k+j } \sigma_{ n - j+1}  x_l     ^{j - k -1}  & = \sigma_{l     , n-k},   \hskip 10 pt  k = 1, ..., r, \\
	\sum_{j=1}^{k } (-)^{ j - 1} \sigma_{k - j}  x_l     ^{j-1}  & = \sigma_{l     , k - 1}, \hskip 10 pt  k = 1, ..., n-r.
	\end{align*}
	The last two identities can be easily seen, actually for all $k = 1,..., n$, from computing the coefficients of $x^{k-1}$ and $x^{2n - k}$ on the two sides of  
	\begin{align*}
	\prod_{h \neq l     } (x - x_h) & = \lp  \sum_{p=0}^\infty x_{l     }^p x^{-p-1} \rp \prod_{h = 1}^{n} (x - x_h), \\
	\lp x^{n} - x_l     ^{n} \rp \prod_{h \neq l     } (x - x_h) & = \lp \sum_{p=1}^{n } x_{l     }^{ p - 1} x^{n-p}  \rp \prod_{h = 1}^{n} (x - x_h),
	\end{align*}
	respectively.
\end{proof}

Applying Lemma \ref{7lem: matrix} with $x_{l     } = e^{- 2\pi i \lambda_{(\umu, \um), l     }^\pm  } = (-)^{m_l     } e^{- 2\pi i  \mu_{l     } }$ to the formula  \eqref{7eq: second matrix form},
we arrive at the following theorem.

\begin{thm}
	\label{7thm: second formula}
	Let $(\umu, \um) \in \BL^{n-1} \times \BZ^n$ and $r \in \{ 0, 1, .., n \}$. 
	Define $\xi _{ j} = e^{ 2 \pi i \frac { {  j -  1}  } n}$, $\zeta_j = e^{ 2 \pi i \frac { {  j -  r}  } n}$, and  denote by $\sigma_{(\umu, \um)}^d $ the elementary symmetric polynomial in $(-)^{m_1} e^{- 2 \pi i  \mu_1 }$, $...$, $(-)^{m_n} e^{- 2 \pi i  \mu_n }$ of degree $d$, with $j = 1, ..., n$ and $d = 0, 1, ..., n$. Then we have
	\begin{equation}\label{7eq: J (z)}
	\begin{split}
	J _{(\umu, \um)} (z) & =  (-)^{ |\um|} \frac {(2 \pi)^{n-1}} n \mathop{\sum\sum}_{\substack{ k, j = 1, ..., r \\ k+j \geq r+1 }}  
	C_{k, j} (\umu, \um)  \\ 
	& \hskip 92 pt
	J \big( 2 \pi z^{\frac 1 n} ; \umu + \tfrac 1 2 \um; \xi_k \big) J \big( 2 \pi \overline z^{\frac 1 n} ; \umu - \tfrac 1 2 \um; \zeta_j \big) \\
	& + (-)^{ |\um|} \frac {(2 \pi)^{n-1}} n \mathop{\sum\sum}_{\substack{ k, j = 1, ..., n-r \\ k+j \leq n-r+1 }}  D_{k, j} (\umu, \um)   \\
	& \hskip 74 pt
	J \big( 2 \pi z^{\frac 1 n} ; \umu + \tfrac 1 2 \um; \xi_{r+k} \big) J \big( 2 \pi \overline z^{\frac 1 n} ; \umu - \tfrac 1 2 \um; \zeta_{r+j} \big).
	\end{split}
	\end{equation}
	\delete{	\begin{equation}\label{7eq: J -}
		\begin{split}
		J^-_{(\umu, \um)} (z) = &\, (-)^{ |\um|} \frac {(2 \pi)^{n-1}} n \mathop{\sum\sum}_{\substack{ k, j = 1, ..., r \\ k+j \geq r+1 }}  C^-_{k, j} (\umu, \um)  \\ 
		& \hskip 82 pt
		J \big( 2 \pi z^{\frac 1 n} ; \umu + \tfrac 1 2 \um; \xi_k \big) J \big( 2 \pi \overline z^{\frac 1 n} ; \umu - \tfrac 1 2 \um; \zeta_j \big)  ,
		\end{split}
		\end{equation}
		\begin{equation}\label{7eq: J +}
		\begin{split}
		J^+_{(\umu, \um)} (z) = &\, (-)^{ |\um|} \frac {(2 \pi)^{n-1}} n \mathop{\sum\sum}_{\substack{ k, j = 1, ..., n-r \\ k+j \leq n-r+1 }}  D_{k, j} (\umu, \um)   \\
		& \hskip 64 pt
		J \big( 2 \pi z^{\frac 1 n} ; \umu - \tfrac 1 2 \um; \xi_{r+k} \big) J \big( 2 \pi \overline z^{\frac 1 n} ; \umu - \tfrac 1 2 \um; \zeta_{r+j} \big),
		\end{split}
		\end{equation}	
	}
	with
	\begin{align}
	\label{7eq: C-} C _{k, j} (\umu, \um) & =  (-)^{ r + k+j + 1} \xi_k^{-\frac {n-1} 2 - \frac 1 2 |\um|} \zeta_j^{-\frac {n-1} 2 + \frac 1 2 |\um|} \sigma_{(\umu, \um) }^{n+r-k-j+1}, \\
	\label{7eq: C+} D _{k, j} (\umu, \um) & = (-)^{ r+k+j } \xi_{r+k}^{-\frac {n-1} 2 - \frac 1 2 |\um|} \zeta_{r+j}^{-\frac {n-1} 2 + \frac 1 2 |\um|} \sigma_{(\umu, \um) }^{n-r-k-j+1}.
	\end{align}
\end{thm}


\begin{lem}\label{7lem: C and D}
	We retain the notations in Theorem {\rm \ref{7thm: second formula}}. Moreover, we define  $\mathfrak I (\umu) = \max \left\{ \left| \Im \mu_l      \right| \right\}$.
	
	{\rm (1.1).} For $k=1, ..., r$, we have 
	$ C_{k, r-k+1} (\umu, \um) = (- \overline\xi_k)^{ |\um|}.$
	
	{\rm (1.2).} Let $k, j = 1, ... , r$ be such that $k + j \geq r + 2$. Denote $p = k+j-r-1$. We have the estimate
	\begin{equation*}
	|C_{k, j} (\umu, \um)| \leq {n \choose p} \exp\big(  2 \pi \min \left\{  n-p, p \right\} \mathfrak I (\umu) \big).
	\end{equation*}
	
	{\rm (2.1).} For $k=1, ..., n-r$, we have 
	$D_{k, n-r-k+1} (\umu, \um) =  (- \overline\xi_{k+r})^{ |\um|}.$
	
	{\rm (2.2).} Let $k, j = 1, ... , n-r$ be such that $k + j \leq n-r  $. Denote $p = n - r - k - j + 1$. We have the estimate
	\begin{equation*}
	|D_{k, j} (\umu, \um)| \leq {n \choose p} \exp\big(   2 \pi \min \left\{  n-p, p \right\} \mathfrak I (\umu) \big).
	\end{equation*}	
	
\end{lem}

\subsection{The Rank-Two Case}\label{sec: rank-two, C}

\begin{example}\label{3prop: n=2, C}
	Let $\mu \in \BC$ and $m \in \BZ$. 
	
	If we define 
	\begin{equation}\label{7def: J mu m (z), n=2, C}
	J_{\mu, m} (z) = J_{- 2\mu - \frac 12 m } \lp  z \rp J_{- 2\mu + \frac 12 m  } \lp  {\overline z} \rp,
	\end{equation}
	then
	\begin{equation}\label{3eq: n=2, C}
	\hskip -1 pt J_{(\mu, - \mu, m, 0)} \lp z \rp \hskip -1 pt =  \hskip -2 pt
	\left\{ 
	\begin{split}
	& \hskip -3 pt \frac {2 \pi^2} {\sin (2\pi \mu)} [\sqrt z]^{-m} \lp J_{\mu, m} (4 \pi \sqrt z) \hskip -1 pt - \hskip -1 pt  J_{-\mu, -m} (4 \pi \sqrt z) \rp \hskip 2 pt \text {if } m \text{ is even},\\
	& \hskip -3 pt \frac {2 \pi^2 i} {\cos (2\pi \mu)} [\sqrt z]^{-m} \lp J_{\mu, m} (4 \pi \sqrt z) \hskip -1 pt + \hskip -1 pt J_{-\mu, -m} (4 \pi \sqrt z) \rp \hskip 2 pt \text {if }  m \text{ is odd},
	\end{split}
	\right.
	\end{equation} 
	which should be interpreted in the way as in Theorem {\rm \ref{4thm: connection formula}}. We remark that the generic case is when  $4 \mu \notin 2\BZ + m$.
	
	On the other hand, using the connection formulae {\rm (\cite[3.61 (1, 2)]{Watson})}
	\begin{equation*}
	J_\nu (z) = \frac {H_\nu^{(1)} (z) + H_\nu^{(2)} (z)} 2, \hskip 10 pt 
	J_{-\nu} (z) =  \frac {e^{\pi i \nu} H_\nu^{(1)} (z) + e^{-\pi i \nu} H_\nu^{(2)} (z) } 2,
	\end{equation*}
	one obtains
	\begin{equation}
	J_{(\mu, - \mu, m, 0)} (z) = \pi^2 i [\sqrt z]^{-m}  \lp e^{2 \pi i \mu} H^{(1)}_{\mu, m} \lp 4 \pi \sqrt z \rp + (-)^{m+1} e^{- 2 \pi i \mu} H^{(2)}_{\mu, m} \lp 4 \pi \sqrt z \rp \rp,
	\end{equation}
	with the definition
	\begin{equation}\label{7def: H (1, 2) mu m (z), n=2, C}
	H^{(1, 2)}_{\mu, m} (z) = H^{(1, 2)}_{2 \mu + \frac 1 2 m} \lp   z \rp  H^{(1, 2)}_{2 \mu - \frac 1 2 m} \lp  { \overline z} \rp.
	\end{equation}
\end{example}

\section{\texorpdfstring{The Asymptotic Expansion of $J_{(\umu, \um)} (z) $}{The Asymptotic Expansion of $J_{(\mu, m)} (z) $}}\label{sec: asymptotics}




In the following, we shall present the asymptotic expansion of $J_{(\umu, \um)} (z) $. 

\vskip 5 pt

First of all, we have the following proposition on the asymptotic expansion of $J (z; \ulambda; \xi)$, which is in substance  Theorem \ref{thm: error bound}. 



\begin{prop}\label{8prop: asymptotic}
	Let $\ulambda \in \BC^n$ and define $\mathfrak C (\ulambda) = \max \left\{ \left| \lambda_l      - \frac 1 n |\ulambda| \right| + 1 \right\}$. Let $\xi $ be a  $2 n$-th root of unity.  For a small positive constant $ \vartheta $, say $0 < \vartheta < \frac 1 2 \pi $, we define the sector
	\begin{equation*}
	\BS'_{\xi } (\vartheta) = \left\{ z : \left| \arg z - \arg ( i \overline \xi) \right|  < \pi + \frac {\pi} n - \vartheta \right \}.
	\end{equation*} 
	For a positive integer $A $, we have  the asymptotic expansion
	\begin{equation*}
	J (z; \ulambda; \xi)   =    e^{i n \xi z} z^{- \frac { n-1} 2 - |\ulambda| }   \lp     \sum_{\alpha=0}^{A-1}   (i \xi)^{-\alpha} B_\alpha   \left(  \ulambda - \tfrac 1 n |\ulambda| \ue^n   \right) z^{-\alpha}   + O_{A, \,\vartheta,\, n} \lp \fC (\ulambda)^{2 A} |z|^{- A} \rp   \rp
	\end{equation*}
	for all $z \in \BS'_{\xi} (\vartheta)$ such that $ |z| \ggg_{\,A,\, \vartheta,\, n} \fC (\ulambda)^2 $. Here $B_{\alpha } (\ulambda) $  is a certain symmetric polynomial function in $\ulambda \in \BL^{n-1}$ of degree $2 \alpha$, with $B_0 (\ulambda) = 1$.
\end{prop}

\begin{lem}\label{8lem: asymptotic}
	Let $r$ be a positive integer. Suppose that either $ n = 2r$ or $n = 2r-1$. Put $\vartheta_n = \frac 1 {\, n \, }  \pi$ if $n = 2 r$ and $\vartheta_n = \frac 1 {2 n} \pi$ if $n = 2r-1$.
	For a given constant $0 < \vartheta < \vartheta_n$  define the sector
	\begin{equation*}
	\BS_n (\vartheta) =  \left\{ \begin{split}
	\left\{ z : - \frac \pi 2 - \frac \pi {n} + \vartheta < \arg z < - \frac \pi 2 + \frac {3 \pi} {n} - \vartheta \right\} & \hskip 10 pt \text{ if } n = 2 r, \\
	\left\{ z : - \frac \pi 2 - \frac \pi {n} + \vartheta < \arg z < - \frac \pi 2 + \frac {2 \pi} {n} - \vartheta \right\} & \hskip 10 pt \text{ if } n = 2 r - 1, 
	\end{split} \right.
	\end{equation*}
	Let $(\umu, \um) \in \BL^{n-1} \times \BZ^n$ and define $\mathfrak C (\umu, \um) = \max \left\{ |\mu_l     | + 1, \left|m_l      - \frac 1 n |\um| \right| + 1 \right\}$. Define $\xi _{ j} = e^{ 2 \pi i \frac { {  j -  1}  } n}$ and $\zeta_j = e^{ 2 \pi i \frac { {  j -  r}  } n}$ for $j=1, ..., n$.  
	Then, for any $z \in \BS_n (\vartheta) $ such that $|z| \ggg_{\, A,\, \vartheta,\, n} \fC (\umu, \um) ^2$, we have
	\begin{align*}
	& J  \lp 2 \pi z  ; \umu + \tfrac 1 2 \um; \xi_k \rp J \lp 2 \pi \overline z  ; \umu - \tfrac 1 2 \um; \zeta_j \rp =  \frac { e \lp n \lp  \xi_k z + \zeta_j \overline z \rp  \rp }   { (2\pi)^{n-1} |z|^{n-1} [z]^{|\um|} }  \\
	& \hskip 50 pt
	\lp \mathop{\sum\sum}_{\substack{ \alpha,\, \beta = 0, ..., A-1 \\ \alpha + \beta \leq A-1 }} (i \xi_k)^{-\alpha} (i \zeta_j)^{-\beta}  B_{\alpha, \beta} (\umu, \um) z^{-\alpha} \overline z^{-\beta} + O_{A,\, \vartheta,\, n} \lp\fC (\umu, \um)^{2 A} |z|^{-A} \rp \rp,
	\end{align*}
	with 
	\begin{equation*}
	B_{\alpha, \beta} (\umu, \um) = (2\pi)^{-\alpha-\beta} B_\alpha \lp \umu + \tfrac 1 2 \um - \tfrac 1 {2 n} |\um| \ue^n \rp B_\beta \lp \umu - \tfrac 1 2 \um + \tfrac 1 {2 n} |\um| \ue^n \rp, \hskip 10 pt \alpha, \beta \in \BN,
	\end{equation*}
	where $B_\alpha (\ulambda)$  
	is the polynomial function in $\ulambda$ of degree $2 \alpha$ given in Proposition {\rm \ref{8prop: asymptotic}}.
\end{lem}

\begin{proof}

	
	Recall that, for an integer $a$, we defined $\xi _{a, j} = e^{ 2 \pi i \frac { {  j  + a  -  1}  } n}$. Note that $\xi_j = \xi _{0, j}$ and $\zeta_j = \xi _{1-r, j}$. It is clear that
	\begin{equation*}
	\bigcap_{j=1}^n \BS'_{\xi_{a, j} } (\vartheta) =  \left\{ z : 
	- \frac \pi 2  -   \frac { 2a+1 } n \pi + \vartheta <  \arg z < - \frac \pi 2  -    \frac { 2a-3 } n \pi - \vartheta \right \}.
	\end{equation*}
	We denote this sector by $\BS' _{ a } (\vartheta)$. 
	Observe that, when $n = 2 r$ or $2r-1$, the intersection $\BS' _{ 0 } (\vartheta) \cap \overline {\BS' _{ 1-r } (\vartheta) } $ is exactly the sector $\BS_n (\vartheta)$. In other words,  for all $j = 1, ..., n$, $z \in \BS'_{\xi_{ j} } (\vartheta) $ and $ \overline z \in \BS'_{\zeta_{ j} } (\vartheta) $ both hold if $z \in \BS_n (\vartheta)$. Therefore, Proposition \ref{8prop: asymptotic} can be applied to yield the asymptotic expansion of  $J \lp 2 \pi z  ; \umu + \tfrac 1 2 \um; \xi_k \rp J \lp 2 \pi \overline z  ; \umu - \tfrac 1 2 \um; \zeta_j \rp$ as above. 
\end{proof}

\begin{rem}  In view of our choice of $\vartheta$, the sector $\BS_{n} (\vartheta) $ is of angle at least $ \frac 2 {\,n\,} \pi$, and therefore the sector $\BS_{n} (\vartheta)^n = \left\{ z^n : z \in \BS_{n} (\vartheta) \right\}$ covers the whole  $\BC  \smallsetminus \{0\}$. 
	
\end{rem}

\begin{lem} \label{8lem: positive im}
	Let notations be as in Lemma {\rm \ref{8lem: asymptotic}}. 
	
	{\rm (1.1).} For $k = 1,..., r$, we have 
	\begin{equation*}
	\Im \lp \xi_{k} z + \zeta_{r-k+1} \overline z \rp = 0.
	\end{equation*}

	{\rm (1.2).} Let $k, j = 1, ... , r$ be such that $k + j \geq r + 2$. For any $z \in \BS_n (\vartheta)$, we have 
	\begin{equation*}
	\Im \lp \xi_{k} z + \zeta_j \overline z \rp \geq 2 \sin \lp\frac {k+j-r-1} n \pi\rp \sin \vartheta \cdot |z|.
	\end{equation*}
	
	{\rm (2.1).} For $k = 1,..., n - r$, we have 
	\begin{equation*}
	\Im \lp \xi_{k+r} z + \zeta_{n-k+1} \overline z \rp = 0.
	\end{equation*}
	
	{\rm (2.2).} Let $k, j = 1, ... , n-r$ be such that $k + j \leq n-r  $. For any $z \in \BS_n (\vartheta)$, we have 
	\begin{equation*}
	\Im \lp \xi_{k + r} z + \zeta_{j+r} \overline z \rp \geq 
	\left\{ \begin{split}
	& 2\sin \lp\frac {n-r -k-j +1} n \pi \rp \sin \vartheta \cdot |z|, \hskip 43 pt \text{ if } n=2r, \\
	& 2\sin \lp\frac {n-r-k-j +1} n \pi \rp \sin \lp \frac \pi n + \vartheta \rp \cdot |z|,  \hskip 10 pt \text{ if } n=2r -1.
	\end{split} \right.
	\end{equation*}
	
\end{lem}

\begin{proof}
	We shall only prove (1.1) and (1.2) in the case $n = 2 r$. The other cases follow in exactly the same way.
	
	Write $z = x e^{i \phi}$.  
	Since
	\begin{align*}
	\xi_{k} z + \zeta_j \overline z & = xe^{2 \pi i \frac {k-1} {2r} + i \phi} + xe^{2 \pi i \frac {j-r} {2r} -i \phi} \\
	& = xe^{\pi i \frac {k+j - r - 1} {2r} } \lp e^{ \pi i \frac {k-j+r-1} {2r} + i \phi}  + e^{ - \pi i \frac {k-j+r-1} {2r} - i \phi} \rp,
	\end{align*}
	(1.1) is then obvious (we also note that $\zeta_{r-k+1} = \overline \xi_k $), whereas (1.2) is equivalent to
	\begin{equation}\label{8eq: sin cos}
	\cos \lp  \frac {k-j+r-1} {2 r} \pi + \phi \rp \geq  \sin \vartheta.
	\end{equation}
	Observe that the condition $z \in \BS_{2r} (\vartheta)$ amounts to 
	\begin{equation*}
	\left|\phi + \frac \pi 2  -\frac \pi {2r} \right| < \frac {\pi} r  - \vartheta.
	\end{equation*} Moreover, under  our assumptions on $k$ and $j$ in (1.2), one has $|k-j| \leq r-2$.  Consequently, these yield the following estimate
	\begin{align*}
	\left|\frac {k-j+r-1} {2 r} \pi + \phi \right| \leq \frac {r-2} {2 r} \pi + \frac \pi r  - \vartheta = \frac \pi 2 - \vartheta.
	\end{align*}
	Thus \eqref{8eq: sin cos} is proven.
\end{proof}

\begin{rem}
	In the cases other than those listed in Lemma {\rm \ref{8lem: positive im}}, $\Im \lp \xi_{k} z + \zeta_{j} \overline z \rp  $ can not always be nonnegative for all $z \in \BS_n (\vartheta)$. 
	Fortunately, these cases are excluded from the second connection formula for $J_{(\umu, \um)} (z)$ in Theorem {\rm \ref{7thm: second formula}}.
\end{rem}

Now the asymptotic expansion of $J _{(\umu, \um)} \left(z \right)$ can   be readily established using Theorem \ref{7thm: second formula} along with Lemma \ref{7lem: C and D}, \ref{8lem: asymptotic} and \ref{8lem: positive im}.

\begin{thm}\label{8thm: asymptotic}
	Denote by $\BX_n  $ the set of $n$-th roots of unity. Let $(\umu, \um) \in \BL^{n-1} \times \BZ^n$ and define $\mathfrak C (\umu, \um) = \max \left\{ |\mu_l     | + 1, \left|m_l      - \frac 1 n |\um| \right| + 1 \right\}$. Let $A$ be a positive integer. Then
	\begin{equation*}
	\begin{split}
	J _{(\umu, \um)} \left(z^n \right)   = \sum_{\xi \in \BX_n } \frac { e \big( n \big( \xi z +   \overline {\xi z } \big) \big) }   { n |z|^{n-1} [\xi z]^{|\um|} }   
	\lp \mathop{\sum\sum}_{\substack{ \alpha,\, \beta = 0, ...,\, A-1 \\ \alpha + \beta \leq A-1 }} i^{- \alpha - \beta} \xi^{ -\alpha + \beta} B_{\alpha, \beta} (\umu, \um) z^{-\alpha} \overline z^{-\beta} \rp & \\
	 \hskip 71 pt + O_{A, \, n} \lp\fC (\umu, \um)^{2 A} |z|^{-A-n+1} \rp   & ,
	\end{split}
	\end{equation*}
	if $|z| \ggg_{\,A,\, n}  \fC (\umu, \um)^2$, with the coefficient $B_{\alpha, \beta} (\umu, \um)$ given 
	in Lemma {\rm \ref{8lem: asymptotic}}. 
	
\end{thm}

We may also prove the following elaborate version of Theorem \ref{8thm: asymptotic}.

\begin{thm}\label{8thm: asymptotic, complex, 2}
	Let notations be as in  Lemma {\rm \ref{8lem: asymptotic}} and Theorem {\rm \ref{8thm: asymptotic}}. Fix the angle $\vartheta$, say  $\vartheta = \frac 1 2 \vartheta_n$.
	Let $\mathfrak I (\umu) = \max \left\{ \left| \Im \mu_l      \right| \right\}$. Then we may write 
	\begin{align*}
	J _{(\umu, \um)} \left(z^n \right) = \sum_{\xi \in \BX_n } \frac { e \big( n \big( \xi z +   \overline {\xi z } \big) \big) }   { n |z|^{n-1} [\xi z]^{|\um|} } W_{(\umu, \um)} \lp  z , \xi \rp + E_{(\umu, \um)} (z),
	\end{align*} 
	such that
	\begin{align*}
	W_{(\umu, \um)} \lp z , \xi \rp =    
	\mathop{\sum\sum}_{\substack{ \alpha,\, \beta = 0, ...,\, A-1 \\ \alpha + \beta \leqslant A-1 }} i^{- \alpha - \beta} \xi^{- \alpha + \beta}  B_{\alpha, \beta} (\umu, \um) z^{-\alpha } \overline z^{-\beta }  + O_{A, \, n} \lp\fC (\umu, \um)^{2 A} |z|^{-A } \rp ,
	\end{align*} 
	and
	\begin{align*}
	E_{ (\umu, \um) } (z) = O_{ n} \left( |z|^{- n + 1} \exp \lp 2 \pi \mathfrak I (\umu)   - 4 \pi n \sin \lp \tfrac 1 {\,n\,} \pi \rp \sin   \vartheta  \cdot |z| \rp \right),
	\end{align*} 
	for $z \in  \BS_n (\vartheta) $ with $|z| \ggg_{\,A,\, n}  \fC (\umu, \um)^2$.  Moreover, $ E_{ (\umu, \um) } (z) \equiv 0 $ when $n = 1, 2$.
\end{thm}

%
%
%


\chapter{Hankel Transforms and Bessel Kernels in Representation Theory}\label{chap: Representation Theory}

This last chapter is devoted to the representation theoretic investigations of Hankel transforms for $\GL_n (\BF)$ and Bessel kernels for $\GL_2 (\BF)$, with $\BF = \BR$ or $\BC$.

\section{Hankel Transforms from the Representation Theoretic Viewpoint}\label{sec: Hankel, Ichino-Templier}

We shall start with a brief review of Hankel transforms over an \textit{Archimedean} local field\footnote{For a non-Archimedean local field, Hankel transforms can also be constructed in the same way.} in the work of Ichino and Templier \cite{Ichino-Templier} on the \Voronoi summation formula. 
For the theory of $L$-functions and local functional equations over a local field the reader is referred to Cogdell's survey \cite{Cogdell}. We shall then study  Hankel transforms using the Langlands classification.
For this,  Knapp's article \cite{Knapp} is used as our reference, with some changes of notations for our convenience.
\vskip 5 pt

Let $\BF$ be an Archimedean local field with normalized absolute value $\|\, \| = \|\, \|_{\BF}$ defined as in \S \ref{sec: R+, Rx and Cx}, 
and let $\psi $ be a given additive character on $\BF$.  For $s \in \BC$  let $\omega_s $ denote the character $\omega_s (x) = \|x\|^s$. Let $\eta (x) = \sgn (x)$ for $x \in \BRx$ and  $\eta (z) = [z]$ for $z \in \BCx$. 


Suppose for the moment $n \geq 2$. Let $\pi $ be an infinite dimensional irreducible admissible generic representation of $\GL_n( \BF)$\footnote{Since $\pi $ is a local component of an irreducible cuspidal automorphic representation in \cite{Ichino-Templier},  \cite{Ichino-Templier} also assumes that $\pi$ is unitary. However, if one only considers the local theory, this assumption is not necessary.}, and $\mathscr W (\pi, \psi)$ be the $\psi$-Whittaker model of $\pi$. Denote by $\omega_\pi$ the central character of $\pi$. Recall that the $\gamma$-factor $\gamma (s, \pi , \psi )$ of $\pi$ is given by
\begin{equation*}
\gamma (s, \pi , \psi ) = \epsilon (s, \pi , \psi ) \frac {L(1 - s, \widetilde \pi )} { L(s,\pi ) }
\end{equation*}
where $\widetilde \pi $ is the contragradient representation of $\pi$,  $\epsilon (s, \pi , \psi )$ and $L( s, \pi)$ are the $\epsilon$-factor and the $L$-function of $\pi$ respectively. 

To a smooth compactly supported function $w$ on $\BF^{\times}$ we associate a dual function $\widetilde w$ on $\BF^{\times}$ defined by \cite[(1.1)]{Ichino-Templier},
\begin{equation}\label{4eq: Ichino-Templier}
\begin{split}
\int_{\BF^\times} \widetilde w (x) \chiup (x)\- & \|x\|^{s - \frac {n-1} 2} d^\times x \\
& = \chiup (-1)^{n-1} \gamma (1-s, \pi \otimes \chiup, \psi ) \int_{\BF^\times} w (x) \chiup (x) \|x\|^{1 - s - \frac {n-1} 2} d^\times x,
\end{split}
\end{equation}
for all $s$ of real part sufficiently large and all unitary multiplicative characters $\chiup $ of $\BF^\times$. (\ref{4eq: Ichino-Templier}) is independent of the chosen Haar measure $d^\times x$ on $\BF^\times$, and uniquely defines $\widetilde w $ in terms of $\pi$, $\psi$ and $w$. We shall let the Haar measure be given as in \S \ref{sec: R+, Rx and Cx}. We call $\widetilde w $ the Hankel transform of $w$ associated with $\pi$.

According to \cite[Lemma 5.1]{Ichino-Templier}, there exists a smooth Whitaker function $W \in \mathscr W (\pi, \psi)$ so that
\begin{equation}\label{4eq: w = Whittaker}
w(x) = W   
\begin{pmatrix}
x & \\
& I_{n-1}
\end{pmatrix},
\end{equation}
for all $x \in \BF^\times$. Denote by $\varw_n$ the  $n$-by-$n$ permutation matrix whose anti-diagonal entries are $1$, that is, the longest Weyl element of rank $n$, 
and define
\begin{equation*}
\varw_{n, 1} = \begin{pmatrix}
1 & \\
& \varw_{n-1}
\end{pmatrix}.
\end{equation*}
In the theory of integral representations of Rankin-Selberg $L$-functions, \eqref{4eq: Ichino-Templier} amounts to the local functional equations of 
zeta integrals for $\pi \otimes \chiup$, with
\begin{equation}\label{4eq: tilde w for n = 2}
\widetilde w (x) = \widetilde W 
\begin{pmatrix}
x & \\
& 1
\end{pmatrix}  = 
W \lp \varw_2
\begin{pmatrix}
x\- & \\
& 1
\end{pmatrix}\rp,
\end{equation}
if $n = 2$, and
\begin{equation}\label{4eq: tilde w for n > 2}
\widetilde w (x) = \int_{\BF^{n-2}} \widetilde W \lp
\begin{pmatrix}
x &  &  \\
y & I_{n-2} & \\
&  & 1
\end{pmatrix} \varw_{n, 1}
\rp d y_{\psi},
\end{equation}
if $n \geq 3$, where $\widetilde W \in \mathscr W(\widetilde \pi, \psi\-)$ is the dual Whittaker function defined by $\widetilde W (g) = W(\varw_n \cdot {^t g\-})$, for $g \in \GL_n( \BF)$, and $d x_\psi$ denotes the self-dual additive Haar measure on $\BF$ with respect to $\psi$. See \cite[Lemma 2.3]{Ichino-Templier}.

\vskip 5 pt

It should be noted that the settings in \cite{Ichino-Templier} can be extended in two aspects.

First, the constraint that $\pi$ be \textit{infinite dimensional} and {\it generic} is actually dispensable for defining the Hankel transform via \eqref{4eq: Ichino-Templier}. In the following, we shall assume that $\pi$ is any irreducible admissible representation of $\GL_n(\BF)$. Moreover, we shall also include the case $n=1$. It will be seen that, after renormalizing  the functions $w$ and $\widetilde w$, the Hankel transform defined by \eqref{4eq: Ichino-Templier} converts into the Hankel transform given by \eqref{3eq: Hankel transform identity, R} or \eqref{3eq: Hankel transform identity, C}.
For this, we shall apply the Langlands classification for irreducible admissible representations of $\GL_n(\BF)$. 

Second, the constraint that the weight function $w$ be compactly supported is not necessary. By the work in \S \ref{sec: all Ssis} and \ref{sec: Hankel transforms}, the Hankel transform defined by \eqref{4eq: Ichino-Templier} may be extended to a larger space of weight functions. Particularly important is that this space contains the Kirillov model when $n = 2$ (see Remark \ref{rem: Kirillov}). 

\subsection{Hankel Transforms over $\BR$} 

Suppose $\BF = \BR$. Recall that $\|\, \|_\BR = |\ |$ is the ordinary absolute value. 
For $  r     \in \BR^\times$ let $\psi (x) = \psi_  r     (x) = e (  r     x)$.

According to \cite[\S 3, Lemma]{Knapp}, every finite dimensional semisimple representation $\varphi$ of the Weil group of $\BR$ may be decomposed into irreducible representations of dimension one or two. The one-dimensional representations are parametrized by $ (\mu, \delta) \in \BC \times \BZT$. We denote by $\varphi_{(\mu, \delta)}$ the representation given by $(\mu, \delta)$. $\varphi_{(\mu, \delta)}$ corresponds to the representation $\chiup_{(\mu, \delta)} = \omega_{\mu} \eta^{\delta}   $ of $\GL_1( \BR)$ under the Langlands correspondence over $\BR$. The irreducible two-dimensional representations are parametrized by $(\mu, m) \in \BC \times \BN_+$. We denote by $\varphi_{( \mu, m)}$ the representation given by $( \mu, m)$. $\varphi_{( \mu, m)}$ corresponds to the representation $\sigma(m) \otimes \omega_{\mu}(\det) $ of $\GL_2( \BR)$, where $\sigma(m)$ denotes the discrete series representation of weight $m$. 

In view of the formulae \cite[(3.6, 3.7)]{Knapp}\footnote{The formulae in \cite[(3.6, 3.7)]{Knapp} are for $\psi_1$. The relation between the epsilon factors $\epsilon (s, \pi , \psi_r )$ and $\epsilon (s, \pi , \psi )$ is given in \cite[\S 3]{Tate} (see in particular \cite[(3.6.6)]{Tate}).} of   $L$-functions and $\epsilon$-factors, the definitions  of $G_\delta$ and $G_m$ in \eqref{1def: G delta} and \eqref{1def: G m (s)}, along with the formula \eqref{1f: G m (s) = G 1 G delta(m)}, we deduce that
\begin{equation}\label{4eq: gamma, character}
\begin{split}
\gamma ( s, \varphi_{( \mu, \delta)} , \psi)  
= \sgn (  r     )^{\delta}  |  r     |^{s  + \mu - \frac 1 2 } G_{\delta} (1 - s  - \mu),
\end{split}
\end{equation} 
whereas
\begin{equation}\label{4eq: gamma, discrete}
\begin{split}
\gamma (s, \varphi_{( \mu, m)} , \psi) & = \sgn (  r     )^{\delta(m)+1} |  r     |^{ 2s + 2\mu - 1} i G_m(1 - s - \mu),
\end{split}
\end{equation}
and
\begin{equation}\label{4eq: gamma, discrete, 2}
\begin{split}
\gamma (s, \varphi_{( \mu, m)} , \psi) & = \gamma  ( s, \varphi_{ ( \mu + \frac 12 {m} ,\, \delta(m) + 1  )} , \psi ) \gamma ( s, \varphi_{ ( \mu - \frac 12 {m} ,\, 0  )} , \psi ) \\
& = \gamma  ( s, \varphi_{( \mu + \frac 12 {m} ,\, \delta (m) )} , \psi ) \gamma  ( s, \varphi_{( \mu - \frac 12 {m} ,\, 1 )} , \psi ).
\end{split}
\end{equation}
To $ \varphi_{( \mu, m)}$ we shall attach either one of the following two parameters 
\begin{equation}\label{4eq: two parameters, discrete}
\textstyle \big( \mu + \frac 12 {m} , \mu - \frac 12 {m} , \delta (m) + 1, 0 \big), \ \big(\mu + \frac 12{m}  , \mu - \frac 12 {m}  , \delta (m), 1 \big).
\end{equation}

\begin{rem}\label{4rem: discrete series}
	\eqref{4eq: gamma, discrete, 2} reflects the isomorphism $\varphi_{( 0, m)} \otimes \varphi_{( 0, 1)} \cong \varphi_{( 0, m)}$ of representations of the Weil group {\rm (}here $(0, 1)$ is an element in $\BC \times \BZT${\rm )}, as well as 
	the isomorphism $\sigma(m) \otimes \eta \cong \sigma(m)$ of representations of $\GL_2(\BR)$. 
\end{rem}

For $\varphi$ reducible, 
$\gamma (s, \varphi , \psi)$ is the product of the $\gamma$-factors of the  irreducible constituents of $\varphi$. Suppose that $\varphi$ is $n$-dimensional. It follows from (\ref{4eq: gamma, character}, \ref{4eq: gamma, discrete}, \ref{4eq: gamma, discrete, 2}) that there is a parameter $(\umu, \udelta) \in \BC^{n} \times (\BZT)^n$ attached to $\varphi$ such that
\begin{equation}\label{4eq: gamma factor of phi}
\gamma ( s, \varphi, \psi) = \sgn(  r     )^{|\udelta|} |  r     |^{n\lp s - \frac 1 2 \rp + |\umu|} G_{(\umu, \udelta)} (1 - s ).
\end{equation}
The irreducible constituents of $\varphi$ are unique up to permutation, but, in view of the two different parameters  attached to $\varphi_{(\mu, m)}$ in \eqref{4eq: two parameters, discrete}, the  parameter $(\umu, \udelta)$ attached to $\varphi$ may not.

Suppose that $\pi$ corresponds to $\varphi $ under the Langlands correspondence over $\BR$. We have $\gamma ( s, \pi, \psi) = \gamma ( s, \varphi, \psi)$. It is known that $\pi$ is an irreducible constituent of the principal series representation unitarily induced from the character $ \bigotimes _{l      = 1}^n \chiup_{(\mu_l     , \delta_l     )}$ of the Borel subgroup. 
In particular, 
\begin{align}\label{4eq: central char, R}
\omega_\pi (x) =  \omega_{|\umu|} (x) \eta^{|\udelta|} (x) = \sgn (x)^{|\udelta|} |x|^{|\umu|}.
\end{align}

Now let $\chiup = \chiup_{(0,\delta)} = \eta^{\delta}$ in \eqref{4eq: Ichino-Templier}, $\delta \in \BZT$.
In view of \eqref{4eq: gamma factor of phi} and \eqref{4eq: central char, R}, one has the following expression of the $\gamma$-factor in   \eqref{4eq: Ichino-Templier},
\begin{equation}\label{4eq: gamma factor real case}
\gamma (1-s, \pi \otimes \eta^{\delta} , \psi) = \omega_\pi (  r     ) \big( \sgn (  r     )^{ \delta} |  r     |^{ \frac 1 2 - s} \big)^n  G_{(\umu, \udelta + \delta \ue^n)}  (s ).
\end{equation}
Some calculations show that (\ref{4eq: Ichino-Templier}) is exactly translated into \eqref{3eq: Hankel transform identity, R} if one lets
\begin{equation}\label{4eq: upsilon and w, real case}
\begin{split}
\upsilon (x) & =  \omega_\pi (  r     ) w \big(  |  r     |^{- \frac n 2}x \big) |x|^{ - \frac {n-1} 2}, \\
\Upsilon (x) & = \widetilde w \big( (-)^{n-1} \sgn (  r     )^{ n } |  r     |^{- \frac n 2}x \big) |x|^{ - \frac {n-1} 2}.
\end{split}
\end{equation}
Then, \eqref{3eq: Hankel transform, with Bessel kernel, R} can be reformulated as
\begin{equation}\label{4eq: Hankel transform tilde w and w, real case}
\widetilde w \big((-)^{n-1} x \big) = \omega_\pi (  r     ) |  r     |^{\frac n 2} |x|^{\frac {n-1} 2} \int_{\BR ^\times} w (y) J_{(\umu, \udelta)} (   r     ^n xy ) |y|^{1 - \frac {n-1} 2} d^\times y.
\end{equation}

\subsection{Hankel Transforms over $\BC$} \label{sec: Hankel transform over C}
Suppose $\BF = \BC$. Recall that $\|\, \|_\BC = \|\, \| = |\  |^2$, where $|\  |$ denotes the ordinary absolute value. 
For $  r     \in \BC^\times$ let $\psi (z) = \psi_  r     (z) = e (  r     z + \overline {  r     z})$.

The Langlands classification and correspondence for $\GL_n( \BC)$ are less complicated. First of all,
the Weil group of $\BC$ is simply $\BC^\times$. Any $n$-dimensional semisimple representation $\varphi$ of the Weil group $\BC^\times$ is the direct sum of one-dimensional representations. The one-dimensional representations are of the form $\chiup_{( \mu, m)}  = \omega_{\mu} \eta^m  $, with $ ( \mu, m) \in \BC \times \BZ.$
In view of the formulae \cite[(4.6, 4.7)]{Knapp} of $L(s, \chiup_{( \mu, m)})$ and $\epsilon (s, \chiup_{( \mu, m)}, \psi)$ as well as the definition  of  $G_m$ in \eqref{1def: G m (s)}, we have
\begin{equation}
\gamma (s, \chiup_{( \mu, m)}, \psi) = [  r     ]^m \|  r     \|^{ s + \mu - \frac 1 2 } G_m (1-s-\mu).
\end{equation}
Thus $\varphi$ is parametrized by some $(\umu, \um)\in \BC^{n} \times \BZ^n$ and
\begin{equation}\label{4eq: gamma factor of chi, complex case}
\gamma (s,  \varphi , \psi) = [  r     ]^{|\um|}  \|  r     \|^{n \lp s - \frac 1 2\rp + |\umu|} G_{(\umu, \um)} (1-s).
\end{equation}
This parametrization is unique up to permutation, in contrast to the case $\BF = \BR$.

If $\pi$ corresponds to $\varphi $ under the Langlands correspondence over $\BC$, then one has $\gamma ( s, \pi, \psi) = \gamma ( s, \varphi, \psi)$. Moreover, $\pi$ is an irreducible constituent of the principal series representation unitarily induced from the character $ \bigotimes _{l      = 1}^n \chiup_{(\mu_l     , m_l     )}$ of the Borel subgroup. Note that 
\begin{align}\label{4eq: central char, C}
\omega_\pi  (z) = \omega_{|\umu|} (z) \eta^{|\um|} (z) = [z]^{|\um|} \|z\|^{|\umu|}. 
\end{align}

Now let $\chiup = \chiup_{(0, m)} =\eta^m$ in \eqref{4eq: Ichino-Templier}, $m \in \BZ $. Then \eqref{4eq: gamma factor of chi, complex case} and \eqref{4eq: central char, C} imply
\begin{equation}\label{4eq: gamma factor, complex case}
\gamma (1-s,  \pi \otimes \eta^m, \psi) = \omega_\pi (  r     ) \big(   [  r     ]^m \|  r     \|^{ \frac 1 2 - s } \big)^n  G_{(\umu, \um + m \ue^n)} (s).
\end{equation}
By putting  
\begin{equation}\label{4eq: upsilon = w Upsilon = tilde w, complex}
\begin{split}
\upsilon (z) &= \omega_\pi (  r     ) w \big( \|  r     \|^{- \frac n 2} z \big) \|z\| ^{ - \frac {n-1} 2}, \\
\Upsilon (z) &=   \widetilde w \big( (-)^{n-1} [  r     ]^{- n} \|  r     \|^{-  \frac n 2}  z \big) \|z\| ^{ - \frac {n-1} 2},
\end{split}
\end{equation}
the identity (\ref{4eq: Ichino-Templier}) is translated into \eqref{3eq: Hankel transform identity, C}, and \eqref{3eq: Hankel transform, with Bessel kernel, C} can be reformulated as
\begin{equation}\label{4eq: Hankel transform tilde w and w, complex case}
\widetilde w \lp (-)^{n-1}  z \rp = \omega_\pi (  r     )  \|  r     \|^{\frac n 2} \| z\|^{\frac {n-1} 2} \int_{\BC^\times } w \lp u \rp J_{(\umu, \um)} (   r     ^n z u) \| u\|^{1 - \frac {n-1} 2} d^\times u.
\end{equation}



\subsection{Some New Notations} Let $\pi$ be an irreducible admissible representation of $\GL_n(\BF)$. 
For $\BF = \BR$, respectively $\BF = \BC$, if $\pi$ is parametrized by $(\umu, \udelta)$, respectively $(\umu, \um)$, we shall denote    by $J_{\pi}$ the Bessel kernel $J_{(\umu, \udelta)}$, respectively $J_{(\umu, \um)}$.
Thus, \eqref{4eq: Hankel transform tilde w and w, real case} and \eqref{4eq: Hankel transform tilde w and w, complex case} can be uniformly combined into one formula
\begin{equation}\label{4eq: Hankel transform tilde w and w}
\widetilde w \lp (-)^{n-1}  x \rp = \omega_\pi (  r     )  \|  r     \|^{\frac n 2} \| x\|^{\frac {n-1} 2} \int_{\BF^\times } w \lp y \rp J_{\pi} (   r     ^n x y) \| y\|^{1 - \frac {n-1} 2} d^\times y.
\end{equation}

Proposition \ref{3prop: properties of J, R} (1) and \ref{3prop: properties of J, C} (1) are translated into the following lemma. 
\begin{lem}\label{4lem: normalizing the index}
	Let $\pi$ be an irreducible admissible representation of $\GL_n(\BF)$, and let $\chiup $ be a character on $\BFx $. We have
	$
	J_{\chiup \otimes \pi}(x) = \chiup \- (x) J_{\pi} (x).
	$
\end{lem}

\begin{rem}\label{4rem: G/Z}
	Let $\mathrm Z_n$ denote the center of $\GL_n$. In view of Lemma {\rm \ref{4lem: normalizing the index}}, no generality will be lost if one only considers $J_{\pi}$ for irreducible admissible  representations $\pi$ of $\GL_n (\BF)/ \mathrm Z_n(\BR _+)$.
\end{rem}

Let $\varphi$ be the $n$-dimensional semisimple representation of the Weil group of $\BF$ corresponding to $\pi$ under the Langlands correspondence over $\BF$.

When $\BF = \BR$,  the function space $\Ssis^{(\umu, \udelta)} (\BRx)$ depends on the choice of the parameter $(\umu, \udelta)$ attached to $\varphi$, if some discrete series $\varphi_{(\mu, m)}$ occurs in its decomposition. Thus one needs to redefine the function spaces  for Hankel transforms according to the Langlands classification rather than   the above parametrization. For this, let $n_1, n_2 \in \BN$, $(\umu^1, \udelta^1) \in \BC^{n_1} \times (\BZT)^{n_1}$ and $(\umu^2, \um^2) \in \BC^{n_2} \times \BN_+ ^{n_2}$ be such that $n_1 + 2 n_2 = n$ and $\varphi = \bigoplus_{l      = 1}^{n_1} \varphi_{(\mu^1_{l     },\, \delta^1_{l     })} \oplus \bigoplus_{l      = 1}^{n_2} \varphi_{(\mu^2_{l     },\, m^2_{l     })}$.
We define the function space $\Ssis^{\pi} (\BRx) = \Ssis^{\varphi} (\BRx)$ to be
\begin{equation}\label{4eq: Ssis phi, R}
\Ssis^{(- \umu^1,\, \udelta^1)} (\BRx) + \mathscr S_{\mathrm {sis} }^{(- \umu^2,\, \um^2)} (\BRx),
\end{equation}
where $\Ssis^{( \umu ,\, \udelta)} (\BRx)$ is   defined by \eqref{3eq: Ssis (lambda, delta), R}, and
\begin{align}
\mathscr S_{\mathrm {sis} }^{( \umu ,\, \um )} (\BRx) =  \sum_{\delta \in \BZT} \sgn(x)^\delta \mathscr S_{\mathrm {sis} }^{  \umu  - \frac 1 2 \um } (\BR_+) ,
\end{align}
with $\mathscr S_{\mathrm {sis} }^{\ulambda} (\BR_+)$ defined by \eqref{3eq: Ssis(R+)}.  

\begin{lem}\label{4lem: Ssis pi}
	Let $(\umu, \um) \in \BC^{n} \times \BN_+^n$. Suppose $ \upsilon \in \mathscr S_{\mathrm {sis} }^{ (- \umu, \um) } (\BRx)$. Then there exists a unique function $\Upsilon \in \mathscr S_{\mathrm {sis} }^{ ( \umu, \um) } (\BRx)$ satisfying the following two identities,
\begin{equation}\label{3eq: Hankel transform identity, R, 2}
\EM _\delta \Upsilon (s ) = i^{n} G_{(\umu,  \um)} (s) \EM _\delta \upsilon ( 1 - s), \hskip 10 pt \delta \in \BZ/2 \BZ.
\end{equation}
Writing $\EuScript H_{(\umu, \um)}^2  \upsilon  = \Upsilon$, we have
\begin{equation}\label{3eq: Hankel inversion, R, 2}
\EuScript H_{(\umu, \um) }^2 \upsilon (x) = \Upsilon (x), \hskip 10 pt \EuScript H_{(-\umu, \um)}^2  \Upsilon (x) = (-)^{|\um|+n} \upsilon \left( x \right).
\end{equation}
\end{lem}


When $\BF = \BC$, we put
\begin{equation}\label{4eq: Ssis phi, C}
\Ssis^{\pi} (\BCx) = \Ssis^{\varphi} (\BCx) = \Ssis^{(- \umu ,\, -\um)} (\BCx).
\end{equation}

Let $d = [\BF : \BR]$. For each character $\chiup$ on $\BFx/\BR _+$ we define the Mellin transform $\EM_{\chiup}$ of a function $\upsilon \in \Ssis (\BFx)$ by 
\begin{equation} \label{4def: Mellin transform over F}
\EM _{\chiup} \upsilon (s) = \int_{\BF^\times} \upsilon (x) \chiup (x) \|x\|^{\frac 1 d s } d^\times x.
\end{equation} 

\begin{thm}\label{thm: uniform}
	Let $\pi$ be an irreducible admissible representation of $\GL_n(\BF)$. Suppose $\upsilon \in \Ssis^{\pi} (\BFx)$. Then there exists a unique $\widetilde \upsilon \in  \Ssis^{\widetilde \pi} (\BFx)$ satisfying the following identity
	\begin{equation*}
	\EM_{\chiup\-} \widetilde \upsilon (d s) = \gamma (1-s, \pi \otimes \chiup, \psi_1) \EM_{\chiup } \upsilon (d (1-s))
	\end{equation*}
	for all  characters $\chiup$ on $\BFx/\BR _+$. We write $\EH_{\pi} \upsilon  = \widetilde \upsilon$ and call $\widetilde \upsilon$ the normalized Hankel transform of  $\upsilon$ over $\BFx$ associated with $\pi$. Moreover, we have the Hankel inversion formula
	\begin{equation*} 
	\EH_{\pi} \upsilon (x) = \widetilde \upsilon (x), \hskip 10 pt \EH_{\widetilde \pi} \widetilde \upsilon (x) = \omega_{\pi} (-1) \upsilon \lp (-)^n x \rp.
	\end{equation*}
\end{thm}

\begin{proof}
	If $\BF = \BR$, 
	this follows from  Theorem  \ref{3prop: H (lambda, delta)} and Lemma \ref{4lem: Ssis pi}. If $\BF = \BC$, this is simply a translation of Theorem \ref{3prop: H (mu, m)}.
\end{proof}

\begin{rem}\label{rem: Kirillov}
	Now let $n = 2$. In view of \cite[\S 5, 6]{J-L} or \cite[\S 2.5]{Godement}, it may be checked that $\|x\|^{-\frac 1 2} w (x)$ lies in  the space   $\Ssis^{\pi} (\BFx)$ if $w$ is a  function in the   Kirillov model of $\pi$, namely, 
	\begin{align*}
	\left\{  \|x\|^{-\frac 1 2} W   
	\begin{pmatrix}
	x & \\
	& 1
	\end{pmatrix} : W \in \mathscr W(  \pi, \psi_1 )  \right\}\subset \Ssis^{\pi} (\BFx).
	\end{align*}
	Moreover, it follows from the formula of $J_{\pi}$ in {\rm\S \ref{sec: n=1,2, R}} and {\rm \S \ref{sec: rank-two, C}}  that the integral in \eqref{4eq: Hankel transform tilde w and w} is absolutely convergent and \eqref{4eq: Hankel transform tilde w and w} is valid for all $\|x\|^{-\frac 1 2} w (x) \in \Ssis^{\pi} (\BFx)$ if   $\pi$ is either a principal series representation with parameter $|\Re \mu| < \frac 1 2$ or   a discrete series representation.  For the principal series case, see also Proposition {\rm \ref{3prop: properties of J, R} (3)} and {\rm \ref{3prop: properties of J, C} (3)}. In particular,  if $\pi$ is unitary then \eqref{4eq: Hankel transform tilde w and w} is valid for all $ w $ in the  Kirillov model of $\pi$.
\end{rem}

	

\section{Bessel Functions for $\GL_2 (\BF)$}\label{sec: Bessel, GL2(F)}

Let $n=2$ and retain the notations from \S \ref{sec: Hankel, Ichino-Templier} except for the different choice of the Weyl element
$\varw_2 = \begin{pmatrix}
& -1 \\
1&
\end{pmatrix}$, which is more often used  for $\GL_2$ in the literature.

Let $\pi$ be an infinite dimensional irreducible  unitary representation  of $\GL_2( \BF)$\footnote{It is well-known that a representation of $\GL_2 (\BF)$ satisfying these conditions is generic.}. Using (\ref{4eq: w = Whittaker}, \ref{4eq: tilde w for n = 2})  and Remark \ref{rem: Kirillov}, one may rewrite \eqref{4eq: Hankel transform tilde w and w} as follows,
\begin{equation}\label{7eq: Hankel transform, W}
W   
\begin{pmatrix}
& 1\\
- x\- &  
\end{pmatrix} 
= \omega_\pi (  r     )  \|  r     \|  \int_{\BF^\times }  \| x y\|^{  \frac {1} 2} J_{\pi} (   r^2 x y) W 
\begin{pmatrix}
y & \\
& 1
\end{pmatrix} 
d^\times y,
\end{equation}
for all $W \in \mathscr W (\pi, \psi_r) $.
We define 
\begin{equation}\label{7eq: EJ}
\EuScript J_{\pi, \psi_r } (x) = \omega_\pi ( - r x )  \|  r     \| \sqrt {\|x\|} J_{\pi} (  r     ^2 x).
\end{equation}
We call $\EuScript J_{\pi, \psi } (x)$ the {Bessel function associated with $\pi$ and $\psi $}. The formula \eqref{7eq: Hankel transform, W} then reads
\begin{equation}\label{7eq: Hankel transform, W, 2}
\omega_\pi ( -  x ) W   
\begin{pmatrix}
& 1\\
- x\- &  
\end{pmatrix}  
= \int_{\BF^\times }  \EJ_{\pi, \psi } ( x y) \omega_\pi (y)\- W 
\begin{pmatrix}
y & \\
& 1
\end{pmatrix} 
d^\times y.
\end{equation}
Moreover, with the observation
\begin{equation*}
\omega_\pi (-x) W   
\begin{pmatrix}
& 1\\
- x\- &  
\end{pmatrix}   
=
  W \lp
\begin{pmatrix}
x & \\
& 1
\end{pmatrix} 
\varw_2 \rp,
\end{equation*}
\eqref{7eq: Hankel transform, W, 2} turns into
\begin{equation}\label{7eq: Weyl element action}
W \lp  
\begin{pmatrix}
x & \\
&  1
\end{pmatrix} 
\varw_2 \rp 
=  \int_{\BF^\times } \omega_\pi (y)\- \EJ_{\pi, \psi } ( x y)  
W 
\begin{pmatrix}
y & \\
& 1
\end{pmatrix} 
d^\times y.\footnote{In the real case, this identity is given in \cite[Theorem 4.1]{CPS}. 
}
\end{equation}
Thus \eqref{7eq: Weyl element action} indicates that the action of the Weyl element $\varw_2$ on the Kirillov model 
$$\mathscr K (\pi, \psi) = \left\{ w(x) = W   \begin{pmatrix}
x & \\
& 1
\end{pmatrix}   : W \in \mathscr W (\pi, \psi) \right\}$$
is essentially a Hankel transform. From this perspective, the Hankel inversion formula follows from the simple identity $\varw_2^2 = - I_2$. This may be seen from the following lemma.
\begin{lem}\label{7lem: n=2, Bessel} Let $\pi$ be an irreducible admissible representation of $\GL_2(\BF)$. Then we have
	$
	J_{\widetilde \pi} (x) = \omega_{\pi}(x) J_{\pi}(x).
	$
\end{lem}
\begin{proof}
	This follows from some straightforward calculations using Proposition \ref{3prop: properties of J, R} (1) and \ref{3prop: properties of J, C} (1). 
\end{proof}

\begin{rem}
	The representation theoretic viewpoint of Lemma {\rm \ref{7lem: n=2, Bessel}} is the isomorphism $\widetilde \pi \cong \omega\- \otimes \pi$. With this, Lemma {\rm \ref{7lem: n=2, Bessel}} is a direct consequence of Lemma {\rm \ref{4lem: normalizing the index}}.
\end{rem}

Finally, we shall summarize the formulae of the Bessel functions associated with infinite dimensional irreducible {unitary} representations of $\GL_2( \BF) $. First of all, in view of Lemma \ref{4lem: normalizing the index} and Remark \ref{4rem: G/Z}, one may assume without loss of generality that $\pi$ is trivial on $\mathrm Z_2 (\BR_+)$. Moreover, with the simple observation 
\begin{equation}\label{7eq: psi a and psi 1}
\EuScript J_{\pi, \psi_r } (x) = \omega_{\pi} (r        ) \- \EuScript J_{\pi, \psi_1 }  (r^2 x  ),
\end{equation}
it is sufficient to consider the Bessel function $\EuScript J_{\pi }  = \EuScript J_{\pi, \psi_1 } $  associated with $\psi_1$.


\subsection{Bessel Functions for $\GL_2(\BR)$}

Under the Langlands correspondence, we have the following classification of infinite dimensional irreducible unitary representations  of $\GL_2( \BR)/\mathrm Z_2 (\BR_+)$.
\begin{itemize}
	\item[-] (principal series and the limit of discrete series) $\varphi _{(i t,  \epsilon + \delta)} \oplus \varphi _{(- it,  \epsilon)}$, with $t \in \BR$ and $ \epsilon, \delta \in \BZT$,
	\item[-] (complementary series) $\varphi _{(t,  \epsilon)} \oplus \varphi _{(- t,  \epsilon)}$, with $t \in \lp 0, \frac 1 2\rp$ and $ \epsilon \in \BZT$,
	\item[-] (discrete series) $\varphi_{(0, m)}$, with $m \in \BN_+$.
\end{itemize}
Here, in the first case, the corresponding representation is a limit of discrete series  if $t = 0$ and $\delta = 1$ and  a principal series representation if otherwise. We shall write the corresponding representations  as
$\eta^{ \epsilon} \otimes \pi^+ (i t)$ if $\delta = 0$, $\eta^{ \epsilon} \otimes \pi^- (i t)$ if $\delta = 1$,
$\eta^{ \epsilon} \otimes \pi ( t)$
and $\sigma (m)$, respectively. We have
\begin{equation}
\omega_{\pi^+ (i t)} = 1, \ \ \omega_{\pi^- (i t)} = \eta, \ \  \omega_{\pi ( t)} = 1, \ \ \omega_{\sigma (m)} = \eta ^{m+1}.
\end{equation}
Furthermore, we have the equivalences $ \pi^+ (i t) \cong \pi^+ (- i t)$ and $ \pi^- (i t) \cong \eta \otimes \pi^- (- i t)$. 

According to \S \ref{sec: n=1,2, R}, we have the following proposition.

\begin{prop} 
	
	\  
	
	{\rm (1).} Let $t \in \BR $. We have for $x \in \BR _+$
	\begin{equation*} 
	\begin{split}
	\EJ_{\pi^+ (i t) } ( x) &= \frac {\pi i} {\sinh (\pi t)} \sqrt x \lp J_{2it} (4\pi \sqrt x ) - J_{-2it} (4\pi \sqrt x ) \rp,\\
	\EJ_{\pi^+ (i t) } (-x) &= 4 \cosh (\pi t)  \sqrt x K_{2it} (4 \pi  \sqrt x ),
	\end{split}
	\end{equation*}
	where it is understood that when $t = 0$ the right hand side of the first formula should be replaced by its limit, and
	\begin{equation*} 
	\begin{split}
	\EJ_{\pi^- (i t) } ( x) &= - \frac {\pi i} {\cosh (\pi t)}  \sqrt x \lp J_{2it} (4\pi   \sqrt x) + J_{-2it} (4\pi \sqrt x ) \rp,\\
	\EJ_{\pi^- (i t) } ( -x) &= 4 \sinh (\pi t)  \sqrt x K_{2it} (4 \pi \sqrt x ).
	\end{split}
	\end{equation*}
	
	{\rm (2).} Let $t \in \lp 0, \frac 1 2\rp$. We have for $x \in \BR _+$
	\begin{equation*} 
	\begin{split}
	\EJ_{\pi ( t) } ( x) &= - \frac {\pi} {\sin  (\pi t)}   \sqrt x \lp J_{2t} (4\pi   \sqrt x ) - J_{-2t} (4\pi  \sqrt x ) \rp,\\
	\EJ_{\pi ( t) } ( -x) &= 4 \cos (\pi t) \sqrt x K_{2t} (4 \pi  \sqrt x ).
	\end{split}
	\end{equation*}
	
	{\rm (3).} Let $m \in \BN_+$. We have for $x \in \BR _+$
	\begin{equation*} 
	\EJ_{\sigma (m) } ( x) =  2 \pi (-i)^{m+1} \sqrt x J_{m} \big(4 \pi \sqrt x \big),
	\hskip 10 pt
	\EJ_{\sigma (m) } ( -x) = 0.
	\end{equation*}
\end{prop}
\begin{rem}
	$\eta^{ \epsilon} \otimes \pi^+ (i t)$, $\eta^{ \epsilon} \otimes \pi (t)$ and $\sigma (2d-1)$ exhaust all the infinite dimensional irreducible unitary representations of $\PGL_2( \BR) $. Their Bessel functions are also given in \cite[Proposition 6.1]{CPS}.
\end{rem}

\subsection{Bessel Functions for $\GL_2(\BC)$}

Under the Langlands correspondence, we have the following classification of infinite dimensional irreducible unitary representations  of $\GL_2( \BC)/\mathrm Z_2 (\BR_+)$.
\begin{itemize}
	\item[-] (principal series) $\chiup _{(i t, k + d + \delta)} \oplus \chiup _{(- it, k - d)}$, with $t \in \BR $, $k, d \in \BZ $ and $\delta \in \BZT = \{0, 1\}$,
	\item[-] (complementary series) $\chiup _{(t, k)} \oplus \chiup _{(- t, k)}$, with $t \in 
	\lp 0, \frac 1 2\rp$ and $k \in\BZ $.
\end{itemize}
We   write the corresponding   representations 
as
$\eta^{k} \otimes \pi_{d}^+ (i t)$ if $\delta = 0$, $\eta^{k} \otimes \pi_{d}^- (i t)$ if $\delta = 1$ and $\eta^{k} \otimes \pi  ( t)$, respectively.
We have
\begin{equation}
\omega_{ \pi_{d}^+ (i t) } = 1, \ \ \omega_{ \pi_{d}^- (i t) } = \eta, \ \ 
\omega_{ \pi  ( t) } = 1.
\end{equation}
Furthermore, we have the equivalences $ \pi_{d}^+ (i t) \cong \pi_{- d}^+ (- i t) $, $ \pi_{d}^- (i t) \cong \pi_{- d - 1}^- (- i t) $. 

According to Example \ref{3prop: n=2, C}, we have the following proposition.

\begin{prop} \label{7prop: Bessel, C} 
	Recall the definitions {\rm(\ref{7def: J mu m (z), n=2, C}, \ref{7def: H (1, 2) mu m (z), n=2, C})} of $J_{\mu, m} (z)$ and $H^{(1, 2) }_{\mu, m} (z)$ in Example {\rm \ref{3prop: n=2, C}}. 
	
	{\rm (1).} Let $t \in \BR $ and $d \in \BZ$. We have for $z \in \BCx$
	\begin{align*}
	\EJ_{\pi_{d}^+ (i t)} (z) & = - \frac {2 \pi^2 i} {\sinh (2\pi t)}   |z| \lp J_{i t, 2 d} (4 \pi  \sqrt z) -  J_{-i t, - 2 d} (4 \pi  \sqrt z) \rp\\
	& = \pi^2 i  |z|  \lp e^{- 2 \pi   t} H^{(1)}_{i t, 2 d} \lp 4 \pi  \sqrt z \rp - e^{  2 \pi  t} H^{(2)}_{it, 2d} \lp 4 \pi   \sqrt z \rp \rp,\\
	\EJ_{\pi_{d}^- (i t)} (z) & = - \frac {2 \pi^2  i } {\cosh (2\pi t)}   {\sqrt {  |z |    z }} \lp J_{i t, 2 d + 1} (4 \pi  \sqrt z) + J_{-i t, - 2 d - 1} (4 \pi  \sqrt z) \rp \\
	& = - \pi^2  i \sqrt {  |z |    z }  \lp e^{- 2 \pi t} H^{(1)}_{i t, 2d+1} \lp 4 \pi  \sqrt z \rp +  e^{ 2 \pi t} H^{(2)}_{it, 2d+1} \lp 4 \pi  \sqrt z \rp \rp.
	\end{align*}
	
	{\rm (2).} Let $t \in 
	\lp 0, \frac 1 2\rp$. We have for $z \in \BCx$
	\begin{align*}
	\EJ_{\pi  ( t)} (z) & =  \frac {2 \pi^2 } {\sin (2\pi t)}  |z| \lp J_{ t, 0} (4 \pi  \sqrt z) -  J_{- t, 0} (4 \pi   \sqrt z) \rp \\
	& =  \pi^2 i  |z|  \lp e^{ 2 \pi i  t} H^{(1)}_{ t, 0} \lp 4 \pi  \sqrt z \rp - e^{-  2 \pi i t} H^{(2)}_{t, 0} \lp 4 \pi   \sqrt z \rp \rp.
	\end{align*}
\end{prop}

In view of Corollary \ref{6cor: d=1, J mu m}, we have the following integral representations of $\EJ_{\pi } \lp x e^{i\phi} \rp$ except for $\pi =  {\pi  ( t)}$ and $ t  \in \left[ \frac 3 8, \frac 1 2 \rp $.

\begin{prop} \label{7prop: Bessel, C, integral} \  
	
	{\rm (1).} Let $t \in \BR $ and $d \in \BZ$. We have for $x \in \BR_+$ and $\phi \in {\BR/2\pi \BZ}$
	\begin{align*}
	\EJ_{\pi_{d}^+ (i t)} \big(  x e^{i\phi} \big) &  =   4 \pi  x  \big(\hskip -2 pt -e^{i \phi} \big)^d \int_0^\infty y^{4 i t - 1} \big[ y\- +  y e^{ i \phi} \big]^{-2d} J_{2 d} \big( 4 \pi \sqrt x \big|y\-  +  y e^{ i \phi} \big| \big) d y, \\
	\EJ_{\pi_{d}^- (i t)} \big( x e^{i\phi} \big) &  =   4 \pi i   x  \big(\hskip -2 pt -e^{i \phi} \big)^{d+1} \hskip -2 pt \int_0^\infty \hskip -2 pt y^{4 i t - 1} \big[ y\- +  y e^{ i \phi} \big]^{-2d - 1} \hskip -2 pt  J_{2 d + 1} \big( 4 \pi \sqrt x \big|y\-  +  y e^{ i \phi} \big| \big) d y.
	\end{align*}
	
	{\rm (2).} Let $t \in        
	\left( 0, \frac 3 8 \right)$. We have for $x \in \BR_+$ and $\phi \in {\BR/2\pi \BZ}$
	\begin{equation*}
	\EJ_{\pi  ( t)} \big( x e^{i\phi} \big)  =  4 \pi   x   \int_0^\infty y^{4 t - 1}   J_{0} \big( 4 \pi \sqrt x \big|y\-  +  y e^{ i \phi}\big| \big) d y,
	\end{equation*}
	in which the integral converges  absolutely only for $ t \in   \left(0, \frac 1 8 \right) $.
\end{prop}

\begin{rem}
	$\pi^+_d (i t)$ and $\pi  (t)$ exhaust all the infinite dimensional irreducible unitary representations of $\PGL_2( \BC) $. Proposition {\rm \ref{7prop: Bessel, C}} shows that the Bessel function for $\pi^+_d (i t)$ actually coincide with that given in \cite{B-Mo}. More precisely, we have the equality $\EJ_{\pi_{d}^+ (i t)} (z) =  {2 \pi^2 } |z| \EuScript K_{2 i t, - d} (4 \pi  \sqrt z),$ with $\EuScript K_{\nu, p}$ given by \cite[(6.21), (7.21)]{B-Mo}. Furthermore, the integral representation of $\EJ_{\pi_{d}^+ (i t)}$ in Proposition {\rm \ref{7prop: Bessel, C, integral}  (1)} is tantamount to \cite[Theorem 12.1]{B-Mo}. We have similar relations between the Bessel function for $\pi^-_d (i t)$ and that given in \cite{B-Mo2}.
\end{rem}


\backmatter
\def\cprime{$'$} \def\cprime{$'$}

\printindex

\end{document}